\documentclass{article}
\usepackage[margin=1in]{geometry}
\usepackage[utf8]{inputenc} 
\usepackage{natbib} 
\usepackage{usebib} 
\usepackage{amsmath} 
\usepackage{mathtools}
\usepackage{amsthm} 
\usepackage{amssymb} 
\usepackage{graphicx} 
\usepackage[ruled,vlined,linesnumbered,hangingcomment]{algorithm2e} 
\usepackage{makecell} 
\usepackage[flushleft]{threeparttable} 
\usepackage{geometry} 
\usepackage{hyperref} 
\usepackage{color}
\usepackage[dvipsnames]{xcolor} 
\usepackage{bm}
\usepackage{enumerate}
\usepackage[explicit]{titlesec} 
\usepackage{xargs}
\usepackage{xifthen}
\usepackage{xparse}
\usepackage{sidecap}
\usepackage{xspace}
\usepackage{xfrac}
\usepackage{arydshln}
\usepackage{booktabs} 
\usepackage{comment}

\bibinput{refs}

\bibpunct{(}{)}{;}{a}{}{,}

\allowdisplaybreaks

\newboolean{showcomments}
\newboolean{showcolors}
\setboolean{showcomments}{true}
\setboolean{showcolors}{true}

\definecolor{todocolor}{rgb}{1,0,0}
\definecolor{cmtcolor}{rgb}{0.1,0.1,0.9}
\definecolor{cmtqcolor}{HTML}{0959aa}

\def\ldelim{(}
\def\rdelim{)}

\newcommand{\IfNE}[2]{\ifthenelse{\isempty{#1}}{}{#2}}
\newcommand{\IfENE}[3]{\ifthenelse{\isempty{#1}}{#2}{#3}}
\NewDocumentCommand{\Sub}{s m O{} O{}}{\IfBooleanTF{#1}{\IfENE{#2}{_{#3#4}}{_{#3#2#4}}}{\IfNE{#2}{_{#3#2#4}}}}
\NewDocumentCommand{\Sup}{s m O{} O{}}{\IfBooleanTF{#1}{\IfENE{#2}{^{#3#4}}{^{#3#2#4}}}{\IfNE{#2}{^{#3#2#4}}}}

\newcommand{\negphantom}[1]{\ifmmode\settowidth{\dimen0}{$#1$}\else\settowidth{\dimen0}{#1}\fi\hspace*{-\dimen0}}

\protected\def\<{%
  \ifmmode
    \mskip0.5\thinmuskip
  \else
    \ifhmode
      \kern0.08334em
    \fi
  \fi
}

\makeatletter 
\titleformat{\paragraph}[runin]{\normalfont\itshape}{}{}{\theparagraph #1.}
\titleformat{\paragraph}[runin]{\normalfont\itshape}{}{}{#1.} 
\makeatother 

\makeatletter
\renewenvironment{proof}[1]{\par
 \pushQED{\qed}%
 \normalfont \topsep6\p@\@plus6\p@\relax
 \trivlist
 \item\relax
 {\itshape
 \proofname\ifthenelse{\equal{#1}{}}{\@addpunct{.}}{ (Proof of #1)\@addpunct{.}}}\hspace\labelsep\ignorespaces
}{%
 \popQED\endtrivlist\@endpefalse \vskip8\p@\@plus6\p@\relax
}
\makeatother

\makeatletter
\newenvironment{subproof}[1]{\par
 \pushQED{\qed}
 \normalfont \topsep6\p@\@plus6\p@\relax
 \trivlist
 \item\relax
 {\itshape
 \proofname\ifthenelse{\equal{#1}{}}{\@addpunct{.}}{ (Proof of #1)\@addpunct{.}}}\hspace\labelsep\ignorespaces
}{%
 \popQED\endtrivlist\@endpefalse \vskip8\p@\@plus6\p@\relax
}
\makeatother



\newtheoremstyle{claim}
{3pt}
{3pt}
{\itshape}
{}
{\bfseries}
{.}
{.5em}
{}

\newtheoremstyle{property}
{3pt}
{3pt}
{\itshape}
{}
{}
{}
{.5em}
{}

\newtheoremstyle{LemmaRepeat}
{3pt}
{3pt}
{\itshape}
{}
{\bfseries}
{.}
{.5em}
{\thmname{#1}\thmnote{ #3}}

\newtheoremstyle{mystyle}
  {}
  {}
  {\itshape}
  {}
  {\bfseries}
  {.}
  { }
  {\thmname{#1}\thmnumber{ #2}\thmnote{ ({\normalfont#3})}}

\theoremstyle{mystyle}

\newtheorem{theorem}{Theorem}[section]
\newtheorem{corollary}[theorem]{Corollary}
\newtheorem{lemma}[theorem]{Lemma}

\newtheorem{assumption}[theorem]{Assumption}
\newtheorem{fact}[theorem]{Fact}
\newtheorem{remark}{Remark}[section]
\newtheorem*{theorem*}{Restatement of Theorem}
\newtheorem*{lemma*}{Restatement of Lemma}
\newtheorem*{claim*}{Restatement of Claim}
\newtheorem*{corollary*}{Restatement of Corollary}
\newtheorem*{fact*}{Restatement of Fact}
\theoremstyle{definition}
\newtheorem{definition}{Definition}[section]
\theoremstyle{remark}
\theoremstyle{claim} 
\newtheorem{claim}[theorem]{Claim}
\theoremstyle{property} 

\theoremstyle{LemmaRepeat} 

\makeatletter
\renewenvironment{remark}[1][]{\par
 \pushQED{\qed}%
 \normalfont \topsep6\p@\@plus6\p@\relax
 \trivlist
 \item\relax
 \refstepcounter{remark}
 {{\bfseries Remark \theremark}\ifthenelse{\equal{#1}{}}{\@addpunct{.}}{ (#1)\@addpunct{.}}}\hspace\labelsep\ignorespaces
}{%
 \popQED\endtrivlist\@endpefalse \vskip8\p@\@plus6\p@\relax
}
\makeatother

\DontPrintSemicolon 

\SetKwInput{Given}{Given}
\SetKwInput{Solution}{Solution}
\SetKwInput{Argument}{Argument}
\SetKwInput{Require}{Require}
\SetKwInput{Let}{Let}
\SetKwInput{Fix}{Fix}
\SetKwInput{Define}{Define}
\SetKwInput{Design}{Design}
\SetKwInput{Construct}{Construct}

\SetKwIF{If}{ElseIf}{Else}%
  {if}%
  {then:}%
  {else if}{else}{endif}
\SetKwFor{ForEach}%
  {for each}%
  {do:}{endfch}
\SetKwFor{ForAll}%
  {for all}%
  {do:}%
  {endfall}

\SetKw{LetStatement}{let:}
\SetKw{InitStatement}{initialize:}
\SetKw{GivenStatement}{given:}
\SetKw{ConstructStatement}{construct:}
\SetKw{CalculateStatement}{calculate:}
\SetKw{EstimateStatement}{estimate:}
\SetKw{FindStatement}{find:}
\SetKw{SolveStatement}{solve:}
\SetKwProg{Stage}{Stage}{:}{end}

\newcommand{\smallsquare}{%
  \text{\fboxsep=-.2pt\raisebox{0.75mm}{\fbox{\rule{0pt}{0.8ex}\rule{0.8ex}{0pt}}}}%
}

\newcommand{\Enum}[2][]{\IfENE{#1}{(#2)}{\Label{#1}{(#2)}}\xspace}
\newcommand{\TagEqn}{\stepcounter{equation}\tag{\theequation}}

\MakeRobust{\ref}

\makeatletter
\NewDocumentCommand{\Label}{s m m}{%
  \@bsphack
  \csname phantomsection\endcsname 
  \IfBooleanF{#1}{#3}%
  \def\@currentlabel{#3}{#2}
  \@esphack
}
\makeatother

\makeatletter
\newcommand{\subalign}[1]{%
  \vcenter{%
    \Let@ \restore@math@cr \default@tag
    \baselineskip\fontdimen10 \scriptfont\tw@
    \advance\baselineskip\fontdimen12 \scriptfont\tw@
    \lineskip\thr@@\fontdimen8 \scriptfont\thr@@
    \lineskiplimit\lineskip
    \ialign{\hfil$\m@th\scriptstyle##$&$\m@th\scriptstyle{}##$\hfil\crcr
      #1\crcr
    }%
  }%
}
\makeatother

\DeclareMathAlphabet{\mathsfit}{T1}{\sfdefault}{\mddefault}{\sldefault}
\SetMathAlphabet{\mathsfit}{bold}{T1}{\sfdefault}{\bfdefault}{\sldefault}

\newcommand{\R}{\mathbb{R}}

\newcommand{\Z}{\mathbb{Z}}

\newcommand{\ZeroTo}[1]{\{ 0, \dots, #1 \}}
\newcommand{\ZminusO}[1][]{\IfENE{#1}{\Z \setminus \{0\}}{( \Z \setminus \{0\} )^{#1}}}
\newcommand{\RminusO}[1][]{\IfENE{#1}{\R \setminus \{0\}}{( \R \setminus \{0\} )^{#1}}}

\let\OperatorName\operatorname
\newcommand{\OperatorNameWithLimits}[1]{\operatornamewithlimits{#1}}

\newcommand{\Distr}[1]{\mathcal{#1}}
\newcommand{\Set}[1]{\mathcal{#1}}
\renewcommand{\Pr}{\operatornamewithlimits{\mathit{P}}}

\newcommand{\N}{\mathcal{N}}

\newcommand{\T}{T}

\newcommand{\IOp}{\mathbb{I}}
\newcommand{\IpmOp}{\mathbb{I}_{\pm}}

\NewDocumentCommand{\Log}{s O{} D(){}}{{%
  \ifthenelse{\isempty{#3}}{\log}{
    {
      \log\IfNE{#2}{_{#2}}
      \mathchoice
        {\left( #3 \right)} 
        {\IfBooleanTF{#1}{\left( #3 \right)}{(#3)}} 
        {\left( #3 \right)}
        {\left( #3 \right)}
}}}}

\NewDocumentCommand{\Sign}{t. s O{} D(){}}{{%
  \def\dfmt{\left( #4 \right)}
  \def\tfmt{( #4 )}
  \IfBooleanT{#1}{\;}
  \SignOp\Sub{#3}
  \IfNE{#4}{
    \IfBooleanTF{#2}{\tfmt}{
      \mathchoice
        {\dfmt} 
        {\tfmt} 
        {\tfmt} 
        {\tfmt} 
}}}}

\NewDocumentCommand{\SignO}{s D(){}}{{%
  \ifthenelse{\isempty{#2}}{}{
    \IfBooleanTF{#1}{\SignOOp(#2)}{
      \mathchoice
        {\SignOOp\left( #2 \right)} 
        {\SignOOp(#2)} 
        {\SignOOp(#2)} 
        {\SignOOp(#2)} 
}}}}

\NewDocumentCommand{\Supp}{s D(){}}{{%
  \def\dfmt{\left( #2 \right)}
  \def\tfmt{( #2 )}
  \SuppOp%
  \IfBooleanTF{#1}{\tfmt}{
    \mathchoice
      {\tfmt} 
      {\tfmt} 
      {\tfmt} 
      {\tfmt} 
}}}

\newcommand{\E}{\operatornamewithlimits{\mathbb{E}}}

\NewDocumentCommand{\I}{s D(){}}{{%
  \IOp%
  \ifthenelse{\isempty{#2}}{}{
    \IfBooleanTF{#1}{(#2)}{
      \mathchoice
        {\bigl( #2 \bigr)} 
        {(#2)} 
        {(#2)} 
        {(#2)} 
}}}}

\NewDocumentCommand{\Ipm}{s D(){}}{{%
  \IpmOp%
  \ifthenelse{\isempty{#2}}{}{
    \IfBooleanTF{#1}{(#2)}{
      \mathchoice
        {\left( #2 \right)} 
        {(#2)} 
        {(#2)} 
        {(#2)} 
}}}}

\NewDocumentCommand{\BigO}{s t' D(){}}{{%
  \def\dfmt{\IfBooleanTF{#1}{\IfNE{#3}{(#3)}}{\IfNE{#3}{\left( #3 \right)}}}
  \def\tfmt{\IfBooleanTF{#1}{\IfNE{#3}{\left( #3 \right)}}{\IfNE{#3}{(#3)}}}
  \IfBooleanTF{#2}{\BigOOp'}{\BigOOp}
  \IfBooleanTF{#1}{\tfmt}{
    \mathchoice
      {\dfmt} 
      {\tfmt} 
      {\tfmt} 
      {\tfmt} 
}}}

\NewDocumentCommand{\BigOmega}{s t' D(){}}{{%
  \def\dfmt{\IfBooleanTF{#1}{(#3)}{\left( #3 \right)}}
  \def\tfmt{\IfBooleanTF{#1}{\left( #3 \right)}{(#3)}}
  \IfBooleanTF{#2}{\BigOmegaOp'}{\BigOmegaOp}
  \IfBooleanTF{#1}{\tfmt}{
    \mathchoice
      {\dfmt} 
      {\tfmt} 
      {\tfmt} 
      {\tfmt} 
}}}

\NewDocumentCommand{\BigTheta}{s t' D(){}}{{%
  \def\dfmt{\IfBooleanTF{#1}{(#3)}{\left( #3 \right)}}
  \def\tfmt{\IfBooleanTF{#1}{\left( #3 \right)}{(#3)}}
  \IfBooleanTF{#2}{\BigThetaOp'}{\BigThetaOp}
  \IfBooleanTF{#1}{\tfmt}{
    \mathchoice
      {\dfmt} 
      {\tfmt} 
      {\tfmt} 
      {\tfmt} 
}}}

\NewDocumentCommand{\LittleO}{s t' D(){}}{{%
  \IfBooleanTF{#2}{\LittleOOp'}{\LittleOOp}
  \ifthenelse{\isempty{#3}}{}{
    \IfBooleanTF{#1}{(#3)}{
      \mathchoice
        {\left( #3 \right)} 
        {(#3)} 
        {(#3)} 
        {(#3)} 
}}}}

\NewDocumentCommand{\LittleOmega}{s t' D(){}}{{%
  \IfBooleanTF{#2}{\LittleOmegaOp'}{\LittleOmegaOp}
  \ifthenelse{\isempty{#3}}{}{
    \IfBooleanTF{#1}{(#3)}{
      \mathchoice
        {\left( #3 \right)} 
        {(#3)} 
        {(#3)} 
        {(#3)} 
}}}}

\NewDocumentCommand{\Paren}{s O{} m}{{%
  #2%
  \ifthenelse{\isempty{#3}}{}{
    \IfBooleanTF{#1}{(#3)}{
      \mathchoice
        {\left( #3 \right)} 
        {(#3)} 
        {(#3)} 
        {(#3)} 
}}}}

\NewDocumentCommand{\BigOOp}{t'}{\IfBooleanTF{#1}{\tilde{O}}{O}}
\NewDocumentCommand{\BigThetaOp}{t'}{\IfBooleanTF{#1}{\tilde{\Theta}}{\Theta}}
\NewDocumentCommand{\BigOmegaOp}{t'}{\IfBooleanTF{#1}{\tilde{\Omega}}{\Omega}}
\newcommand{\LittleOOp}{o}
\newcommand{\LittleOmegaOp}{\omega}

\newcommand{\SignOp}{\OperatorName{sign}}
\newcommand{\SignOOp}{\OperatorName{sign}_{0}}
\newcommand{\SuppOp}{\OperatorName{supp}}

\NewDocumentCommand{\Ker}{s O{}}{\IfBooleanTF{#1}{\widehat{\OperatorName{ker}}\Sub{#2}}{\OperatorName{ker}}}
\newcommand{\Span}{\OperatorName{span}}
\newcommand{\Diag}{\OperatorName{diag}}

\newcommand{\Dim}{\OperatorName{dim}}

\NewDocumentCommand{\Proj}{s m D(){}}{{%
  \def\dfmt{\IfNE{#3}{\left( #3 \right)}}
  \def\tfmt{\IfNE{#3}{( #3 )}}
  \mathsf{proj}_{#2}
  \IfBooleanTF{#1}{\tfmt}{
    \mathchoice
      {\dfmt} 
      {\tfmt} 
      {\tfmt} 
      {\tfmt} 
}}}

\newcommand{\lnorm}[1]{\ell_{#1}}

\RenewDocumentCommand{\Vec}{s +m}{\IfBooleanTF{#1}{#2}{\VecFmt{#2}}}
\newcommand{\VecFmt}[1]{\mathbf{#1}}
\newcommand{\MatFmt}[1]{\mathbf{#1}}
\NewDocumentCommand{\SVec}{s +m O{} +m}{\IfBooleanTF{#1}{#2^{\IfNE{#3}{(#3);}#4}}{\mathbf{#2}^{\IfNE{#3}{(#3);}#4}}}
\NewDocumentCommand{\BVec}{s O{} +m O{}}{\IfNE{#4}{\,(}\IfBooleanTF{#1}{(\mathbf{1}^{#3})\S{#2}}{\mathbf{1}^{\IfNE{#2}{(#2);}#3}}\IfNE{#4}{)_{#4}}}
\NewDocumentCommand{\SMat}{s +m O{} +m}{\IfBooleanTF{#1}{#2^{\IfNE{#3}{(#3);}#4}}{\bm{\mathrm{#2}}^{\IfNE{#3}{(#3);}#4}}}
\NewDocumentCommand{\PMat}{s +m O{} +m}{\IfBooleanTF{#1}{#2_{\IfNE{#3}{(#3);}(#4)}}{\bm{\mathrm{#2}}_{\IfNE{#3}{(#3);}(#4)}}}

\NewDocumentCommand{\Mat}{s +m}{\IfBooleanTF{#1}{#2}{\MatFmt{#2}}}
\NewDocumentCommand{\MatRow}{s +m +m O{}}{\IfBooleanTF{#1}{#2}{\bm{\mathrm{#2}}}\IfENE{#4}{_{#3}}{_{#3,#4}}}
\NewDocumentCommand{\MatCol}{s +m +m O{}}{\IfBooleanTF{#1}{#2}{\bm{\mathrm{\underline{#2}}}}_{#3}\IfNE{#4}{^{#4}}}
\NewDocumentCommand{\Submat}{s +m +m}{{\bm{\mathrm{#2}}}^{(#3)}}
\NewDocumentCommand{\SubmatRow}{s +m +m +m O{}}{\IfBooleanTF{#1}{#2}{\bm{\mathrm{#2}}}\IfENE{#5}{_{#4}}{_{#4,#5}}^{(#3)}}

\newcommand{\Var}{\OperatorName{Var}}

\newcommand{\Binomial}{\OperatorNameWithLimits{Binomial}}

\newcommand{\iid}[1][~]{i.i.d.#1}
\NewDocumentCommand{\MGF}{s}{\IfBooleanTF{#1}{moment generating function}{mgf}\xspace}
\NewDocumentCommand{\MGFs}{s}{\IfBooleanTF{#1}{moment generating functions}{mgfs}\xspace}

\newcommand{\subgaussian}{subgaussian\xspace}

\newcommand{\SphereSym}{S}
\newcommand{\Sphere}[2][]{\ifthenelse{\isempty{#1}}{\SphereSym^{#2-1}}{\mathcal{S}^{#2}}}
\newcommand{\SparseRealSubspace}[2]{{\R^{#2} \cap \Sigma_{#1}^{#2}}}
\newcommand{\SparseSphereSubspace}[2]{{\SphereSym^{#2-1} \cap \SparseSubspace{#1}{#2}}}

\newcommand{\SparseSubspace}[2]{\Sigma_{#1}^{#2}}

\newcommand{\hyphen}{\text{-}}
\NewDocumentCommand{\Th}{t'}{\IfBooleanTF{#1}{^{\prime\mathrm{th}}}{\!\,^{\mathrm{th}}}}
\newcommand{\defeq}{\triangleq}

\NewDocumentCommand{\dCmt}{s O{\qquad} m}{#2\blacktriangleright \IfBooleanTF{#1}{\text{#3}}{\IfNE{#3}{\text{#3} \ }}}
\NewDocumentCommand{\dCmtx}{s O{} m}{#2\dCmtIndent \IfBooleanTF{#1}{\text{#3}}{\IfNE{#3}{\text{#3} \ }}}
\newcommand{\dCmtIndent}{\phantom{\qquad \blacktriangleright \ \ }}

\newcommand{\cIf}{\text{if}\ }

\newcommand{\cWP}{\mathrm{with\ probability}\ }

\newcommand{\SubjectTo}{\text{subject to}}

\newcommand{\THEOREM}[1][~]{Theorem#1\ignorespaces}
\newcommand{\COROLLARY}[1][~]{Corollary#1\ignorespaces}
\newcommand{\LEMMA}[1][~]{Lemma#1\ignorespaces}

\newcommand{\FACT}[1][~]{Fact#1\ignorespaces}

\newcommand{\CLAIM}[1][~]{Claim#1\ignorespaces}

\newcommand{\EQUATION}[1][~]{Equation#1\ignorespaces}
\newcommand{\DEFINITION}[1][~]{Definition#1\ignorespaces}
\newcommand{\FIGURE}[1][~]{Figure#1\ignorespaces}
\newcommand{\ALGORITHM}[1][~]{Algorithm#1\ignorespaces}

\newcommand{\ASSUMPTION}[1][~]{Assumption#1\ignorespaces}

\newcommand{\STEP}[1][~]{Step#1\ignorespaces}
\newcommand{\CASE}[1][~]{Case#1\ignorespaces}
\newcommand{\CONDITION}[1][~]{Condition#1\ignorespaces}

\newcommand{\TASK}[1][~]{Task#1\ignorespaces}

\newcommand{\COROLLARIES}[1][~]{Corollaries#1\ignorespaces}
\newcommand{\LEMMAS}[1][~]{Lemmas#1\ignorespaces}

\newcommand{\EQUATIONS}[1][~]{Equations#1\ignorespaces}
\newcommand{\DEFINITIONS}[1][~]{Definitions#1\ignorespaces}

\newcommand{\SECTION}[1][~]{Section#1\ignorespaces}
\newcommand{\SECTIONS}[1][~]{Sections#1\ignorespaces}
\newcommand{\APPENDIX}[1][~]{Appendix#1\ignorespaces}
\newcommand{\APPENDICES}[1][~]{Appendices#1\ignorespaces}
\newcommand{\STEPS}[1][~]{Steps#1\ignorespaces}
\newcommand{\CASES}[1][~]{Cases#1\ignorespaces}

\NewDocumentCommand{\dueto}{s O{~}}{\IfBooleanTF{#1}{Due~to}{due~to}#2\ignorespaces}
\NewDocumentCommand{\see}{s}{\IfBooleanTF{#1}{See}{see},~\ignorespaces}
\NewDocumentCommand{\seealso}{s}{\IfBooleanTF{#1}{See~also}{see~also},~\ignorespaces}
\NewDocumentCommand{\seeeg}{s}{\IfBooleanTF{#1}{See}{see},~e.g.,~\ignorespaces}

\NewDocumentCommand{\Dueto}{s O{~}}{\IfBooleanTF{#1}{due~to}{Due~to}#2\ignorespaces}
\NewDocumentCommand{\See}{s}{\IfBooleanTF{#1}{see}{See},~\ignorespaces}
\NewDocumentCommand{\Seealso}{s}{\IfBooleanTF{#1}{see~also}{See~also},~\ignorespaces}
\NewDocumentCommand{\Seeeg}{s}{\IfBooleanTF{#1}{see}{See},~e.g.,~\ignorespaces}

\newcommand{\tab}{\ensuremath{~~~~}}
\newcommand{\Tab}[1][1]{%
    \ifthenelse{#1>0}{\tab}{}%
    \ifthenelse{#1>1}{\tab}{}%
    \ifthenelse{#1>2}{\tab}{}%
    \ifthenelse{#1>3}{\tab}{}%
    \ifthenelse{#1>4}{\tab}{}%
    \ifthenelse{#1>5}{\tab}{}%
    \ifthenelse{#1>6}{\tab}{}%
    \ifthenelse{#1>7}{\tab}{}%
    \ifthenelse{#1>8}{\tab}{}%
    \ifthenelse{#1>9}{\tab}{}%
    \ifthenelse{#1>10}{\tab}{}%
    \ifthenelse{#1>11}{\tab}{}%
    \ifthenelse{#1>12}{\tab}{}%
    \ifthenelse{#1>13}{\tab}{}%
    \ifthenelse{#1>14}{\tab}{}%
    \ifthenelse{#1>15}{\tab}{}%
    \ifthenelse{#1>16}{\tab}{}%
    \ifthenelse{#1>17}{\tab}{}%
    \ifthenelse{#1>18}{\tab}{}%
    \ifthenelse{#1>19}{\tab}{}%
    \ifthenelse{#1>20}{\tab}{}%
}

\newcommand{\AlignSp}{\tab}
\newcommand{\AlignIndent}{\tab}

\newcommand{\hfrac}[2]{#1/#2}

\NewDocumentCommand{\SBinom}{s +m +m}{\IfBooleanTF{#1}{\binom{#2}{#3}}{\binom{[#2]}{#3}}}

\newcommand{\+}{\phantom{+}}

\newcommand{\Ell}{\ell}

\newcommand{\Id}[1]{\Mat{I}_{#1}}

\newcommand{\DistSOp}[1]{d_{\Sphere{#1}}}

\NewDocumentCommand{\Dist}{s +m +m}{d\IfBooleanTF{#1}{( #2, #3 )}{\left( #2, #3 \right)}}
\NewDocumentCommand{\DistH}{s O{} +m +m}{{%
 \def\dfmt{\bigl( #3, #4 \bigr)}
 \def\tfmt{( #3, #4 )}
 d_{H}#2
 \IfBooleanTF{#1}{\dfmt}{
    \mathchoice
     {\tfmt} 
     {\tfmt} 
     {\tfmt} 
     {\tfmt} 
}}}
\NewDocumentCommand{\DistS}{s O{n} O{} +m +m}{{%
 \def\dfmt{\bigl( #4, #5 \bigr)}
 \def\tfmt{( #4, #5 )}
 \DistSOp{#2}#3
 \IfBooleanTF{#1}{\dfmt}{
    \mathchoice
     {\tfmt} 
     {\tfmt} 
     {\tfmt} 
     {\tfmt} 
}}}

\NewDocumentCommand{\LatestUpdate}{s +m}{\IfBooleanTF{#1}{(Latest update on: #2)}{(\emph{Latest update on: #2}.)}}

\newcommand{\thsp}{\hspace{0.2pt}}

\NewDocumentCommand{\Mid}{s m}{%
  \def\dfmt{\,#2\,}
  \def\tfmt{\,#2\,}
  \def\sfmt{\thsp#2\thsp}
  \def\ssfmt{\thsp#2\thsp}
  \IfBooleanTF{#1}{\tfmt}{%
    \mathchoice
      {\dfmt} 
      {\tfmt} 
      {\sfmt} 
      {\ssfmt} 
}}

\newcommand{\onebitcs}{1-bit compressed sensing\xspace}

\newcommand{\BIHT}{BIHT\xspace}
\NewDocumentCommand{\NBIHT}{s}{\IfBooleanTF{#1}{(normalized)}{normalized} \BIHT}

\NewDocumentCommand{\VIx}{+m O{}}{\Sub{#1\IfNE{#2}{,#2}}}

\newcommand{\IIx}[2][]{^{(#2)#1}}
\NewDocumentCommand{\RVIx}{+m O{}}{\Sub{#1\IfNE{#2}{,#2}}}

\newcommand{\Net}[2][]{\mathcal{C}_{#2\IfNE{#1}{;#1}}}
\newcommand{\BallNet}[2][]{\mathcal{D}_{#2\IfNE{#1}{;#1}}}
\newcommand{\Ball}[2][]{\mathcal{B}_{#2}\IfNE{#1}{^{(#1)}}}
\newcommand{\BallX}[2][]{\mathcal{B}'_{#2}\IfNE{#1}{^{(#1)}}}
\newcommand{\BallAngular}[2][]{\mathcal{B}_{\theta\leq #2}\IfNE{#1}{^{(#1)}}}
\NewDocumentCommand{\BallSparseSphere}{O{} +m D(){}}{\mathcal{B}_{#2}\IfNE{#1}{^{(#1)}}( #3 ) \cap \Sphere{n} \cap \SparseSubspace{k}{n}}

\newcommand{\ThresholdOp}{T\!}
\NewDocumentCommand{\Threshold}{s +m D(){}}{{%
  \def\dfmt{\left( #3 \right)}
  \def\tfmt{( #3 )}
  \ThresholdOp\Sub{#2}%
  \ifthenelse{\isempty{#3}}{}{
  \IfBooleanTF{#1}{\tfmt}{
    \mathchoice
      {\dfmt} 
      {\tfmt} 
      {\tfmt} 
      {\tfmt} 
}}}}

\NewDocumentCommand{\ThresholdSet}{s t' O{} +m D(){}}{{%
  \def\dfmt{\left( #5 \right)}
  \def\tfmt{( #5 )}
  \ThresholdOp_{\IfBooleanF{#2}{#3}#4}%
  \ifthenelse{\isempty{#5}}{}{%
  \IfBooleanTF{#1}{\tfmt}{
    \mathchoice
      {\dfmt} 
      {\tfmt} 
      {\tfmt} 
      {\tfmt} 
}}}}

\newcommand{\realvalued}{real-valued\xspace}
\newcommand{\pointwise}{point-wise\xspace}
\newcommand{\coordinatewise}{coordinate-wise\xspace}
\newcommand{\entrywise}{entry-wise\xspace}
\newcommand{\orderwise}{order-wise\xspace}

\newcommand{\eg}{e.g.,\xspace}
\newcommand{\ie}{i.e.,\xspace}
\newcommand{\WLOG}{without loss of generality\xspace}

\newcommand{\errorrate}{error-rate\xspace}
\newcommand{\topk}[1][k]{top-\(  #1  \)\xspace}
\newcommand{\polytime}{polynomial time\xspace}
\newcommand{\informationtheoretical}{information theoretical\xspace}
\newcommand{\informationtheoretically}{information theoretically\xspace}
\newcommand{\nonseparable}{non-separable\xspace}

\newcommand{\SACap}[2][]{\mathsf{A}\Sub{#2\IfNE{#1}{,#1}}}

\newcommand{\RegularizedIncompleteBetaFunctionOp}[1]{I_{#1}}
\NewDocumentCommand{\RegularizedIncompleteBetaFunction}{s m m m}{{%
  \def\dfmt{\left( #3, #4 \right)}
  \def\tfmt{( #3, #4 )}
  \RegularizedIncompleteBetaFunctionOp{#2}
  \IfBooleanTF{#1}{\tfmt}{
    \mathchoice
      {\dfmt} 
      {\tfmt} 
      {\tfmt} 
      {\tfmt} 
}}}

\NewDocumentCommand{\RHS}{s}{\IfBooleanTF{#1}{RHS}{right-hand-side}\xspace}
\NewDocumentCommand{\LHS}{s}{\IfBooleanTF{#1}{LHS}{left-hand-side}\xspace}

\renewcommand{\k}{k}
\newcommand{\kO}{k_{0}}
\newcommand{\kOExpr}{\min \{ 2\k+\max_{\JCoords \in \JS} | \JCoords |, \n \}}

\newcommand{\kX}{k'}
\newcommand{\kOX}{k''_{0}}
\newcommand{\kOXExpr}{\min \{ \k+\max_{\JCoordsX \in \JSX} | \JCoordsX |, \n \}}
\newcommand{\kOXX}{k'_{0}}
\newcommand{\kOXXExpr}{\min \{ \max \{ 2\nO, \max_{\JCoordsXX' \in \JSXX} | \JCoordsXX' | \}, \n \}}
\newcommand{\kOXXExprX}{\min \{ \max \{ {| \Supp( \thetaX ) \cup \Supp( \thetaXX ) |}, \allowbreak | \JCoordsXX | \}, \n \}}

\newcommand{\m}{n}

\newcommand{\mOneS}{\m_{1}}
\newcommand{\mTwoS}{\m_{2}}
\newcommand{\mThreeS}{\m_{3}}
\newcommand{\mFourS}{\m_{4}}
\newcommand{\mFiveS}{\m_{5}}
\newcommand{\n}{d}
\newcommand{\nO}{\k}
\newcommand{\epsilonX}{\epsilon}

\newcommand{\Iter}{t}
\newcommand{\IterX}{t'}
\newcommand{\InductiveClaim}[1]{C(#1)}

\newcommand{\rhoX}{\rho}
\newcommand{\rhoLDX}{\rho_{1}}
\newcommand{\rhoLDXX}{\rho_{3}}
\newcommand{\rhoSD}{\rho_{2}}

\newcommand{\betaX}{\beta}

\newcommand{\betaXX}{r}
\newcommand{\gammaX}{\gamma}
\newcommand{\GAMMAX}{\gamma}
\newcommand{\zetaX}{\zeta}
\newcommand{\alphaX}{\alpha}
\newcommand{\alphaO}{\alphaX_{0}}
\newcommand{\alphaOExpr}[1][\alphaOthreshold]{\max \{ \alphaX, {#1} \}}

\newcommand{\deltaX}{\delta}
\newcommand{\tauX}{\tau}
\newcommand{\tauXEXPR}[1][]{\frac{\nuX \IfNE{#1}{( #1 )}}{\ConstdSD \log \left( \frac{2e}{\nuX \IfNE{#1}{( #1 )}} \right)}}
\newcommand{\nuX}{\eta}
\newcommand{\nuXEXPR}[1][]{\frac{\GAMMAX \Constb \deltaX}{\sqrt{\frac{2}{\pi} \log \left( \frac{4e}{\nuX \IfNE{#1}{( #1 )}} \right)}}}
\newcommand{\nuXEXPRLROne}{\displaystyle \frac{\bigl( 1 - \frac{\cO^{2}}{6} \bigr) \Constb \betaX \deltaX}{\sqrt{2\pi \log \left( \frac{4e}{\nuX} \right)}}}
\newcommand{\nuXEXPRLRTwo}{\displaystyle \frac{\bO \Constb \deltaX}{\sqrt{\log \left( \frac{4e}{\nuX} \right)}}}
\newcommand{\nuXEXPRLRThree}{\displaystyle \frac{\bO \Constb \deltaX}{\sqrt{\log \left( \frac{4e}{\nuX} \right)}}}

\newcommand{\nuXEXPRPR}[1][]{\frac{\Constb \deltaX}{\sqrt{\frac{\pi}{2} (1 + \frac{1}{\betaX^{2}} )\log \left( \frac{4e}{\nuX \IfNE{#1}{( #1 )}} \right)}}}

\newcommand{\uX}{u}
\newcommand{\vX}{v}

\newcommand{\ConstbLD}{c}
\newcommand{\ConstbSD}{c'}
\newcommand{\ConstdSD}{\ConstCFive}
\newcommand{\ConstC}{a_{1}}
\newcommand{\Constc}{a_{3}}
\newcommand{\Constb}{a_{2}}
\newcommand{\cO}{b_{1}}
\newcommand{\bO}{b_{0}}
\newcommand{\ConstCOne}{c_{1}}
\newcommand{\ConstCTwo}{c_{1}}
\newcommand{\ConstCThree}{c_{2}}
\newcommand{\ConstCFour}{c_{3}}
\newcommand{\ConstCFive}{c_{4}}

\newcommand{\ConstCTwoValue}{\frac{192}{\ConstbLD}}
\newcommand{\ConstCThreeValue}{\sqrt{\frac{800}{\pi}} \ConstC}
\newcommand{\ConstCFourValue}{64}
\newcommand{\ConstCFiveValue}{256}
\newcommand{\ConstA}{a}
\newcommand{\ConstB}{b}
\newcommand{\ConstAValue}{24}
\newcommand{\ConstBValue}{3}

\newcommand{\ParamSpace}{\Theta}
\newcommand{\ParamCover}{\Set{C}}
\newcommand{\ParamCoverJ}[1][\JCoords]{\Set{C}_{#1}}
\newcommand{\ParamCoverX}{\Set{\tilde{C}}}
\newcommand{\BallXCover}[1]{\Set{D}( #1 )}
\newcommand{\THETA}{\theta}
\NewDocumentCommand{\thetaStar}{s}{\IfBooleanTF{#1}{\THETA}{\boldsymbol{\THETA}}^{\ast}}
\NewDocumentCommand{\thetaHat}{s O{}}{\IfBooleanTF{#1}{\hat{\THETA}}{\boldsymbol{\hat{\THETA}}}\IfNE{#2}{\IIx{#2}}}
\NewDocumentCommand{\thetaHatX}{s O{}}{\IfBooleanTF{#1}{\tilde{\THETA}}{\boldsymbol{\tilde{\THETA}}}\IfNE{#2}{\IIx{#2}}}
\NewDocumentCommand{\thetaX}{s}{\IfBooleanTF{#1}{\THETA}{\boldsymbol{\THETA}}}
\NewDocumentCommand{\thetaXX}{s}{\IfBooleanTF{#1}{\hat{\THETA}}{\boldsymbol{\hat{\THETA}}}}
\NewDocumentCommand{\thetaXXX}{s}{\IfBooleanTF{#1}{w_{\thetaXX}}{\Vec{w}_{\thetaXX}}}

\NewDocumentCommand{\CovV}{s}{\IfBooleanTF{#1}{x}{\VecFmt{x}}}
\NewDocumentCommand{\CovVX}{s}{\IfBooleanTF{#1}{x}{\VecFmt{\tilde{x}}}}
\NewDocumentCommand{\CovM}{s}{\IfBooleanTF{#1}{X}{\MatFmt{X}}}
\NewDocumentCommand{\CovMX}{s}{\IfBooleanTF{#1}{\tilde{X}}{\MatFmt{\tilde{X}}}}
\NewDocumentCommand{\RespV}{s}{\IfBooleanTF{#1}{y}{\VecFmt{y}}}

\newcommand{\gFn}[1][]{\bar{h}\Sub{#1}}
\newcommand{\hFn}[1][]{h\Sub{#1}}
\NewDocumentCommand{\gXFn}{s O{} m m}{\tilde{h}'\Sub{#2}\IfBooleanF{#1}{(#4;#3)}}
\NewDocumentCommand{\hXFn}{s O{} m m}{\tilde{h}\Sub{#2}\IfBooleanF{#1}{(#4;#3)}}
\newcommand{\hfFn}[1][]{h_{\fFn \IfNE{#1}{;#1}}}
\newcommand{\gfFn}[1][]{\bar{h}_{\fFn \IfNE{#1}{;#1}}}
\newcommand{\fFn}[1][]{f\Sub{#1}}
\newcommand{\pFn}{p}
\newcommand{\pbetaFn}[2][\betaX]{\pFn \IfNE{#2}{( #2 )}}
\newcommand{\betaXParam}{}

\newcommand{\linkFn}{g}
\newcommand{\intlinkFn}{G}

\newcommand{\objectiveFn}{\mathcal{J}}
\newcommand{\pFnL}{\pFn_{\operatorname{logistic}}}
\newcommand{\pFnP}{\pFn_{\operatorname{probit}}}
\newcommand{\alphaXL}{\alphaX_{\operatorname{logistic}}}
\newcommand{\alphaXP}{\alphaX_{\operatorname{probit}}}
\newcommand{\gammaXL}{\gammaX_{\operatorname{logistic}}}
\newcommand{\gammaXP}{\gammaX_{\operatorname{probit}}}
\newcommand{\integrandL}{f_{\operatorname{logistic}}}
\newcommand{\integrandP}{f_{\operatorname{probit}}}

\newcommand{\ICoords}{J''}
\newcommand{\JCoords}{J}
\newcommand{\JCoordsX}{J''}
\newcommand{\JCoordsXX}{J'}
\newcommand{\JCoordsXXX}{J'''}
\newcommand{\JCoordsu}{J'}
\newcommand{\JCoordsv}{J''}
\newcommand{\JS}{\Set{J}}
\newcommand{\JSX}{\Set{J}''}
\newcommand{\JSXX}{\Set{J}'}
\newcommand{\JSXXX}{\Set{J}'''}
\newcommand{\iIx}{i}
\newcommand{\jIx}{j}
\newcommand{\lIx}{\ell}

\newcommand{\wX}{w}
\newcommand{\zX}{z}
\newcommand{\zO}{z_{0}}
\newcommand{\zXX}{z'}
\newcommand{\rX}[1][]{r\Sub{#1}}
\newcommand{\yX}{y}

\newcommand{\sX}{s}
\newcommand{\sXX}{s}
\newcommand{\sXXX}{s'}
\newcommand{\tX}{t}

\newcommand{\tXX}{t'}
\newcommand{\tXXX}{t''}

\NewDocumentCommand{\pExprx}{s +m}{\IfBooleanTF{#1}{1 - \pFn( #2 ) + \pFn( -(#2) )}{1 - \pFn( #2 ) + \pFn( -#2 )}}
\NewDocumentCommand{\pExprxx}{s +m}{\IfBooleanTF{#1}{1 - \pFn( #2 ) + \pFn( -#2 )}{1 - \pFn( \betaX #2 ) + \pFn( -\betaX #2 )}}
\let\pExpr\pExprx
\newcommand{\pExprFn}{\nu}

\newcommand{\ZRV}{Z}
\newcommand{\ZRVX}{Z}
\newcommand{\RRV}{R}
\newcommand{\LRV}{L}
\newcommand{\LRVO}{L_{0}}
\newcommand{\lX}{\ell}
\newcommand{\EllX}{\ell}
\newcommand{\Zi}[1][\iIx]{\ZRV\RVIx{#1}}

\NewDocumentCommand{\Zij}{O{\iIx} O{\jIx}}{Z\RVIx{#1,#2}}
\NewDocumentCommand{\Wij}{O{\iIx} O{\jIx}}{W\RVIx{#1,#2}}
\NewDocumentCommand{\Yi}{O{\iIx}}{Y\RVIx{#1}}
\NewDocumentCommand{\Uij}{O{\iIx} O{\jIx}}{U\RVIx{#1,#2}}
\NewDocumentCommand{\Uj}{O{\jIx}}{U\RVIx{#1}}
\NewDocumentCommand{\Ui}{O{\iIx}}{U\RVIx{#1}}
\NewDocumentCommand{\URV}{}{U}
\NewDocumentCommand{\Vi}{O{\iIx}}{V\RVIx{#1}}
\NewDocumentCommand{\VRV}{}{V}
\newcommand{\Ri}[1][\iIx]{\RRV\RVIx{#1}}
\newcommand{\Wi}[1][\iIx]{W\RVIx{#1}}
\newcommand{\WRV}{W}
\newcommand{\muX}[1][]{\mu\Sub{#1}}

\newcommand{\LRVX}[1]{L(#1)}

\newcommand{\LRVXX}[2]{L(#1,#2)}
\newcommand{\ARV}[3]{A\Sub{#1}(#2,#3,\JCoordsXX)}

\newcommand{\XRV}{X}
\newcommand{\dX}{m}
\newcommand{\ej}[1][\jIx]{\Vec{e}\Sub{#1}}

\newcommand{\WS}[1]{\Set{W}( #1 )}
\newcommand{\wS}{\Vec{w}}

\newcommand{\pdf}[1]{f\Sub{#1}}
\newcommand{\mgf}[1]{\psi\Sub{#1}}

\newcommand{\fFnX}{f_{1}}
\newcommand{\fFnXX}{f_{2}}

\newcommand{\uV}{u}
\newcommand{\vV}{v}
\newcommand{\wV}{w}

\newcommand{\ux}{u}
\newcommand{\vx}{v}
\newcommand{\wx}{w}

\newcommand{\tx}{t}
\newcommand{\fx}{f_{1}}
\newcommand{\fxx}{f_{2}}

\newcommand{\zx}{z}

\newcommand{\ExistsST}[2]{\exists #1 \ \ #2}
\newcommand{\Forall}[2]{\forall #1 \ \ #2}
\newcommand{\VEE}{\text{ or }}

\newcommand{\ADIST}{\arccos( \langle \thetaStar, \thetaX \rangle )}
\newcommand{\EDIST}{\| \thetaStar - \thetaX \|_{2}}
\newcommand{\ADISTX}{\arccos( \langle \thetaX, \thetaXX \rangle )}

\newcommand{\JSSIZES}{\sum_{\ell=0}^{\k} \binom{\n}{\ell}}
\newcommand{\JSPCSIZES}{\sum_{\ell=0}^{\k} \binom{\n}{\ell} \sum_{\ell'=0}^{\k} \binom{\n}{\ell'} \left( \frac{\ConstB}{\tauX} \right)^{\ell'}}
\newcommand{\PCSIZES}{\sum_{\ell=0}^{\k} \binom{\n}{\ell} \left( \frac{3}{\tauX} \right)^{\ell}}
\newcommand{\JSPCSIZED}{\left( \frac{6}{\tauX} + 2 \right)^{\n}}

\newcommand{\MaxBallXCoverSize}{Q}
\newcommand{\qX}{q}
\newcommand{\qXExpr}{\sum_{\ell=0}^{\nuX \m} \binom{\m}{\ell}}

\let\BallXX\BallX

\newcommand{\SX}{\Set{S}}
\newcommand{\SXX}{\Set{T}}
\newcommand{\sSX}{s}
\newcommand{\sSXX}{t}
\newcommand{\dSX}{d_{\SX}}

\newcommand{\dDim}{m}

\newcommand{\RAICs}{RAICs\xspace}

\newcommand{\DENOM}{\| \E[ \thetaX + \hfFn( \thetaStar, \thetaX ) ] \|_{2}}

\NewDocumentCommand{\mEXPR}{O{d} O{}}{%
  \max \!\!\begin{array}[t]{l} \displaystyle \Biggl\{
    \frac{\ConstCTwo \alphaO}{\GAMMAX^{2} \deltaX^{2}}
    \log \left( \frac{\ConstA}{\rhoX} \ifthenelse{\equal{#1}{s}}{\JSPCSIZES}{\ifthenelse{\equal{#1}{d}}{\JSPCSIZED}{| \JS | | \ParamCover |}} \right)
    ,\\ \displaystyle \phantom{\Bigg\{}
    \frac{\ConstCThree}{\GAMMAX \deltaX {\sqrt{\log \left( \frac{4e}{\nuX} \right)}}}
    \log \left( \frac{\ConstA}{\rhoX} \ifthenelse{\equal{#1}{s}}{\JSPCSIZES}{\ifthenelse{\equal{#1}{d}}{\JSPCSIZED}{| \JS | | \ParamCover |}} \right)
    ,\\ \displaystyle \phantom{\Bigg\{}
    \frac{\ConstCFour}{\nuX} \log \left( \frac{\ConstA}{\rhoX} \right)
    ,
    \frac{\ConstCFive \ifthenelse{\equal{#1}{s}}{\k}{\n}}{\nuX} \log \left( \frac{1}{\nuX} \right)
  \Biggr\} #2 \end{array}
}

\newcommand{\tauXInvOEXPR}[2][]{\frac{\IfENE{#1}{1}{#1}}{#2}}
\newcommand{\nuXInvOEXPR}[2][]{\frac{\IfENE{#1}{1}{#1}}{#2}}

\NewDocumentCommand{\mOEXPRD}{+m O{}}{%
  \BigO \Biggl( \max \!\!\begin{array}[t]{l} \displaystyle \Biggl\{
    \frac{\alphaO \n}{\GAMMAX^{2} #1^{2}}
    \log \left( \tauXInvOEXPR{#1} \right)
    +
    \frac{\alphaO}{\GAMMAX^{2} #1^{2}}
    \log \left( \frac{1}{\rhoX} \right)
    ,
    \nuXInvOEXPR[\n]{#1} \log^{3/2} \left( \nuXInvOEXPR{#1} \right)
    ,
    \nuXInvOEXPR{#1} \sqrt{\log \left( \nuXInvOEXPR{#1} \right)} \log \left( \frac{1}{\rhoX} \right)
  \Biggr\} \Biggr) #2 \end{array}
}

\NewDocumentCommand{\mOEXPRS}{+m O{}}{%
  \BigO \Biggl( \max \!\!\begin{array}[t]{l} \displaystyle \Biggl\{
    \frac{\alphaO \k}{\GAMMAX^{2} #1^{2}}
    \log \left( \tauXInvOEXPR[\n]{#1\k} \right)
    +
    \frac{\alphaO}{\GAMMAX^{2} #1^{2}}
    \log \left( \frac{1}{\rhoX} \right)
    ,
    \nuXInvOEXPR[\k]{#1} \log^{3/2} \left( \nuXInvOEXPR{#1} \right)
    ,
    \nuXInvOEXPR{#1} \sqrt{\log \left( \nuXInvOEXPR{#1} \right)} \log \left( \frac{1}{\rhoX} \right)
  \Biggr\} \Biggr) #2 \end{array}
}

\newcommand{\mOneEXPR}[1][d]{%
  \begin{cases}
  \displaystyle
  \frac{432\pi^{3}}{\bigl( 1 - \frac{\cO^{2}}{6} \bigr)^{2} \betaX^{2} \deltaX}
  \log \left( \frac{24}{\rhoX} \ifthenelse{\equal{#1}{s}}{\JSPCSIZES}{\JSPCSIZED} \right)
  ,& \cIf \betaX \in (0, \betaXThresholdLROne) ,\\
  \displaystyle
  \frac{216\pi^{2}}{\bO^{2} \deltaX}
  \log \left( \frac{24}{\rhoX} \ifthenelse{\equal{#1}{s}}{\JSPCSIZES}{\JSPCSIZED} \right)
  ,& \cIf \betaX \in [\betaXThresholdLROne, \betaXThresholdLRTwo) ,\\
  \displaystyle
  \frac{216\pi^{2}}{\bO^{2} \deltaX}
  \log \left( \frac{24}{\rhoX} \ifthenelse{\equal{#1}{s}}{\JSPCSIZES}{\JSPCSIZED} \right)
  ,& \cIf \betaX \in [\betaXThresholdLRTwo, \infty) ,
  \end{cases}
}

\NewDocumentCommand{\mOneOEXPR}{O{} +m O{,}}{%
  \newcommand{\SIZEEXPR}{\ifthenelse{\equal{#1}{s}}{\tauXInvOEXPR[\n]{#2\k}}{\tauXInvOEXPR{#2}}}
  \newcommand{\NK}{\ifthenelse{\equal{#1}{s}}{\k}{\n}}
  \begin{cases}
  \displaystyle
  \BigO'(
  \frac{\NK}{\betaX^{2} #2}
  \log \left( \SIZEEXPR \right)
  +
  \frac{1}{\betaX^{2} #2}
  \log \left( \frac{1}{\rhoX} \right)
  )
  ,& \cIf \betaX \in (0, \betaXThresholdLROne) ,\\
  \displaystyle
  \BigO'(
  \frac{\NK}{#2}
  \log \left( \SIZEEXPR \right)
  +
  \frac{1}{#2}
  \log \left( \frac{1}{\rhoX} \right)
  )
  ,& \cIf \betaX \in [\betaXThresholdLROne, \betaXThresholdLRTwo) ,\\
  \displaystyle
  \BigO'(
  \frac{\NK}{#2}
  \log \left( \SIZEEXPR \right)
  +
  \frac{1}{#2}
  \log \left( \frac{1}{\rhoX} \right)
  )
  ,& \cIf \betaX \in [\betaXThresholdLRTwo, \infty) #3
  \end{cases}
}

\newcommand{\mTwoEXPR}[1][d]{%
  \begin{cases}
  \displaystyle
  \frac{96\pi^{2}}{\bigl( 1 - \frac{\cO^{2}}{6} \bigr)^{2} \ConstbLD^{2} \betaX^{2} \deltaX^{2}}
  \log \left( \frac{24}{\rhoX} \ifthenelse{\equal{#1}{s}}{\JSPCSIZES}{\JSPCSIZED} \right)
  ,& \cIf \betaX \in (0, \betaXThresholdLROne) ,\\
  \displaystyle
  \frac{\sqrt{2\pi} 96}{\bO^{2} \ConstbLD^{2} \betaX \deltaX^{2}}
  \log \left( \frac{24}{\rhoX} \ifthenelse{\equal{#1}{s}}{\JSPCSIZES}{\JSPCSIZED} \right)
  ,& \cIf \betaX \in [\betaXThresholdLROne, \betaXThresholdLRTwo) ,\\
  \displaystyle
  \frac{96 \pi}{\bO^{2} \ConstbLD^{2} \deltaX}
  \log \left( \frac{24}{\rhoX} \ifthenelse{\equal{#1}{s}}{\JSPCSIZES}{\JSPCSIZED} \right)
  ,& \cIf \betaX \in [\betaXThresholdLRTwo, \infty) ,
  \end{cases}
}

\NewDocumentCommand{\mTwoOEXPR}{O{} +m O{,}}{%
  \newcommand{\SIZEEXPR}{\ifthenelse{\equal{#1}{s}}{\tauXInvOEXPR[\n]{#2\k}}{\tauXInvOEXPR{#2}}}
  \newcommand{\NK}{\ifthenelse{\equal{#1}{s}}{\k}{\n}}
  \begin{cases}
  \displaystyle
  \BigO'(
  \frac{\NK}{\betaX^{2} #2^{2}}
  \log \left( \SIZEEXPR \right)
  +
  \frac{1}{\betaX^{2} #2^{2}}
  \log \left( \frac{1}{\rhoX} \right)
  )
  ,& \cIf \betaX \in (0, \betaXThresholdLROne) ,\\
  \displaystyle
  \BigO'(
  \frac{\NK}{\betaX #2^{2}}
  \log \left( \SIZEEXPR \right)
  +
  \frac{1}{\betaX #2^{2}}
  \log \left( \frac{1}{\rhoX} \right)
  )
  ,& \cIf \betaX \in [\betaXThresholdLROne, \betaXThresholdLRTwo) ,\\
  \displaystyle
  \BigO'(
  \frac{\NK}{#2}
  \log \left( \SIZEEXPR \right)
  +
  \frac{1}{#2}
  \log \left( \frac{1}{\rhoX} \right)
  )
  ,& \cIf \betaX \in [\betaXThresholdLRTwo, \infty) #3
  \end{cases}
}

\newcommand{\mThreeEXPR}[1][d]{%
  \begin{cases}
  \displaystyle
  \frac{200 \sqrt{2\pi} \Constb}{\bigl( 1 - \frac{\cO^{2}}{6} \bigr) \Constc^{2} \betaX \deltaX}
  \frac{\log \left( \frac{24}{\rhoX} \ifthenelse{\equal{#1}{s}}{\JSPCSIZES}{\JSPCSIZED} \right)}{\sqrt{\log \left( \frac{4e}{\nuX} \right)}}
  ,& \cIf \betaX \in (0, \betaXThresholdLROne) ,\\
  \displaystyle
  \frac{200 \Constb}{\bO \Constc^{2} \deltaX}
  \frac{\log \left( \frac{24}{\rhoX} \ifthenelse{\equal{#1}{s}}{\JSPCSIZES}{\JSPCSIZED} \right)}{\sqrt{\log \left( \frac{4e}{\nuX} \right)}}
  ,& \cIf \betaX \in [\betaXThresholdLROne, \betaXThresholdLRTwo) ,\\
  \displaystyle
  \frac{200 \Constb}{\bO \Constc^{2} \deltaX}
  \frac{\log \left( \frac{24}{\rhoX} \ifthenelse{\equal{#1}{s}}{\JSPCSIZES}{\JSPCSIZED} \right)}{\sqrt{\log \left( \frac{4e}{\nuX} \right)}}
  ,& \cIf \betaX \in [\betaXThresholdLRTwo, \infty) ,
  \end{cases}
}

\NewDocumentCommand{\mThreeOEXPR}{O{} +m O{,}}{%
  \newcommand{\SIZEEXPR}{\ifthenelse{\equal{#1}{s}}{\tauXInvOEXPR[\n]{#2\k}}{\tauXInvOEXPR{#2}}}
  \newcommand{\NK}{\ifthenelse{\equal{#1}{s}}{\k}{\n}}
  \begin{cases}
  \displaystyle
  \BigO'(
  \frac{\NK}{\betaX #2}
  \ifthenelse{\equal{#1}{s}}{\frac{\log \left( \SIZEEXPR \right)}{\sqrt{\log \left( \nuXInvOEXPR{#2} \right)}}}{\sqrt{\log \left( \SIZEEXPR \right)}}
  +
  \frac{\log \left( \frac{1}{\rhoX} \right)}{\betaX #2 \sqrt{\log \left( \nuXInvOEXPR{#2} \right)}}
  )
  ,& \cIf \betaX \in (0, \betaXThresholdLROne) ,\\
  \displaystyle
  \BigO'(
  \frac{\NK}{#2}
  \ifthenelse{\equal{#1}{s}}{\frac{\log \left( \SIZEEXPR \right)}{\sqrt{\log \left( \nuXInvOEXPR{#2} \right)}}}{\sqrt{\log \left( \SIZEEXPR \right)}}
  +
  \frac{\log \left( \frac{1}{\rhoX} \right)}{#2 \sqrt{\log \left( \nuXInvOEXPR{#2} \right)}}
  )
  ,& \cIf \betaX \in [\betaXThresholdLROne, \betaXThresholdLRTwo) ,\\
  \displaystyle
  \BigO'(
  \frac{\NK}{#2}
  \ifthenelse{\equal{#1}{s}}{\frac{\log \left( \SIZEEXPR \right)}{\sqrt{\log \left( \nuXInvOEXPR{#2} \right)}}}{\sqrt{\log \left( \SIZEEXPR \right)}}
  +
  \frac{\log \left( \frac{1}{\rhoX} \right)}{#2 \sqrt{\log \left( \nuXInvOEXPR{#2} \right)}}
  )
  ,& \cIf \betaX \in [\betaXThresholdLRTwo, \infty) #3
  \end{cases}
}

\newcommand{\mFourEXPR}[1][d]{%
  \begin{cases}
  \displaystyle
  \frac{64 \sqrt{2\pi}}{\bigl( 1 - \frac{\cO^{2}}{6} \bigr) \Constb \betaX \deltaX}
  \sqrt{\log \left( \frac{4e}{\nuX} \right)}
  \log \left( \frac{24}{\rhoX} \right)
  ,& \cIf \betaX \in (0, \betaXThresholdLROne) ,\\
  \displaystyle
  \frac{64}{\bO \Constb \deltaX}
  \sqrt{\log \left( \frac{4e}{\nuX} \right)}
  \log \left( \frac{24}{\rhoX} \right)
  ,& \cIf \betaX \in [\betaXThresholdLROne, \betaXThresholdLRTwo) ,\\
  \displaystyle
  \frac{64}{\bO \Constb \deltaX}
  \sqrt{\log \left( \frac{4e}{\nuX} \right)}
  \log \left( \frac{24}{\rhoX} \right)
  ,& \cIf \betaX \in [\betaXThresholdLRTwo, \infty) ,
  \end{cases}
}

\NewDocumentCommand{\mFourOEXPR}{O{} +m O{,}}{%
  \newcommand{\SIZEEXPR}{\ifthenelse{\equal{#1}{s}}{\tauXInvOEXPR[\n]{#2\k}}{\tauXInvOEXPR{#2}}}
  \newcommand{\NK}{\ifthenelse{\equal{#1}{s}}{\k}{\n}}
  \begin{cases}
  \displaystyle
  \BigO'(
  \ifthenelse{\equal{#1}{s}}{
  \frac{\k}{\betaX #2}
  \sqrt{\log \left( \nuXInvOEXPR{#2} \right)}
  \log \left( \frac{\n}{\k} \right)
  +
  }{}
  \frac{1}{\betaX #2}
  \sqrt{\log \left( \nuXInvOEXPR{#2} \right)}
  \log \left( \frac{1}{\rhoX} \right)
  )
  ,& \cIf \betaX \in (0, \betaXThresholdLROne) ,\\
  \displaystyle
  \BigO'(
  \ifthenelse{\equal{#1}{s}}{
  \frac{\k}{#2}
  \sqrt{\log \left( \nuXInvOEXPR{#2} \right)}
  \log \left( \frac{\n}{\k} \right)
  +
  }{}
  \frac{1}{#2}
  \sqrt{\log \left( \nuXInvOEXPR{#2} \right)}
  \log \left( \frac{1}{\rhoX} \right)
  )
  ,& \cIf \betaX \in [\betaXThresholdLROne, \betaXThresholdLRTwo) ,\\
  \displaystyle
  \BigO'(
  \ifthenelse{\equal{#1}{s}}{
  \frac{\k}{#2}
  \sqrt{\log \left( \nuXInvOEXPR{#2} \right)}
  \log \left( \frac{\n}{\k} \right)
  +
  }{}
  \frac{1}{#2}
  \sqrt{\log \left( \nuXInvOEXPR{#2} \right)}
  \log \left( \frac{1}{\rhoX} \right)
  )
  ,& \cIf \betaX \in [\betaXThresholdLRTwo, \infty) #3
  \end{cases}
}

\newcommand{\mFiveEXPR}[1][d]{%
  \begin{cases}
  \displaystyle
  \frac{\sqrt{2\pi} \ConstdSD \ifthenelse{\equal{#1}{s}}{\k}{\n}}{\bigl( 1 - \frac{\cO^{2}}{6} \bigr) \Constb \betaX \deltaX}
  \sqrt{\log \left( \frac{4e}{\nuX} \right)}
  \log \left( \frac{1}{\nuX} \right)
  ,& \cIf \betaX \in (0, \betaXThresholdLROne) ,\\
  \displaystyle
  \frac{\ConstdSD \ifthenelse{\equal{#1}{s}}{\k}{\n}}{\bO \Constb \deltaX}
  \sqrt{\log \left( \frac{4e}{\nuX} \right)}
  \log \left( \frac{1}{\nuX} \right)
  ,& \cIf \betaX \in [\betaXThresholdLROne, \betaXThresholdLRTwo) ,\\
  \displaystyle
  \frac{\ConstdSD \ifthenelse{\equal{#1}{s}}{\k}{\n}}{\bO \Constb \deltaX}
  \sqrt{\log \left( \frac{4e}{\nuX} \right)}
  \log \left( \frac{1}{\nuX} \right)
  ,& \cIf \betaX \in [\betaXThresholdLRTwo, \infty) ,
  \end{cases}
}

\NewDocumentCommand{\mFiveOEXPR}{O{} +m O{,}}{%
  \newcommand{\SIZEEXPR}{\ifthenelse{\equal{#1}{s}}{\tauXInvOEXPR[\n]{#2\k}}{\tauXInvOEXPR{#2}}}
  \newcommand{\NK}{\ifthenelse{\equal{#1}{s}}{\k}{\n}}
  \begin{cases}
  \displaystyle
  \BigO'(
  \frac{\NK}{\betaX #2}
  \log^{3/2} \left( \nuXInvOEXPR{#2} \right)
  )
  ,& \cIf \betaX \in (0, \betaXThresholdLROne) ,\\
  \displaystyle
  \BigO'(
  \frac{\NK}{#2}
  \log^{3/2} \left( \nuXInvOEXPR{#2} \right)
  )
  ,& \cIf \betaX \in [\betaXThresholdLROne, \betaXThresholdLRTwo) ,\\
  \displaystyle
  \BigO'(
  \frac{\NK}{#2}
  \log^{3/2} \left( \nuXInvOEXPR{#2} \right)
  )
  ,& \cIf \betaX \in [\betaXThresholdLRTwo, \infty) #3
  \end{cases}
}

\NewDocumentCommand{\mOEXPRLR}{O{d} +m O{,}}{
\newcommand{\NK}{\ifthenelse{\equal{#1}{s}}{\k}{\n}}
  \begin{cases}
  \displaystyle
  \BigO'(
  \frac{\NK}{\betaX^{2} #2^{2}}
  )
  ,& \cIf \betaX \in (0, \betaXThresholdLROne) ,\\
  \displaystyle
  \BigO'(
  \frac{\NK}{\betaX #2^{2}}
  )
  ,& \cIf \betaX \in [\betaXThresholdLROne, \betaXThresholdLRTwo) ,\\
  \displaystyle
  \BigO'(
  \frac{\NK}{#2}
  )
  ,& \cIf \betaX \in [\frac{\ConstbetaXThrsholdLR}{\epsilonX}, \infty) #3
  \end{cases}
}

\NewDocumentCommand{\mOEXPRPR}{O{d} +m O{,}}{
  \begin{cases}
  \displaystyle
  \BigO'( \frac{\k}{\betaX^{2} \epsilonX^{2}} ) ,&\cIf \betaX \in (0,\betaXThresholdPROne) ,\\
  \displaystyle
  \BigO'( \frac{\k}{\betaX \epsilonX^{2}} )     ,&\cIf \betaX \in [\betaXThresholdPROne,\betaXThresholdPRTwo) ,\\
  \displaystyle
  \BigO'( \frac{\k}{\epsilonX} )                ,&\cIf \betaX \in [\betaXThresholdPRTwo, \infty) #3
  \end{cases}
}

\newcommand{\TS}{\Supp( \thetaStar ) \cup \Supp( \thetaHat[\Iter-1] ) \cup \Supp( \thetaHat[\Iter] )}

\newcommand{\sep}{\;}

\newcommand{\MeasM}{\Mat{X}}
\newcommand{\MeasV}{\Vec{x}}

\newcommand{\Jbiht}{\mathcal{J}_{\mathrm{BIHT}}}
\newcommand{\Jglmtron}{\mathcal{J}_{\mathrm{\GLMtronX}}}
\newcommand{\GLMtron}{GLM\hyphen tron\xspace}
\newcommand{\GLMtronX}{GLM\hyphen tron\xspace}

\newcommand{\betaXnamelr}{inverse temperature\xspace}
\newcommand{\betaXnamepr}{SNR\xspace}

\newcommand{\betaXThresholdLROne}{\cO}
\newcommand{\betaXThresholdLRTwo}{\frac{\ConstbetaXThrsholdLR}{\epsilonX}}
\newcommand{\ConstbetaXThrsholdValueLR}{\frac{3}{\sqrt{2 \pi}} ( 5+\sqrt{21} )}
\newcommand{\ConstbetaXThrsholdValuedeltaLR}{\sqrt{\frac{2}{\pi}}}
\newcommand{\ConstbetaXThrsholdLR}{b_{2}}

\newcommand{\betaXThresholdPROne}{\ConstbetaXThresholdPROne}
\newcommand{\betaXThresholdPRTwo}{\frac{\ConstbetaXThresholdPRTwo}{\epsilonX}}
\newcommand{\ConstbetaXThresholdPROne}{b_{3}}
\newcommand{\ConstbetaXThresholdPROneValue}{1}
\newcommand{\ConstbetaXThresholdPRTwo}{b_{4}}
\newcommand{\ConstbetaXThresholdPRTwoValue}{\frac{3}{2}( 5+\sqrt{21} )}

\newcommand{\GTR}{>}

\newcommand{\EWRValue}{\frac{\sqrt{\hfrac{2}{\pi}}-\gammaX}{2\alphaX}}

\definecolor{LinkColor}{HTML}{667699}
\definecolor{CiteColor}{HTML}{507395}
\definecolor{URLColor}{rgb}{0.3,0.5,0.7}
\hypersetup{
    colorlinks=true,
    linkcolor=LinkColor,
    filecolor=magenta,
    urlcolor=URLColor, 
    citecolor=CiteColor,
    anchorcolor=blue
}

\definecolor{MathInlineColor}{HTML}{000000}

\def\lightgraycolor{0.9}
\def\mediumgraycolor{0.55}
\def\darkgraycolor{0.4}
\definecolor{light-gray}{gray}{\lightgraycolor}
\definecolor{medium-gray}{gray}{\mediumgraycolor}
\definecolor{dark-gray}{gray}{0.4}

\definecolor{cmtc}{rgb}{0.1,0.1,0.9}

\NewDocumentCommand{\StartComment}{s}{%
\IfBooleanTF{#1}{%
\color{light-gray}\everymath{\color{light-gray}}%
\def\graycolor{\lightgraycolor}
\definecolor{todocolor}{gray}{\graycolor}%
\definecolor{cmtcolor}{gray}{\graycolor}%
\definecolor{cmtqcolor}{gray}{\graycolor}%
\definecolor{MathInlineColor}{gray}{\graycolor}%
\definecolor{LinkColor}{gray}{\graycolor}%
\definecolor{CiteColor}{gray}{\graycolor}%
\definecolor{URLColor}{gray}{\graycolor}%
}{%
\color{medium-gray}\everymath{\color{medium-gray}}%
\def\graycolor{\mediumgraycolor}%
\definecolor{todocolor}{gray}{\graycolor}%
\definecolor{cmtcolor}{gray}{\graycolor}%
\definecolor{cmtqcolor}{gray}{\graycolor}%
\definecolor{MathInlineColor}{gray}{\graycolor}%
\definecolor{LinkColor}{gray}{\graycolor}%
\definecolor{CiteColor}{gray}{\graycolor}%
\definecolor{URLColor}{gray}{\graycolor}%
}}

\NewDocumentCommand{\StartRemark}{s}{%
\IfBooleanTF{#1}{%
\color{dark-gray}\everymath{\color{dark-gray}}%
\def\graycolor{\darkgraycolor}
\definecolor{todocolor}{gray}{\graycolor}%
\definecolor{cmtcolor}{gray}{\graycolor}%
\definecolor{cmtqcolor}{gray}{\graycolor}%
\definecolor{MathInlineColor}{gray}{\graycolor}%
\definecolor{LinkColor}{gray}{\graycolor}%
\definecolor{CiteColor}{gray}{\graycolor}%
\definecolor{URLColor}{gray}{\graycolor}%
}{%
\color{dark-gray}\everymath{\color{dark-gray}}%
\def\graycolor{\darkgraycolor}
\definecolor{todocolor}{gray}{\graycolor}%
\definecolor{cmtcolor}{gray}{\graycolor}%
\definecolor{cmtqcolor}{gray}{\graycolor}%
\definecolor{MathInlineColor}{gray}{\graycolor}%
\definecolor{LinkColor}{gray}{\graycolor}%
\definecolor{CiteColor}{gray}{\graycolor}%
\definecolor{URLColor}{gray}{\graycolor}%
}}

\newcommand{\xToDo}[1]{\normalfont\textcolor{todocolor}{$\textcolor{todocolor}{\ldelim}$\textbf{To\,do}\IfENE{#1}{\textbf{:}\;Fill this in.}{\textbf{:}\;{#1}}$\textcolor{todocolor}{\rdelim}$}}
\newcommand{\ToDo}[1]{%
  \colorlet{origcolor}{.}\everymath{\color{todocolor}}%
  \ifthenelse{\boolean{showcomments}}%
  {\relax\ifmmode\text{\xToDo{#1}}\else\xToDo{#1}\fi}%
  {}%
  \everymath{\color{origcolor}}
}%

\newcommand{\xCOMMENT}[1]{\normalfont\textcolor{cmtcolor}{$\textcolor{cmtcolor}{\ldelim}$\textbf{Note}\IfNE{#1}{\textbf{:}\;{#1}}$\textcolor{cmtcolor}{\rdelim}$}}
\newcommand{\COMMENT}[1]{%
  \colorlet{origcolor}{.}\everymath{\color{cmtcolor}}%
  \ifthenelse{\boolean{showcomments}}%
  {\relax\ifmmode\text{\xCOMMENT{#1}}\else\xCOMMENT{#1}\fi}%
  {}%
  \everymath{\color{origcolor}}
}%

\NewDocumentCommand{\?}{s +m}{%
  \colorlet{origcolor}{.}\everymath{\color{cmtqcolor}}%
  \ifthenelse{\boolean{showcomments}}%
  {\normalfont\textcolor{cmtqcolor}{{\IfBooleanF{#1}{$\textcolor{cmtqcolor}{\ldelim}$}#2$\textcolor{cmtqcolor}{\rdelim}$}}}%
  {}%
  \everymath{\color{origcolor}}
}%

\NewDocumentCommand{\BEGINMARK}{s}{\everymath{\color{blue}}\color{blue}}

\definecolor{cmtoutcolor}{gray}{0.7}
\newcommand{\xGRAYOUT}[1]{\normalfont\textcolor{cmtoutcolor}{#1}}
\newcommand{\GRAYOUT}[1]{%
  \colorlet{origcolor}{.}\everymath{\color{cmtoutcolor}}%
  \ifthenelse{\boolean{showcomments}}%
  {\relax\ifmmode\text{\xGRAYOUT{\ensuremath{#1}}}\else\xGRAYOUT{#1}\fi}%
  {}%
  \everymath{\color{MathInlineColor}}%
}%

\newcommand{\BEGINGRAYOUT}{\everymath{\color{cmtoutcolor}}\color{cmtoutcolor}}
\newcommand{\ENDGRAYOUT}{\everymath{\color{MathInlineColor}}\color{black}}

\newcommand{\checkoff}{}
\newcommand{\mostlycheckoff}{}
\newcommand{\checkoffbutmayberecheck}{}
\newcommand{\checkthis}[1][]{}

\newcommand{\XXX}[1]{}
\usepackage{enumitem}
\newcommand\blfootnote[1]{%
  \begingroup
  \renewcommand\thefootnote{}\footnote{#1}%
  \addtocounter{footnote}{-1}%
  \endgroup
}
\title{Learning sparse generalized linear models with binary outcomes via iterative hard thresholding}
\usepackage{times}
\author{Namiko Matsumoto \and Arya Mazumdar\blfootnote{The authors are with the University of California San Diego. This work is supported in part by NSF awards 2133484, 2127929 and  NSF GRFP No. DGE-2038238. Emails: \texttt{\{nmatsumoto,arya\}@ucsd.edu}.}}
\begin{document}
 \maketitle

\begin{abstract}

In statistics, generalized linear models (GLMs) are widely used for modeling
data and can expressively capture 
potential nonlinear dependence of the model's outcomes on its covariates.
Within the broad family of GLMs, those with binary outcomes, which include logistic and probit regressions, are motivated by common tasks such as  binary classification with (possibly) non-separable data. In addition, in modern machine learning and statistics, data is often high-dimensional yet has a low intrinsic dimension, making sparsity constraints in models another reasonable consideration. In this work, we propose to use and analyze an iterative hard thresholding (projected gradient descent on the ReLU loss) algorithm, called \emph{binary iterative hard thresholding (BIHT)}, for parameter estimation in sparse GLMs with binary outcomes. We establish that BIHT is statistically efficient and converges to the correct solution for parameter estimation in  a general class of sparse binary GLMs.
Unlike many other methods for learning GLMs, including maximum likelihood estimation, generalized approximate message passing, and GLM-tron \citep{kakade2011efficient,bahmani2016learning}, BIHT does not require knowledge of the GLM's link function, offering flexibility and generality in allowing the algorithm to learn arbitrary binary GLMs.
As two applications, logistic and probit regression are additionally studied.
In this regard, it is shown that in logistic regression, the algorithm is in fact statistically optimal in the sense that the order-wise sample complexity matches (up to logarithmic factors) the lower bound obtained previously.
To the best of our knowledge, this is the first work achieving statistical optimality for logistic regression in all noise regimes with a computationally efficient algorithm. Moreover, for probit regression, our sample complexity is on the same order as that obtained for logistic regression. 
\end{abstract}
\section{Introduction}
\label{outline:intro}

\subsection{Generalized Linear Models}
\label{outline:intro|>glm}

Generalized linear models (GLMs) are a popular statistical paradigm that has been extensively studied since their introduction by \cite{nelder1972generalized} as a generalization
and unifying framework encompassing several
common
statistical models.
In a GLM, each response (random) variable, \(  \RespV* \in \R \), has distribution with a parameter, \(  \thetaStar \in \ParamSpace  \), taken from a parameter space, \(  \ParamSpace \subseteq \R^{\n}  \), and dependent on a covariate, \(  \CovV  \in \R^{\n} \), such that,
\(  \E[\RespV* \Mid| \CovV] = \linkFn^{-1}( \langle \CovV, \thetaStar \rangle )  \),
where \(  \linkFn  \) is the \emph{link function} that ``links'' the linear combination, \(  \langle \CovV, \thetaStar \rangle  \), to the conditional expectation of the response, \(  \RespV* \Mid| \CovV  \).
This
framework
offers a flexible extension of the popular linear regression model, 
\(  \E[\RespV* \Mid| \CovV] = \langle \CovV, \thetaStar \rangle  \)
---to allow for nonlinearities.
The reader is referred to \cite{mccullagh2019generalized,dobson2018introduction,fahrmeir1994multivariate,hardin2007generalized} for background on GLMs.

A fundamental problem in GLMs is parameter estimation---that is, the estimation of the parameter, \(  \thetaStar \in \ParamSpace  \), when given \(  \m  \)  i.i.d. samples,
\(  ( \CovV\VIx{1}, \RespV*_{1} ), \dots, ( \CovV\VIx{\m}, \RespV*_{\m} )  \),
where the observed responses,
\(  \RespV*_{1}, \dots, \RespV*_{\m}  \),
and known covariates,
\(  \CovV\VIx{1}, \dots, \CovV\VIx{\m}  \),
are related
in the manner stated earlier.
%
%
Maximum likelihood estimation (MLE) \citep{myung2003tutorial,richards1961method,gallant1987semi,wald1949note,aitchison1960maximum} is a predominant approach to parameter estimation in GLMs \citep{mcculloch1997maximum,hardin2007generalized}, where the estimates can be obtained through techniques such as iterative weighted least-squares methods \citep{nelder1972generalized,firth1992generalized,hardin2007generalized}, the Newton-Raphson method \citep{jin2022statistical,hardin2007generalized}, and the Gauss-Newton method \citep{wedderburn1974quasi}.
Gradient descent can also compute maximum likelihood estimates for parameter estimation in GLMs.
One such line of work has studied gradient descent on the objective,
\(  \objectiveFn( \thetaX ) = \sum_{\iIx} \intlinkFn( \langle \CovV\VIx{\iIx}, \thetaX \rangle ) - \RespV*_{\iIx} \langle \CovV\VIx{\iIx}, \thetaX \rangle  \),
where \(  \intlinkFn  \) is defined such that
\(  \linkFn^{-1} = \partial \intlinkFn  \).
\cite{kakade2011efficient} propose a perceptron-like algorithm, \emph{\GLMtron}, for learning 
GLMs by performing gradient descent on this loss function.
\cite{bahmani2016learning} study a similar gradient decent algorithm which incorporates (sparse) projections.
When the outcomes, \(  \RespV*_{\iIx}  \), are \iid and follow an exponential distribution, \(  \objectiveFn  \) is the negative log-likelihood function, and therefore, under these conditions, the minimizer of \(  \objectiveFn  \) is the maximum likelihood estimate.
Note, however, that \cite{kakade2011efficient,bahmani2016learning} consider more general classes of GLMs, 
meaning that their results extend beyond MLE in some cases.
These works will be discussed further and compared to the contributions of this work in Appendix \ref{outline:intro|>contributions|>comparison}.
Of note, the \emph{Sparsitron} algorithm of \cite{klivans2017learning}, a multiplicative-weights update method, improves on the error of \GLMtron. 

One consideration in learning GLMs is whether the link function is known.
In contrast to many other works, \eg \cite{nelder1972generalized,kakade2011efficient,bahmani2016learning,barbier2019optimal}, the algorithm studied in this work does not require knowledge of the specific link function. 
Some other works do consider learning a class of single-index models (SIMs) agnostically, without access to the link function, 
e.g. \cite{gollakota2024agnostically} via omnipredictors,  which are predictors that optimize over all loss functions in a collection. The setup and guarantees of this line of work is quite different than ours.

While the discussion thus far has not constrained the outcomes, \(  \RespV*  \), GLMs with binary outcomes are popular for classification and particularly useful when data is \nonseparable.
For brevity, this class of GLMs will be referred to as \emph{binary GLMs} throughout this manuscript.
They contain several important families of models, such as a subset of the exponential family that includes the ubiquitous logistic and probit regression models.
This work is concerned with binary GLMs with a mild assumption on their link functions, which indeed holds for logistic and probit regression.
Appendix \ref{outline:intro|>contributions|>comparison} briefly surveys some relevant prior works on parameter estimation in binary GLMs.
 
In treating binary GLMs as (linear) classifiers, where \(  \thetaX  \) becomes a feature vector, the assumption of sparsity in a high-dimensional parameter space---or in
this analogy
the feature space---is a commonplace paradigm in machine learning.
Moreover, interpreting the parameter \(  \thetaX  \) as a signal or data vector, parameter estimation in GLMs can be framed as
the inverse problem of
signal reconstruction from noisy measurements.
In this regard, requiring high-dimensional parameters---or signals in this perspective---to be contained in a sparse subspace is again a frequent consideration. 
In fact, this present work is motivated by the connection of binary GLMs to \onebitcs, a topic within compressed sensing where the entries of the compressed signal representations are quantized to single bits: the \(  \pm  \) signs of the unquantized values.
The next section, \SECTION \ref{outline:intro|>1bcs}, briefly introduces \onebitcs and explores this connection.


\subsection{1-Bit Compressed Sensing and Binary Iterative Hard Thresholding}
\label{outline:intro|>1bcs}

The \onebitcs problem~\cite{boufounos20081} seeks to recover an unknown vector \(  \thetaStar  \) when given \(  \m  \) responses,
\(  \RespV*_{\iIx} = \Sign(\langle \MeasV\VIx{\iIx}, \thetaStar \rangle )  \),
where
\(  \MeasV\VIx{1}, \dots, \MeasV\VIx{\m} \in \R^{\n}  \) are measurement vectors, and only $\pm$ signs of the linear measurements are being kept.
Letting
\(  \MeasM = ( \MeasV\VIx{1} \;\cdots\; \MeasV\VIx{\m} )^{\T}  \)
and extending the notation of signs
to vectors by applying it \coordinatewise,
this representation can be written concisely as
\(  \RespV = ( \RespV*_{1}, \dots, \RespV*_{\m} ) = \Sign( \MeasM \thetaStar )  \).
Typically, it is assumed that \(  \m \ll \n  \), and thus,
\(  \RespV  \)
is a compressed representation of the original signal vector, \(  \thetaStar  \).
The compressibility of \(  \thetaStar  \) is often incorporated by a notion of sparsity.
In this work, it is assumed that \(  \thetaStar  \) is \emph{\(  \k  \)-sparse}---that is, the vector is supported on at most \(  \k \leq \n  \) nonzero entries.
Additionally, one of the most common choices of measurements is \iid standard Gaussian random vectors, as is the design studied in this work.
Most often, in \onebitcs, the unknown vector, \(  \thetaStar  \), is assumed to have unit norm since information about the norm is lost by the binarization of the responses.
In relating \onebitcs to binary classification, this noiseless model corresponds with classification of separable data.
Alternatively, in connecting \onebitcs to binary GLMs, the sign function can be replaced by a random function \(  \fFn  \), to incorporate noise,   defined such that each \(  \iIx\Th  \) response takes the value \(  1  \) with probability \(  \pFn( \langle \MeasV\VIx{\iIx}, \thetaStar \rangle )  \) and the value \(  -1  \) with probability \(  1-\pFn( \langle \MeasV\VIx{\iIx}, \thetaStar \rangle )  \) for some function, \(  \pFn : \R \to [0,1]  \).
This setting, which can also be interpreted as binary classification of non-separable data, is studied in this work.

For reconstruction in \onebitcs, \cite{jacques2013robust} propose the \emph{binary iterative hard thresholding (BIHT)} algorithm, inspired by the existing {\em iterative hard thresholding} algorithm for compressed sensing~\cite{blumensath2009iterative}. BIHT is a projected, (sub)gradient descent algorithm on the (negative) ReLU loss, given for
\(  \thetaX \in \R^{\n}  \)
by
\begin{gather}
\label{eqn:notations:objective-function}
  \mathcal{J}( \thetaX )
  =
  \| [ \Diag( \RespV ) \CovM \thetaX ]_{-} \|_{1}
,\end{gather}
where
\(  \RespV \defeq \Sign( \CovM \thetaStar ) \in \{ -1,1 \}^{\m}  \)
is the vector of binary responses and where
\(  ( [ \Vec{v} ]_{-} )_{\iIx} = [ \Vec*{v}_{\iIx} ]_{-} = \min \{ \Vec*{v}_{\iIx}, 0 \}  \),
\(  \iIx \in [\m]  \),
for a vector \(  \Vec{v} \in \R^{\m}  \).
\cite{jacques2013robust} shows that
\begin{gather}
\label{eqn:notations:grad-objective-function}
  \CovM^{\T} \sep \frac{1}{2} ( \Sign( \CovM \thetaX ) - \RespV ) \in \nabla_{\thetaX} \mathcal{J}( \thetaX )
,\end{gather}
leading to their BIHT algorithm.
This present work studies the normalized variant of the BIHT algorithm of \cite{jacques2013robust}, presented in \ALGORITHM \ref{alg:biht:normalized}, which iteratively performs
\Enum[{\label{step:alg:biht:normalized:i}}]{i} first a (sub)gradient descent step on the objective function \(  \objectiveFn  \) in \EQUATION \eqref{eqn:notations:objective-function}, 
followed by
\Enum[{\label{step:alg:biht:normalized:ii}}]{ii} a projection onto \(  \ParamSpace  \),
by performing the \topk thresholding operation and then normalizing.
Note that in the dense regime, when \(  \k=\n  \), the algorithm applies no thresholding in the second step but still normalizes the approximation.
%
\par 
%
For this noiseless response setting in \onebitcs,
\cite{friedlander2021nbiht} shows that under the Gaussian design, the BIHT algorithm converges to the correct solutions with a suboptimal sample complexity. 
Subsequently, \cite{matsumoto2022binary} improves the sample complexity to the theoretically \orderwise optimal (up to logarithmic factors) sample complexity:
\(  \BigO'( \frac{\k}{\epsilonX} \log( \frac{\n}{\k} ) \sqrt{\log( \smash[b]{\frac{1}{\epsilonX}} )} + \frac{\k}{\epsilonX} \log^{3/2}( \frac{1}{\epsilonX} ) )  \),
matching lower bound on the sample complexity for recovery in \onebitcs established by \cite{jacques2013robust}.
One limitation of \cite{friedlander2021nbiht,matsumoto2022binary} is the inability to immediately extend the convergence result to \nonseparable data and settings with noise.
This was partially addressed in later work, in which \cite{matsumoto2024robust} shows robustness properties of BIHT when \(  \fFn  \) incorporates adversarial noise: convergence of BIHT is possible with a similar sample complexity to that stated for the noiseless (or \nonseparable) setting.
This begs the question: what convergence guarantees are possible in other models of noise or \nonseparable data?
This work seeks an answer to this question for one such class of models: binary GLMs.


\begin{algorithm}
\caption{Normalized binary iterative hard thresholding (BIHT)\label{alg:biht:normalized}}
\Given{\(  \RespV, \CovM  \)}
\(
  \thetaHat^{(0)}
  \sim
  \SparseSphereSubspace{\k}{\n}
\)\;

\For
{\( \Iter = 1, 2, 3, \dots \)}
{
  \(
    \thetaHatX[\Iter]
    \gets
    \thetaHat^{(\Iter-1)}
    +
    \frac{\sqrt{2\pi}}{\m}
    \sep
    \CovM^{\T}
    \sep
    \frac{1}{2}
    \left(
      \RespV - \Sign{} \big( \CovM \thetaHat[\Iter-1] \big)
    \right)
  \)\; \label{alg:biht:normalized:intermediate}

  \(
    \thetaHat[\Iter]
    \gets
    \frac
    {\Threshold{\k}{}(\thetaHatX[\Iter])}
    {\| \Threshold{\k}{}(\thetaHatX[\Iter]) \|_{2}}
  \)\; \label{alg:biht:normalized:projection}
}
\end{algorithm}
\subsection{Main Contributions}
\label{outline:intro|>contributions}


One realization that leads to this work is that the BIHT algorithm (Algorithm~\ref{alg:biht:normalized}) can be applicable to any possibly randomized $1$-bit  quantization function, \(  \fFn  \), without any change, and not restricted to the sign function. Therefore, in particular,  Algorithm~\ref{alg:biht:normalized} is a perfect candidate for efficient parameter estimation in binary GLMs, and can be advantageous over some existing estimation methods as it is oblivious of the link function.

 This work establishes that, when applied to binary GLMs where the link function satisfies a reasonable property (\ie \ASSUMPTION \ref{assumption:p}), the BIHT algorithm iteratively produces approximations that converge to the \(  \epsilonX  \)-ball around the true parameter, \(  \thetaStar  \), with high probability under the Gaussian covariates design with a sample complexity of
\begin{gather*}
  \m = \BigO'(\max\Bigg[\frac{\max(\alpha,\epsilonX)}{\gammaX^2\epsilonX^{2}} \k \log \left( \frac{\n}{\epsilonX \k} \right), \frac{k}{\epsilonX} \log^{3/2} \left( \frac{1}{\epsilonX} \right) \Bigg] )
,\end{gather*}
where 
the big-O term hides  some less-significant factors for the sake of readability, and
\(  \alphaX  \) and \(  \gammaX  \) are  properties of the link function defined in Equations~\eqref{eqn:notations:alpha:def} and \eqref{eqn:notations:gamma:def}
which can be explicitly computed for standard GLMs (leading to  tighter sample complexity \(\BigO'(\k/\epsilonX)\) in some regimes - equivalent to {\em optimistic rate} in learning theory). In addition the estimation error decays at an exponential rate with respect to the number of iterations, \(  \Iter \in \Z_{\geq 0}  \), of the algorithm, bounded from above by
\begin{gather*}
  \| \thetaStar - \thetaHat[\Iter] \|_{2}
  =
  \BigO( \epsilon^{1-2^{-\Iter}} )
,\end{gather*}
and hence, asymptotically approaches the \errorrate of \(  \epsilonX  \) as \(  \Iter \to \infty  \).
%
\par 
%
When specialized to two of the most popular binary GLMs, logistic and probit regressions, convergence of BIHT can be shown and explicit expressions
can derived for the sample complexities (via closed-form bounds on \(  \alphaX  \) and \(  \gammaX  \)).
For logistic and probit regressions these
lead to
the optimal scaling with the ``signal-to-noise ratio'' SNR, \(  \betaX  \), and \errorrate, \(  \epsilonX  \), as \(  \betaX  \) is varied.
Notably, as \(  \betaX \to \infty  \), the dependence of the sample complexity on the \errorrate, reduces to \(  \frac{1}{\epsilonX}  \).
Due to lower bounds of \cite{hsu2024sample},
the resultant sample complexity for logistic regression is optimal up to logarithmic factors in all 
regimes.
As discussed in \cite{hsu2024sample}, for logistic regression with Gaussian covariates,  computationally efficient algorithms  with optimal sample complexity were known only for the high noise regime, 
e.g., \cite{plan2017high}, and 
the noiseless regime~\cite{matsumoto2022binary}, while no such optimal algorithms were known for the intermediate regimes\footnote{However, for $\k=\n$, MLE achieves optimal samples for low noise case, by recent work of~ \cite{chardon2024finite}.}.
Thus, to the best of our knowledge, BIHT is the first computationally efficient algorithm with the optimal sample complexity in all noise regimes
for logistic regression under the Gaussian covariate design.

\paragraph{Organization}
In \SECTION \ref{outline:intro|>contributions|>techniques} we give an overview of how our results are obtained. In \SECTION \ref{outline:notations}, we provide the notations used throughout the paper, and formally define binary GLMs and related quantities and assumptions.
In \SECTION \ref{outline:main-result}, the main theorems of this work, which establish the convergence of BIHT to the correct solution for parameter estimation in binary GLMs, are formally stated.
\SECTION \ref{outline:main-result|outline-of-pf} outlines the key steps in the proof of the main results.
\SECTION \ref{outline:main-technical-result} presents the main technical theorem: the invertibility condition for Gaussian covariates. 
In \APPENDIX \ref{outline:intro|>contributions|>comparison} we compare our techniques with existing results, notably with~\cite{matsumoto2022binary}. 
The main theorems are formally proved in \APPENDIX \ref{outline:pf-main-result}, while formal proofs of the main technical theorems are in \APPENDIX \ref{outline:pf-main-technical-result}.
Lastly, \APPENDIX \ref{outline:concentration-ineq} derives some concentration inequalities that are needed in the proof of the main technical theorems.
%
%

\section{Overview of  Techniques}
\label{outline:intro|>contributions|>techniques}

The proof of the convergence of the BIHT approximations for any GLM satisfying \ASSUMPTION \ref{assumption:p} adapts the following approach.
It consists of two primary bounds on the approximation error:
\Enum[{\label{enum:contributions:1:a}}]{a}
a deterministic bound that relates the approximation error to an {\em invertibility condition} satisfied by Gaussian matrices, and
\Enum[{\label{enum:contributions:1:b}}]{b}
a probabilistic bound that describes this invertibility condition for Gaussian matrices.
The former of these bounds, \ref{enum:contributions:1:a}, is relatively straightforward to establish using standard techniques, including 
bounding of a recurrence relation.
On the other hand, the derivation of the latter bound, \ref{enum:contributions:1:b}, which is the primary technical contribution of this work, entails extensive analysis that constitutes the majority of this manuscript.
The establishment of the  invertibility condition for Gaussian covariate matrices in regard to this latter bound follows a similar approach to  \cite{matsumoto2022binary} (i.e., the noiseless case) to prove an analogous invertibility condition for Gaussian matrices therein, though there are some major technical differences, which are highlighted in \APPENDIX \ref{outline:intro|>contributions|>comparison-biht}.
%
\par 
%
The intuition behind the argument for the probabilistic bound, \ref{enum:contributions:1:b}, is as follows.
The aforementioned invertibility condition (see Theorem~\ref{thm:main-technical:sparse}) upper bounds
\begin{gather}
\label{eqn:contributions:1}
  \left\|
    \thetaStar
    -
    \frac{\thetaXXX}{\| \thetaXXX \|_{2}}
  \right\|_{2}
,\end{gather}
where
\(
    \thetaXXX
    \defeq
    \thetaXX + \hfFn[\JCoords]( \thetaStar, \thetaXX )
,\)
uniformly for every \(  \thetaXX \in \ParamSpace  \) and every \(  \JCoords \subseteq [\n]  \), \(  | \JCoords | \leq \k  \), and for a particular random function, \(  \hfFn[\JCoords] : \R^{\n} \times \R^{\n} \to \R^{\n}  \), parameterized by the coordinate subset, \(  \JCoords  \), and dependent on the covariate matrix, \(  \CovM  \).
The specification of the function, \(  \hfFn[\JCoords]  \), which is determined by the GLM function \(  \fFn  \), is deferred to \SECTION 
\ref{outline:main-result|outline-of-pf} as this informal overview can be understood without its formal definition.
One can view an upper bound on the quantity of \EQUATION \eqref{eqn:contributions:1} to be a {\em single-step progress} towards estimating \(\thetaStar\) via the BIHT algorithm.
One
salient
characteristic
of the invertibility condition is that
stronger guarantees are provided for points, \(  \thetaXX  \), which are closer to \(  \thetaStar  \).
The invertibility condition will be proved to hold for Gaussian covariate matrices with high probability.
Towards this, the quantity in \EQUATION \eqref{eqn:contributions:1} will be shown to describe a notion of deviation of the random vector \(  \thetaXXX  \) around its mean in the sense that
\(
  \left\|
    \thetaStar
    -
    \frac{\thetaXXX}{\| \thetaXXX \|_{2}}
  \right\|_{2}
  =
  \left\|
    \frac{\thetaXXX}{\| \thetaXXX \|_{2}}
    -
    \frac{\E[ \thetaXXX ]}{\| \E[ \thetaXXX ] \|_{2}}
  \right\|_{2}
.\) 
Furthermore, it can be shown that this deviation is roughly proportional to the deviation of the random function \(  \hfFn[\JCoords]  \) around its mean:
\begin{gather}
\label{eqn:contributions:2b}
  \left\|
    \frac{\thetaXXX}{\| \thetaXXX \|_{2}}
    -
    \frac{\E[ \thetaXXX ]}{\| \E[ \thetaXXX ] \|_{2}}
  \right\|_{2}
  \propto
  \| \hfFn[\JCoords]( \thetaStar, \thetaXX ) - \E[ \hfFn[\JCoords]( \thetaStar, \thetaXX ) ] \|_{2}
.\end{gather}
The deviation of \(  \hfFn[\JCoords]  \) is then decomposed into (and upper bounded by) three components of deviation:
\Enum[{\label{enum:contributions:2:i}}]{i}
a global component for each point in a cover over the parameter space, \(  \ParamSpace  \), that is sufficiently far from \(  \thetaStar  \);
\Enum[{\label{enum:contributions:2:iii}}]{ii}
a local component for each point in the cover and every point in a small region surrounding it; and
\Enum[{\label{enum:contributions:2:ii}}]{iii}
a component which arises from the randomness of the GLM---specifically, through the function \(  \fFn  \).
Upper bounds on the first and third components, \ref{enum:contributions:2:i} and \ref{enum:contributions:2:ii}, can be established  following Gaussian concentration.
Meanwhile, the second component, \ref{enum:contributions:2:iii}, relies on the local binary embeddings of \cite{oymak2015near}.

Notably, the two components, \ref{enum:contributions:2:i} and \ref{enum:contributions:2:iii}, do not depend on the random function \(  \fFn  \): they replace the random function \(  \fFn  \) with the deterministic \(  \Sign  \) function.
Thus, the third component, \ref{enum:contributions:2:ii}, entirely captures the deviation associated with the randomness induced by \(  \fFn  \).
The upper bounding of all three components exploits the statistical ``niceness'' of the Gaussian covariates.
In particular, for the first component, \ref{enum:contributions:2:i}, the angular uniformity of \iid Gaussian random vectors is crucial because
it
controls the number of covariates involved in the computation of the random function, \(  \hfFn[\JCoords]  \). 
However, this breaks down when points are too close together as the number of samples involved in the computation of \(  \hfFn[\JCoords]  \) cannot be guaranteed to further decrease beyond a certain threshold (a distance on the order of \(  \epsilonX  \)).
Once this occurs, the local guarantees provided by the local binary embeddings of \cite{oymak2015near} take over through \ref{enum:contributions:2:iii} to ensure that the randomness of \(  \hfFn[\JCoords]  \) is well-controlled to provide sufficient guarantees within these local regions.
The (sub)gaussianity of the covariates is also critical for bounding the third component of deviation, \ref{enum:contributions:2:ii}, though not
strictly
through angular uniformity.
In effect, it limits the amount of deviation that can be introduced by the randomness of the GLM, which comes from standard knowledge.
%
\par 
%
Taking a step back to examine the relationship between this invertibility condition and the convergence of BIHT, there are a couple key ideas underlying the behavior exhibited by the algorithm.
First, the last two components of deviation, \ref{enum:contributions:2:iii} and \ref{enum:contributions:2:ii},
introduce error to the algorithm's approximations which is more or less ``baked in'' once the model and its covariates are fixed---that is, these contributions to the approximation error will not decay as the algorithm continues to iterate.
In contrast, the first component of deviation, \ref{enum:contributions:2:i}, contributes error into the approximation which indeed decays.
Simply put, this is the consequence of the invertibility condition imposing a stronger bound for points which are closer to the true parameter, \(  \thetaStar  \).
In essence, since the number of covariates---and hence also the variance---involved in the evaluation of \(  \hfFn[\JCoords]  \) at a pair of points
decreases as the distance between the points decreases, the improvement of the approximation in one iteration of BIHT leads to even better control over \(  \hfFn[\JCoords]  \), and thus a better approximation, in the next iteration.
However, because the invertibility condition only guarantees this improvement up to but not within small, local regions, the error cannot be guaranteed to reduce once the approximations reach the \(  \epsilonX  \)-ball around \(  \thetaStar  \).
At the same time, the error will remain within a small threshold due to the local result, \ie \ref{enum:contributions:2:iii}.


\section{Notations and Key Properties of GLMs}
\label{outline:notations}

For a set of real numbers, \(  \Set{S} \subseteq \R  \), let \(  \Set{S}_{\geq 0}, \Set{S}_{+} \subseteq \Set{S}  \) denote the sets of nonnegative elements and, respectively, positive elements in \(  \Set{S}  \)---formally,
\(  \Set{S}_{\geq 0} \defeq \{ s \in \Set{S} : s \geq 0 \}  \) and
\(  \Set{S}_{+} \defeq \{ s \in \Set{S} : s > 0 \}  \).
For \(  t \in \Z_{+}  \), let \(  [t] \defeq \{ 1, \dots, t \}  \).
Vectors and matrices are denoted in lowercase and uppercase bold typeface, respectively, e.g.,
\(  \Vec{v} \in \R^{n}  \) and \(  \Mat{M} \in \R^{m \times n}  \),
with their entries in italic font such that, e.g., \(  \Vec{v} = ( \Vec*{v}_{j} )_{j \in [\n]}  \) and \(  \Mat{M} = ( \Mat*{M}_{i,j} )_{(i,j) \in [\m] \times [\n]}  \).
The all-zeros vector is written in boldface: \(  \Vec{0} = ( 0, \dots, 0 )  \).
For a coordinate subset \(  J \subseteq [\n]  \), the vector with entries indexed by \(  J  \) taking value \(  1  \) and with all other entries set to \(  0  \) is written as \(  \BVec{J} \in \R^{\n}  \).
The set of all \(  s  \)-sparse, real-valued vectors is denoted by
\(  \SparseSubspace{s}{\n} \defeq \{ \Vec{v} \in \R^{\n} : \| \Vec{v} \|_{0} \leq s \}  \).
For \(  r > 0  \) and \(  \Vec{v} \in \R^{\n}  \), specify the radius-\(  r  \) ball around \(  \Vec{v}  \) by
\(  \Ball{r}( \Vec{v} ) \defeq \{ \Vec{u} \in \R^{\n} : \| \Vec{u} - \Vec{v} \|_{2} \leq r \}  \),
and let
\(  \BallX{r}( \Vec{v} ) \defeq \Ball{r}( \Vec{v} ) \cap \Sphere{\n} \cap \{ \Vec{u} \in \R^{\n} : \Supp( \Vec{u} ) = \Supp( \Vec{v} ) \}  \).
Let \(  \Set{X}  \), \(  \Set{Y}  \), and \(  \Set{U} \subseteq \Set{X}  \) be sets, and let \(  f : \Set{X} \to \Set{Y}  \).
The image of \(  \Set{U}  \) under \(  f  \) is denoted by
\(  f[ \Set{U} ] \subseteq \Set{Y}  \).
The natural logarithm is denoted by \(  \log : \R \to \R  \).
The indicator function, \(  \I  \), is given for a true/false condition, \(  C  \), by \(\I( C )
  = 0\) if \(C\) is false and \(\I( C )
  = 1\) if \(C\) is true,
where this notation extends to vectors by applying it \entrywise.
The sign function, \(  \SignOp : \R \to \{ -1, 1 \}  \), simply returns the \(  \pm  \) sign of the input, i.e., \(\Sign( a )
  = +1\) if and only if \(a\geq 0\),
 for \(  a \in \R  \) with this notation extending to vectors by applying it \entrywise.
For \(  s \in \Z_{+}  \), \emph{the \topk[s] thresholding operation}, written \(  \Threshold{s} : \R^{\n} \to \R^{\n}  \), maps \(  \Vec{v} \mapsto \Threshold{s}( \Vec{v} )  \), where \(  \Threshold{s}( \Vec{v} )  \) retains the largest magnitude entries in \(  \Vec{v}  \) and sets all other entries to \(  0  \) with ties broken arbitrarily.
Similarly, for a set \(  J \subseteq [\n]  \), define the \emph{subset thresholding operation}, denoted by \(  \ThresholdSet{J} : \R^{\n} \to \R^{\n}  \), to be the map which takes a vector \(  \Vec{v} \in \R^{\n}  \) to a vector with \(  j\Th  \) entries \(  \ThresholdSet{J}( \Vec{v} )_{j} = \Vec*{v}_{j} \I( j \in J )  \), \(  j \in [\n]  \).
Note that the latter thresholding operation is a linear transformation given by
\(  \ThresholdSet{J}( \Vec{v} ) = \Diag( \BVec{J} ) \Vec{v}  \)
for \(  \Vec{v} \in \R^{\n}  \).
%
\par 
\par 
%
Denote by \(  X \sim \Distr{D}  \) a random variable \(  X  \), which follows a distribution \(  \Distr{D}  \). If \(\Set{S}\) is a set then  \(  X \sim \Set{S}  \) means \(X\) follows the uniform distribution over  \(  \Set{S}  \).
Additionally, the density function and \MGF* (\MGF) (when well-defined) of a random variable, \(  X  \), are written \(  \pdf{X}  \) and \(  \mgf{X}  \), respectively.
%
\par 
%
\paragraph{\bf Binary GLMs} 
Throughout this manuscript,
\(  \n, \k, \m \in \Z_{+}  \)
denote, in order, the dimension of the parameter space, the sparsity, and the number of samples (or measurements), and
the \errorrate is denoted by
\(  \epsilonX \in (0,1)  \).
The parameter space is written as
\(  \ParamSpace = \SparseSphereSubspace{\k}{\n} \subseteq \R^{\n}  \).
Note that the results in this manuscript extend to the dense regime by taking
\(  \k = \n  \) and \(  \ParamSpace = \Sphere{\n}  \).
The covariates are \(  \n  \)-variate \iid standard Gaussian random vectors, written as
\(  \CovV\VIx{1}, \dots, \CovV\VIx{\m} \sim \N( \Vec{0}, \Id{\n} )  \),
which are stacked up into the covariate matrix,
\(  \CovM = ( \CovV\VIx{1} \,\cdots\, \CovV\VIx{\m} )^{\T} \in \R^{\m \times \n}  \).
The unknown parameter vector which is being estimated is denoted by
\(  \thetaStar \in \ParamSpace  \),
and the \(  \Iter\Th  \) approximations produced by the \(  \Iter\Th  \) iteration of the recovery algorithms are written as
\(  \thetaHat[\Iter] \in \ParamSpace  \),
where \(  \Iter \in \Z_{\geq 0}  \).
There is assumed access to the covariates, \(  \CovV\VIx{\iIx}  \), \(  \iIx \in [\m]  \), as well as \(  \m  \) binary measurement responses specified as the vector
\(  \RespV \in \{ -1,1 \}^{\m}  \).
%

For a function,
\(  \pFn : \R \to [0,1]  \),
the \(  \iIx\Th  \) measurement responses, \(  \RespV*\VIx{\iIx} \in \{ -1,1 \}  \), \(  \iIx \in [\m]  \), are obtained through a random function \(  \fFn : \R \to \{ -1,1 \}  \), given by
\begin{gather}
\label{eqn:notations:f:def}
  \fFn( z )
  =
  \begin{cases}
  -1 ,&\cWP 1 - \pFn( z )  ,\\
  \+1 ,&\cWP \pFn( z ) ,
  \end{cases}
\end{gather}
for \(  z \in \R  \),
such that
\(  \RespV*\VIx{\iIx} = \fFn( \langle \CovV\VIx{\iIx}, \thetaStar \rangle )  \),
where the notation of \(  \fFn  \) extends to vectors, i.e., \(  \fFn : \R^{\m} \to \{ -1,1 \}^{\m}  \), by applying it \entrywise independently so that the response vector is given concisely by
\(  \RespV = \fFn( \CovM \thetaStar )  \).

\begin{definition}[``Noise'' and ``Slope'']
    There are two important quantities related to the function \(  \pFn  \), which are concisely represented as the variables
\(  \alphaX >0\) that measures the amount of ``noise'' in the random function $f$ compared to the $\Sign$ function; and \( \gammaX > 0  \) that measures the average ``slope'' of the function:
\begin{gather}
\label{eqn:notations:alpha:def}
  \alphaX
  \defeq
  \Pr(
    \fFn( Z ) \neq \Sign( Z )
  )
  ,\\
\label{eqn:notations:gamma:def}
  \gammaX
  \defeq
  \E[ Z \fFn( Z ) ]
,\end{gather}
where \(  Z \sim \N(0,1)  \) is a standard univariate Gaussian random variable and the probabilities and expectations are with respect to $Z$ and the randomness of $f$. Note that, \(\gammaX
  \leq
  \E[|Z|] = \sqrt{\frac{2}{\pi}} \). From Stein's lemma, if the function $p$ is differentiable then \[\gamma = 2\E[p'(Z)], \]
where the expectation is now with respect to $Z$.
\end{definition}

In addition, for \(  \epsilon > 0  \), define
\begin{gather}
\label{eqn:notations:alpha_0:def}
  \alphaO \defeq \max \{ \alphaX, \frac{\epsilonX}{\frac{3}2(5+\sqrt{21})} \}
.\end{gather}
Note that when \(  \pFn(-z) = 1-\pFn(z)  \)---as is the case in logistic and probit regression (\see the formal definitions of these models in \DEFINITIONS \ref{def:p:logistic-regression} and \ref{def:p:probit} below)---the expression for \(  \alphaX  \) simplifies to
\begin{gather*}
  \alphaX
  =
  \frac{1}{\sqrt{2\pi}}
  \int_{\zX=0}^{\zX=\infty}
  e^{-\frac{1}{2} \zX^{2}}
  (\pExpr{\zX})
  d\zX
  =
  \sqrt{\frac{2}{\pi}}
  \int_{\zX=0}^{\zX=\infty}
  e^{-\frac{1}{2} \zX^{2}}
  \pFn( -\zX )
  d\zX
.\end{gather*}
%
In addition, the function \(  \pFn  \) must satisfy two assumptions, stated together in \ASSUMPTION \ref{assumption:p}, below.
%
\begin{assumption}
\label{assumption:p}
The following conditions are enforced on \(  \pFn  \):
\Enum[\label{condition:assumption:p:i}]{i} \(  \pFn  \) monotonically increases over the real line; and
\Enum[\label{condition:assumption:p:ii}]{ii} let $\nu(z) \equiv \pExpr{\zX}$; the function \begin{gather}
\label{eqn:assumption:p:ii}
\frac{\nu(\zX+\wX)}{\nu(\zX)}
 \end{gather}
is non-increasing in $\zX \geq 0,$ for any \(  \wX > 0  \).
\end{assumption}
Intuitively, the second condition means that the ``noise'' of the GLM defined above decreases at a faster rate away from ``margin'' ($z=0)$. Indeed, $\nu(z) = P(f(\zX)=-1)+P(f(-\zX)=1)$ for $\zX \ge 0$ can be thought of as a proxy for the noise with respect to the $\Sign$ function, and the ratio in \eqref{eqn:assumption:p:ii} quantifies the growth-rate of the function.

Here, it is worth noting---and later, it will be proved---that \ASSUMPTION \ref{assumption:p} is satisfied by two ubiquitous models in binary classification and statistical modeling with binary outcomes: logistic and probit regression.
As these two models will be studied later on, the functions, \(  \pFn  \), corresponding with these models are formally defined below in \DEFINITION \ref{def:p:logistic-regression} and \ref{def:p:probit}.
To provide greater generality with these models, these definitions introduce an addition parameter: \(  \betaX \GTR 0  \), which denotes the \betaXnamelr and signal-to-noise ratio (\betaXnamepr) in logistic and probit regression, respectively.
%
\begin{definition}
\label{def:p:logistic-regression}
For {\bf logistic regression} with \betaXnamelr \(  \betaX \GTR 0  \), the function
\(  \pFn : \R \to [0,1]  \)
is given at \(  z \in \R  \) by
\begin{gather}
  \pFn( z )
  =
  \frac{1}{1+e^{-\betaX z}}
.\end{gather}
\end{definition}
%
\begin{definition}
\label{def:p:probit}
For the {\bf probit model} with signal-to-noise ratio (SNR) \(  \betaX \GTR 0  \), the function
\(  \pFn : \R \to [0,1]  \)
is given at \(  z \in \R  \) by
\begin{gather}
  \pFn( z )
  =
  \frac{1}{\sqrt{2\pi}}
  \int_{u=-\infty}^{u=\betaX z}
  e^{-\frac{1}{2} u^{2}}
  du
.\end{gather}
Note that, equivalently, \(  \pFn  \) is simply the distribution function of a standard Gaussian random variable composed with multiplication by \(  \betaX  \).
\end{definition}


\section{Main Results\label{outline:main-result}}
The main result for the convergence of \ALGORITHM \ref{alg:biht:normalized} to the correct solution is stated below as \THEOREM \ref{thm:approx-error:sparse}, whose proof is overviewed in \SECTION \ref{outline:main-result|outline-of-pf} and presented in full in \APPENDIX \ref{outline:pf-main-result}.
Its analog for the dense parameter regime, when \(  \k = \n  \) and \(  \ParamSpace = \Sphere{\n}  \), is provided as \COROLLARY \ref{corollary:approx-error:dense}.
Additionally, the specializations of the main result to logistic and probit regression are presented below in \COROLLARY \ref{corollary:approx-error:logistic-regression} 
which is also proved in \APPENDIX \ref{outline:pf-main-result}.
Essentially, these theorems say that for a sufficiently large number of samples, \(  \m  \), the approximations produced by \ALGORITHM \ref{alg:biht:normalized} converge to the \(  \epsilonX  \)-ball around the true parameter, \(  \thetaStar \in \ParamSpace  \).

\begin{theorem}
\label{thm:approx-error:sparse}
%
Fix
\(  \n, \k, \m \in \Z_{+}  \), \(  \k \leq \n  \),
and
\(  \epsilonX, \rhoX \in (0,1)  \).
Write
\(  \alphaO  \defeq \max \{ \alphaX, \frac{\epsilonX}{\frac{3}2(5+\sqrt{21})} \}  \)
as in \EQUATION \eqref{eqn:notations:alpha_0:def}.
Let
\(  \ParamSpace = \SparseSphereSubspace{\k}{\n}  \),
and fix
\(  \thetaStar \in \ParamSpace  \)
as the unknown parameter.
Fix \(  \m  \) \iid standard Gaussian covariates,
\(  \CovV\VIx{1}, \dots, \CovV\VIx{\m} \sim \N( \Vec{0}, \Id{\n} )  \),
and let
\(  \CovM = ( \CovV\VIx{1} \, \cdots \, \CovV\VIx{\m} )^{\T}  \)
be the covariate matrix.
For a number of samples
\begin{align}
\nonumber
  \m
  &=
  \mOEXPRS{\epsilonX}[,]
  \\
\label{eqn:approx-error:sparse:m}
.\end{align}
of the model specified in \EQUATION \eqref{eqn:notations:f:def} and under \ASSUMPTION \ref{assumption:p}, with probability at least \(  1-\rhoX  \), the sequence of approximations,
\(  \{ \thetaHat[\Iter] \in \ParamSpace \}_{\Iter \in \Z_{\geq 0}}  \),
produced by \ALGORITHM \ref{alg:biht:normalized} with the covariate matrix \(  \CovM  \) converges to the \(  \epsilonX  \)-ball around \(  \thetaStar  \) such that
\begin{gather}
\label{eqn:approx-error:sparse:asymptotic}
  \lim_{\Iter \to \infty} \| \thetaStar-\thetaHat[\Iter] \|_{2}
  \leq
  \epsilonX
,\end{gather}
with the rate of convergence upper bounded at each \(  \Iter\Th  \) iteration, \(  \Iter \in \Z_{\geq 0}  \) by
\begin{gather}
\label{eqn:approx-error:sparse:iterative}
  \| \thetaStar-\thetaHat[\Iter] \|_{2}
  \leq
  2^{2^{-\Iter}}
  \epsilonX^{1-2^{-\Iter}}
.\end{gather}
\end{theorem}

There are a few special cases of interest, formalized below as \COROLLARIES \ref{corollary:approx-error:dense}--\ref{corollary:approx-error:logistic-regression}.
Unless stated otherwise, the following corollaries to \THEOREM \ref{thm:approx-error:sparse} use notations consistent with the theorem.
As the first of these, \COROLLARY \ref{corollary:approx-error:dense} takes a look at the dense parameter regime.

\begin{corollary}
\label{corollary:approx-error:dense}
Let
\(  \ParamSpace = \Sphere{\n}  \).
Then, under \ASSUMPTION \ref{assumption:p}, the convergence guarantees for \ALGORITHM \ref{alg:biht:normalized} stated in \EQUATIONS \eqref{eqn:approx-error:sparse:asymptotic} and \eqref{eqn:approx-error:sparse:iterative} of \THEOREM \ref{thm:approx-error:sparse} hold for a number of samples
\begin{align}
\nonumber
  \m
  &=
\mOEXPRD{\epsilonX}[.]
\end{align}
\end{corollary}

We now proceed to applications of the main result to two well-studied GLMs: logistic and probit regression models.
For these models, not only can \THEOREM \ref{thm:approx-error:sparse} be shown to be valid, but also closed-form bounds on the sample complexity in the theorem can be obtained, as per \COROLLARY \ref{corollary:approx-error:logistic-regression} 
below.
For conciseness, only \orderwise sample complexities are presented in the corollary's statement, while precise bounds on the sample complexity are specified in its proof in \APPENDIX \ref{outline:pf-main-result|pf-main-corollaries}.

\begin{corollary}
\label{corollary:approx-error:logistic-regression}
%
When \(  \pFn  \) is the logistic function with \betaXnamelr \(  \betaX \geq 0  \), as in \DEFINITION \ref{def:p:logistic-regression} (or the probit regression with \betaXnamepr \(  \betaX \geq 0  \) defined as in \DEFINITION \ref{def:p:probit}), there exist positive constants, \(  b_{1}, b_{2} > 0  \), such that the convergence guarantees for \ALGORITHM \ref{alg:biht:normalized} stated in \EQUATIONS \eqref{eqn:approx-error:sparse:asymptotic} and \eqref{eqn:approx-error:sparse:iterative} of \THEOREM \ref{thm:approx-error:sparse} hold if
\begin{align}
\label{eqn:corollary:approx-error:logistic-regression:m:sparse}
   \m &= \mOEXPRLR[s]{\epsilonX}[.]
\end{align}
Note that the constants \(  b_{1}, b_{2} > 0  \) can be different for the logistic and probit cases.
\end{corollary}
\begin{remark}
\label{remark:}
%
In \COROLLARY \ref{corollary:approx-error:logistic-regression}, \ALGORITHM \ref{alg:biht:normalized} achieves the \orderwise optimal sample complexity (up to logarithmic factors) for parameter estimation in logistic regression under the Gaussian design.
\See \cite{hsu2024sample} for the establishment of the optimal sample complexity.
\end{remark}


  \subsection{Overview of the Proof of the Main Result\label{outline:main-result|outline-of-pf}}

\checkoffbutmayberecheck%
While the formal proof of the main theorem, \THEOREM \ref{thm:approx-error:sparse}, is deferred to \APPENDIX \ref{outline:pf-main-result|pf-main}, the arguments are outlined here.
(Meanwhile, \COROLLARY \ref{corollary:approx-error:logistic-regression} 
is proved in \APPENDIX \ref{outline:pf-main-result|pf-main-corollaries} but not outlined here.)
This proof resembles the approach in \cite{matsumoto2022binary}, but some important differences are necessary to handle the randomness introduced into the responses. 
In particular, the analysis in this work relies on the normalization in \STEP \ref{alg:biht:normalized:projection} of \ALGORITHM \ref{alg:biht:normalized} in order to reduce the error induced by \(  \fFn  \).
However, this feature of the analysis will not be apparent until the technical proofs.
%
\par 
%
To facilitate this overview, as well as the upcoming formal analysis, the following notations are defined.
For \(  \Vec{u}, \Vec{v} \in \R^{\n}  \) and \(  \JCoords \subseteq [\n]  \), let
\begin{gather}
  \label{eqn:notations:h:def}
  \hFn( \Vec{u}, \Vec{v} )
  \defeq
  \frac{\sqrt{2\pi}}{\m}
  \sep
  \CovM^{\T}
  \sep
  \frac{1}{2}
  \left( \Sign( \CovM \Vec{u} ) - \Sign( \CovM \Vec{v} ) \right)
  ,\\ \label{eqn:notations:hJ:def}
  \hFn[\JCoords]( \Vec{u}, \Vec{v} )
  \defeq
  \ThresholdSet{\Supp( \Vec{u} ) \cup \Supp( \Vec{v} ) \cup \JCoords}(
  \hFn( \Vec{u}, \Vec{v} )
  )
  ,\\ \label{eqn:notations:hf:def}
  \hfFn( \Vec{u}, \Vec{v} )
  \defeq
  \frac{\sqrt{2\pi}}{\m}
  \sep
  \CovM^{\T}
  \sep
  \frac{1}{2}
  \left( \fFn( \CovM \Vec{u} ) - \Sign( \CovM \Vec{v} ) \right)
  ,\\ \label{eqn:notations:hfJ:def}
  \hfFn[\JCoords]( \Vec{u}, \Vec{v} )
  \defeq
  \ThresholdSet{\Supp( \Vec{u} ) \cup \Supp( \Vec{v} ) \cup \JCoords}(
  \hfFn( \Vec{u}, \Vec{v} )
  )
.\end{gather}

\subsubsection{Outline of the Proof}
\label{outline:main-result|outline-of-pf|outline}

The proof of \THEOREM \ref{thm:approx-error:sparse} is outlined as follows.
\checkoff%

\begin{enumerate}[nolistsep,leftmargin=*]
\item \label{enum:outline-pf-main:1}
The error of the \(  0\Th  \) approximation,
\(  \thetaHat[0] \sim \ParamSpace  \),
produced by BIHT is clearly bounded from above by the diameter of the unit sphere \(  \Sphere{\n}  \), i.e., no more than \(  2  \).
\item \label{enum:outline-pf-main:2}
The vast majority of the work thus falls onto analyzing any subsequent \(  \Iter\Th  \) approximation, \(  \thetaHat[\Iter] \in \ParamSpace  \), \(  \Iter \in \Z_{+}  \).
For this, the analysis is divided into establishing two main bounds:
\Enum[{\label{enum:outline-pf-main:2:i}}]{i}
a deterministic bound on the error of the \(  \Iter\Th  \) approximation obtained from BIHT (\see \LEMMA \ref{lemma:error:deterministic}), and
\Enum[{\label{enum:outline-pf-main:2:ii}}]{ii}
a probabilistic bound, which is amounts to a {\em restricted  invertibility} property that holds for Gaussian matrices with high probability (\see \THEOREM \ref{thm:main-technical:sparse}).
Then, these bounds are combined into the convergence guarantees for BIHT stated in the main theorem.
\item \label{enum:outline-pf-main:3}
Regarding the first bound, \ref{enum:outline-pf-main:2:i}, it can be shown that the error of the \(  \Iter\Th  \) approximation obtained via \ALGORITHM \ref{alg:biht:normalized} is bounded from above by (\see \LEMMA \ref{lemma:error:deterministic} and \ALGORITHM \ref{alg:biht:normalized})
\begin{gather*}
  \| \thetaStar - \thetaHat[\Iter] \|_{2}
  =
  \BigO(
  \left\|
    \thetaStar
    -
    \frac
    {\ThresholdSet*{\TS}( \thetaHat[\Iter-1] + \hfFn( \thetaStar, \thetaHat[\Iter-1] ) )}
    {\| \ThresholdSet*{\TS}( \thetaHat[\Iter-1] + \hfFn( \thetaStar, \thetaHat[\Iter-1] ) ) \|_{2}}
  \right\|_{2}
  ).
\end{gather*}

\item \label{enum:outline-pf-main:4}
For the second bound, \ref{enum:outline-pf-main:2:ii}, stated in \STEP \ref{enum:outline-pf-main:2}, a variant of the restricted approximate invertibility condition that appeared in \cite[{\DEFINITION 3.1}]{matsumoto2022binary} is established for Gaussian matrices when the number of rows in the covariate matrix (alternatively, the number of covariates or measurements), \(  \m  \), is sufficiently large.
More precisely, Gaussian matrices are shown to have the property that
\begin{gather*}
  \left\|
    \thetaStar
    -
    \frac
    {\thetaXX + \hfFn[\JCoords]( \thetaStar, \thetaXX )}
    {\| \thetaXX + \hfFn[\JCoords]( \thetaStar, \thetaXX ) \|_{2}}
  \right\|_{2}
  =
  \BigO(
  \sqrt{\epsilonX \| \thetaStar-\thetaXX \|_{2}}
  +
  \epsilonX
  )
\end{gather*}
uniformly for all \(  \thetaXX \in \ParamSpace  \) and \(  \JCoords \subseteq [\n]  \), \(  | \JCoords | \leq \k  \), with high probability when \(  \m  \) is at least what is specified in \EQUATION \eqref{eqn:approx-error:sparse:m} (\see \THEOREM \ref{thm:main-technical:sparse}).
\item \label{enum:outline-pf-main:5}
The results of \STEPS \ref{enum:outline-pf-main:2} and \ref{enum:outline-pf-main:3} are then combined in order to upper bound the error of the \(  \Iter\Th  \) approximation by the following recurrence relation, which holds with bounded probability dictated by \THEOREM \ref{thm:main-technical:sparse}, \ie by the probability that the bound \ref{enum:outline-pf-main:2:ii} holds:
\begin{gather*}
  \| \thetaStar - \thetaHat[0] \|_{2} \leq 2
  ,\\
  \| \thetaStar - \thetaHat[\Iter] \|_{2}
  =
  \BigO(
  \sqrt{\epsilonX \| \thetaStar - \thetaHat[\Iter-1] \|_{2}}
  +
  \epsilonX
  )
  ,\qquad
  \Iter \in \Z_{+}
.\end{gather*}
\item \label{enum:outline-pf-main:6}
Per \FACT \ref{fact:recurrence}, the above recurrence relation is \pointwise bounded from above to yield the rate of convergence and, 
consequently,
the asymptotic convergence in the limit as \(  \Iter \to \infty  \)
of the approximations produced by BIHT:
\begin{gather*}
  \| \thetaStar - \thetaHat[\Iter] \|_{2} \leq 2^{2^{-\Iter}} \epsilonX^{1-2^{-\Iter}}
  ,\quad \Iter \in \Z_{\geq 0}
  ,\\
  \lim_{\Iter \to \infty} \| \thetaStar - \thetaHat[\Iter] \|_{2} \leq \epsilonX
,\end{gather*}
completing the proof of the main theorem.
\end{enumerate}
 \section{Restricted Approximate Invertibility of GLMs}
\label{outline:main-technical-result}

The crux of the analysis for the convergence of \ALGORITHM \ref{alg:biht:normalized} is a variant of the restricted approximate invertibility conditions (\RAICs) studied in \cite{friedlander2021nbiht,matsumoto2022binary}, which is established for Gaussian matrices in \THEOREM \ref{thm:main-technical:sparse}.
The formal proofs of these technical results, which constitute the primary technical contributions of this work, are located in \APPENDIX \ref{outline:pf-main-technical-result} and overviewed in \SECTION \ref{outline:main-technical-result|outline-of-pf|outline}.

%

%

The main technical theorem will be formalized next.
\begin{theorem}
\label{thm:main-technical:sparse}
%
Fix
\(  \n, \k, \m \in \Z_{+}  \), \(  \k \leq \n  \), and \(  \rhoX, \deltaX \in (0,1)  \) where
\(
  \deltaX \defeq \frac{\epsilonX}{\frac{3}{2} ( 5+\sqrt{21} )}.
\)
Write
\(  \alphaO = \alphaO( \deltaX ) \defeq \alphaOExpr[\deltaX]  \)
as in \EQUATION \eqref{eqn:notations:alpha_0:def}.
Let
\(  \ParamSpace = \SparseSphereSubspace{\k}{\n}  \),
and fix
\(  \thetaStar \in \ParamSpace  \).
Under \ASSUMPTION \ref{assumption:p}, for a number of samples 
\begin{align}
\nonumber
  \m
  &=
  \mOEXPRS{\deltaX}[,]
  \\
\label{eqn:main-technical:sparse:m}
\end{align}
 with probability at least \(  1-\rhoX  \), uniformly for all \(  \thetaXX \in \ParamSpace  \) and all \(  \JCoords \subseteq [\n]  \), \(  | \JCoords | \leq \k  \),
\begin{gather}
\label{eqn:main-technical:sparse:1}
  \left\|
    \thetaStar
    -
    \frac
    {\thetaXX + \hfFn[\JCoords]( \thetaStar, \thetaXX )}
    {\| \thetaXX + \hfFn[\JCoords]( \thetaStar, \thetaXX ) \|_{2}}
  \right\|_{2}
  \leq
  \sqrt{\deltaX \| \thetaStar-\thetaXX \|_{2}}
  +
  \deltaX
.\end{gather}
\end{theorem}

The main technical theorem holds for logistic and probit regression---as formalized below in \COROLLARY \ref{corollary:main-technical:logistic-regression}
---because \ASSUMPTION \ref{assumption:p} is satisfied by both models.
Moreover,
closed-form bounds on the sample complexity in \THEOREM \ref{thm:main-technical:sparse} can be derived for these two
exemplary
GLMs, which are stated as \orderwise results in \COROLLARY \ref{corollary:main-technical:logistic-regression} 
with the specification of precise bounds and constants left to the proof of the corollary in \APPENDIX \ref{outline:pf-main-technical-result|pf-main-corollaries}.

\begin{corollary}
\label{corollary:main-technical:logistic-regression}
%
Let \(  \pFn  \) be the logistic function with \betaXnamelr \(  \betaX \GTR 0  \), as in \DEFINITION \ref{def:p:logistic-regression} (or the probit function with SNR \(  \betaX \GTR 0  \), as in \DEFINITION \ref{def:p:probit}).
If there exist absolute constants \(  \cO, \ConstbetaXThrsholdLR > 0  \) then for a number of samples given by
 \EQUATION \eqref{eqn:corollary:approx-error:logistic-regression:m:sparse},
 the bound stated as \EQUATION \eqref{eqn:main-technical:sparse:1} in \THEOREM \ref{thm:main-technical:sparse} holds uniformly for all \(  \thetaXX \in \ParamSpace  \) and all \(  \JCoords \subseteq [\n]  \), \(  | \JCoords | \leq \k  \), with probability at least \(  1-\rhoX  \).
\end{corollary}

\section{Conclusion}
In this paper, we  made a case for binary iterative hard thresholding, (projected) gradient descent on the ReLU loss, as a universal learning algorithm for classification tasks. Under very general models of nonseparable (and separable) data, that include logistic, probit, and random classification noise models, BIHT is statistically optimal in parameter estimation. We observe this in practice as well. 

We have restricted ourselves to Gaussian covariates, for which our results are tight. However it will be worth exploring the performance of BIHT for more general classes of distributions. We also note an observation that contrasts the noiseless case from generalized linear models, as far as the dynamics of BIHT is concerned. It is known that the normalization step in BIHT, though essential for the convergence proof of \cite{friedlander2021nbiht,matsumoto2022binary}, can be redundant in practice for the noiseless case, see Figure~\ref{fig:error-decay-nbiht-biht-comparison-beta=1-plot:sign} in the appendix. On the other hand, for the setting of this paper, the normalization seems to be crucial for the stability of the algorithm especially in the high-noise regime (as can be seen in Figure~\ref{fig:error-decay-nbiht-biht-comparison-beta=1-plot:logistic-regression} in the appendix). 

Finally, it will  be interesting to analyze BIHT (and a stochastic perceptron-like version of it) from a learning theory perspective, especially in the agnostic setting, where data not necessarily comes from a GLM. Other noise models, such as Massart noise, can also be interesting.

\bibliography{refs}


\appendix

\let\SECTIONREF\SECTION
\let\SECTIONSREF\SECTIONS
\let\SECTION\APPENDIX
\let\SECTIONS\APPENDICES
\section{Comparison to Prior Works}
\label{outline:intro|>contributions|>comparison}

A special case of the projected gradient decent algorithm studied by \cite{bahmani2016learning} offers an alternative gradient-based method to BIHT for parameter estimation in some classes of GLMs, including those with nondecreasing, Lipschitz transfer functions, \(  \linkFn^{-1}  \).
Note that this class  encompasses GLMs whose responses
follow
an exponential distribution, which is one of the most widely studied families of GLMs.
Like BIHT, this algorithm projects its approximations onto the set of \(  \k  \)-sparse vectors, where the sparsity can be taken as \(  \k=\n  \) in the dense parameter regime so as to effectively eliminate the projection step of the algorithm.
In the dense parameter regime, this algorithm becomes the perceptron-like \emph{\GLMtron} algorithm of the earlier work, \cite{kakade2011efficient}, which learns GLMs with (possibly nonstrictly) monotonically increasing, Lipschitz transfer functions.
For concise nomenclature, we will borrow the name of ``\GLMtronX'' to refer to the original \GLMtron algorithm with the addition of a sparse projection, as analyzed in \cite{bahmani2016learning}.
%
\par 
%
BIHT and \GLMtronX have a few key differences.
While other covariate designs are possible, the analysis in \cite{bahmani2016learning} assumes that the norm of each covariate is almost surely at most \(  1  \), in contrast to the Gaussian covariate design considered in this present work.
As another distinction
between BIHT and \GLMtronX, BIHT requires that the GLM has binary outcomes with a mild condition on the link function, \(  \linkFn  \), but otherwise need not know the specific choice of link function, while \GLMtron can learn a larger class of GLMs that only necessitates that the transfer function, \(  \linkFn^{-1}  \), satisfies a certain derivative condition (which indeed holds when the transfer function is nondecreasing and Lipschitz).
However, unlike BIHT, \GLMtronX requires knowledge of the specific choice of link function.
(Note that \cite{kakade2011efficient} proposes a second algorithm for learning single-index models which estimates an unknown link function, but this is outside the scope of this work.)
%
%
Most significantly, BIHT and \GLMtronX ``try to'' minimize different objective functions.
Whereas BIHT performs gradient descent on the (negative) ReLU loss,
\begin{gather*}
  \Jbiht( \thetaX )
  =
  \sum_{\iIx=1}^{\m}
  | [ \RespV*_{\iIx} \langle \CovV\VIx{\iIx}, \thetaX \rangle ]_{-} |
,\end{gather*}
by taking gradient steps in the negated direction of
\begin{gather*}
  \nabla_{\thetaX}
  \Jbiht( \thetaX )
  \ni
  -\CovM^{\T}
  \sep
  \frac{1}{2}
  \left(
    \RespV - \Sign{} \big( \CovM \thetaX \big)
  \right)
,\end{gather*}
\GLMtronX is a gradient descent procedure on the loss
\begin{gather*}
  \Jglmtron( \thetaX )
  =
  \sum_{\iIx=1}^{\m}
  \intlinkFn( \langle \CovV\VIx{\iIx}, \thetaX \rangle )
  -
  \RespV*_{\iIx}
  \langle \CovV\VIx{\iIx}, \thetaX \rangle
\end{gather*}
with gradient steps in the negated direction of
\begin{gather*}
  \nabla_{\thetaX}
  \Jglmtron( \thetaX )
  =
  -\CovM^{\T}
  \left(
    \RespV - \linkFn^{-1} \big( \CovM \thetaX \big)
  \right)
,\end{gather*}
where the function, \(  \intlinkFn  \), is defined such that
\(  \linkFn^{-1} = \partial \intlinkFn  \).
When the responses,
\(  \RespV*_{\iIx} \Mid| \CovV\VIx{\iIx}  \),
\(  \iIx \in [\m]  \),
follow an exponential distribution, \(  \Jglmtron  \) becomes the negative log-likelihood function, and hence, roughly speaking, \GLMtronX essentially ``tries to'' compute the MLE in this case.
The difference in objective functions is fundamental:
it precludes the application of the analysis for BIHT in this work to \GLMtronX.
Conversely, adapting the approach in \cite{bahmani2016learning} is insufficient to achieve the sample complexity established for BIHT here.
%
\par 
%
Assuming that the responses, \(  \RespV*_{\iIx}  \), \(  \iIx \in [\m]  \), are bounded, as is the case for binary GLMs, \cite{bahmani2016learning} shows that \GLMtronX achieves an \errorrate of \(  \epsilonX  \) provided the number of covariates, \(  \m  \), is at least
\begin{gather*}
  \m = \BigO'( \max \left\{ \frac{1}{\epsilonX^{4}}, \k \log \left( \frac{\n}{\k} \right) \right\} )
,\end{gather*}
where this hides some terms.
Notice that the dependency on the \errorrate, \(  \epsilonX  \), is \(  \epsilonX^{-4}  \) compared to \(  \epsilonX^{-2}  \) obtained in this work for BIHT, though in fairness, neither result should be considered ``superior'' to the other since, although this work obtains a smaller dependence on \(  \epsilonX  \) and need not know the link function, the analysis \cite{bahmani2016learning} applies to a larger class of GLMs.
%
\par 
%
A generalized version of the popular ``LASSO'' algorithm has been proposed for GLMs in \cite{plan2016generalized}.
Our results for BIHT are analogous to their results for parameter estimation with LASSO; in particular their sample complexity is \(  \BigO'( \frac{\k {\log( \frac{\n}{\k} )}}{\epsilonX^{2}} ) \); and their results also depends on properties of the link function, namely a scaling factor and noise variance, the former being same as the quantity
 \(  \gammaX  \) as defined in Equation \eqref{eqn:notations:gamma:def}. Subsequently, a very precise error-analysis for generalized LASSO has been performed in \cite{thrampoulidis2015lasso}. To  compare with \THEOREM \ref{thm:approx-error:sparse}, their sample complexity is \(  \BigO'( \frac{( 1-\gammaX^{2} ) \k {\log( \frac{\n}{\k} )}}{\epsilonX^{2}} ) \) (see, \cite[Equation (8)]{thrampoulidis2015lasso}); however, due to the differences in assumption and applicability as explained above, one should exercise caution in such comparisons.
%
\par 
%
A few other prior works are worth remarking on.
In the following discussion, the parameter, \(  \thetaStar  \), is assumed to have unit norm, but the models incorporate SNR denoted by \(  \betaX \GTR 0  \).
Formulating the (sparse) estimation problem as a convex program, \cite{plan2012robust} shows that the estimation of \(  \thetaStar  \) from binary responses is possible with
\(  \BigO'( \frac{\k {\log( \frac{\n}{\k} )}}{\min \{ \betaX^{2}, 1 \} \epsilonX^{4}} ) \)
samples under the Gaussian covariate design.
\cite{plan2017high} improves this sample complexity to
\( \BigO'( \frac{\k {\log( \frac{\n}{\k} )}}{\min \{ \betaX^{2}, 1 \} \epsilonX^{2}} )\)
using a method that effectively amounts to the ``Average'' algorithm of \cite{servedio1999pac}.
%
\par 
%
Subsequently, in the case of logistic regression with Gaussian covariates, maximum likelihood estimators and their regularized versions have  recently received renewed attention. In this regard, \cite{sur2019modern} and \cite{salehi2019impact} are notable; however their precise asymptotic results are  given in terms of solutions of a system of equations, and are not directly comparable to our sample complexity results. On the other hand, recently
\cite{hsu2024sample} established a lower bound on sample complexity for this case via a variant of Fano's inequality that is \orderwise tight (up to logarithmic factors) when \(  \betaX \leq 1  \).
\cite{hsu2024sample} additionally obtains \orderwise tight (again, up to logarithmic factors) bounds on the sample complexity in logistic regression with the dense parameter space (when \(  \k=\n  \)) for any \(  \betaX  \), summarized in Equation~\eqref{eqn:corollary:approx-error:logistic-regression:m:sparse}.
Very recently, ~\cite{chardon2024finite} show that for the $\k= \n$ case MLE achieves optimal sample complexity for $\betaX = \Omega(1).$
While \cite{plan2012robust,plan2017high}
pair
their sample complexity bounds with efficient algorithms for the Gaussian design, \polytime algorithms achieving the optimal sample complexity for logistic regression in the \(  \betaX > 1  \) regimes were not known in general. 
This work settles this question by proving that BIHT is a computationally efficient algorithm that in fact simultaneously achieves the \orderwise optimal sample complexity (up to logarithmic factors) for all choices of \(  \betaX  \) even with the sparsity constraint.
Here, it is worth remarking that, although we are not aware of a lower bound on the sample complexity for probit regression in the literature, our result for probit model shares the same order sample complexity (which we believe to be tight) as our result for logistic regression.
It should be noted that, for the probit model in the the non-sparse $k=d$ case, when restricted to \(  \betaX > 1  \),
the same sample complexity (up to log factors)  is also achieved by \cite{kuchelmeister2024finite}.

\subsection{Comparison to \cite{matsumoto2022binary,matsumoto2024robust}}
\label{outline:intro|>contributions|>comparison-biht}

Although numerous prior works have studied BIHT, \eg \cite{friedlander2021nbiht,jacques2013quantized,jacques2013robust,liu2019one,plan2017high}, the works  most closely aligned with the analysis in this manuscript are \cite{matsumoto2022binary,matsumoto2024robust}, and indeed, some elements of the approach in this work are analogous to components of the analyses in \cite{matsumoto2022binary,matsumoto2024robust}.
However, handling the GLM's randomness---introduced into the model through the function \(  \fFn  \)---requires a novel approach.
It turns out that the normalization step in each iteration of BIHT is
crucial
to obtain our bound on the approximation error, which distinguishes this work from \cite{matsumoto2022binary,matsumoto2024robust}, even despite the fact that \cite{matsumoto2024robust} considers an alternative (adversarially) noisy setting.
As discussed in Section \ref{outline:intro|>contributions|>techniques}, the analysis in this work largely centers around an invertibility condition that uniformly bounds an expression of the form
\begin{gather}
\label{eqn:contributions:3}
  \left\|
    \thetaStar
    -
    \frac{\thetaXX + \hfFn[\JCoords]( \thetaStar, \thetaXX )}{\| \thetaXX + \hfFn[\JCoords]( \thetaStar, \thetaXX ) \|_{2}}
  \right\|_{2}
\end{gather}
for all \(  \thetaXX \in \ParamSpace  \) and all \(  \JCoords \subseteq [\n]  \), \(  | \JCoords | \leq \k  \).
In contrast, \cite{matsumoto2022binary,matsumoto2024robust} consider invertibility conditions that bounds expressions of the respective forms
\begin{gather}
\label{eqn:contributions:4}
  \|
    \thetaStar
    -
    \thetaXX - \hFn[\Sign;\JCoords]( \thetaStar, \thetaXX )
  \|_{2}
  ,\qquad
  \|
    \thetaStar
    -
    \thetaXX - \hFn[f_{\mathrm{adv}};\JCoords]( \thetaStar, \thetaXX )
  \|_{2}
,\end{gather}
where the parameterizations by the \(  \Sign  \) and \(  f_{\mathrm{adv}}  \) functions can be thought of as the noiseless and adversarially noisy analogs, respectively, to \(  \fFn  \) in our setting for GLMs.
Notice that both expressions in \eqref{eqn:contributions:4} omit the sort of normalization that appears in \eqref{eqn:contributions:3}.
But following \cite{matsumoto2022binary,matsumoto2024robust} in this way turns out to be problematic when applying the analysis to GLMs:
in ``lower'' SNR regimes, it would effectively lead to an \(  \Omega(1)  \) additive term in the bound, meaning that the bound on the \errorrate for BIHT would become \(  \Omega(1)  \), rather than \(  \epsilonX  \), regardless of the number of covariates (\ie the sample complexity).
However, accounting for the normalization mitigates this issue to give the desired \(  \epsilonX  \)-\errorrate.
%
\par 
%
The intuition behind this is the following.
In expectation, the vector
\(  \thetaXX + \hfFn[\JCoords]( \thetaStar, \thetaXX )  \)
is aligned with \(  \thetaStar  \), as is the noise introduced by the GLM's randomness.
Therefore, the normalization essentially eliminates (or at least reduces) this noise, leading to the desired \errorrate.
Note that, on the other hand, if the noise was instead adversarial, as in \cite{matsumoto2024robust}, it is unlikely that accounting for such normalization in this way would help as the noise can be (adversarially) chosen to be in more or less any direction.
%
\par 
%
In fact, an empirical study with logistic regression (\see \FIGURE \ref{fig:error-decay-nbiht-biht-comparison-beta=1-plot:logistic-regression}) corroborates these observations and suggests that BIHT may exhibit different convergence and stability behaviors when it does or does not normalize its approximations, at least at ``higher'' noise levels, which is a notable distinction from observations made in the ``noiseless'' setting, where BIHT has been empirically seen to converge well (and potentially even less brittlely) when the algorithm's approximations are not normalized (\see \FIGURE \ref{fig:error-decay-nbiht-biht-comparison-beta=1-plot:sign}).
Similar empirical behavior is exhibited with probit regression, as well, but such empirical results have been omitted to avoid redundancy.
%
\par 
%
All in all, this key observation and distinguishing approach are essential for our analysis of and convergence guarantees for learning GLMs with BIHT, and may potentially even be less an artifact of our analysis and
more an inherent feature of
the algorithm itself.
\subsection{Other Related Work}
Stochastic gradient descent (SGD) \citep{robbins1951stochastic,sakrison1965efficient} offers an alternative gradient-based method for parameter estimation in GLMs.
\cite{toulis2014statistical} studies the statistical properties of SGD estimates in GLMs when updates are both explicit and implicit.
However, such SGD estimates are asymptotically sub-optimal compared to the maximum likelihood estimates, as noted by \cite{toulis2014statistical}.
As another common approach, approximate message passing (AMP) and its extensions have also been heavily applied to parameter estimation in GLMs, \eg \cite{mondelli2021approximate,venkataramanan2022estimation,zhang2024spectral,barbier2019optimal,zhu2018amp,schniter2016vector,zhao2024vector}.
This includes generalized approximate message passing (GAMP), an algorithm first proposed by \cite{rangan2011generalized}.
While the \errorrate of GAMP is \informationtheoretically optimal for some GLM's, it falls short of
the \informationtheoretical optimum for GLMs in some paradigms \citep{barbier2019optimal}.
In fact,
\cite{barbier2019optimal} characterizes the regions of the parameter space in which GAMP achieves the \informationtheoretical optimal \errorrate or is \informationtheoretically sub-optimal for GLMs.
As a relative to GAMP, another variant of approximation passing called vector approximate message passing (VAMP), introduced by \cite{rangan2019vector}, has been used for estimation in GLMs, initially by \cite{schniter2016vector} and subsequently by, \eg \cite{zhao2024vector}.

\begin{figure}
%
%
\includegraphics[width=\textwidth]{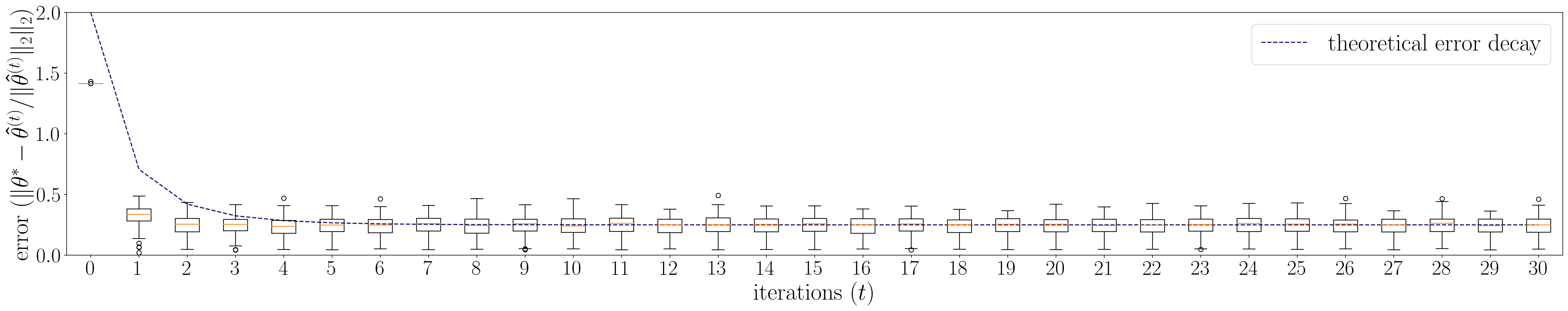}
\includegraphics[width=\textwidth]{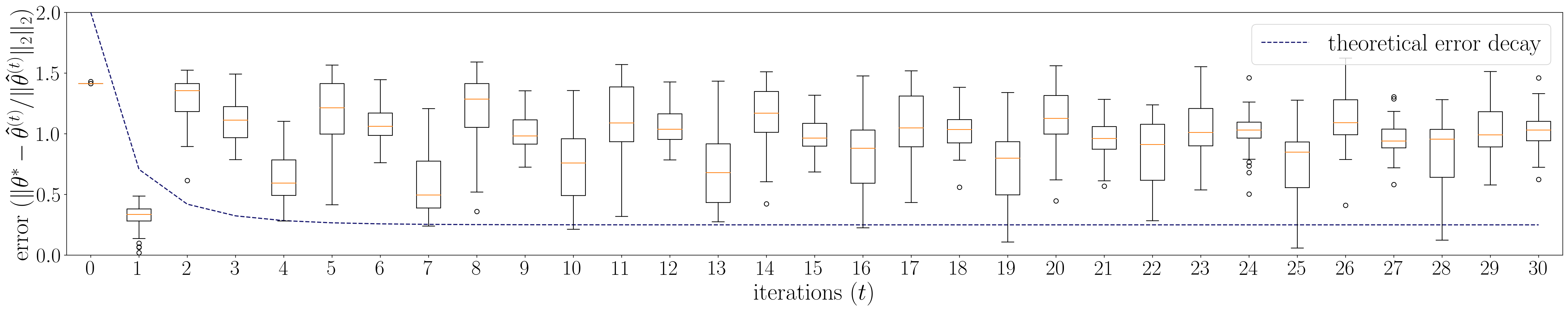}
\caption{\label{fig:error-decay-nbiht-biht-comparison-beta=1-plot:logistic-regression}
This experiment compares the iterative approximation errors for BIHT with (top plot) and without (bottom plot) the normalization step of the algorithm under logistic regression with \betaXnamelr \(  \betaX = 1  \).
The error is the \(  \lnorm{2}  \)-distance between the normalized approximation and the true parameter.
In both plots, the theoretical error decay for the normalized version of BIHT with logistic regression is displayed for reference.
The experiment ran \(  100  \) trials of recovery for \(  30  \) iterations with parameters: \(  d = 2000  \), \(  k = 5  \), \(  n = 3000  \), \(  \epsilonX = 0.25  \), and \(  \rhoX = 0.25  \).%
}
%
\end{figure}

\begin{figure}
%
%
\includegraphics[width=\textwidth]{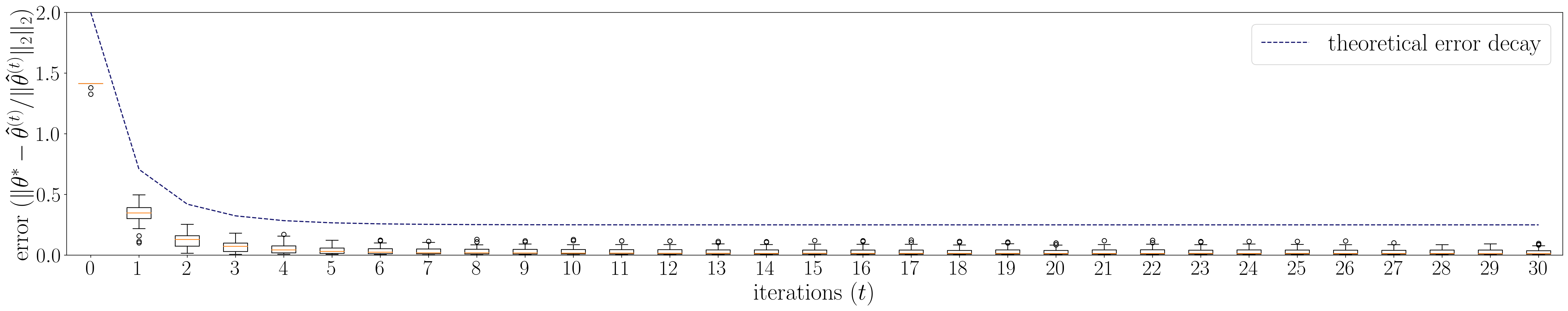}
\includegraphics[width=\textwidth]{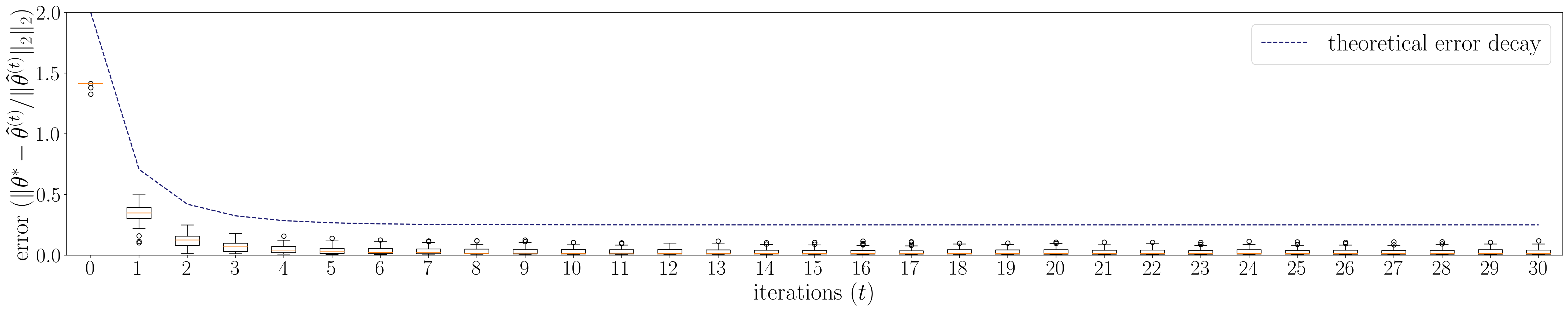}
\caption{\label{fig:error-decay-nbiht-biht-comparison-beta=1-plot:sign}
This experiment compares the iterative approximation errors for BIHT with (top plot) and without (bottom plot) the normalization step of the algorithm under the noiseless model.
The error is the \(  \lnorm{2}  \)-distance between the normalized approximation and the true parameter.
In both plots, the theoretical error decay for the normalized version of BIHT in the noiseless setting---established by \cite{matsumoto2022binary}---is displayed for reference.
The experiment ran \(  100  \) trials of recovery for \(  30  \) iterations with parameters: \(  d = 2000  \), \(  k = 5  \), \(  n = 700  \), \(  \epsilonX = 0.25  \), and \(  \rhoX = 0.25  \).%
}
%
\end{figure}
\section{Proof of the Main Results}
\label{outline:pf-main-result}

In this section, the main results---\THEOREM \ref{thm:approx-error:sparse} and \COROLLARY \ref{corollary:approx-error:logistic-regression}
---are proved, contingent on the correctness of the main technical results---\THEOREM \ref{thm:main-technical:sparse} and \COROLLARY \ref{corollary:main-technical:logistic-regression}
---and some auxiliary results, whose proofs are deferred to \SECTIONS \ref{outline:pf-main-technical-result} and \ref{outline:concentration-ineq}.


\subsection{Intermediate Results}
\label{outline:pf-main-result|intermediate}

Before \THEOREM \ref{thm:approx-error:sparse} can be proved, two auxiliary results, stated below as \LEMMA \ref{lemma:error:deterministic} and \FACT \ref{fact:recurrence}, are needed.
The first of these intermediate results---whose proof is deferred to \SECTION \ref{outline:pf-main-result|pf-intermediate}---will allow the main technical result, \THEOREM \ref{thm:main-technical:sparse}, as well as its corollaries, to be related to the error of the approximations iteratively produced by BIHT (\ALGORITHM \ref{alg:biht:normalized}).

\begin{lemma}
\label{lemma:error:deterministic}
Let
\(  \Vec{\uV} \in \SparseRealSubspace{\k}{\n}  \) and
\(  \Vec{\vV} \in \R^{\n}  \),
and let
\(  \JCoords, \JCoordsu, \JCoordsv \subseteq [\n]  \),
where
\(  | \JCoords | \leq \k  \),
\(  \JCoordsu \defeq \Supp( \Vec{\uV} )  \), and
\(  \JCoordsv \defeq \Supp( \Threshold{\k}( \Vec{\vV} ) )  \).
Then,
\begin{gather}
  \left\|
    \Vec{\uV}
    -
    \frac
    {\Threshold{\k}(\Vec{\vV})}
    {\| \Threshold{\k}(\Vec{\vV}) \|_{2}}
  \right\|_{2}
  \leq
  3
  \left\|
    \Vec{\uV}
    -
    \frac
    {\ThresholdSet{\JCoords \cup \JCoordsu \cup \JCoordsv}(\Vec{\vV})}
    {\| \ThresholdSet{\JCoords \cup \JCoordsu \cup \JCoordsv}(\Vec{\vV}) \|_{2}}
  \right\|_{2}
.\end{gather}
\end{lemma}

The proof of the main theorem will additionally utilize the following fact from \cite{matsumoto2022binary}.
The iterative approximation errors will turn out to be upper bounded by the functions in this fact, and thus, this fact will facilitate the calculation of a close-form bound on the iterative approximation errors, much like the approach in \cite{matsumoto2022binary}.

\begin{fact}[{\cite[{\FACT 4.1}]{matsumoto2022binary}}]
\label{fact:recurrence}
Let
\(  \ux, \vx, \wx \in \R_{+}  \), where
\(  \ux \defeq \frac{1}{2} ( 1 + \sqrt{1 + 4\wx} )  \) and
\(  1 \leq \ux \leq \sqrt{\frac{2}{\vx}}  \).
Let
\(  \fx, \fxx : \Z_{\geq 0} \to \R  \)
be functions given by
\begin{gather*}
  \fx( 0 ) = 2
  ,\\
  \fx( \tx ) = \sqrt{\vx \fx( \tx-1 )} + \vx \wx
  ,\quad \tx \in \Z_{+}
  ,\\
  \fxx( \tx ) = 2^{2^{-\tx}} ( \ux^{2} \vx )^{1-2^{-\tx}}
  ,\quad \tx \in \Z_{\geq 0}
.\end{gather*}
Then,
\begin{gather*}
  \fx( \tx ) > \fx( \tx' )
  ,\quad \tx < \tx' \in \Z_{\geq 0}
  ,\\
  \fxx( \tx ) > \fxx( \tx' )
  ,\quad \tx < \tx' \in \Z_{\geq 0}
  ,\\
  \fx( \tx ) \leq \fxx( \tx )
  ,\quad \tx \in \Z_{\geq 0}
  ,\\
  \lim_{\tx \to \infty} \fx( \tx ) \leq \lim_{\tx \to \infty} \fxx( \tx ) = \ux^{2} \vx
.\end{gather*}
\end{fact}


\subsection{Proof of \THEOREM \ref{thm:approx-error:sparse}}
\label{outline:pf-main-result|pf-main}

With the above results in \SECTION \ref{outline:pf-main-result|intermediate}, the convergence of the BIHT approximations, as stated in the main theorem, can now be proved.

\begin{proof}
{\THEOREM \ref{thm:approx-error:sparse}}
\checkoff%
Setting
\begin{gather}
  \deltaX
  =
  \frac{\epsilonX}{\frac{3}{2} ( 5+\sqrt{21} )}
  =
  \frac{\epsilonX}{9 \left( \frac{1}{2} ( 1+\sqrt{\frac{7}{3}} ) \right)^{2}}
,\end{gather}
and taking
\begin{gather}
  \m \geq \mEXPR[s][,]
\end{gather}
the following bound holds for all \(  \thetaXX \in \ParamSpace  \) and all \(  \JCoords \subseteq[\n]  \), \(  | \JCoords | \leq \k  \), uniformly with probability at least \(  1-\rhoX  \) due to \THEOREM \ref{thm:main-technical:sparse}:
\begin{gather}
\label{eqn:pf:thm:approx-error:sparse:ic}
  \left\|
    \thetaStar
    -
    \frac
    {\thetaXX + \hfFn[\JCoords]( \thetaStar, \thetaXX )}
    {\| \thetaXX + \hfFn[\JCoords]( \thetaStar, \thetaXX ) \|_{2}}
  \right\|_{2}
  \leq
  \sqrt{\deltaX \| \thetaStar - \thetaXX \|_{2}} + \deltaX
.\end{gather}
The remainder of the proof will assume that the inequality in \EQUATION \eqref{eqn:pf:thm:approx-error:sparse:ic} holds uniformly, which occurs with bounded probability, as just stated.
Additionally, using the notations in \FACT \ref{fact:recurrence}---wherein the variables are set as
\(  \ux \defeq \frac{1}{2} ( 1+\sqrt{\frac{7}{3}} )  \),
\(  \vx \defeq 9 \deltaX  \), and
\(  \wx \defeq \frac{1}{3}  \)
and satisfy the fact's requirement, 
\(
  \sqrt{\frac{2}{\vx}}
  =
  \sqrt{\frac{2}{9 \deltaX}}
  =
  \sqrt{\frac{2 \cdot 9 \ux^{2}}{9 \epsilonX}}
  =
  \ux \sqrt{\frac{2}{\epsilonX}}
  >
  \ux 
\)
---define the functions
\(  \fx, \fxx : \Z_{\geq 0} \to \R  \)
by
\begin{gather}
  \label{eqn:pf:thm:approx-error:sparse:f1(0)}
  \fx( 0 ) = 2
  ,\\ \label{eqn:pf:thm:approx-error:sparse:f1(t)}
  \fx( \tx )
  = \sqrt{\vx \fx( \tx-1 )} + \wx \vx
  = \sqrt{9 \deltaX \fx( \tx-1 )} + 3 \deltaX
  ,\quad \tx \in \Z_{+}
  ,\\ \label{eqn:pf:thm:approx-error:sparse:f2(t)}
  \fxx( \tx )
  = 2^{2^{-\tx}} ( \ux^{2} \vx )^{1-2^{-\tx}}
  = 2^{2^{-\tx}} \left( \frac{3}{2} ( 5+\sqrt{21} ) \delta \right)^{1-2^{-\tx}}
  = 2^{2^{-\tx}} \epsilonX^{1-2^{-\tx}}
  ,\quad \tx \in \Z_{\geq 0}
.\end{gather}
Then, by \FACT \ref{fact:recurrence}, for all \(  \tx \in \Z_{\geq 0}  \),
\begin{gather}
\label{eqn:pf:thm:approx-error:sparse:1}
  \fx( \tx )
  \leq \fxx( \tx )
  = 2^{2^{-\tx}} \epsilonX^{1-2^{-\tx}}
,\end{gather}
and asymptotically,
\begin{gather}
\label{eqn:pf:thm:approx-error:sparse:2}
  \lim_{\tx \to \infty} \fx( \tx )
  \leq
  \lim_{\tx \to \infty} \fxx( \tx )
  =
  \epsilonX
.\end{gather}
%
\par 
%
With these preliminaries laid out, we are ready to verify \EQUATIONS \eqref{eqn:approx-error:sparse:asymptotic} and \eqref{eqn:approx-error:sparse:iterative} in \THEOREM \ref{thm:approx-error:sparse}, which can be argued inductively.
Inducting on the iterations, \(  \Iter = 0,1,2,3,\dots  \), the following inductive claim will be shown:
\begin{gather}
\label{eqn:pf:thm:approx-error:sparse:inductive-claim}
  \InductiveClaim{\Iter}
  \defeq
  \text{``}
  \| \thetaStar-\thetaHat[\Iter] \|_{2}
  \leq
  \fx( \Iter )
  .\text{''}
\end{gather}
The base case, when \(  \Iter=0  \), is trivial:
since there is the membership of
\(  \thetaStar, \thetaHat[0] \in \SparseSphereSubspace{\k}{\n} \subseteq \Sphere{\n}  \),
the Euclidean distance between \(  \thetaStar  \) and \(  \thetaHat[0]  \) cannot exceed the diameter of the unit sphere (distance \(  2  \)), i.e.,
\begin{gather*}
  \| \thetaStar - \thetaHat[0] \|_{2} \leq 2 = \fx( 0 )
,\end{gather*}
where the rightmost equality is due to the definition of \(  \fx  \) in \EQUATION \eqref{eqn:pf:thm:approx-error:sparse:f1(0)}.
Next, consider some arbitrary choice of \(  \Iter \in \Z_{+}  \), and suppose that for every \(  \IterX < \Iter  \), the \(  \IterX\Th  \) inductive claim, \(  \InductiveClaim{\IterX}  \), holds.
Then, under this inductive assumption, the \(  \Iter\Th  \) inductive claim, \(  \InductiveClaim{\Iter}  \), needs to be verified.
Recall that for \(  \Iter > 0  \), \ALGORITHM \ref{alg:biht:normalized} sets
\begin{gather}
\label{eqn:pf:thm:approx-error:sparse:3}
  \thetaHatX[\Iter]
  =
  \thetaHat[\Iter-1]
  +
  \frac{\sqrt{2\pi}}{\m}
  \sep
  \CovM^{\T}
  \sep
  \frac{1}{2}
  \left( \fFn( \CovM \thetaStar ) - \Sign*( \CovM \thetaHat[\Iter-1] ) \right)
  =
  \thetaHat[\Iter-1] + \hfFn( \thetaStar, \thetaHat[\Iter-1] )
  ,\\
\label{eqn:pf:thm:approx-error:sparse:4}
  \thetaHat[\Iter]
  =
  \frac{\Threshold*{\k}( \thetaHatX[\Iter] )}{\| \Threshold*{\k}( \thetaHatX[\Iter] ) \|_{2}}
.\end{gather}
\renewcommand{\TS}{\Supp( \thetaStar ) \cup \Supp( \thetaHat[\Iter-1] ) \cup \Supp( \thetaHat[\Iter] )}%
Additionally, due to \LEMMA \ref{lemma:error:deterministic}---where, in the context of this proof, the sets \(  \JCoords, \JCoordsX, \JCoordsXX \subseteq [\n]  \) in the lemma are taken to be
\(  \JCoords \defeq \Supp( \thetaHat[\Iter-1] )  \),
\(  \JCoordsX \defeq \Supp( \thetaStar )  \), and
\(  \JCoordsXX \defeq \Supp( \thetaHat[\Iter] )  \)%
---the following holds:
\begin{gather}
\label{eqn:pf:thm:approx-error:sparse:5}
  \| \thetaStar - \thetaHat[\Iter] \|_{2}
  =
  \left\|
    \thetaStar
    -
    \frac
    {\Threshold*{\k}( \thetaHatX[\Iter] )}
    {\| \Threshold*{\k}( \thetaHatX[\Iter] ) \|_{2}}
  \right\|_{2}
  \leq
  3
  \left\|
    \thetaStar
    -
    \frac
    {\ThresholdSet*{\TS}( \thetaHatX[\Iter] )}
    {\| \ThresholdSet*{\TS}( \thetaHatX[\Iter] ) \|_{2}}
  \right\|_{2}
.\end{gather}
Then, the \(  \Iter\Th  \) inductive claim, \(  \InductiveClaim{\Iter}  \), can now be established:
\begin{align*}
  \| \thetaStar - \thetaHat[\Iter] \|_{2}
  &\leq
  3
  \left\|
    \thetaStar
    -
    \frac
    {\ThresholdSet*{\TS}( \thetaHatX[\Iter] )}
    {\| \ThresholdSet*{\TS}( \thetaHatX[\Iter] ) \|_{2}}
  \right\|_{2}
  \\
  &\dCmt{by \EQUATION \eqref{eqn:pf:thm:approx-error:sparse:5}}
  \\
  &=
  3
  \left\|
    \thetaStar
    -
    \frac
    {\ThresholdSet*{\TS}( \thetaHat[\Iter-1] + \hfFn( \thetaStar, \thetaHat[\Iter-1] ) )}
    {\| \ThresholdSet*{\TS}( \thetaHat[\Iter-1] + \hfFn( \thetaStar, \thetaHat[\Iter-1] ) ) \|_{2}}
  \right\|_{2}
  \\
  &\dCmt{by \EQUATION \eqref{eqn:pf:thm:approx-error:sparse:3}}
  \\
  &=
  3
  \left\|
    \thetaStar
    -
    \frac
    {\thetaHat[\Iter-1] + \hfFn[{\Supp( \thetaHat[\Iter] )}]( \thetaStar, \thetaHat[\Iter-1] )}
    {\| \thetaHat[\Iter-1] + \hfFn[{\Supp( \thetaHat[\Iter] )}]( \thetaStar, \thetaHat[\Iter-1] ) \|_{2}}
  \right\|_{2}
  \\
  &\dCmt{by the definitions of the subset thresholding operation and \(  \hfFn[\JCoords]  \) (\(  \JCoords \subseteq [\n]  \))}
  \\
  &\leq
  3
  \left( \sqrt{\deltaX \| \thetaStar - \thetaHat[\Iter-1] \|_{2}} + \deltaX \right)
  \\
  &\dCmt{by \EQUATION \eqref{eqn:pf:thm:approx-error:sparse:ic}}
  \\
  &=
  \sqrt{9 \deltaX \| \thetaStar - \thetaHat[\Iter-1] \|_{2}} + 3 \deltaX
  \\
  &\leq
  \sqrt{9 \deltaX \fx( \Iter-1 )} + 3 \deltaX
  \\
  &\dCmt{by the inductive hypothesis, i.e., the assumed correctness of \(  \InductiveClaim{\Iter-1}  \)}
  \\
  &=
  \fx( \Iter )
  ,\\
  &\dCmt{by the definition of \(  \fx  \) in \EQUATION \eqref{eqn:pf:thm:approx-error:sparse:f1(t)}}
\end{align*}
as desired.
%
\par 
%
Having verified the \(  \Iter\Th  \) inductive claim, \(  \InductiveClaim{\Iter}  \), under the inductive assumption, it follows by induction that for all \(  \Iter \in \Z_{\geq 0}  \), the \(  \Iter\Th  \) inductive claim, \(  \InductiveClaim{\Iter}  \), holds:
\begin{gather*}
  \| \thetaStar - \thetaHat[\Iter] \|_{2}
  \leq
  \fx( \Iter )
.\end{gather*}
Therefore, the assumption that \EQUATION \eqref{eqn:pf:thm:approx-error:sparse:ic} holds uniformly---which occurs with probability at least \(  1-\rhoX  \)---and \EQUATIONS \eqref{eqn:pf:thm:approx-error:sparse:1} and \eqref{eqn:pf:thm:approx-error:sparse:2} together imply that
\begin{gather*}
  \| \thetaStar - \thetaHat[\Iter] \|_{2}
  \leq
  \fx( \Iter )
  \leq \fxx( \Iter )
  = 2^{2^{-\Iter}} \epsilonX^{1-2^{-\Iter}}
\end{gather*}
for every \(  \Iter \in \Z_{\geq 0}  \) and that
\begin{gather*}
  \lim_{\Iter \to \infty} \| \thetaStar - \thetaHat[\Iter] \|_{2}
  \leq
  \lim_{\Iter \to \infty} \fx( \Iter )
  \leq
  \lim_{\Iter \to \infty} \fxx( \Iter )
  =
  \epsilonX
,\end{gather*}
concluding the theorem's proof.
\end{proof}


\subsection{Proof of \COROLLARY \ref{corollary:approx-error:logistic-regression}} 
\label{outline:pf-main-result|pf-main-corollaries}

\begin{proof}
{\COROLLARY \ref{corollary:approx-error:logistic-regression}} 
%
%
Under the presumed correctness of the main technical corollary, \COROLLARY \ref{corollary:main-technical:logistic-regression}, 
\COROLLARY \ref{corollary:approx-error:logistic-regression} 
for the convergence of BIHT (\ALGORITHM \ref{alg:biht:normalized}) in logistic and probit regressions, 
now follow along the same arguments as in the proof of \THEOREM \ref{thm:approx-error:sparse}. 
In this analogous proof, the use of \COROLLARY \ref{corollary:main-technical:logistic-regression} replaces \THEOREM \ref{thm:main-technical:sparse} in order to establish the logistic and probit cases of \COROLLARY \ref{corollary:approx-error:logistic-regression}. 
\end{proof}


\subsection{Proof of the Intermediate Result, \LEMMA \ref{lemma:error:deterministic}}
\label{outline:pf-main-result|pf-intermediate}

This section verifies the intermediate result, \LEMMA \ref{lemma:error:deterministic}, which was introduced in \SECTION \ref{outline:pf-main-result|intermediate}.
The proof of \LEMMA \ref{lemma:error:deterministic} will use the following fact.

\begin{fact}
\label{fact:dist-btw-normalized-vectors}
Let
\(  \Vec{\uV}, \Vec{\vV} \in \R^{\n}  \).
Then,
\begin{gather}
  \left\| \frac{\Vec{\uV}}{\| \Vec{\uV} \|_{2}} - \frac{\Vec{\vV}}{\| \Vec{\vV} \|_{2}} \right\|_{2}
  \leq
  2
  \min \left\{
    \frac{\| \Vec{\uV} - \Vec{\vV} \|_{2}}{\| \Vec{\uV} \|_{2}},
    \frac{\| \Vec{\uV} - \Vec{\vV} \|_{2}}{\| \Vec{\vV} \|_{2}}
  \right\}
.\end{gather}
\end{fact}

\begin{proof}
{\FACT \ref{fact:dist-btw-normalized-vectors}}
\checkoff%
Before the fact is verified, the following easily verifiable claim is derived. 
%
\begin{claim}
\label{claim:pf:fact:dist-btw-normalized-vectors:1}
Let
\(  \zx \geq 0  \).
Then,
\(  | 1 - | 1 - \zx | | \leq \zx  \).
\end{claim}
Now, returning to the proof of \FACT \ref{fact:dist-btw-normalized-vectors}, fix
\(  \Vec{\uV}, \Vec{\vV} \in \R^{\n}  \)
arbitrarily.
Observe:
\begin{align*}
  \left\| \frac{\Vec{\uV}}{\| \Vec{\uV} \|_{2}} - \frac{\Vec{\vV}}{\| \Vec{\vV} \|_{2}} \right\|_{2}
  &=
  \left\|
    \left(
      \frac{\Vec{\uV}}{\| \Vec{\uV} \|_{2}}
      -
      \frac{\Vec{\vV}}{\| \Vec{\uV} \|_{2}}
    \right)
    +
    \left(
      \frac{\Vec{\vV}}{\| \Vec{\uV} \|_{2}}
      -
      \frac{\Vec{\vV}}{\| \Vec{\vV} \|_{2}}
    \right)
  \right\|_{2}
  \\
  &\leq
  \left\|
    \frac{\Vec{\uV}}{\| \Vec{\uV} \|_{2}}
    -
    \frac{\Vec{\vV}}{\| \Vec{\uV} \|_{2}}
  \right\|_{2}
  +
  \left\|
    \frac{\Vec{\vV}}{\| \Vec{\uV} \|_{2}}
    -
    \frac{\Vec{\vV}}{\| \Vec{\vV} \|_{2}}
  \right\|_{2}
  \\
  &\dCmt{by the triangle inequality}
  \\
  &=
  \frac{\| \Vec{\uV} - \Vec{\vV} \|_{2}}{\| \Vec{\uV} \|_{2}}
  +
  \left|
    1 - \frac{\| \Vec{\uV} - ( \Vec{\uV} - \Vec{\vV} ) \|_{2}}{\| \Vec{\uV} \|_{2}}
  \right|
  \\
  &\leq
  \frac{\| \Vec{\uV} - \Vec{\vV} \|_{2}}{\| \Vec{\uV} \|_{2}}
  +
  \max \left\{
    \left|
      1 - \frac{\| \Vec{\uV} \|_{2} + \|  \Vec{\uV} - \Vec{\vV} \|_{2}}{\| \Vec{\uV} \|_{2}}
    \right|
    ,~
    \left|
      1 -
      \left| \frac{\| \Vec{\uV} \|_{2} - \|  \Vec{\uV} - \Vec{\vV} \|_{2}}{\| \Vec{\uV} \|_{2}} \right|
    \right|
  \right\}
  \\
  &\dCmt{since \(  | \| \Vec{\uV} \|_{2} - \| \Vec{\uV} - \Vec{\vV} \|_{2} | \leq \| \Vec{\uV} - ( \Vec{\uV} - \Vec{\vV} ) \|_{2} \leq \| \Vec{\uV} \|_{2} + \| \Vec{\uV} - \Vec{\vV} \|_{2}  \)}
  \\
  &\dCmtIndent\text{by the triangle inequality}
  \\
  &=
  \frac{\| \Vec{\uV} - \Vec{\vV} \|_{2}}{\| \Vec{\uV} \|_{2}}
  +
  \max \left\{
    \left|
      1 - \left( 1 + \frac{\|  \Vec{\uV} - \Vec{\vV} \|_{2}}{\| \Vec{\uV} \|_{2}} \right)
    \right|
    ,~
    \left|
      1 - \left| 1 - \frac{\|  \Vec{\uV} - \Vec{\vV} \|_{2}}{\| \Vec{\uV} \|_{2}} \right|
    \right|
  \right\}
  \\
  &=
  \frac{\| \Vec{\uV} - \Vec{\vV} \|_{2}}{\| \Vec{\uV} \|_{2}}
  +
  \max \left\{
    \frac{\|  \Vec{\uV} - \Vec{\vV} \|_{2}}{\| \Vec{\uV} \|_{2}}
    ,~
    \frac{\|  \Vec{\uV} - \Vec{\vV} \|_{2}}{\| \Vec{\uV} \|_{2}}
  \right\}
  \\
  &\dCmt{by \CLAIM \ref{claim:pf:fact:dist-btw-normalized-vectors:1}}
  \\
  &=
  \frac{2 \|  \Vec{\uV} - \Vec{\vV} \|_{2}}{\| \Vec{\uV} \|_{2}}
.\end{align*}
A nearly identical derivation obtains
\begin{align*}
  \left\| \frac{\Vec{\uV}}{\| \Vec{\uV} \|_{2}} - \frac{\Vec{\vV}}{\| \Vec{\vV} \|_{2}} \right\|_{2}
  \leq
  \frac{2 \|  \Vec{\uV} - \Vec{\vV} \|_{2}}{\| \Vec{\vV} \|_{2}}
.\end{align*}
Combining the two acquired bounds implies the fact:
\begin{align*}
  \left\| \frac{\Vec{\uV}}{\| \Vec{\uV} \|_{2}} - \frac{\Vec{\vV}}{\| \Vec{\vV} \|_{2}} \right\|_{2}
  \leq
  2
  \min \left\{
    \frac{\|  \Vec{\uV} - \Vec{\vV} \|_{2}}{\| \Vec{\uV} \|_{2}},
    \frac{\|  \Vec{\uV} - \Vec{\vV} \|_{2}}{\| \Vec{\vV} \|_{2}}
  \right\}
,\end{align*}
as desired.
\end{proof}

We now proceed to the proof of \LEMMA \ref{lemma:error:deterministic}.

\begin{proof}
{\LEMMA \ref{lemma:error:deterministic}}
\checkoff%
Consider any \(  \Vec{\uV} \in \SparseRealSubspace{\k}{\n}  \), \(  \Vec{\vV} \in \R^{\n}  \), and \(  \JCoords \subseteq [\n]  \), \(  | \JCoords | \leq \k  \).
Recall the notations of the coordinate subsets
\(  \JCoordsu, \JCoordsv \subseteq [\n]  \),
where
\(  \JCoordsu \defeq \Supp( \Vec{\uV} )  \) and
\(  \JCoordsv \defeq \Supp( \Threshold{\k}( \Vec{\vV} ) )  \).
Due to the definition of \(  \JCoordsv  \),
\begin{align*}
  \left\|
    \Vec{\uV}
    -
    \frac
    {\Threshold{\k}(\Vec{\vV})}
    {\| \Threshold{\k}(\Vec{\vV}) \|_{2}}
  \right\|_{2}
  =
  \left\|
    \Vec{\uV}
    -
    \frac
    {\ThresholdSet{\JCoordsv}(\Vec{\vV})}
    {\| \ThresholdSet{\JCoordsv}(\Vec{\vV}) \|_{2}}
  \right\|_{2}
.\end{align*}
Then,
\begin{align*}
  \left\|
    \Vec{\uV}
    -
    \frac
    {\Threshold{\k}(\Vec{\vV})}
    {\| \Threshold{\k}(\Vec{\vV}) \|_{2}}
  \right\|_{2}
  &=
  \left\|
    \Vec{\uV}
    -
    \frac
    {\ThresholdSet{\JCoordsv}(\Vec{\vV})}
    {\| \ThresholdSet{\JCoordsv}(\Vec{\vV}) \|_{2}}
  \right\|_{2}
  \\
  &=
  \left\|
    \left(
    \Vec{\uV}
    -
    \frac
    {\ThresholdSet{\JCoords \cup \JCoordsu \cup \JCoordsv}(\Vec{\vV})}
    {\| \ThresholdSet{\JCoords \cup \JCoordsu \cup \JCoordsv}(\Vec{\vV}) \|_{2}}
    \right)
    +
    \left(
    \frac
    {\ThresholdSet{\JCoords \cup \JCoordsu \cup \JCoordsv}(\Vec{\vV})}
    {\| \ThresholdSet{\JCoords \cup \JCoordsu \cup \JCoordsv}(\Vec{\vV}) \|_{2}}
    -
    \frac
    {\ThresholdSet{\JCoordsv}(\Vec{\vV})}
    {\| \ThresholdSet{\JCoordsv}(\Vec{\vV}) \|_{2}}
    \right)
  \right\|_{2}
  \\
  &\leq
  \left\|
    \Vec{\uV}
    -
    \frac
    {\ThresholdSet{\JCoords \cup \JCoordsu \cup \JCoordsv}(\Vec{\vV})}
    {\| \ThresholdSet{\JCoords \cup \JCoordsu \cup \JCoordsv}(\Vec{\vV}) \|_{2}}
  \right\|_{2}
  +
  \left\|
    \frac
    {\ThresholdSet{\JCoords \cup \JCoordsu \cup \JCoordsv}(\Vec{\vV})}
    {\| \ThresholdSet{\JCoords \cup \JCoordsu \cup \JCoordsv}(\Vec{\vV}) \|_{2}}
    -
    \frac
    {\ThresholdSet{\JCoordsv}(\Vec{\vV})}
    {\| \ThresholdSet{\JCoordsv}(\Vec{\vV}) \|_{2}}
  \right\|_{2}
\TagEqn{\label{eqn:pf:lemma:error:deterministic:5}}
,\end{align*}
where the last line follows from the triangle inequality.
Focusing in on the second term in the last line above, it follows from \FACT \ref{fact:dist-btw-normalized-vectors} that
\begin{align*}
  \left\|
    \frac
    {\ThresholdSet{\JCoords \cup \JCoordsu \cup \JCoordsv}(\Vec{\vV})}
    {\| \ThresholdSet{\JCoords \cup \JCoordsu \cup \JCoordsv}(\Vec{\vV}) \|_{2}}
    -
    \frac
    {\ThresholdSet{\JCoordsv}(\Vec{\vV})}
    {\| \ThresholdSet{\JCoordsv}(\Vec{\vV}) \|_{2}}
  \right\|_{2}
  \leq
  \frac
  {2 \| \ThresholdSet{\JCoords \cup \JCoordsu \cup \JCoordsv}(\Vec{\vV}) - \ThresholdSet{\JCoordsv}(\Vec{\vV}) \|_{2}}
  {\| \ThresholdSet{\JCoords \cup \JCoordsu \cup \JCoordsv}(\Vec{\vV}) \|_{2}}
.\end{align*}
Note that
\(  \ThresholdSet{\JCoords \cup \JCoordsu \cup \JCoordsv}(\Vec{\vV}) - \ThresholdSet{\JCoordsv}(\Vec{\vV}) = \ThresholdSet{( \JCoords \cup \JCoordsu ) \setminus \JCoordsv}(\Vec{\vV})  \),
and hence,
\begin{align}
\label{eqn:pf:lemma:error:deterministic:1}
  \left\|
    \frac
    {\ThresholdSet{\JCoords \cup \JCoordsu \cup \JCoordsv}(\Vec{\vV})}
    {\| \ThresholdSet{\JCoords \cup \JCoordsu \cup \JCoordsv}(\Vec{\vV}) \|_{2}}
    -
    \frac
    {\ThresholdSet{\JCoordsv}(\Vec{\vV})}
    {\| \ThresholdSet{\JCoordsv}(\Vec{\vV}) \|_{2}}
  \right\|_{2}
  \leq
  \frac
  {2 \| \ThresholdSet{\JCoords \cup \JCoordsu \cup \JCoordsv}(\Vec{\vV}) - \ThresholdSet{\JCoordsv}(\Vec{\vV}) \|_{2}}
  {\| \ThresholdSet{\JCoords \cup \JCoordsu \cup \JCoordsv}(\Vec{\vV}) \|_{2}}
  =
  \frac
  {2 \| \ThresholdSet{( \JCoords \cup \JCoordsu ) \setminus \JCoordsv}(\Vec{\vV}) \|_{2}}
  {\| \ThresholdSet{\JCoords \cup \JCoordsu \cup \JCoordsv}(\Vec{\vV}) \|_{2}}
.\end{align}
Since
\(  | \JCoordsu | = | \Supp( \Vec{\uV} ) | \leq \k  \),
the definitions of \(  \JCoordsXX  \) and the top-\(  \k  \) thresholding operation imply that
\(  | \Supp( \ThresholdSet{\JCoordsu}( \Vec{\vV} ) ) | \leq | \Supp( \ThresholdSet{\JCoordsv}( \Vec{\vV} ) ) |  \),
as well as that
\begin{gather}
\label{eqn:pf:lemma:error:deterministic:2}
  \| \ThresholdSet{( \JCoords \cup \JCoordsu ) \setminus \JCoordsv}(\Vec{\vV}) \|_{2}
  =
  \| \ThresholdSet{( \JCoords \cup \JCoordsu \cup \JCoordsv ) \setminus \JCoordsv}(\Vec{\vV}) \|_{2}
  \leq
  \| \ThresholdSet{( \JCoords \cup \JCoordsu \cup \JCoordsv ) \setminus \JCoordsu}(\Vec{\vV}) \|_{2}
  =
  \| \ThresholdSet{( \JCoords \cup \JCoordsv ) \setminus \JCoordsu}(\Vec{\vV}) \|_{2}
.\end{gather}
Additionally, observe:
\begin{align*}
  \left\|
    \Vec{\uV}
    -
    \frac
    {\ThresholdSet{\JCoords \cup \JCoordsu \cup \JCoordsv}(\Vec{\vV})}
    {\| \ThresholdSet{\JCoords \cup \JCoordsu \cup \JCoordsv}(\Vec{\vV}) \|_{2}}
  \right\|_{2}^{2}
  &=
  \left\|
    \Vec{\uV}
    -
    \frac
    {\ThresholdSet{\JCoordsu}(\Vec{\vV})}
    {\| \ThresholdSet{\JCoords \cup \JCoordsu \cup \JCoordsv}(\Vec{\vV}) \|_{2}}
    -
    \frac
    {\ThresholdSet{(\JCoords \cup \JCoordsv ) \setminus \JCoordsu}(\Vec{\vV})}
    {\| \ThresholdSet{\JCoords \cup \JCoordsu \cup \JCoordsv}(\Vec{\vV}) \|_{2}}
  \right\|_{2}^{2}
  \\
  &=
  \left\|
    \Vec{\uV}
    -
    \frac
    {\ThresholdSet{\JCoordsu}(\Vec{\vV})}
    {\| \ThresholdSet{\JCoords \cup \JCoordsu \cup \JCoordsv}(\Vec{\vV}) \|_{2}}
  \right\|_{2}^{2}
  +
  \left\|
    \frac
    {\ThresholdSet{(\JCoords \cup \JCoordsv ) \setminus \JCoordsu}(\Vec{\vV})}
    {\| \ThresholdSet{\JCoords \cup \JCoordsu \cup \JCoordsv}(\Vec{\vV}) \|_{2}}
  \right\|_{2}^{2}
\TagEqn{\label{eqn:pf:lemma:error:deterministic:3}}
,\end{align*}
where the first equality holds since
\(  \ThresholdSet{\JCoords \cup \JCoordsu \cup \JCoordsv}(\Vec{\vV}) = \ThresholdSet{\JCoordsu}(\Vec{\vV}) + \ThresholdSet{( \JCoords \cup \JCoordsu \cup \JCoordsv ) \setminus \JCoordsu}(\Vec{\vV}) = \ThresholdSet{\JCoordsu}(\Vec{\vV}) + \ThresholdSet{( \JCoords \cup \JCoordsu ) \setminus \JCoordsu}(\Vec{\vV})  \),
and where the second equality is due to the orthogonality of
\(  \Vec{\uV} + w \ThresholdSet{\JCoordsu}(\Vec{\vV})  \)
and
\(  \ThresholdSet{( \JCoords \cup \JCoordsv ) \setminus \JCoordsu}(\Vec{\vV})  \)
for any scalar \(  w \in \R  \).
This orthogonality is the result of disjoint support sets:
\(  \Supp( \Vec{\uV} + w \ThresholdSet{\JCoordsu}(\Vec{\vV}) ) \cap \Supp( \ThresholdSet{( \JCoords \cup \JCoordsv ) \setminus \JCoordsu}(\Vec{\vV}) ) \subseteq {\JCoordsu \cap ( ( \JCoords \cup \JCoordsv ) \setminus \JCoordsu )} = \emptyset  \).
Rearranging the terms in \EQUATION \eqref{eqn:pf:lemma:error:deterministic:3} and taking the square root yields:
\begin{align}
\label{eqn:pf:lemma:error:deterministic:3b}
  \left\|
    \frac
    {\ThresholdSet{(\JCoords \cup \JCoordsv ) \setminus \JCoordsu}(\Vec{\vV})}
    {\| \ThresholdSet{\JCoords \cup \JCoordsu \cup \JCoordsv}(\Vec{\vV}) \|_{2}}
  \right\|_{2}
  =
  \sqrt{
  \left\|
    \Vec{\uV}
    -
    \frac
    {\ThresholdSet{\JCoords \cup \JCoordsu \cup \JCoordsv}(\Vec{\vV})}
    {\| \ThresholdSet{\JCoords \cup \JCoordsu \cup \JCoordsv}(\Vec{\vV}) \|_{2}}
  \right\|_{2}^{2}
  -
  \left\|
    \Vec{\uV}
    -
    \frac
    {\ThresholdSet{\JCoordsu}(\Vec{\vV})}
    {\| \ThresholdSet{\JCoords \cup \JCoordsu \cup \JCoordsv}(\Vec{\vV}) \|_{2}}
  \right\|_{2}^{2}
  }
.\end{align}
From \EQUATION \eqref{eqn:pf:lemma:error:deterministic:3b}, it follows that
\begin{align}
  \frac
  {\| \ThresholdSet{( \JCoords \cup \JCoordsv ) \setminus \JCoordsu}(\Vec{\vV}) \|_{2}}
  {\| \ThresholdSet{\JCoords \cup \JCoordsu \cup \JCoordsv}(\Vec{\vV}) \|_{2}}
  &=
  \sqrt{
  \left\|
    \Vec{\uV}
    -
    \frac
    {\ThresholdSet{\JCoords \cup \JCoordsu \cup \JCoordsv}(\Vec{\vV})}
    {\| \ThresholdSet{\JCoords \cup \JCoordsu \cup \JCoordsv}(\Vec{\vV}) \|_{2}}
  \right\|_{2}^{2}
  -
  \left\|
    \Vec{\uV}
    -
    \frac
    {\ThresholdSet{\JCoordsu}(\Vec{\vV})}
    {\| \ThresholdSet{\JCoords \cup \JCoordsu \cup \JCoordsv}(\Vec{\vV}) \|_{2}}
  \right\|_{2}^{2}
  }
  \\
  &\leq
  \left\|
    \Vec{\uV}
    -
    \frac
    {\ThresholdSet{\JCoords \cup \JCoordsu \cup \JCoordsv}(\Vec{\vV})}
    {\| \ThresholdSet{\JCoords \cup \JCoordsu \cup \JCoordsv}(\Vec{\vV}) \|_{2}}
  \right\|_{2}
\label{eqn:pf:lemma:error:deterministic:4}
.\end{align}
Combining \EQUATIONS \eqref{eqn:pf:lemma:error:deterministic:1}, \eqref{eqn:pf:lemma:error:deterministic:2}, and \eqref{eqn:pf:lemma:error:deterministic:4},
\begin{align*}
  \left\|
    \frac
    {\ThresholdSet{\JCoords \cup \JCoordsu \cup \JCoordsv}(\Vec{\vV})}
    {\| \ThresholdSet{\JCoords \cup \JCoordsu \cup \JCoordsv}(\Vec{\vV}) \|_{2}}
    -
    \frac
    {\ThresholdSet{\JCoordsv}(\Vec{\vV})}
    {\| \ThresholdSet{\JCoordsv}(\Vec{\vV}) \|_{2}}
  \right\|_{2}
  &\leq
  \frac
  {2 \| \ThresholdSet{( \JCoords \cup \JCoordsu ) \setminus \JCoordsv}(\Vec{\vV}) \|_{2}}
  {\| \ThresholdSet{\JCoords \cup \JCoordsu \cup \JCoordsv}(\Vec{\vV}) \|_{2}}
  \\
  &\dCmt{by \EQUATION \eqref{eqn:pf:lemma:error:deterministic:1}}
  \\
  &\leq
  \frac
  {2 \| \ThresholdSet{( \JCoords \cup \JCoordsv ) \setminus \JCoordsu}(\Vec{\vV}) \|_{2}}
  {\| \ThresholdSet{\JCoords \cup \JCoordsu \cup \JCoordsv}(\Vec{\vV}) \|_{2}}
  \\
  &\dCmt{by \EQUATION \eqref{eqn:pf:lemma:error:deterministic:2}}
  \\
  &\leq
  2
  \left\|
    \Vec{\uV}
    -
    \frac
    {\ThresholdSet{\JCoords \cup \JCoordsu \cup \JCoordsv}(\Vec{\vV})}
    {\| \ThresholdSet{\JCoords \cup \JCoordsu \cup \JCoordsv}(\Vec{\vV}) \|_{2}}
  \right\|_{2}
  . \TagEqn{\label{eqn:pf:lemma:error:deterministic:6}} \\
  &\dCmt{by \EQUATION \eqref{eqn:pf:lemma:error:deterministic:4}}
\end{align*}
Now, returning to \EQUATION \eqref{eqn:pf:lemma:error:deterministic:5}, the proof is completed as follows:
\begin{align*}
  \left\|
    \Vec{\uV}
    -
    \frac
    {\Threshold{\k}(\Vec{\vV})}
    {\| \Threshold{\k}(\Vec{\vV}) \|_{2}}
  \right\|_{2}
  &\leq
  \left\|
    \Vec{\uV}
    -
    \frac
    {\ThresholdSet{\JCoords \cup \JCoordsu \cup \JCoordsv}(\Vec{\vV})}
    {\| \ThresholdSet{\JCoords \cup \JCoordsu \cup \JCoordsv}(\Vec{\vV}) \|_{2}}
  \right\|_{2}
  +
  \left\|
    \frac
    {\ThresholdSet{\JCoords \cup \JCoordsu \cup \JCoordsv}(\Vec{\vV})}
    {\| \ThresholdSet{\JCoords \cup \JCoordsu \cup \JCoordsv}(\Vec{\vV}) \|_{2}}
    -
    \frac
    {\ThresholdSet{\JCoordsv}(\Vec{\vV})}
    {\| \ThresholdSet{\JCoordsv}(\Vec{\vV}) \|_{2}}
  \right\|_{2}
  \\
  &\dCmt{by \EQUATION \eqref{eqn:pf:lemma:error:deterministic:5}}
  \\
  &\leq
  \left\|
    \Vec{\uV}
    -
    \frac
    {\ThresholdSet{\JCoords \cup \JCoordsu \cup \JCoordsv}(\Vec{\vV})}
    {\| \ThresholdSet{\JCoords \cup \JCoordsu \cup \JCoordsv}(\Vec{\vV}) \|_{2}}
  \right\|_{2}
  +
  2
  \left\|
    \Vec{\uV}
    -
    \frac
    {\ThresholdSet{\JCoords \cup \JCoordsu \cup \JCoordsv}(\Vec{\vV})}
    {\| \ThresholdSet{\JCoords \cup \JCoordsu \cup \JCoordsv}(\Vec{\vV}) \|_{2}}
  \right\|_{2}
  \\
  &\dCmt{by \EQUATION \eqref{eqn:pf:lemma:error:deterministic:6}}
  \\
  &=
  3
  \left\|
    \Vec{\uV}
    -
    \frac
    {\ThresholdSet{\JCoords \cup \JCoordsu \cup \JCoordsv}(\Vec{\vV})}
    {\| \ThresholdSet{\JCoords \cup \JCoordsu \cup \JCoordsv}(\Vec{\vV}) \|_{2}}
  \right\|_{2}
,\end{align*}
as desired.
\end{proof}


\section{Proof of the Main Technical Results}
\label{outline:pf-main-technical-result}

\subsection{Overview of the Proof of \THEOREM \ref{thm:main-technical:sparse}}
\label{outline:main-technical-result|outline-of-pf}

\checkoffbutmayberecheck%
The proof of the main technical theorem, \THEOREM \ref{thm:main-technical:sparse}, takes up the majority of the work in this manuscript.
This section provides an overview of the proof.
The proof in full is located in \APPENDIX \ref{outline:pf-main-technical-result} with some auxiliary results therein proved in \APPENDIX \ref{outline:concentration-ineq}.
Before outlining the arguments, recall the  definitions of Equations \eqref{eqn:notations:h:def}--\eqref{eqn:notations:hfJ:def} from \SECTION \ref{outline:main-result|outline-of-pf}. 
Additionally, define the following related notations for \(  \Vec{u}, \Vec{v} \in \R^{\n}  \) and \(  \JCoords \subseteq [\n]  \):
\begin{gather}
  \label{eqn:notations:g:def}
  \gFn( \Vec{u}, \Vec{v} )
  \defeq
  \hFn( \Vec{u}, \Vec{v} )
  -
  \left\langle \hFn( \Vec{u}, \Vec{v} ), \frac{\Vec{u}-\Vec{v}}{\| \Vec{u}-\Vec{v} \|_{2}} \right\rangle \frac{\Vec{u}-\Vec{v}}{\| \Vec{u}-\Vec{v} \|_{2}}
  -
  \left\langle \hFn( \Vec{u}, \Vec{v} ), \frac{\Vec{u}+\Vec{v}}{\| \Vec{u}+\Vec{v} \|_{2}} \right\rangle \frac{\Vec{u}+\Vec{v}}{\| \Vec{u}+\Vec{v} \|_{2}}
  ,\\ \label{eqn:notations:gJ:def}
  \gFn[\JCoords]( \Vec{u}, \Vec{v} )
  \defeq
  \ThresholdSet{\Supp( \Vec{u} ) \cup \Supp( \Vec{v} ) \cup \JCoords}( \gFn( \Vec{u}, \Vec{v} ) )
  ,\\ \label{eqn:notations:gf:def}
  \gfFn( \Vec{u}, \Vec{u} )
  \defeq
  \hfFn( \Vec{u}, \Vec{u} ) - \left\langle \hfFn( \Vec{u}, \Vec{u} ), \Vec{u} \right\rangle \Vec{u}
  ,\\ \label{eqn:notations:gfJ:def}
  \gfFn[\JCoords]( \Vec{u}, \Vec{u} )
  \defeq
  \ThresholdSet{\Supp( \Vec{u} ) \cup \Supp( \Vec{v} ) \cup \JCoords}( \gfFn( \Vec{u}, \Vec{u} ) )
.\end{gather}
Note that
\begin{gather*}
  \gFn[\JCoords]( \Vec{u}, \Vec{v} )
  =
  \hFn[\JCoords]( \Vec{u}, \Vec{v} )
  -
  \left\langle \hFn[\JCoords]( \Vec{u}, \Vec{v} ), \frac{\Vec{u}-\Vec{v}}{\| \Vec{u}-\Vec{v} \|_{2}} \right\rangle \frac{\Vec{u}-\Vec{v}}{\| \Vec{u}-\Vec{v} \|_{2}}
  -
  \left\langle \hFn[\JCoords]( \Vec{u}, \Vec{v} ), \frac{\Vec{u}+\Vec{v}}{\| \Vec{u}+\Vec{v} \|_{2}} \right\rangle \frac{\Vec{u}+\Vec{v}}{\| \Vec{u}+\Vec{v} \|_{2}}
\end{gather*}
and that
\begin{gather*}
  \gfFn[\JCoords]( \Vec{u}, \Vec{u} )
  =
  \hfFn[\JCoords]( \Vec{u}, \Vec{u} )
  -
  \left\langle \hfFn[\JCoords]( \Vec{u}, \Vec{u} ), \Vec{u} \right\rangle \Vec{u}
.\end{gather*}


\subsubsection{Key Steps of the Proof}
\label{outline:main-technical-result|outline-of-pf|outline}

The proof of \THEOREM \ref{thm:main-technical:sparse} are sketched as follows.
\checkoff%

\begin{enumerate}
\item \label{enum:outline-pf-main-technical:1}
Recall that the aim is to bound
\begin{gather}
\label{eqn:enum:outline-pf-main-technical:1}
  \left\|
    \thetaStar
    -
    \frac
    {\thetaXX + \hfFn[\JCoords]( \thetaStar, \thetaXX )}
    {\| \thetaXX + \hfFn[\JCoords]( \thetaStar, \thetaXX ) \|_{2}}
  \right\|_{2}
\end{gather}
from above with high probability uniformly for all \(  \thetaXX \in \ParamSpace  \) and all \(  \JCoords \subseteq [\n]  \), \(  | \JCoords | \leq \k  \).
\item \label{enum:outline-pf-main-technical:2}
To obtain a uniform result, a \(  \tauX  \)-net, \(  \ParamCover \subset \ParamSpace  \), over the parameter space, \(  \ParamSpace  \), is constructed with a particular design, the details of which are left to the formal proof of \THEOREM \ref{thm:main-technical:sparse}.
For the purpose of this overview, its suffices to say that, crucially, the design of \(  \ParamCover  \) ensures that for each \(  \thetaXX \in \ParamSpace  \), there exists an element, \(  \thetaX \in \ParamCover  \), such that both
\(  \| \thetaX - \thetaXX \|_{2} \leq \tauX  \) and
\(  \Supp( \thetaX ) = \Supp( \thetaXX )  \).
This cover, \(  \ParamCover  \), will allow the establishment of a global result for points, \(  \thetaX \in \ParamCover  \), within it, which can subsequently be extended to arbitrary points, \(  \thetaXX \in \ParamSpace  \), in the entire parameter space via a local analysis.
\item \label{enum:outline-pf-main-technical:3}
As another preliminary step, it will be shown that for any \(  \thetaXX \in \ParamSpace  \) and \(  \JCoords \subseteq [\n]  \),
\begin{gather*}
  \frac
  {\E[ \thetaXX + \hfFn[\JCoords]( \thetaStar, \thetaXX ) ]}
  {\| \E[ \thetaXX + \hfFn[\JCoords]( \thetaStar, \thetaXX ) ] \|_{2}}
  =
  \thetaStar
.\end{gather*}
In other words, the quantity in \eqref{eqn:enum:outline-pf-main-technical:1}---which we seek to bound---describes a notion of deviation of
\(  \thetaXX + \hfFn[\JCoords]( \thetaStar, \thetaXX )  \)
from its mean (after normalization):
\begin{gather*}
  \left\|
    \thetaStar
    -
    \frac
    {\thetaXX + \hfFn[\JCoords]( \thetaStar, \thetaXX )}
    {\| \thetaXX + \hfFn[\JCoords]( \thetaStar, \thetaXX ) \|_{2}}
  \right\|_{2}
  =
  \left\|
    \frac
    {\thetaXX + \hfFn[\JCoords]( \thetaStar, \thetaXX )}
    {\| \thetaXX + \hfFn[\JCoords]( \thetaStar, \thetaXX ) \|_{2}}
    -
    \frac
    {\E[ \thetaXX + \hfFn[\JCoords]( \thetaStar, \thetaXX ) ]}
    {\| \E[ \thetaXX + \hfFn[\JCoords]( \thetaStar, \thetaXX ) ] \|_{2}}
  \right\|_{2}
.\end{gather*}
In fact, this deviation turns out to roughly scale with the deviation of the random function \(  \hfFn[\JCoords]  \) around its mean:
\begin{gather}
\label{eqn:enum:outline-pf-main-technical:1b}
  \left\|
    \frac
    {\thetaXX + \hfFn[\JCoords]( \thetaStar, \thetaXX )}
    {\| \thetaXX + \hfFn[\JCoords]( \thetaStar, \thetaXX ) \|_{2}}
    -
    \frac
    {\E[ \thetaXX + \hfFn[\JCoords]( \thetaStar, \thetaXX ) ]}
    {\| \E[ \thetaXX + \hfFn[\JCoords]( \thetaStar, \thetaXX ) ] \|_{2}}
  \right\|_{2}
  \propto
  \| \hfFn[\JCoords]( \thetaStar, \thetaXX ) - \E[ \hfFn[\JCoords]( \thetaStar, \thetaXX ) ] \|_{2}
.\end{gather}
Analyzing (a decomposition of) the deviation of \(  \hfFn[\JCoords]  \) will be at the core of the proof.
\item \label{enum:outline-pf-main-technical:4}
Letting
\(  \thetaStar, \thetaXX \in \ParamSpace  \)
be arbitrary, and using the observations in \STEP \ref{enum:outline-pf-main-technical:3}, the triangle inequality, algebraic manipulations, and other standard techniques, the expression in \eqref{eqn:enum:outline-pf-main-technical:1}
is bounded by the sum of
three terms which will admit an easier analysis than directly handling \eqref{eqn:enum:outline-pf-main-technical:1}:
\begin{subequations}
\label{eqn:enum:outline-pf-main-technical:2}
\begin{align}
  \left\| \thetaStar - \frac{\thetaXX+\hfFn[\JCoords]( \thetaStar, \thetaXX )}{\| \thetaXX+\hfFn[\JCoords]( \thetaStar, \thetaXX ) \|_{2}} \right\|_{2}
  &\leq
  \label{enum:outline-pf-main-technical:5:i}
  \frac
  {2 \| \hFn[\JCoords]( \thetaStar, \thetaX ) - \E[ \hFn[\JCoords]( \thetaStar, \thetaX ) ] \|_{2}}
  {\DENOM}
  \\ \label{enum:outline-pf-main-technical:5:iii}
  &\AlignSp+
  \frac
  {2 \| \hFn[\Supp( \thetaStar ) \cup \JCoords]( \thetaX, \thetaXX ) - \E[ \hFn[\Supp( \thetaStar ) \cup \JCoords]( \thetaX, \thetaXX ) ] \|_{2}}
  {\DENOM}
  \\ \label{enum:outline-pf-main-technical:5:ii}
  &\AlignSp+
  \frac
  {2 \| \hfFn[\Supp( \thetaX ) \cup \JCoords]( \thetaStar, \thetaStar ) - \E[ \hfFn[\Supp( \thetaX ) \cup \JCoords]( \thetaStar, \thetaStar ) ] \|_{2}}
  {\DENOM}
,\end{align}
\end{subequations}
where \(  \JCoords \subseteq [\n]  \), \(  | \JCoords | \leq \k  \), is arbitrary,
and where
\(  \thetaX \in \ParamCover \setminus \Ball{\tauX}( \thetaStar )  \)
such that
\(  \| \thetaX - \thetaXX \|_{2} \leq 2\tauX  \) and
\(  \Supp( \thetaX ) \cup \JCoords = \Supp( \thetaXX ) \cup \JCoords  \)
(\see \LEMMA \ref{lemma:combine}).
Per the design of the \(  \tauX  \)-net, \(  \ParamCover \subset \ParamSpace  \), in \STEP \ref{enum:outline-pf-main-technical:2}, such a point \(  \thetaX \in \ParamCover  \) exists for any choice of \(  \thetaXX \in \ParamSpace  \).
\item \label{enum:outline-pf-main-technical:5}
The three terms on the \RHS of \EQUATION \eqref{eqn:enum:outline-pf-main-technical:2} can be viewed as bounding \eqref{eqn:enum:outline-pf-main-technical:1} by relating it (with appropriate scaling) to the deviation of \(  \hfFn  \) specified on the \RHS of \eqref{eqn:enum:outline-pf-main-technical:1b}, and then controlling the \RHS of \eqref{eqn:enum:outline-pf-main-technical:1b} by
decomposing the deviation of \(  \hfFn  \) into three components of deviation,
in order:
\eqref{enum:outline-pf-main-technical:5:i},
a component handling points in the cover, \(  \ParamCover  \), over \(  \ParamSpace  \) that are sufficiently far from \(  \thetaStar  \)---a ``global'' result;
\eqref{enum:outline-pf-main-technical:5:iii},
a component reconciling the discrepancy between the original point, \(  \thetaXX \in \ParamSpace  \), which may be outside the cover, and a nearby neighbor in the cover, \(  \thetaX \in \ParamCover \setminus \Ball{\tauX}( \thetaStar )  \)---a ``local'' result; and
\eqref{enum:outline-pf-main-technical:5:ii},
a component handling the ``noise'' introduced into the GLM through the randomness of \(  \fFn  \).
\item \label{enum:outline-pf-main-technical:6}
The technical work in this manuscript then lies largely with bounding the three terms on the \RHS of \EQUATION \eqref{eqn:enum:outline-pf-main-technical:2}.
While most of the details of this analysis are left to the formal proofs (\see \APPENDICES \ref{outline:pf-main-technical-result|pf-intermediate}--\ref{outline:concentration-ineq}), a few salient ideas in the approach are mentioned here.
\item \label{enum:outline-pf-main-technical:7}
For the three terms, \eqref{enum:outline-pf-main-technical:5:i}--\eqref{enum:outline-pf-main-technical:5:ii}, the (shared) denominator can be calculated directly.
\item \label{enum:outline-pf-main-technical:8}
On the other hand, the numerators in \eqref{enum:outline-pf-main-technical:5:i}--\eqref{enum:outline-pf-main-technical:5:ii} are upper bounded with bounded probability
through concentration inequalities derived with standard techniques.
To do so, each numerator is orthogonally decomposed into two to three components for which derivations of concentration inequalities are easier.
Subsequently, for each numerator, the concentration inequalities for its associated components are combined via the triangle inequality.
Then, these are extended into uniform results by appropriate union bounds.
\item \label{enum:outline-pf-main-technical:9}
For the first and last terms, \eqref{enum:outline-pf-main-technical:5:i} and \eqref{enum:outline-pf-main-technical:5:ii}, the union bounds are straightforward: simply taken over the coordinate subsets of cardinality at most \(  \k  \), as well as, in the case of \eqref{enum:outline-pf-main-technical:5:i}, over the cover, \(  \ParamCover  \).
\item \label{enum:outline-pf-main-technical:10}
In contrast, the second term, \eqref{enum:outline-pf-main-technical:5:iii},  requires a more careful---and somewhat indirect---argument.
In this case, the union bound is taken over the set
\(  \{ \hFn[\Supp( \thetaStar ) \cup \JCoords]( \thetaX, \thetaXX ) : \thetaXX \in \BallX{2\tauX}( \thetaX ) \}  \),
which has a sufficiently small cardinality due to the local binary embeddings of \cite[{\COROLLARY 3.3}]{oymak2015near}.
\item \label{enum:outline-pf-main-technical:11}
Using the uniform results obtained in \STEPS \ref{enum:outline-pf-main-technical:7}--\ref{enum:outline-pf-main-technical:10}, the number of covariates, \(  \m  \), can then be determined such that desired bounds on the terms \eqref{enum:outline-pf-main-technical:5:i}--\eqref{enum:outline-pf-main-technical:5:ii},
and hence also the desired bound on \eqref{eqn:enum:outline-pf-main-technical:1}, hold uniformly with high probability.
This will establish the invertibility condition for Gaussian covariate matrices claimed in \THEOREM \ref{thm:main-technical:sparse}.
\end{enumerate}

\subsection{Detailed Proofs}
Several constants will  appear throughout this section.
For convenient reference later, they are specified in the following definition.
\begin{definition}
\label{def:univ-const}
Let
\(  \ConstC, \Constb, \Constc, \ConstbLD, \ConstbSD > 0  \)
be (absolute) constants such that
\(  \ConstC \defeq \frac{\Constb}{\Constc^{2}}  \geq \frac{1}{50}  \),
\(  \sqrt{8} \Constb < \frac{\ConstbSD}{2}  \),
\(  \Constc \defeq \frac{\ConstbSD}{2}-\sqrt{8} \Constb = \BigOmega(1)  \),
\(  \ConstbLD < 1 - \sqrt{\frac{2\Constb}{\ConstdSD}}  \), and
\(  \ConstbSD \leq 1 - \ConstbLD - \sqrt{\frac{2\Constb}{\ConstdSD}}  \).
Additionally, let
\(  \ConstA, \ConstB, \ConstCTwo, \ConstCThree, \ConstCFour, \ConstCFive > 0  \)
be the (absolute) constants given by:
\(  \ConstA = \ConstAValue  \),
\(  \ConstB = \ConstBValue  \),
\(  \ConstCTwo = \ConstCTwoValue  \),
\(  \ConstCThree = \ConstCThreeValue  \),
\(  \ConstCFour = \ConstCFourValue  \), and
\(  \ConstCFive = \ConstCFiveValue  \).
\end{definition}


Two additional notations will be used in this manuscript, which are introduced in \DEFINITION \ref{def:nu-and-tau}, below.
%
\begin{definition}
\label{def:nu-and-tau}
For \(  \deltaX > 0  \), let \(  \nuX( \deltaX ), \tauX( \deltaX ) > 0  \) be given by
\begin{gather}
\label{eqn:def:nu-and-tau:nu}
  \nuX( \deltaX )
  =
  \nuXEXPR[\deltaX]
,\end{gather}
and
\begin{gather}
\label{eqn:def:nu-and-tau:tau}
  \tauX( \deltaX ) \defeq \tauXEXPR[\deltaX]
,\end{gather}
where \(  \ConstCFive > 0  \) is given in \DEFINITION \ref{def:univ-const}.
To condense notation, the explicit parameterization by \(  \deltaX  \) will in general be dropped and left implicit in this manuscript, \ie
\(  \nuX = \nuX( \deltaX )  \) and \(  \tauX = \tauX( \deltaX )  \),
where the specific choice of \(  \deltaX  \) may vary but will be clear from the context.
\end{definition}
%

We restate \THEOREM \ref{thm:main-technical:sparse} below with the sample complexity more specific with the above-defined constants.

\begin{theorem}[Restatement of \THEOREM \ref{thm:main-technical:sparse}]
\label{thm:main-technical:sparse:detail}
Let
\(  \ConstA, \ConstB, \ConstCOne, \ConstCTwo, \ConstCThree, \ConstCFour, \ConstCFive > 0  \)
be absolute constants as specified in \DEFINITION \ref{def:univ-const}, and fix
\(  \n, \k, \m \in \Z_{+}  \), \(  \k \leq \n  \), and \(  \rhoX, \deltaX \in (0,1)  \) where
\begin{gather}
  \deltaX \defeq \frac{\epsilonX}{\frac{3}{2} ( 5+\sqrt{21} )}
.\end{gather}
Set \(  \nuX = \nuX( \deltaX ) > 0  \)
and let \(  \tauX = \tauX( \deltaX ) > 0  \)
as in \DEFINITION \ref{def:nu-and-tau}.
Write
\(  \alphaO = \alphaO( \deltaX ) \defeq \alphaOExpr[\deltaX]  \)
as in \EQUATION \eqref{eqn:notations:alpha_0:def}.
Let
\(  \ParamSpace = \SparseSphereSubspace{\k}{\n}  \),
and fix
\(  \thetaStar \in \ParamSpace  \).
Under \ASSUMPTION \ref{assumption:p}, if
\begin{align}
\nonumber
  \m
  &\geq
  \mEXPR[s]
  \\ \nonumber
  &=
  \mOEXPRS{\deltaX}[,]
  \\
\end{align}
then with probability at least \(  1-\rhoX  \), uniformly for all \(  \thetaXX \in \ParamSpace  \) and all \(  \JCoords \subseteq [\n]  \), \(  | \JCoords | \leq \k  \),
\begin{gather}
  \left\|
    \thetaStar
    -
    \frac
    {\thetaXX + \hfFn[\JCoords]( \thetaStar, \thetaXX )}
    {\| \thetaXX + \hfFn[\JCoords]( \thetaStar, \thetaXX ) \|_{2}}
  \right\|_{2}
  \leq
  \sqrt{\deltaX \| \thetaStar-\thetaXX \|_{2}}
  +
  \deltaX
.\end{gather}
\end{theorem}

\subsection{Intermediate Results for the Proof of \THEOREM \ref{thm:main-technical:sparse}}
\label{outline:pf-main-technical-result|intermediate}

\checkoff

Lemmas  \ref{lemma:combine}--\ref{lemma:large-dist:2}, stated below in this section, lay the groundwork for proving the main technical theorem, \THEOREM \ref{thm:main-technical:sparse}.
The proofs of these intermediate results can be found in \SECTION \ref{outline:pf-main-technical-result|pf-intermediate}.
Recall that the ultimate goal is to uniformly bound
\begin{gather}
\label{eqn:pf-main-technical-result:intermediate:discussion:1}
  \left\| \thetaStar - \frac{\thetaXX+\hfFn[\JCoords]( \thetaStar, \thetaXX )}{\| \thetaXX+\hfFn[\JCoords]( \thetaStar, \thetaXX ) \|_{2}} \right\|_{2}
\end{gather}
from above.
\LEMMA \ref{lemma:combine} starts off by upper bounding \eqref{eqn:pf-main-technical-result:intermediate:discussion:1} by the sum of three terms (with some scaling), each of which describes how much
the functions \(  \hFn  \) and \(  \hfFn  \) (with thresholding)
deviate from their means.
Subsequently, \LEMMAS \ref{lemma:large-dist:1}--\ref{lemma:large-dist:2} provide bounds on these deviations.

\begin{lemma}
\label{lemma:combine}
%
Let
\(  \JCoords \subseteq [\n]  \),
and fix
\(  \thetaStar, \thetaX, \thetaXX \in \ParamSpace  \)
such that
\(  \Supp( \thetaXX ) \cup \JCoords = \Supp( \thetaX ) \cup \JCoords  \).
Then,
\begin{align}
  \left\| \thetaStar - \frac{\thetaXX+\hfFn[\JCoords]( \thetaStar, \thetaXX )}{\| \thetaXX+\hfFn[\JCoords]( \thetaStar, \thetaXX ) \|_{2}} \right\|_{2}
  &\leq
  \frac
  {2 \| \hFn[\JCoords]( \thetaStar, \thetaX ) - \E[ \hFn[\JCoords]( \thetaStar, \thetaX ) ] \|_{2}}
  {\DENOM}
  \nonumber \\
  &\AlignSp+
  \frac
  {2 \| \hFn[\Supp( \thetaStar ) \cup \JCoords]( \thetaX, \thetaXX ) - \E[ \hFn[\Supp( \thetaStar ) \cup \JCoords]( \thetaX, \thetaXX ) ] \|_{2}}
  {\DENOM}
  \nonumber \\
  &\AlignSp+
  \frac
  {2 \| \hfFn[\Supp( \thetaX ) \cup \JCoords]( \thetaStar, \thetaStar ) - \E[ \hfFn[\Supp( \thetaX ) \cup \JCoords]( \thetaStar, \thetaStar ) ] \|_{2}}
  {\DENOM}
\label{eqn:lemma:combine:1}
.\end{align}
\end{lemma}

\LEMMA \ref{lemma:combine} motivates three additional results, presented next in \LEMMAS \ref{lemma:large-dist:1}--\ref{lemma:large-dist:2}.
Note that whereas \LEMMA \ref{lemma:combine} holds deterministically, \LEMMAS \ref{lemma:large-dist:1}--\ref{lemma:large-dist:2} are probabilistic results.

\begin{lemma}
\label{lemma:large-dist:1}
Let
\(  \rhoLDX, \deltaX \in (0,1)  \),
and define
\(  \tauX = \tauX( \deltaX )  \)
according to \DEFINITION \ref{def:nu-and-tau}.
Fix
\(  \thetaStar \in \ParamSpace  \),
and let \(  \JS \subseteq 2^{[\n]}  \) and \(  \ParamCover \subset \ParamSpace  \) be finite sets.
Define \(  \kO \defeq \kOExpr  \).
If
\begin{gather}
\label{eqn:lemma:large-dist:1:m}
  \m
  \geq
  \frac{16}{\GAMMAX^{2} \deltaX}
  \max \left\{
    27\pi \log \left( \frac{12}{\rhoLDX} | \JS | | \ParamCover | \right)
    ,
    4 ( \kO-2 )
  \right\}
%
%
%
,\end{gather}
then with probability at least \(  1-\rhoLDX  \), uniformly for all \(  \JCoords \in \JS  \) and all \(  \thetaX \in \ParamCover \setminus \Ball{\tauX}( \thetaStar )  \),
\begin{gather}
\label{eqn:lemma:large-dist:1:ub}
  \frac
  {2 \| \hFn[\JCoords]( \thetaStar, \thetaX ) - \E[ \hFn[\JCoords]( \thetaStar, \thetaX ) ] \|_{2}}
  {\DENOM}
  \leq
  \sqrt{\deltaX \EDIST} 
.\end{gather}
\end{lemma}

\begin{lemma}
\label{lemma:small-dist}
Let
\(  \ConstbSD, \ConstdSD > 0  \)
be constants specified in \DEFINITION \ref{def:univ-const}.
Let
\(  \rhoSD, \deltaX \in (0,1)  \),
and define
\(  \nuX = \nuX( \deltaX )  \) and \(  \tauX = \tauX( \deltaX )  \)
according to \DEFINITION \ref{def:nu-and-tau}.
Let
\(  \ParamCover \subset \ParamSpace  \)
be a finite set, 
and fix
\(  \thetaStar \in \ParamSpace  \).
Let \(  \JS, \JSXX \subseteq 2^{[\n]}  \), where
\(  \JSXX \defeq \{ \Supp( \thetaStar ) \cup \JCoords : \JCoords \in \JS \}  \).
Set
\(  \kOXX \defeq \kOXXExpr  \).
If
\begin{gather}
\label{eqn:lemma:small-dist:m}
  \m
  \geq
  \max \left\{
  \frac{200 \nuX \log \left( \frac{6}{\rhoSD} | \JS | | \ParamCover | \right)}{\left( \sqrt{\frac{\pi}{8}} \GAMMAX \ConstbSD \deltaX - \nuX \sqrt{8 \log \left( \frac{e}{\nuX} \right)} \right)^{2}}
  ,
  \frac{200 \nuX \kOXX}{\pi \GAMMAX^{2} \ConstbSD^{2} \deltaX^{2}}
  ,
  \frac{64}{\nuX} \log \left( \frac{6}{\rhoSD} \binom{\n}{\nO} \right)
  ,
  \frac{\ConstdSD \nO}{\nuX} \log \left( \frac{1}{\nuX} \right)
  \right\}
,\end{gather}
then with probability at least \(  1-\rhoSD  \), uniformly for all
\(  \JCoordsXX \in \JSXX  \),
\(  \thetaX \in \ParamCover \setminus \Ball{\tauX}( \thetaStar )  \), and
\(  \thetaXX \in \BallX{2\tauX}( \thetaX )  \),
\begin{gather}
\label{eqn:lemma:small-dist:ub}
  \frac
  {2 \| \hFn[\JCoordsXX]( \thetaX, \thetaXX ) - \E[ \hFn[\JCoordsXX]( \thetaX, \thetaXX ) ] \|_{2}}
  {\DENOM}
  \leq
  \ConstbSD \deltaX
.\end{gather}
\end{lemma}

\begin{lemma}
\label{lemma:large-dist:2}
Let
\(  \ConstbLD > 0  \)
be a constant specified in \DEFINITION \ref{def:univ-const}.
Let
\(  \rhoLDXX, \deltaX \in (0,1)  \),
and define
\(  \alphaO = \alphaO( \deltaX ) \defeq \alphaOExpr  \).
Fix
\(  \thetaStar \in \ParamSpace  \).
Let \(  \JS \subseteq 2^{[\n]}  \) and \(  \ParamCover \subset \ParamSpace  \) be finite sets,
and let
\(  \JSX \defeq \{ \Supp( \thetaX ) \cup \JCoords : \thetaX \in \ParamCover, \JCoords \in \JS \}  \).
Define \(  \kOX \defeq \kOXExpr  \).
If
\begin{gather}
\label{eqn:lemma:large-dist:2:m}
  \m
  \geq
%
  \max \left\{
  \frac{64 \alphaO}{\GAMMAX^{2} \ConstbLD^{2} \deltaX^{2}}
  \max \left\{
    3 \log \left( \frac{6}{\rhoLDXX} | \JS | | \ParamCover | \right)
    ,
    2 ( \kOX-1 )
  \right\}
  ,
  \frac{4}{\alphaO} \log \left( \frac{6}{\rhoLDXX} | \JS | | \ParamCover | \right)
  \right\}
%
,\end{gather}
then with probability at least \(  1-\rhoLDXX  \), uniformly for all \(  \JCoordsX \in \JSX  \),
\begin{gather}
\label{eqn:lemma:large-dist:2:ub}
  \frac
  {2 \| \hfFn[\JCoordsX]( \thetaStar, \thetaStar ) - \E[ \hfFn[\JCoordsX]( \thetaStar, \thetaStar ) ] \|_{2}}
  {\DENOM}
  \leq
  \ConstbLD \deltaX
,\end{gather}
for every \(  \thetaX \in \ParamSpace  \).
\end{lemma}


\subsection{Proof of \THEOREM  \ref{thm:main-technical:sparse}}
\label{outline:pf-main-technical-result|pf}



An important (and standard) construct for the analysis in this work is a \(  \tauX  \)-net. 
%
Fix \(  \tauX > 0  \).
Let \(  ( \SX, \dSX )  \) be a metric space.
A subset,
\(  \SXX \subseteq \SX  \),
is a \emph{\(  \tauX  \)-net} over \(  \SX  \) if
\(  \inf_{\sSXX \in \SXX} \dSX( \sSX, \sSXX ) \leq \tauX  \)
for all \(  \sSX \in \SX  \).
%
We will use the following upper bound on the minimal cardinality of a \(  \tauX  \)-net of a sphere.
%
\begin{lemma}[{\see \eg \cite{vershynin2018high}}]
\label{lemma:tau-net-cardinality}
Fix \(  \tauX > 0  \), and
let \(  \dDim \in \Z_{+}  \).
There exists an $\ell_2$ \(  \tauX  \)-net, \(  \SXX \subset \Sphere{\dDim}  \), over \(  \Sphere{\dDim}  \) of cardinality not exceeding
\(  | \SXX | \leq ( \frac{3}{\tauX} )^{\dDim}  \).
\end{lemma}
\begin{proof}{\THEOREM \ref{thm:main-technical:sparse}}
\mostlycheckoff%
Fix
\(  \thetaStar \in \ParamSpace  \).
Let
\(  \JS \subseteq 2^{[\n]}  \)
be the set of coordinate subsets with cardinality at most \(  \k  \)---that is, the set given by
\(
  \JS
  \defeq
  {\{ \JCoords \subseteq [\n] : | \JCoords | \leq \k \}}
\).
Construct a \(  \tauX  \)-net, \(  \ParamCover \subset \ParamSpace  \), over \(  \ParamSpace  \) with the following design.
\checkthis[Done]%
For each \(  \JCoords \in \JS  \), let \(  \ParamCoverJ \subset \ParamSpace  \) be a \(  \tauX  \)-net over the set of points in \(  \ParamSpace  \) whose support is a subset of \(  \JCoords  \)---formally, over the set
\(  \{ \Vec{v} \in \ParamSpace : \Supp( \Vec{v} ) \subseteq  \JCoords \}  \)%
---such that each vector in the cover, \(  \ParamCoverJ  \), has support exactly \(  \JCoords  \), \ie
\(  \Supp( \Vec{v} ) = \JCoords  \) for all \(  \Vec{v} \in \ParamCoverJ  \).
(This last condition on the support of elements in the \(  \tauX  \)-net is possible through a rotation.)
Then, let
\(  \ParamCover \defeq \bigcup_{\JCoords \in \JS} \ParamCoverJ  \).
Note that this construction ensures that for every point in the parameter space, \(  \ParamSpace  \), the cover, \(  \ParamCover  \), contains at least one point within distance \(  \tauX  \) of it and with precisely the same support.
Additionally, the cardinalities of \(  \JS  \) and \(  \ParamCover  \) satisfy
\(  | \JS | = \JSSIZES  \) and
\(  | \ParamCover | \leq \sum_{\JCoords \in \JS} ( \frac{3}{\tauX} )^{| \JCoords |} = \PCSIZES  \),
where the bound on \(  | \ParamCover |  \) is due to \LEMMA \ref{lemma:tau-net-cardinality} combined with a union bound.
%
\par 
%
Consider an arbitrary choice of
\(  \thetaXX \in \ParamSpace  \),
to later be varied over the entire parameter space, \(  \ParamSpace  \), and let
\(  \thetaX \in \ParamCover  \)
satisfy
\(  \thetaX \notin \Ball{\tauX}( \thetaStar )  \) and
\(  \thetaXX \in \BallXX{2\tauX}( \thetaX )  \),
where such a point, \(  \thetaX  \), exists in the \(  \tauX  \)-net, \(  \ParamCover  \), by its design.
Note that this ensures that
\(  \Supp( \thetaX ) \cup \JCoords = \Supp( \thetaXX ) \cup \JCoords  \)
for all \(  \JCoords \in \JS  \), and hence,
by \LEMMA \ref{lemma:combine},
\begin{align*}
  \left\| \thetaStar - \frac{\thetaXX+\hfFn[\JCoords]( \thetaStar, \thetaXX )}{\| \thetaXX+\hfFn[\JCoords]( \thetaStar, \thetaXX ) \|_{2}} \right\|_{2}
  &\leq
  \frac
  {2 \| \hFn[\JCoords]( \thetaStar, \thetaX ) - \E[ \hFn[\JCoords]( \thetaStar, \thetaX ) ] \|_{2}}
  {\DENOM}
  \\
  &\AlignSp+
  \frac
  {2 \| \hFn[\Supp( \thetaStar ) \cup \JCoords]( \thetaX, \thetaXX ) - \E[ \hFn[\Supp( \thetaStar ) \cup \JCoords]( \thetaX, \thetaXX ) ] \|_{2}}
  {\DENOM}
  \\
  &\AlignSp+
  \frac
  {2 \| \hfFn[\Supp( \thetaX ) \cup \JCoords]( \thetaStar, \thetaStar ) - \E[ \hfFn[\Supp( \thetaX ) \cup \JCoords]( \thetaStar, \thetaStar ) ] \|_{2}}
  {\DENOM}
\TagEqn{\label{eqn:pf:thm:main-technical:1}}
,\end{align*}
where this bound holds deterministically.
Now, suppose
\begin{align*}
  \m
  &\geq
  \mEXPR[s][.]
\end{align*}
This choice of \(  \m  \) is sufficiently large so that taking
\(  \rhoLDX = \frac{\rhoX}{2}  \),
\(  \rhoSD = \frac{\rhoX}{4}  \), and
\(  \rhoLDXX = \frac{\rhoX}{4}  \)%
---such that
\(  \rhoLDX + \rhoSD + \rhoLDXX = \rhoX  \)%
---in \LEMMAS \ref{lemma:large-dist:1}--\ref{lemma:large-dist:2}, respectively, and then combining the bounds in these lemmas with a union bound, the three terms on the \RHS of the inequality in \EQUATION \eqref{eqn:pf:thm:main-technical:1} are simultaneously bounded from above with probability at least
\(  1 - \rhoLDX - \rhoSD - \rhoLDXX = 1 - \rhoX  \)
by
first,
\begin{align*}
  \sup_{\substack{\JCoords \in \JS ,\\
                  \thetaX \in \ParamCover \setminus \Ball{\tauX}( \thetaStar )}}
  \frac
  {2 \| \hFn[\JCoords]( \thetaStar, \thetaX ) - \E[ \hFn[\JCoords]( \thetaStar, \thetaX ) ] \|_{2}}
  {\DENOM}
  \leq
  \sqrt{\deltaX \EDIST}
,\end{align*}
second,
\begin{align*}
  &\negphantom{\AlignSp}
  \sup_{\substack{\JCoords \in \JS ,\\
                  \thetaX \in \ParamCover \setminus \Ball{\tauX}( \thetaStar ) ,\\
                  \thetaXX \in \BallXX{2\tauX}( \thetaX )}}
  \frac
  {2 \| \hFn[\Supp( \thetaStar ) \cup \JCoords]( \thetaX, \thetaXX ) - \E[ \hFn[\Supp( \thetaStar ) \cup \JCoords]( \thetaX, \thetaXX ) ] \|_{2}}
  {\DENOM}
  \\
  &\leq
  \sup_{\substack{\JCoordsXX \in \JSXX ,\\
                  \thetaX \in \ParamCover \setminus \Ball{\tauX}( \thetaStar ) ,\\
                  \thetaXX \in \BallXX{2\tauX}( \thetaX )}}
  \frac
  {2 \| \hFn[\JCoordsXX]( \thetaX, \thetaXX ) - \E[ \hFn[\JCoordsXX]( \thetaX, \thetaXX ) ] \|_{2}}
  {\DENOM}
  \\
  &\leq
  \ConstbSD \deltaX
  \\
  &\leq
  \left( 1-\ConstbLD-\sqrt{\frac{2\Constb}{\ConstdSD}} \right) \deltaX
  \\
  &\dCmt{by \DEFINITION \ref{def:univ-const}}
  \\
  &\leq
  \left( 1-\ConstbLD-\sqrt{\frac{2\tauX}{\deltaX}} \right) \deltaX
  ,\\
  &\dCmt{by the definitions of \(  \deltaX, \tauX  \)}
\end{align*}
and
third,
\begin{align*}
  \sup_{\substack{\JCoords \in \JS ,\\
                  \thetaX \in \ParamCover \setminus \Ball{\tauX}( \thetaStar )}}
  \frac
  {2 \| \hfFn[\Supp( \thetaX ) \cup \JCoords]( \thetaStar, \thetaStar ) - \E[ \hfFn[\Supp( \thetaX ) \cup \JCoords]( \thetaStar, \thetaStar ) ] \|_{2}}
  {\DENOM}
  &\leq
  \sup_{\JCoordsX \in \JSX}
  \frac
  {2 \| \hfFn[\JCoordsX]( \thetaStar, \thetaStar ) - \E[ \hfFn[\JCoordsX]( \thetaStar, \thetaStar ) ] \|_{2}}
  {\DENOM}
  \leq
  \ConstbLD \deltaX
,\end{align*}
where
\(  \JSXX \defeq \{ \Supp( \thetaStar ) \cup \JCoords : \JCoords \in \JS \}  \)
and
\(  \JSX \defeq \{ \Supp( \thetaX ) \cup \JCoords : \thetaX \in \ParamCover, \JCoords \in \JS \}  \).
It follows that under the stated condition on \(  \m  \), with probability at least \(  1 - \rhoX  \), for all \(  \thetaXX \in \ParamSpace  \) and \(  \JCoords \in \JS  \),
\begin{align*}
  \left\| \thetaStar - \frac{\thetaXX+\hfFn[\JCoords]( \thetaStar, \thetaXX )}{\| \thetaXX+\hfFn[\JCoords]( \thetaStar, \thetaXX ) \|_{2}} \right\|_{2}
  &\leq
  \sqrt{\deltaX \EDIST}
  +
  \left( 1-\ConstbLD-\sqrt{\frac{2\tauX}{\deltaX}} \right) \deltaX
  +
  \ConstbLD \deltaX
  \\
  &
  \dCmt{for some \(  \thetaX \in ( \ParamCover \cap \BallXX{2\tauX}( \thetaXX ) ) \setminus \Ball{\tauX}( \thetaStar )  \)}
  \\
  &=
  \sqrt{\deltaX \| ( \thetaStar-\thetaXX ) - ( \thetaX-\thetaXX ) \|_{2}}
  +
  \left( 1-\ConstbLD-\sqrt{\frac{2\tauX}{\deltaX}} \right) \deltaX
  +
  \ConstbLD \deltaX
  \\
  &\leq
  \sqrt{\deltaX \| \thetaStar-\thetaXX \|_{2}}
  +
  \sqrt{\deltaX \| \thetaX-\thetaXX \|_{2}}
  +
  \left( 1-\ConstbLD-\sqrt{\frac{2\tauX}{\deltaX}} \right) \deltaX
  +
  \ConstbLD \deltaX
  \\
  &\dCmt{by the triangle inequality}
  \\
  &\leq
  \sqrt{\deltaX \| \thetaStar-\thetaXX \|_{2}}
  +
  \sqrt{2 \deltaX \tauX}
  +
  \left( 1-\ConstbLD-\sqrt{\frac{2\tauX}{\deltaX}} \right) \deltaX
  +
  \ConstbLD \deltaX
  \\
  &\dCmt{\(  \because \thetaX \in \BallXX{2\tauX}( \thetaXX )  \)}
  \\
  &=
  \sqrt{\deltaX \| \thetaStar-\thetaXX \|_{2}}
  +
  \sqrt{\frac{2\tauX}{\deltaX}} \delta
  +
  \left( 1-\ConstbLD-\sqrt{\frac{2\tauX}{\deltaX}} \right) \deltaX
  +
  \ConstbLD \deltaX
  \\
  &=
  \sqrt{\deltaX \| \thetaStar-\thetaXX \|_{2}}
  +
  \deltaX
,\end{align*}
as claimed.
\end{proof}


\subsection{Proof of \COROLLARY \ref{corollary:main-technical:logistic-regression}}
\label{outline:pf-main-technical-result|pf-main-corollaries}

%

\begin{proof}
{\COROLLARY \ref{corollary:main-technical:logistic-regression}}
\mostlycheckoff%
The specialization of the main technical result to logistic regression in \COROLLARY \ref{corollary:main-technical:logistic-regression} requires two arguments:
\Enum[{\label{enum:pf:corollary:main-technical:logistic-regression:a}}]{a}
\ASSUMPTION \ref{assumption:p} needs to be shown to hold for logistic regression, i.e., when \(  \pFn  \) is the logistic function with \betaXnamelr \(  \betaX \GTR 0  \), as in \DEFINITION \ref{def:p:logistic-regression};
and
\Enum[{\label{enum:pf:corollary:main-technical:logistic-regression:b}}]{b}
explicit forms for the variables \(  \alphaX  \) (and \(  \alphaO  \)) and \(  \gammaX  \) need specification.
Once these are achieved, the corollary will follow from combining the bounds on \(  \alphaX  \) and \(  \gammaX  \) obtained from \TASK \ref{enum:pf:corollary:main-technical:logistic-regression:b} with \THEOREM \ref{thm:main-technical:sparse}.
Throughout this proof, \(  \pFn  \) is taken to be the logistic function, parameterized by the \betaXnamelr, \(  \betaX \GTR 0  \), per \DEFINITION \ref{def:p:logistic-regression}, which is recalled for convenience:
\begin{gather*}
  \pbetaFn{z}
  =
  \frac{1}{1+e^{-\betaX z}}
.\end{gather*}
%
\par 
%
For \TASK \ref{enum:pf:corollary:main-technical:logistic-regression:a}, recall that \ASSUMPTION \ref{assumption:p} imposes two conditions: \ref{condition:assumption:p:i} that \(  \pFn  \) is nondecreasing over the entire real line, and \ref{condition:assumption:p:ii} that
\(  \frac{\partial}{\partial \zX} \frac{\pExpr*{\zX+\wX \betaXParam}}{\pExpr{\zX \betaXParam}} \leq 0  \).
Let
\(  \zX < \zXX \in \R  \).
Then, clearly,
\begin{gather*}
  \pbetaFn{ \zX }
  =
  \frac{1}{1+e^{-\betaX \zX}}
  <
  \frac{1}{1+e^{-\betaX \zXX}}
  =
  \pbetaFn{ \zXX }
,\end{gather*}
and thus, \CONDITION \ref{condition:assumption:p:i} holds.
On the other hand, \CONDITION \ref{condition:assumption:p:ii} can be established via basis calculus.
First, note that for any \(  \zX \in \R  \),
\begin{align}
\label{eqn:pf:corollary:main-technical:logistic-regression:7}
  1-\pbetaFn{ \zX }
  =
  1 - \frac{1}{1+e^{-\betaX \zX}}
  =
  \frac{e^{-\betaX \zX}}{1+e^{-\betaX \zX}}
  =
  \frac{1}{1+e^{\betaX \zX}}
  =
  \pbetaFn{ -\zX }
,\end{align}
and hence, for \(  \wX, \zX \in \R  \), \(  \wX > 0  \),
\begin{align*}
  \frac{\pExpr*{\zX+\wX \betaXParam}}{\pExpr{\zX \betaXParam}}
  =
  \frac{2 \pbetaFn{ -( \zX+\wX } ) }{2 \pbetaFn{ -\zX }}
  =
  \frac{\pbetaFn{ -( \zX+\wX } ) }{\pbetaFn{ -\zX }}
  =
  \frac{1+e^{\betaX \zX}}{1+e^{\betaX ( \zX+\wX )}}
.\end{align*}
Then,
\begin{align*}
  \frac{\partial}{\partial \zX} \frac{\pExpr*{\zX+\wX}}{\pExpr{\zX}}
  &=
  \frac{\partial}{\partial \zX} \frac{1+e^{\betaX \zX}}{1+e^{\betaX ( \zX+\wX )}}
  =
  -\frac
  {\betaX e^{\betaX \zX} ( e^{\betaX \wX}-1 )}
  {( 1+e^{\betaX ( \zX+\wX )} )^{2}}
  \leq 0
,\end{align*}
as desired.
Thus, \CONDITION \ref{condition:assumption:p:ii} also holds when \(  \pFn  \) is the logistic function.
This complete \TASK \ref{enum:pf:corollary:main-technical:logistic-regression:a}.
%
\par 
%
Proceeding to \TASK \ref{enum:pf:corollary:main-technical:logistic-regression:b}, the aim now is to derive closed-form bounds on \(  \alphaX   \) and \(  \gammaX  \).
Looking first at \(  \alphaX  \), recall its definition from \EQUATION \eqref{eqn:notations:alpha:def}:
\begin{gather*}
  \alphaX
  \defeq
  \Pr_{Z \sim \N(0,1)} (
    \fFn( Z ) \neq \Sign( Z )
  )
  =
  \frac{1}{\sqrt{2\pi}}
  \int_{\zX=0}^{\zX=\infty}
  e^{-\frac{1}{2} \zX^{2}}
  (\pExpr{\zX})
  d\zX
.\end{gather*}
By the earlier observation in \eqref{eqn:pf:corollary:main-technical:logistic-regression:7}, when \(  \pFn  \) is the logistic function,
\begin{align*}
  \alphaX
  &=
  \frac{1}{\sqrt{2\pi}}
  \int_{\zX=0}^{\zX=\infty}
  e^{-\frac{1}{2} \zX^{2}}
  (\pExpr{\zX})
  d\zX
  \\
  &\dCmt{by \EQUATION \eqref{eqn:notations:alpha:def}}
  \\
  &=
  \sqrt{\frac{2}{\pi}}
  \int_{\zX=0}^{\zX=\infty}
  e^{-\frac{1}{2} \zX^{2}}
  \pbetaFn{ -\zX }
  d\zX
  \\
  &\dCmt{by \EQUATION \eqref{eqn:pf:corollary:main-technical:logistic-regression:7}}
  \\
  &=
  \E \left[ \pbetaFn{ -| \ZRV | } \right]
  \\
  &\dCmt{by the law of the lazy statistician and the density function}
  \\
  &\dCmtIndent \text{for standard half-normal random variables}
  \\
  &=
  \E \left[ \frac{1}{1 + e^{\betaX | \ZRV |}} \right]
\TagEqn{\label{eqn:pf:corollary:main-technical:logistic-regression:1}}
,\end{align*}
where \(  \ZRV \sim \N(0,1)  \) is a standard univariate Gaussian random variable.
Note that
\begin{gather}
\label{eqn:pf:corollary:main-technical:logistic-regression:3}
  \pbetaFn{ 0 }
  =
  \frac{1}{1+e^{0}}
  =
  \frac{1}{2}
.\end{gather}
Hence, when \(  \betaX=0  \), \EQUATION \eqref{eqn:pf:corollary:main-technical:logistic-regression:1} trivially evaluates to
\(  \alphaX = \frac{1}{2}  \).
To bound \EQUATION \eqref{eqn:pf:corollary:main-technical:logistic-regression:1} when \(  \betaX > 0  \), we can directly apply the following result from \cite{hsu2024sample}.
%
\begin{lemma}[{\cite[{\LEMMA 13}]{hsu2024sample}}]
\label{lemma:gaussian-integral:E[p]}
Fix \(  \betaXX > 0  \), and let \(  \ZRV \sim \N(0,1)  \) be a standard univariate Gaussian random variable.
Then,
\begin{gather}
  \E \left[ \frac{1}{1 + e^{\betaXX | \ZRV |}} \right]
  \leq
  \min
  \left\{
    \frac{1}{2},
    \frac{1}{2} - \sqrt{\frac{2}{\pi}} \left( 1-\frac{\betaXX^{2}}{6} \right) \frac{\betaXX}{4},
    \sqrt{\frac{2}{\pi}} \frac{1}{\betaXX}
  \right\}
.\end{gather}
\end{lemma}
%
It immediately follows from \EQUATION \eqref{eqn:pf:corollary:main-technical:logistic-regression:1} and \LEMMA \ref{lemma:gaussian-integral:E[p]} that for \(  \betaX > 0  \),
\begin{align}
\label{eqn:pf:corollary:main-technical:logistic-regression:alpha:ub}
  \alphaX
  &=
  \E \left[ \frac{1}{1 + e^{\betaX | \ZRV |}} \right]
  \leq
  \min
  \left\{
    \frac{1}{2},
    \frac{1}{2} - \sqrt{\frac{2}{\pi}} \left( 1-\frac{\betaX^{2}}{6} \right) \frac{\betaX}{4},
    \sqrt{\frac{2}{\pi}} \frac{1}{\betaX}
  \right\}
.\end{align}
\checkthis%
Using the above bound on \(  \alphaX  \) in \EQUATION \eqref{eqn:pf:corollary:main-technical:logistic-regression:alpha:ub}, an upper bound on \(  \alphaO  \) can also be obtained.
Noting that
\(  \deltaX \leq \frac{1}{2}  \),
and letting
\(  \ConstbetaXThrsholdLR \defeq \ConstbetaXThrsholdValueLR  \),
if
\(  \betaX < \betaXThresholdLRTwo = \ConstbetaXThrsholdValuedeltaLR \frac{1}{\deltaX}  \),
then
\(  \alphaO = \max \{ \alphaX, \deltaX \} \leq \min \{ \frac{1}{2}, \sqrt{\frac{2}{\pi}} \frac{1}{\betaX} \}  \),
whereas if
\(  \betaX \geq \betaXThresholdLRTwo = \ConstbetaXThrsholdValuedeltaLR \frac{1}{\deltaX}  \),
then
\(  \alphaO = \max \{ \alphaX, \deltaX \} = \deltaX  \).
%
\par 
%
Next, an explicit form for an lower  bound on \(  \gammaX  \) will be derived.
This will largely hinge on showing that
\(  \gammaX \geq \sqrt{\frac{2}{\pi}} ( 1-2 \alphaX )  \),
from where the above bound on \(  \alphaX  \) can subsequently provide a closed-form bound on \(  \gammaX  \).
Towards this, define
\(  \zetaX \defeq 1 - \sqrt{\frac{\pi}{2}} \gammaX  \).
Then, the inequality
\(  \gammaX \geq \sqrt{\frac{2}{\pi}} ( 1-2 \alphaX )  \)
is equivalent to
\(  \zetaX \leq 2\alphaX  \),
the latter of which will be our focus.
Note that \(  \zetaX  \) can be calculated by the following expression:
\begin{align*}
  \zetaX
  =
  \int_{\zX=0}^{\zX=\infty}
  \zX
  e^{-\frac{1}{2} \zX^{2}}
  (\pExpr{\zX})
  d\zX
.\end{align*}
It is convenient to view \(  \alphaX  \) and \(  \zetaX  \) as being parameterized by \(  \betaX  \), and hence, the following argument will use the notations: \(  \alphaX( \betaX )  \) and \(  \zetaX( \betaX )  \).
Note that the \betaXnamelr, \(  \betaX \GTR 0  \), is left as 
implicit in \(  \pFn  \) to simplify the notation.
Then, when \(  \pFn  \) is taken as in \DEFINITION \ref{def:p:logistic-regression} for logistic regression, \(  \zetaX( \betaX )  \) has the form:
\begin{align*}
  \zetaX( \betaX )
  %
  %
  %
  %
  =
  \int_{\zX=0}^{\zX=\infty}
  \zX
  e^{-\frac{1}{2} \zX^{2}}
  (\pExpr{\zX})
  d\zX
  =
  2
  \int_{\zX=0}^{\zX=\infty}
  \zX
  e^{-\frac{1}{2} \zX^{2}}
  \pbetaFn{ -\zX }
  d\zX
,\end{align*}
where the second equality applies \EQUATION \eqref{eqn:pf:corollary:main-technical:logistic-regression:7}.
Notice that due to \EQUATION \eqref{eqn:pf:corollary:main-technical:logistic-regression:3}, when \(  \betaX = 0  \),
\begin{gather*}
  \zetaX( 0 )
  =
  2
  \int_{\zX=0}^{\zX=\infty}
  \zX
  e^{-\frac{1}{2} \zX^{2}}
  \pbetaFn{ 0 }
  d\zX
  =
  \int_{\zX=0}^{\zX=\infty}
  \zX
  e^{-\frac{1}{2} \zX^{2}}
  d\zX
  =
  1
,\end{gather*}
where the last equality is obtained by scaling the expected value of a standard half-normal random variable by \(  \sqrt{\frac{\pi}{2}}  \).
Thus, in this scenario,
\(  \zetaX(0) = 1 = 2 \cdot \frac{1}{2} = 2 \alphaX(0)  \),
where the last equality follows from the earlier observation
that \(  \alphaX(0) = \frac{1}{2}  \).
Given this, it suffices to show that the ratio \(  \frac{\zetaX( \betaX )}{\alphaX( \betaX )}  \) is maximized when \(  \betaX=0  \), i.e., that
\(  \sup_{\betaX \GTR 0} \frac{\zetaX( \betaX )}{\alphaX( \betaX )} = \frac{\zetaX( 0 )}{\alphaX( 0 )} = 2  \).
This would indeed be true if
\(  \frac{\partial}{\partial \betaX} \frac{\zetaX( \betaX )}{\alphaX( \betaX )} \leq 0  \)
for all \(  \betaX \GTR 0  \).
As scaling this by a positive constant will not affect the inequality, it will be more convenient to establish that
\(  \frac{\partial}{\partial \betaX} \frac{\zetaX( \betaX )}{\sqrt{2\pi} \alphaX( \betaX )} \leq 0  \),
which will be done next.
%
\par 
%
To begin, note the following partial derivatives.
\begin{gather}
\label{eqn:pf:corollary:main-technical:logistic-regression:4:alpha}
  \frac{\partial}{\partial \betaX} \sqrt{2\pi} \alphaX( \betaX )
  =
  \frac{\partial}{\partial \betaX}
  2
  \int_{\zX=0}^{\zX=\infty}
  e^{-\frac{1}{2} \zX^{2}}
  \pbetaFn{ -\zX }
  d\zX
  =
  2
  \int_{\zX=0}^{\zX=\infty}
  e^{-\frac{1}{2} \zX^{2}}
  \frac{\partial}{\partial \betaX} \pbetaFn{ -\zX }
  d\zX
,\\ \label{eqn:pf:corollary:main-technical:logistic-regression:4:gamma}
  \frac{\partial}{\partial \betaX} \zetaX( \betaX )
  =
  \frac{\partial}{\partial \betaX}
  2
  \int_{\zX=0}^{\zX=\infty}
  \zX
  e^{-\frac{1}{2} \zX^{2}}
  \pbetaFn{ -\zX }
  d\zX
  =
  2
  \int_{\zX=0}^{\zX=\infty}
  \zX
  e^{-\frac{1}{2} \zX^{2}}
  \frac{\partial}{\partial \betaX} \pbetaFn{ -\zX }
  d\zX
,\end{gather}
where due to the quotient rule,
\begin{align*}
  \frac{\partial}{\partial \betaX} \pbetaFn{ -\zX }
  &=
  \frac{\partial}{\partial \betaX} \frac{1}{1+e^{\betaX \zX}}
  =
  -\frac{\zX e^{\betaX \zX}}{( 1+e^{\betaX \zX} )^{2}}
  =
  -\frac{\zX}{( 1+e^{-\betaX \zX} ) ( 1+e^{\betaX \zX} )}
  =
  -\zX \pbetaFn{ \zX } \pbetaFn{ -\zX }
\TagEqn{\label{eqn:pf:corollary:main-technical:logistic-regression:4:p}}
.\end{align*}
Plugging \eqref{eqn:pf:corollary:main-technical:logistic-regression:4:p} into \eqref{eqn:pf:corollary:main-technical:logistic-regression:4:alpha} and \eqref{eqn:pf:corollary:main-technical:logistic-regression:4:gamma} and scaling by a factor of \(  \frac{1}{2}  \) yields
\begin{gather}
\label{eqn:pf:corollary:main-technical:logistic-regression:4:alpha:b}
  \frac{1}{2}
  \frac{\partial}{\partial \betaX} \sqrt{2\pi} \alphaX( \betaX )
  =
  \int_{\zX=0}^{\zX=\infty}
  e^{-\frac{1}{2} \zX^{2}}
  \frac{\partial}{\partial \betaX} \pbetaFn{ -\zX }
  d\zX
  =
  -
  \int_{\zX=0}^{\zX=\infty}
  \zX
  e^{-\frac{1}{2} \zX^{2}}
  \pbetaFn{ \zX } \pbetaFn{ -\zX }
  d\zX
,\\ \label{eqn:pf:corollary:main-technical:logistic-regression:4:gamma:b}
  \frac{1}{2}
  \frac{\partial}{\partial \betaX} \zetaX( \betaX )
  =
  \int_{\zX=0}^{\zX=\infty}
  \zX
  e^{-\frac{1}{2} \zX^{2}}
  \frac{\partial}{\partial \betaX} \pbetaFn{ -\zX }
  d\zX
  =
  -
  \int_{\zX=0}^{\zX=\infty}
  \zX^{2}
  e^{-\frac{1}{2} \zX^{2}}
  \pbetaFn{ \zX } \pbetaFn{ -\zX }
  d\zX
.\end{gather}
Then, by applying the quotient rule and plugging in \eqref{eqn:pf:corollary:main-technical:logistic-regression:4:alpha:b} and \eqref{eqn:pf:corollary:main-technical:logistic-regression:4:gamma:b},
\begin{align*}
  &
  \frac{\partial}{\partial \betaX} \frac{\zetaX( \betaX )}{\sqrt{2\pi} \alphaX( \betaX )}
  \\
  &=
  \frac
  {
    \sqrt{2\pi} \alphaX( \betaX ) \frac{\partial}{\partial \betaX} \zetaX( \betaX )
    -
    \zetaX( \betaX ) \frac{\partial}{\partial \betaX} \sqrt{2\pi} \alphaX( \betaX )
  }
  {2\pi \alphaX( \betaX )^{2}}
  \\
  &=
  \frac
  {
    ( -\zetaX( \betaX ) \frac{\partial}{\partial \betaX} \sqrt{2\pi} \alphaX( \betaX ) )
    -
    ( -\sqrt{2\pi} \alphaX( \betaX ) \frac{\partial}{\partial \betaX} \zetaX( \betaX ) )
  }
  {2\pi \alphaX( \betaX )^{2}}
  \\
  &=
  \frac
  {
    ( -\frac{1}{2} \zetaX( \betaX ) \frac{1}{2} \frac{\partial}{\partial \betaX} \sqrt{2\pi} \alphaX( \betaX ) )
    -
    ( -\frac{1}{2} \sqrt{2\pi} \alphaX( \betaX ) \frac{1}{2} \frac{\partial}{\partial \betaX} \zetaX( \betaX ) )
  }
  {2\pi ( \frac{1}{2} \alphaX( \betaX ) )^{2}}
  \\
  &=
  \tfrac
  {
    \left( \int_{\zX=0}^{\zX=\infty} \zX e^{-\frac{1}{2} \zX^{2}} \pbetaFn{ -\zX } d\zX \right)
    \left( \int_{\zX=0}^{\zX=\infty} \zX e^{-\frac{1}{2} \zX^{2}} \pbetaFn{ \zX } \pbetaFn{ -\zX } d\zX \right)
    -
    \left( \int_{\zX=0}^{\zX=\infty} e^{-\frac{1}{2} \zX^{2}} \pbetaFn{ -\zX } d\zX \right)
    \left( \int_{\zX=0}^{\zX=\infty} \zX^{2} e^{-\frac{1}{2} \zX^{2}} \pbetaFn{ \zX } \pbetaFn{ -\zX } d\zX \right)
  }
  {\left( \int_{\zX=0}^{\zX=\infty} e^{-\frac{1}{2} \zX^{2}} \pbetaFn{ -\zX } d\zX \right)^{2}}
.\end{align*}
In the last line, the numerator clearly determines the sign of \(  \frac{\partial}{\partial \betaX} \frac{\zetaX( \betaX )}{\sqrt{2\pi} \alphaX( \betaX )}  \).
Focusing in on this expression, the following claim provides an upper bound.
Its verification is deferred to the end of this proof of the corollary.
%
\begin{claim}
\label{claim:pf:corollary:main-technical:logistic-regression:1}
Using the notations of this proof,
\begin{align*}
  &
  \left( \int_{\zX=0}^{\zX=\infty} \zX e^{-\frac{1}{2} \zX^{2}} \pbetaFn{ -\zX } d\zX \right)
  \left( \int_{\zX=0}^{\zX=\infty} \zX e^{-\frac{1}{2} \zX^{2}} \pbetaFn{ \zX } \pbetaFn{ -\zX } d\zX \right)
  \\
  &-
  \left( \int_{\zX=0}^{\zX=\infty} e^{-\frac{1}{2} \zX^{2}} \pbetaFn{ -\zX } d\zX \right)
  \left( \int_{\zX=0}^{\zX=\infty} \zX^{2} e^{-\frac{1}{2} \zX^{2}} \pbetaFn{ \zX } \pbetaFn{ -\zX } d\zX \right)
  \\
  &\AlignSp \leq
  \left( \int_{\zX=0}^{\zX=\infty} e^{-\frac{1}{2} \zX^{2}} \pbetaFn{ -\zX } d\zX \right)
  \left( \int_{\zX=0}^{\zX=\infty} e^{-\frac{1}{2} \zX^{2}} \pbetaFn{ \zX } \pbetaFn{ -\zX } d\zX \right)
  \\
  &\AlignSp\AlignSp
  \left(
    \left(
      \frac
      {\int_{\zX=0}^{\zX=\infty} \zX e^{-\frac{1}{2} \zX^{2}} \pbetaFn{ \zX } \pbetaFn{ -\zX } d\zX}
      {\int_{\zX=0}^{\zX=\infty} e^{-\frac{1}{2} \zX^{2}} \pbetaFn{ \zX } \pbetaFn{ -\zX } d\zX}
    \right)^{2}
    -
    \frac
    {\int_{\zX=0}^{\zX=\infty} \zX^{2} e^{-\frac{1}{2} \zX^{2}} \pbetaFn{ \zX } \pbetaFn{ -\zX } d\zX}
    {\int_{\zX=0}^{\zX=\infty} e^{-\frac{1}{2} \zX^{2}} \pbetaFn{ \zX } \pbetaFn{ -\zX } d\zX}
  \right)
\TagEqn{\label{eqn:claim:pf:corollary:main-technical:logistic-regression:1}}
.\end{align*}
\end{claim}
%
Under the assumed correctness of \CLAIM \ref{claim:pf:corollary:main-technical:logistic-regression:1}, the proof of \COROLLARY \ref{corollary:main-technical:logistic-regression} can be completed.
The \RHS of the inequality in \EQUATION \eqref{eqn:claim:pf:corollary:main-technical:logistic-regression:1} has the same sign as
\newcommand{\EXPRVARIABLE}{y}
\begin{gather}
  \EXPRVARIABLE
  \defeq
  \left(
    \frac
    {\int_{\zX=0}^{\zX=\infty} \zX e^{-\frac{1}{2} \zX^{2}} \pbetaFn{ \zX } \pbetaFn{ -\zX } d\zX}
    {\int_{\zX=0}^{\zX=\infty} e^{-\frac{1}{2} \zX^{2}} \pbetaFn{ \zX } \pbetaFn{ -\zX } d\zX}
  \right)^{2}
  -
  \frac
  {\int_{\zX=0}^{\zX=\infty} \zX^{2} e^{-\frac{1}{2} \zX^{2}} \pbetaFn{ \zX } \pbetaFn{ -\zX } d\zX}
  {\int_{\zX=0}^{\zX=\infty} e^{-\frac{1}{2} \zX^{2}} \pbetaFn{ \zX } \pbetaFn{ -\zX } d\zX}
\end{gather}
since the product of the first two integrals is positive, \ie
\begin{gather*}
  \left( \int_{\zX=0}^{\zX=\infty} e^{-\frac{1}{2} \zX^{2}} \pbetaFn{ -\zX } d\zX \right)
  \left( \int_{\zX=0}^{\zX=\infty} e^{-\frac{1}{2} \zX^{2}} \pbetaFn{ \zX } \pbetaFn{ -\zX } d\zX \right)
  >
  0
.\end{gather*}
Hence, if \(  \EXPRVARIABLE \leq 0  \), then due to \CLAIM \ref{claim:pf:corollary:main-technical:logistic-regression:1} and the earlier discussion, it must also happen that \(  \frac{\partial}{\partial \zetaX} \frac{\zetaX( \betaX )}{\sqrt{2\pi} \alphaX( \betaX )} \leq 0  \).
%
\par 
%
\let\oldwX\wX%
\let\wX\zX%
\let\oldWRV\WRV%
\let\WRV\ZRV%
To establish the nonpositivity of \(  \EXPRVARIABLE  \), consider a univariate standard Gaussian random variable, \(  \WRV \sim \N(0,1)  \), and a random variable \(  \URV  \) which takes values in \(  \{ 0,1 \}  \) such that for \(  \wX \geq 0  \),
\begin{gather*}
  ( \URV \Mid| | \WRV |=\wX )
  =
  \begin{cases}
    0 ,& \cWP  1 - \pbetaFn{ \wX } \pbetaFn{ -\wX }, \\
    1 ,& \cWP \pbetaFn{ \wX } \pbetaFn{ -\wX }.
  \end{cases}
\end{gather*}
The mass function of this conditioned random variable, \(  \URV \Mid| | \WRV |  \), is given for \(  \wX \geq 0  \) by
\begin{gather*}
  \pdf{\URV \Mid| | \WRV |}( \uX \Mid| \wX )
  =
  \begin{cases}
  1 - \pbetaFn{ \wX } \pbetaFn{ -\wX } ,& \cIf \uX = 0 ,\\
  \pbetaFn{ \wX } \pbetaFn{ -\wX }     ,& \cIf \uX = 1 .
  \end{cases}
\end{gather*}
In addition, by the law of the total probability and the definition of conditional probabilities,
\begin{gather*}
  \pdf{\URV}( 1 )
  =
  \int_{\zX=-\infty}^{\zX=\infty} \pdf{\URV \Mid| | \WRV |}( 1 \Mid| \zX ) \pdf{| \WRV |}( \zX ) d\zX
  =
  \sqrt{\frac{2}{\pi}} \int_{\zX=0}^{\zX=\infty} e^{-\frac{1}{2} \zX^{2}} \pbetaFn{ \zX } \pbetaFn{ -\zX } d\zX
.\end{gather*}
Then, by Bayes' theorem, the density function of the conditioned random variable \(  | \WRV | \Mid| \URV = 1  \) is given for \(  \wX \geq 0  \) by
\begin{gather*}
  \pdf{| \WRV | \Mid| \URV} ( \wX \Mid| 1 )
  =
  \frac
  {\pdf{\URV \Mid| | \WRV |}( 1 \Mid| \wX ) \pdf{| \WRV |}( \wX )}
  {\pdf{\URV}( 1 )}
  =
  \frac
  {\sqrt{\frac{2}{\pi}} e^{-\frac{1}{2} \zX^{2}} \pbetaFn{ \zX } \pbetaFn{ -\zX } d\zX}
  {\sqrt{\frac{2}{\pi}} \int_{\zXX=0}^{\zXX=\infty} e^{-\frac{1}{2} \zXX^{2}} \pbetaFn{ \zXX } \pbetaFn{ -\zXX } d\zXX}
,\end{gather*}
while for \(  \wX < 0  \),
\(  \pdf{| \WRV | \Mid| \URV} ( \wX \Mid| 1 ) = 0  \).
The variance of \(  | \WRV | \Mid| \URV=1  \) is therefore:
\begin{align*}
  \Var( | \WRV | \Mid| \URV=1 )
  &=
  \E[ | \WRV |^{2} \Mid| \URV=1 ]
  -
  \E[ | \WRV | \Mid| \URV=1 ]^{2}
  \\
  &=
  \frac
  {\sqrt{\frac{2}{\pi}} \int_{\wX=0}^{\wX \infty} \zX^{2} e^{-\frac{1}{2} \zX^{2}} \pbetaFn{ \zX } \pbetaFn{ -\zX } d\zX}
  {\sqrt{\frac{2}{\pi}} \int_{\zX=0}^{\zX=\infty} e^{-\frac{1}{2} \zX^{2}} \pbetaFn{ \zX } \pbetaFn{ -\zX } d\zX}
  -
  \left(
    \frac
    {\sqrt{\frac{2}{\pi}} \int_{\wX=0}^{\wX \infty} \zX e^{-\frac{1}{2} \zX^{2}} \pbetaFn{ \zX } \pbetaFn{ -\zX } d\zX}
    {\sqrt{\frac{2}{\pi}} \int_{\zX=0}^{\zX=\infty} e^{-\frac{1}{2} \zX^{2}} \pbetaFn{ \zX } \pbetaFn{ -\zX } d\zX}
  \right)^{2}
  \\
  &=
  \frac
  {\int_{\wX=0}^{\wX \infty} \zX^{2} e^{-\frac{1}{2} \zX^{2}} \pbetaFn{ \zX } \pbetaFn{ -\zX } d\zX}
  {\int_{\zX=0}^{\zX=\infty} e^{-\frac{1}{2} \zX^{2}} \pbetaFn{ \zX } \pbetaFn{ -\zX } d\zX}
  -
  \left(
    \frac
    {\int_{\wX=0}^{\wX \infty} \zX e^{-\frac{1}{2} \zX^{2}} \pbetaFn{ \zX } \pbetaFn{ -\zX } d\zX}
    {\int_{\zX=0}^{\zX=\infty} e^{-\frac{1}{2} \zX^{2}} \pbetaFn{ \zX } \pbetaFn{ -\zX } d\zX}
  \right)^{2}
  \\
  &=
  -\EXPRVARIABLE
.\end{align*}
The variance of a random variable is always nonnegative, which implies that
\begin{gather*}
  \EXPRVARIABLE = -\Var( | \WRV | \Mid| \URV=1 ) \leq 0
,\end{gather*}
and thus, combining this with some previous remarks, it follows that \(  \frac{\partial}{\partial \betaX} \frac{\zetaX( \betaX )}{\sqrt{2\pi} \alphaX( \betaX )} \leq 0  \) when \(  \betaX \GTR 0  \), as claimed.
%
\par 
%
As noted earlier, this further implies that
\begin{gather*}
  \sup_{\betaX \GTR 0} \frac{\zetaX( \betaX )}{\alphaX( \betaX )}
  =
  \frac{\zetaX( 0 )}{\alphaX( 0 )}
  =
  2
.\end{gather*}
Then, by \LEMMA \ref{lemma:gaussian-integral:E[p]},
\begin{gather}
\label{eqn:pf:corollary:main-technical:logistic-regression:zeta:ub}
  \zetaX
  \leq
  2\alphaX
  \leq
  \min
  \left\{
    1,
    1 - \sqrt{\frac{2}{\pi}} \left( 1-\frac{\betaX^{2}}{6} \right) \frac{\betaX}{2},
    \sqrt{\frac{2}{\pi}} \frac{2}{\betaX}
  \right\}
.\end{gather}
This upper bound on \(  \zetaX  \) now gives the following lower bound on \(  \gammaX  \):
\begin{gather}
\label{eqn:pf:corollary:main-technical:logistic-regression:gamma:ub}
  \gammaX
  =
  \sqrt{\frac{2}{\pi}} ( 1-\zetaX )
  \geq
  \sqrt{\frac{2}{\pi}} ( 1-2\alphaX )
  \geq
  \sqrt{\frac{2}{\pi}}
  \left(
  1 -
  \min
  \left\{
    1,
    1 - \sqrt{\frac{2}{\pi}} \left( 1-\frac{\betaX^{2}}{6} \right) \frac{\betaX}{2},
    \sqrt{\frac{2}{\pi}} \frac{2}{\betaX}
  \right\}
  \right)
.\end{gather}
This completes \TASK \ref{enum:pf:corollary:main-technical:logistic-regression:b}.
%
\par 
%
The above work sets up the realization of the corollary's proof.
With the explicit bounds on \(  \alphaX  \) and \(  \gammaX  \), note the following:
\begin{gather}
\label{eqn:pf:corollary:main-technical:logistic-regression:5:a}
  \gammaX
  \geq
  \begin{cases}
  \left( 1 - \frac{\betaX^{2}}{6} \right) \frac{\betaX}{{\pi}}
  ,& \cIf \betaX \in (0, \betaXThresholdLROne) ,\\
  \sqrt{\frac{2}{\pi}} \left( 1 - \sqrt{\frac{2}{\pi}} \frac{2}{\betaX} \right)
  ,& \cIf \betaX \in [\betaXThresholdLROne, \betaXThresholdLRTwo) ,\\
  \sqrt{\frac{2}{\pi}} \left( 1 - \sqrt{\frac{2}{\pi}} \frac{2}{\betaX} \right)
  ,& \cIf \betaX \in [\betaXThresholdLRTwo, \infty) ,
  \end{cases}
  \geq
  \begin{cases}
  \left( 1 - \frac{\cO^{2}}{6} \right) \frac{\betaX}{{\pi}}
  ,& \cIf \betaX \in (0, \betaXThresholdLROne) ,\\
  \bO
  ,& \cIf \betaX \in [\betaXThresholdLROne, \betaXThresholdLRTwo) ,\\
  \bO
  ,& \cIf \betaX \in [\betaXThresholdLRTwo, \infty) .
  \end{cases}
\end{gather}
where
\(  \bO, \cO, \ConstbetaXThrsholdLR > 0  \)
are constants such that
\(  \cO = \frac{\sqrt{\hfrac{8}{\pi}}}{1+\bO} \geq 1  \) and
\(  \ConstbetaXThrsholdLR \defeq \ConstbetaXThrsholdValueLR  \).
Recall from an earlier discussion that
\begin{gather*}
  \alphaO
  \leq
  \begin{cases}
  \min \left\{ \frac{1}{2}, \sqrt{\frac{2}{\pi}} \frac{1}{\betaX} \right\} ,&\cIf \betaX < \frac{\ConstbetaXThrsholdLR}{\deltaX} ,\\
  \deltaX ,&\cIf \betaX \geq \frac{\ConstbetaXThrsholdLR}{\deltaX}.
  \end{cases}
\end{gather*}
As a result,
\begin{gather}
\label{eqn:pf:corollary:main-technical:logistic-regression:6}
  \frac{\alphaO}{\GAMMAX^{2}}
  \leq
  \begin{cases}
  \frac{\pi^{2}}{2 \bigl( 1 - \frac{\betaX^{2}}{6} \bigr)^{2} \betaX^{2}}
  ,& \cIf \betaX \in (0, \betaXThresholdLROne) ,\\
  \frac{\sqrt{\hfrac{\pi}{2}}}{\bigl( 1 - \sqrt{\frac{2}{\pi}} \frac{2}{\betaX} \bigr)^{2} \betaX}
  ,& \cIf \betaX \in [\betaXThresholdLROne, \betaXThresholdLRTwo) ,\\
  \frac{\pi \deltaX}{2 \bigl( 1 - \sqrt{\frac{2}{\pi}} \frac{2}{\betaX} \bigr)^{2}}
  ,& \cIf \betaX \in [\betaXThresholdLRTwo, \infty) ,
  \end{cases}
  \leq
  \begin{cases}
  \frac{\pi^{2}}{2 \bigl( 1 - \frac{\cO^{2}}{6} \bigr)^{2} \betaX^{2}}
  ,& \cIf \betaX \in (0, \betaXThresholdLROne) ,\\
  \frac{\sqrt{\hfrac{\pi}{2}}}{\bO^{2} \betaX}
  ,& \cIf \betaX \in [\betaXThresholdLROne, \betaXThresholdLRTwo) ,\\
  \frac{\pi \deltaX}{2 \bO^{2}}
  ,& \cIf \betaX \in [\betaXThresholdLRTwo, \infty) .
  \end{cases}
\end{gather}
Now, the corollary's result for logistic regression can be established by substituting the bounds in \EQUATIONS \eqref{eqn:pf:corollary:main-technical:logistic-regression:5:a}--\eqref{eqn:pf:corollary:main-technical:logistic-regression:6} into \EQUATION \eqref{eqn:main-technical:sparse:m} of \THEOREM \ref{thm:main-technical:sparse} and the definitions of \(  \nuX( \deltaX )  \) and \(  \tauX( \deltaX )  \).
Take \(  \m  \) to be at least
\begin{align*}
  \m
  &\geq
  \max \{ \mOneS, \mTwoS, \mThreeS, \mFourS, \mFiveS \}
  =
  \mOEXPRLR[s]{\epsilonX}[,]
\end{align*}
where
\checkthis[Done]%
\begin{gather*}
  \mOneS = \mOneEXPR[s]
  \\
  \mTwoS = \mTwoEXPR[s]
  \\
  \mThreeS = \mThreeEXPR[s]
  \\
  \mFourS = \mFourEXPR[s]
  \\
  \mFiveS = \mFiveEXPR[s]
\end{gather*}
and where, as a result of \EQUATION \eqref{eqn:pf:corollary:main-technical:logistic-regression:5:a}, \(  \nuX  \) is bounded from below by
\checkthis%
\begin{gather*}
  \nuX
  =
  \nuXEXPR
  \geq
  \begin{cases}
  \nuXEXPRLROne
  ,& \cIf \betaX \in (0, \betaXThresholdLROne) ,\\
  \nuXEXPRLRTwo
  ,& \cIf \betaX \in [\betaXThresholdLROne, \betaXThresholdLRTwo) ,\\
  \nuXEXPRLRThree
  ,& \cIf \betaX \in [\betaXThresholdLRTwo, \infty) .
  \end{cases}
\end{gather*}
Then, due to \EQUATIONS \eqref{eqn:pf:corollary:main-technical:logistic-regression:5:a} and \eqref{eqn:pf:corollary:main-technical:logistic-regression:6},
\begin{gather*}
  \m \geq \mEXPR[s][.]
\end{gather*}
Therefore, by \THEOREM \ref{thm:main-technical:sparse} and this choice of \(  \m  \), with probability at least \(  1-\rho  \), uniformly for all \(  \thetaXX \in \ParamSpace  \) and \(  \JCoords \subseteq [\n]  \), \(  | \JCoords | \leq \k  \), \EQUATION \eqref{eqn:main-technical:sparse:1} holds:
\begin{gather*}
  \left\|
    \thetaStar
    -
    \frac
    {\thetaXX + \hfFn[\JCoords]( \thetaStar, \thetaXX )}
    {\| \thetaXX + \hfFn[\JCoords]( \thetaStar, \thetaXX ) \|_{2}}
  \right\|_{2}
  \leq
  \sqrt{\deltaX \| \thetaStar-\thetaXX \|_{2}}
  +
  \deltaX
.\end{gather*}
%
\par 
%
Barring the verification of \CLAIM \ref{claim:pf:corollary:main-technical:logistic-regression:1}, this concludes
the proof of \COROLLARY \ref{corollary:main-technical:logistic-regression} for logistic regression.
The last remaining task is returning to and proving \CLAIM \ref{claim:pf:corollary:main-technical:logistic-regression:1}.
%
\begin{subproof}
{\CLAIM \ref{claim:pf:corollary:main-technical:logistic-regression:1}}
Looking at the \LHS of \EQUATION \eqref{eqn:claim:pf:corollary:main-technical:logistic-regression:1}, observe:
\begin{align*}
  &
  \left( \int_{\zX=0}^{\zX=\infty} \zX e^{-\frac{1}{2} \zX^{2}} \pbetaFn{ -\zX } d\zX \right)
  \left( \int_{\zX=0}^{\zX=\infty} \zX e^{-\frac{1}{2} \zX^{2}} \pbetaFn{ \zX } \pbetaFn{ -\zX } d\zX \right)
  \\
  &-
  \left( \int_{\zX=0}^{\zX=\infty} e^{-\frac{1}{2} \zX^{2}} \pbetaFn{ -\zX } d\zX \right)
  \left( \int_{\zX=0}^{\zX=\infty} \zX^{2} e^{-\frac{1}{2} \zX^{2}} \pbetaFn{ \zX } \pbetaFn{ -\zX } d\zX \right)
  \\
  &\AlignSp=
  \left( \int_{\zX=0}^{\zX=\infty} e^{-\frac{1}{2} \zX^{2}} \pbetaFn{ -\zX } d\zX \right)
  \left( \int_{\zX=0}^{\zX=\infty} e^{-\frac{1}{2} \zX^{2}} \pbetaFn{ \zX } \pbetaFn{ -\zX } d\zX \right)
  \\
  &\AlignSp\AlignSp \left(
  \left( \tfrac{\int_{\zX=0}^{\zX=\infty} \zX e^{-\frac{1}{2} \zX^{2}} \pbetaFn{ -\zX } d\zX}{\int_{\zX=0}^{\zX=\infty} e^{-\frac{1}{2} \zX^{2}} \pbetaFn{ -\zX } d\zX} \right)
  \left( \tfrac{\int_{\zX=0}^{\zX=\infty} \zX e^{-\frac{1}{2} \zX^{2}} \pbetaFn{ \zX } \pbetaFn{ -\zX } d\zX}{\int_{\zX=0}^{\zX=\infty} e^{-\frac{1}{2} \zX^{2}} \pbetaFn{ \zX } \pbetaFn{ -\zX } d\zX} \right)
  -
  \left( \tfrac{\int_{\zX=0}^{\zX=\infty} \zX^{2} e^{-\frac{1}{2} \zX^{2}} \pbetaFn{ \zX } \pbetaFn{ -\zX } d\zX}{\int_{\zX=0}^{\zX=\infty} e^{-\frac{1}{2} \zX^{2}} \pbetaFn{ \zX } \pbetaFn{ -\zX } d\zX} \right)
  \right)
\TagEqn{\label{eqn:pf:claim:pf:corollary:main-technical:logistic-regression:1:1}}
.\end{align*}
In the above equation, \eqref{eqn:pf:claim:pf:corollary:main-technical:logistic-regression:1:1}, the term
\begin{gather*}
  \frac
  {\int_{\zX=0}^{\zX=\infty} \zX e^{-\frac{1}{2} \zX^{2}} \pbetaFn{ -\zX } d\zX}
  {\int_{\zX=0}^{\zX=\infty} e^{-\frac{1}{2} \zX^{2}} \pbetaFn{ -\zX } d\zX}
\end{gather*}
can be bounded from above as follows.
Similarly to earlier in the proof, let \(  \WRV \sim \N(0,1)  \) be a univariate standard Gaussian random variable such that \(  | \WRV |  \) is a standard half-normal random variable with density
\begin{gather}
\label{eqn:pf:claim:pf:corollary:main-technical:logistic-regression:2}
  \pdf{\WRV}( \wX )
  =
  \begin{cases}
  0 ,& \cIf \wX < 0 ,\\
  \sqrt{\frac{2}{\pi}} e^{-\frac{1}{2} \wX^{2}} ,& \cIf \wX \geq 0,
  \end{cases}
\end{gather}
and let \(  \URV  \) and \(  \VRV  \) be random variables taking values in \(  \{ 0,1 \}  \), where for \(  \wX \geq 0  \),
\begin{gather*}
  ( \URV \Mid| | \WRV |=\wX )
  =
  \begin{cases}
    0 ,& \cWP  1 - \pbetaFn{ \wX } \pbetaFn{ -\wX }, \\
    1 ,& \cWP \pbetaFn{ \wX } \pbetaFn{ -\wX },
  \end{cases}
  \\
  ( \VRV \Mid| | \WRV |=\wX )
  =
  \begin{cases}
    0 ,& \cWP  1 - \pbetaFn{ -\wX }, \\
    1 ,& \cWP \pbetaFn{ -\wX }.
  \end{cases}
\end{gather*}
These conditioned random variables, \(  \URV \Mid| | \WRV |  \) and \(  \VRV \Mid| | \WRV |  \), have mass functions:
\begin{gather}
\label{eqn:pf:claim:pf:corollary:main-technical:logistic-regression:3}
  \pdf{\URV \Mid| | \WRV |}( \uX \Mid| \wX )
  =
  \begin{cases}
  1 - \pbetaFn{ \wX } \pbetaFn{ -\wX } ,& \cIf \uX=0 ,\\
  \pbetaFn{ \wX } \pbetaFn{ -\wX }     ,& \cIf \uX=1 ,
  \end{cases}
  \\
\label{eqn:pf:claim:pf:corollary:main-technical:logistic-regression:4}
  \pdf{\VRV \Mid| | \WRV |}( \vX \Mid| \wX )
  =
  \begin{cases}
  1 - \pbetaFn{ -\wX } ,& \cIf \vX=0 ,\\
  \pbetaFn{ -\wX }     ,& \cIf \vX=1 ,
  \end{cases}
\end{gather}
for \(  \uX, \vX \in \{ 0,1 \}  \) and \(  \wX \geq 0  \).
Applying, in order twice, the law of total probability, the definition of conditional probabilities, and \EQUATIONS \eqref{eqn:pf:claim:pf:corollary:main-technical:logistic-regression:2}-\eqref{eqn:pf:claim:pf:corollary:main-technical:logistic-regression:4} obtains:
\begin{gather*}
  \pdf{\URV}( 1 )
  =
  \int_{\zX=-\infty}^{\zX=\infty} \pdf{\URV, | \WRV |}( 1, \wX ) d\zX
  =
  \int_{\zX=0}^{\zX=\infty} \pdf{\URV \Mid| | \WRV |}( 1 \Mid| \wX ) \pdf{| \WRV |}( \wX ) d\zX
  =
  \sqrt{\frac{2}{\pi}} \int_{\zX=0}^{\zX=\infty} e^{-\frac{1}{2} \zX^{2}} \pbetaFn{ \zX } \pbetaFn{ -\zX } d\zX
  ,\\
  \pdf{\VRV}( 1 )
  =
  \int_{\zX=-\infty}^{\zX=\infty} \pdf{\VRV, | \WRV |}( 1, \wX ) d\zX
  =
  \int_{\zX=0}^{\zX=\infty} \pdf{\VRV \Mid| | \WRV |}( 1 \Mid| \wX ) \pdf{| \WRV |}( \wX ) d\zX
  =
  \sqrt{\frac{2}{\pi}} \int_{\zX=0}^{\zX=\infty} e^{-\frac{1}{2} \zX^{2}} \pbetaFn{ -\wX } d\zX
.\end{gather*}
Using Bayes' theorem, for \(  \wX \geq 0  \), the density functions of the conditioned random variables \(  | \WRV | \Mid| \URV = 1  \) and \(  | \WRV | \Mid| \VRV = 1  \) are respectively given by
\begin{gather*}
  \pdf{| \WRV | \Mid| \URV} ( \wX \Mid| 1 )
  =
  \frac
  {\pdf{\URV \Mid| | \WRV |}( 1 \Mid| \wX ) \pdf{| \WRV |}( \wX )}
  {\pdf{\URV}( 1 )}
  =
  \frac
  {\sqrt{\frac{2}{\pi}} e^{-\frac{1}{2} \wX^{2}} \pbetaFn{ \wX } \pbetaFn{ -\wX } d\zX}
  {\sqrt{\frac{2}{\pi}} \int_{\zXX=0}^{\zXX=\infty} e^{-\frac{1}{2} \zXX^{2}} \pbetaFn{ \zXX } \pbetaFn{ -\zXX } d\zXX}
  ,\\
  \pdf{| \WRV | \Mid| \VRV} ( \wX \Mid| 1 )
  =
  \frac
  {\pdf{\VRV \Mid| | \WRV |}( 1 \Mid| \wX ) \pdf{| \WRV |}( \wX )}
  {\pdf{\VRV}( 1 )}
  =
  \frac
  {\sqrt{\frac{2}{\pi}} e^{-\frac{1}{2} \wX^{2}} \pbetaFn{ -\wX } d\zX}
  {\sqrt{\frac{2}{\pi}} \int_{\zXX=0}^{\zXX=\infty} e^{-\frac{1}{2} \zXX^{2}} \pbetaFn{ -\zXX } d\zXX}
,\end{gather*}
while for \(  \wX < 0  \),
\(  \pdf{| \WRV | \Mid| \URV} ( \wX \Mid| 1 ) = 0  \) and
\(  \pdf{| \WRV | \Mid| \VRV} ( \wX \Mid| 1 ) = 0  \).
Then, in expectation,
\begin{gather*}
  \E[ | \WRV | \Mid| \URV=1 ]
  =
  \frac
  {\sqrt{\frac{2}{\pi}} \int_{\zX=0}^{\zX=\infty} \zX e^{-\frac{1}{2} \zX^{2}} \pbetaFn{ \zX } \pbetaFn{ -\zX } d\zX}
  {\sqrt{\frac{2}{\pi}} \int_{\zX=0}^{\zX=\infty} e^{-\frac{1}{2} \zX^{2}} \pbetaFn{ \zX } \pbetaFn{ -\zX } d\zX}
  =
  \frac
  {\int_{\zX=0}^{\zX=\infty} \zX e^{-\frac{1}{2} \zX^{2}} \pbetaFn{ \zX } \pbetaFn{ -\zX } d\zX}
  {\int_{\zX=0}^{\zX=\infty} e^{-\frac{1}{2} \zX^{2}} \pbetaFn{ \zX } \pbetaFn{ -\zX } d\zX}
  ,\\
  \E[ | \WRV | \Mid| \VRV=1 ]
  =
  \frac
  {\sqrt{\frac{2}{\pi}} \int_{\zX=0}^{\zX=\infty} \zX e^{-\frac{1}{2} \zX^{2}} \pbetaFn{ -\zX } d\zX}
  {\sqrt{\frac{2}{\pi}} \int_{\zX=0}^{\zX=\infty} e^{-\frac{1}{2} \zX^{2}} \pbetaFn{ -\zX } d\zX}
  =
  \frac
  {\int_{\zX=0}^{\zX=\infty} \zX e^{-\frac{1}{2} \zX^{2}} \pbetaFn{ -\zX } d\zX}
  {\int_{\zX=0}^{\zX=\infty} e^{-\frac{1}{2} \zX^{2}} \pbetaFn{ -\zX } d\zX}
.\end{gather*}
Since \(  \pbetaFn{ \zX } = \frac{1}{1+e^{-\betaX \zX}}  \) is a nondecreasing function and the support of \(  \pdf{| \WRV | \Mid| \URV=1}  \) and \(  \pdf{| \WRV | \Mid| \VRV=1}  \) lies on the nonnegative real line, it follows that
\(  \E[ | \WRV | \Mid| \URV=1 ] \geq \E[ | \WRV | \Mid| \VRV=1 ]  \),
and hence, by the above pair of equations,
\begin{gather}
\label{eqn:pf:claim:pf:corollary:main-technical:logistic-regression:1:2}
  \frac
  {\int_{\zX=0}^{\zX=\infty} \zX e^{-\frac{1}{2} \zX^{2}} \pbetaFn{ -\zX } d\zX}
  {\int_{\zX=0}^{\zX=\infty} e^{-\frac{1}{2} \zX^{2}} \pbetaFn{ -\zX } d\zX}
  =
  \E[ | \WRV | \Mid| \VRV=1 ]
  \leq
  \E[ | \WRV | \Mid| \URV=1 ]
  =
  \frac
  {\int_{\zX=0}^{\zX=\infty} \zX e^{-\frac{1}{2} \zX^{2}} \pbetaFn{ \zX } \pbetaFn{ -\zX } d\zX}
  {\int_{\zX=0}^{\zX=\infty} e^{-\frac{1}{2} \zX^{2}} \pbetaFn{ \zX } \pbetaFn{ -\zX } d\zX}
.\end{gather}
Therefore, returning to \EQUATION \eqref{eqn:pf:claim:pf:corollary:main-technical:logistic-regression:1:1} and applying the above inequality in \EQUATION \eqref{eqn:pf:claim:pf:corollary:main-technical:logistic-regression:1:2}, \CLAIM \ref{claim:pf:corollary:main-technical:logistic-regression:1} follows:
\begin{align*}
  &
  \left( \int_{\zX=0}^{\zX=\infty} \zX e^{-\frac{1}{2} \zX^{2}} \pbetaFn{ -\zX } d\zX \right)
  \left( \int_{\zX=0}^{\zX=\infty} \zX e^{-\frac{1}{2} \zX^{2}} \pbetaFn{ \zX } \pbetaFn{ -\zX } d\zX \right)
  \\
  &-
  \left( \int_{\zX=0}^{\zX=\infty} e^{-\frac{1}{2} \zX^{2}} \pbetaFn{ -\zX } d\zX \right)
  \left( \int_{\zX=0}^{\zX=\infty} \zX^{2} e^{-\frac{1}{2} \zX^{2}} \pbetaFn{ \zX } \pbetaFn{ -\zX } d\zX \right)
  \\
  &\AlignSp=
  \left( \int_{\zX=0}^{\zX=\infty} e^{-\frac{1}{2} \zX^{2}} \pbetaFn{ -\zX } d\zX \right)
  \left( \int_{\zX=0}^{\zX=\infty} e^{-\frac{1}{2} \zX^{2}} \pbetaFn{ \zX } \pbetaFn{ -\zX } d\zX \right)
  \\
  &\AlignSp\AlignSp \left(
  \left( \tfrac{\int_{\zX=0}^{\zX=\infty} \zX e^{-\frac{1}{2} \zX^{2}} \pbetaFn{ -\zX } d\zX}{\int_{\zX=0}^{\zX=\infty} e^{-\frac{1}{2} \zX^{2}} \pbetaFn{ -\zX } d\zX} \right)
  \left( \tfrac{\int_{\zX=0}^{\zX=\infty} \zX e^{-\frac{1}{2} \zX^{2}} \pbetaFn{ \zX } \pbetaFn{ -\zX } d\zX}{\int_{\zX=0}^{\zX=\infty} e^{-\frac{1}{2} \zX^{2}} \pbetaFn{ \zX } \pbetaFn{ -\zX } d\zX} \right)
  -
  \left( \tfrac{\int_{\zX=0}^{\zX=\infty} \zX^{2} e^{-\frac{1}{2} \zX^{2}} \pbetaFn{ \zX } \pbetaFn{ -\zX } d\zX}{\int_{\zX=0}^{\zX=\infty} e^{-\frac{1}{2} \zX^{2}} \pbetaFn{ \zX } \pbetaFn{ -\zX } d\zX} \right)
  \right)
  \\
  &\AlignSp \leq
  \left( \int_{\zX=0}^{\zX=\infty} e^{-\frac{1}{2} \zX^{2}} \pbetaFn{ -\zX } d\zX \right)
  \left( \int_{\zX=0}^{\zX=\infty} e^{-\frac{1}{2} \zX^{2}} \pbetaFn{ \zX } \pbetaFn{ -\zX } d\zX \right)
  \\
  &\AlignSp\AlignSp
  \left(
    \left(
      \tfrac
      {\int_{\zX=0}^{\zX=\infty} \zX e^{-\frac{1}{2} \zX^{2}} \pbetaFn{ \zX } \pbetaFn{ -\zX } d\zX}
      {\int_{\zX=0}^{\zX=\infty} e^{-\frac{1}{2} \zX^{2}} \pbetaFn{ \zX } \pbetaFn{ -\zX } d\zX}
    \right)^{2}
    -
    \tfrac
    {\int_{\zX=0}^{\zX=\infty} \zX^{2} e^{-\frac{1}{2} \zX^{2}} \pbetaFn{ \zX } \pbetaFn{ -\zX } d\zX}
    {\int_{\zX=0}^{\zX=\infty} e^{-\frac{1}{2} \zX^{2}} \pbetaFn{ \zX } \pbetaFn{ -\zX } d\zX}
  \right)
,\end{align*}
as desired.
\end{subproof}
Having established \CLAIM \ref{claim:pf:corollary:main-technical:logistic-regression:1}, \COROLLARY \ref{corollary:main-technical:logistic-regression} for logistic regression is thus proved.
\let\wX\oldwX%
\let\WRV\oldWRV%

Next, the specialization to \textbf{probit regression} in \COROLLARY \ref{corollary:main-technical:logistic-regression} is addressed.

%
\mostlycheckoff%
As in the proof of 
the logistic case, the probit case follows from a two-step argument---the first being \STEP \Enum[{\label{enum:pf:corollary:main-technical:probit:a}}]{a}, verifying that \ASSUMPTION \ref{assumption:p} holds when \(  \pFn  \) is parameterized by the \betaXnamepr, \(  \betaX \GTR 0  \), and defined at \(  \zX \in \R  \) as
\begin{gather}
\label{eqn:pf:corollary:main-technical:probit:p}
  \pbetaFn{ \zX } = \frac{1}{\sqrt{2\pi}} \int_{\uX=-\infty}^{\uX=\betaX \zX} e^{-\frac{1}{2} \uX^{2}} d\uX
\end{gather}
for probit regression, as per \DEFINITION \ref{def:p:probit}; and the second being \STEP \Enum[{\label{enum:pf:corollary:main-technical:probit:b}}]{b}, deriving \(  \alphaX  \) and \(  \gammaX  \).
%
\par 
%
For  \ref{enum:pf:corollary:main-technical:probit:a},
beginning with \CONDITION \ref{condition:assumption:p:i} of \ASSUMPTION \ref{assumption:p}---that \(  \pFn  \) nondecreasing over the real line---notice that when \(  \betaX = 0  \), the function \(  \pFn  \) is constant:
\(  \pFn( \zX ) = \frac{1}{2}  \)
for all \(  \zX \in \R  \), and therefore, in this scenario, \(  \pFn  \) is nondecreasing.
Otherwise, when \(  \betaX > 0  \), \(  \pFn  \) can be rewritten as follows for \(  \zX \in \R  \):
\begin{gather*}
  \pbetaFn{ \zX }
  =
  \frac{1}{\sqrt{2\pi}} \int_{\uX=-\infty}^{\uX=\betaX \zX} e^{-\frac{1}{2} \uX^{2}} d\uX
  =
  \frac{\betaX}{\sqrt{2\pi}} \int_{\uX=-\infty}^{\uX=\zX} e^{-\frac{1}{2} \betaX^{2} \uX^{2}} d\uX
,\end{gather*}
where the second equality applies a change of variables.
Hence, \(  \pFn  \) is the distribution function of a mean\nobreakdash-\(  0  \), variance\nobreakdash-\(  \frac{1}{\betaX^{2}}  \) Gaussian random variable.
Since distribution functions are nondecreasing, clearly the first condition of \ASSUMPTION \ref{assumption:p} holds under this model.
Proceeding to the second requirement, \CONDITION \ref{condition:assumption:p:ii}, of \ASSUMPTION \ref{assumption:p}---that
\begin{gather}
\label{eqn:pf:corollary:main-technical:probit:1}
  \frac{\partial}{\partial \zX}
  \frac{\pExpr*{\zX+\wX \betaXParam}}{\pExpr{\zX \betaXParam}}
  \leq 0
\end{gather}
for all
\(  \zX \in \R  \) and \(  \wX > 0  \)%
---note the following properties of \(  \pFn  \):
\begin{gather}
\label{eqn:pf:corollary:main-technical:probit:6}
  1 - \pbetaFn{ \zX }
  =
  1 - \frac{1}{\sqrt{2\pi}} \int_{\uX=-\infty}^{\uX=\betaX \zX} e^{-\frac{1}{2} \uX^{2}} d\uX
  =
  \frac{1}{\sqrt{2\pi}} \int_{\uX=\betaX \zX}^{\uX=\infty} e^{-\frac{1}{2} \uX^{2}} d\uX
  =
  \frac{1}{\sqrt{2\pi}} \int_{\uX=-\infty}^{\uX=-\betaX \zX} e^{-\frac{1}{2} \uX^{2}} d\uX
  =
  \pbetaFn{ -\zX }
  ,\\
\label{eqn:pf:corollary:main-technical:probit:7}
  1 - \pbetaFn{ \zX }
  =
  1 - \frac{1}{\sqrt{2\pi}} \int_{\uX=-\infty}^{\uX=\betaX \zX} e^{-\frac{1}{2} \uX^{2}} d\uX
  =
  \frac{1}{\sqrt{2\pi}} \int_{\uX=\betaX \zX}^{\uX=\infty} e^{-\frac{1}{2} \uX^{2}} d\uX
  =
  \frac{1}{\sqrt{2\pi}} \int_{\uX=0}^{\uX=\infty} e^{-\frac{1}{2} ( \uX+\betaX \zX )^{2}} d\uX
,\end{gather}
where
\(  \zX \in \R  \).
Additionally, it suffices to verify \ASSUMPTION \ref{condition:assumption:p:ii} for \(  \betaX = 1  \).
Thus, for \(  \wX > 0  \),
\begin{align*}
  \frac{\pExpr*{\zX+\wX}}{\pExpr{\zX}}
  &=
  \frac{2( 1-\pbetaFn{ \zX+\wX })}{2( 1-\pbetaFn{ \zX } )}
  \\
  &=
  \frac{1-\pbetaFn{ \zX+\wX }}{1-\pbetaFn{ \zX }}
  \\
  &=
  \frac
  {\frac{1}{\sqrt{2\pi}} \int_{\uX=0}^{\uX=\infty} e^{-\frac{1}{2} ( \uX + \zX+\wX )^{2}} d\uX}
  {\frac{1}{\sqrt{2\pi}} \int_{\uX=0}^{\uX=\infty} e^{-\frac{1}{2} ( \uX + \zX )^{2}} d\uX}
  \\
  &=
  \frac
  {\int_{\uX=0}^{\uX=\infty} e^{-\frac{1}{2} ( \uX + \zX+\wX )^{2}} d\uX}
  {\int_{\uX=0}^{\uX=\infty} e^{-\frac{1}{2} ( \uX + \zX )^{2}} d\uX}
,\end{align*}
and hence, the desired inequality in \EQUATION \eqref{eqn:pf:corollary:main-technical:probit:1} is equivalently stated as:
\begin{align*}
  \frac{\partial}{\partial \zX}
  \frac{\pExpr*{\zX+\wX}}{\pExpr{\zX}}
  &=
  \frac{\partial}{\partial \zX}
  \frac
  {\int_{\uX=0}^{\uX=\infty} e^{-\frac{1}{2} ( \uX + \zX+\wX )^{2}} d\uX}
  {\int_{\uX=0}^{\uX=\infty} e^{-\frac{1}{2} ( \uX + \zX )^{2}} d\uX}
  \leq 0
.\end{align*}
To evaluate this partial derivative, observe:
\begin{align*}
  &\AlignSp
  \frac{\partial}{\partial \zX}
  \frac{\pExpr*{\zX+\wX}}{\pExpr{\zX}}
  \\
  &\AlignSp=
  \frac{\partial}{\partial \zX}
  \frac
  {\int_{\uX=0}^{\uX=\infty} e^{-\frac{1}{2} ( \uX + \zX+\wX )^{2}} d\uX}
  {\int_{\uX=0}^{\uX=\infty} e^{-\frac{1}{2} ( \uX + \zX )^{2}} d\uX}
  \\
  &\AlignSp=
  \frac
  {\left( \int_{\uX=0}^{\uX=\infty} e^{-\frac{1}{2} ( \uX + \zX )^{2}} d\uX \right)
   \left( \frac{\partial}{\partial \zX} \int_{\uX=0}^{\uX=\infty} e^{-\frac{1}{2} ( \uX + \zX+\wX )^{2}} d\uX \right)}
   {\left( \int_{\uX=0}^{\uX=\infty} e^{-\frac{1}{2} ( \uX + \zX )^{2}} d\uX \right)^{2}}
   \\
   &\AlignSp\AlignSp-
   \frac{
   \left( \int_{\uX=0}^{\uX=\infty} e^{-\frac{1}{2} ( \uX + \zX+\wX )^{2}} d\uX \right)
   \left( \frac{\partial}{\partial \zX} \int_{\uX=0}^{\uX=\infty} e^{-\frac{1}{2} ( \uX + \zX )^{2}} d\uX \right)}
  {\left( \int_{\uX=0}^{\uX=\infty} e^{-\frac{1}{2} ( \uX + \zX )^{2}} d\uX \right)^{2}}
  \\
  &\dCmt{by the quotient rule}
  \\
  &\AlignSp=
  \frac{1}{\left( \int_{\uX=0}^{\uX=\infty} e^{-\frac{1}{2} ( \uX + \zX )^{2}} d\uX \right)^{2}}
  \left(
  \left( \int_{\uX=0}^{\uX=\infty} e^{-\frac{1}{2} ( \uX + \zX )^{2}} d\uX \right)
   \left( \frac{\partial}{\partial \zX} \int_{\uX=0}^{\uX=\infty} e^{-\frac{1}{2} ( \uX + \zX+\wX )^{2}} d\uX \right)
   \right.
   \\
   &\phantom{\AlignSp \displaystyle \frac{1}{\displaystyle \left( \int_{\uX=0}^{\uX=\infty} e^{-\frac{1}{2} ( \uX + \zX )^{2}} d\uX \right)^{2}} \Biggl(} \AlignSp
   \left.
   -
   \left( \int_{\uX=0}^{\uX=\infty} e^{-\frac{1}{2} ( \uX + \zX+\wX )^{2}} d\uX \right)
   \left( \frac{\partial}{\partial \zX} \int_{\uX=0}^{\uX=\infty} e^{-\frac{1}{2} ( \uX + \zX )^{2}} d\uX \right)
   \right)
  \\
  &\AlignSp=
  \frac{1}{\left( \int_{\uX=0}^{\uX=\infty} e^{-\frac{1}{2} ( \uX + \zX )^{2}} d\uX \right)^{2}}
  \left(
  \left( \int_{\uX=0}^{\uX=\infty} e^{-\frac{1}{2} ( \uX + \zX )^{2}} d\uX \right)
  \left( \int_{\uX=0}^{\uX=\infty} \frac{\partial}{\partial \zX} e^{-\frac{1}{2} ( \uX + \zX+\wX )^{2}} d\uX \right)
  \right.
  \\
  &\phantom{\AlignSp \displaystyle \frac{1}{\displaystyle \left( \int_{\uX=0}^{\uX=\infty} e^{-\frac{1}{2} ( \uX + \zX )^{2}} d\uX \right)^{2}} \Biggl(} \AlignSp
  \left.
  -
  \left( \int_{\uX=0}^{\uX=\infty} e^{-\frac{1}{2} ( \uX + \zX+\wX )^{2}} d\uX \right)
  \left( \int_{\uX=0}^{\uX=\infty} \frac{\partial}{\partial \zX} e^{-\frac{1}{2} ( \uX + \zX )^{2}} d\uX \right)
  \right)
  \\
  &\dCmt{\(  \zX  \) does not depend on the variable of integration}
  \\
  &\AlignSp=
  \frac{1}{\left( \int_{\uX=0}^{\uX=\infty} e^{-\frac{1}{2} ( \uX + \zX )^{2}} d\uX \right)^{2}}
  \left(
  \left( \int_{\uX=0}^{\uX=\infty} e^{-\frac{1}{2} ( \uX + \zX )^{2}} d\uX \right)
  \left( -\int_{\uX=0}^{\uX=\infty} ( \uX + \zX+\wX ) e^{-\frac{1}{2} ( \uX + \zX+\wX )^{2}} d\uX \right)
  \right.
  \\
  &\phantom{\AlignSp \displaystyle \frac{1}{\displaystyle \left( \int_{\uX=0}^{\uX=\infty} e^{-\frac{1}{2} ( \uX + \zX )^{2}} d\uX \right)^{2}} \Biggl(} \AlignSp
  \left.
  -
  \left( \int_{\uX=0}^{\uX=\infty} e^{-\frac{1}{2} ( \uX + \zX+\wX )^{2}} d\uX \right)
  \left( -\int_{\uX=0}^{\uX=\infty} ( \uX + \zX ) e^{-\frac{1}{2} ( \uX + \zX )^{2}} d\uX \right)
  \right)
  \\
  &\dCmt{via the chain rule}
  \\
  &\AlignSp=
  \frac
  {
   \left( \int_{\uX=0}^{\uX=\infty} e^{-\frac{1}{2} ( \uX + \zX )^{2}} d\uX \right)
   \left( \int_{\uX=0}^{\uX=\infty} e^{-\frac{1}{2} ( \uX + \zX+\wX )^{2}} d\uX \right)}
  {\left( \int_{\uX=0}^{\uX=\infty} e^{-\frac{1}{2} ( \uX + \zX )^{2}} d\uX \right)^{2}}
  \\
  &\AlignSp\AlignSp
  \left(
  \frac
  {\int_{\uX=0}^{\uX=\infty} ( \uX + \zX ) e^{-\frac{1}{2} ( \uX + \zX )^{2}} d\uX}
  {\int_{\uX=0}^{\uX=\infty} e^{-\frac{1}{2} ( \uX + \zX )^{2}} d\uX}
  -
  \frac
  {\int_{\uX=0}^{\uX=\infty} ( \uX + \zX+\wX ) e^{-\frac{1}{2} ( \uX + \zX+\wX )^{2}} d\uX}
  {\int_{\uX=0}^{\uX=\infty} e^{-\frac{1}{2} ( \uX + \zX+\wX )^{2}} d\uX}
  \right)
  \\
  &\dCmt{by distributivity and commutativity}
  \\
  &\AlignSp=
  \frac
  {
   \left( \int_{\uX=\zX}^{\uX=\infty} e^{-\frac{1}{2} \uX^{2}} d\uX \right)
   \left( \int_{\uX=\zX+\wX}^{\uX=\infty} e^{-\frac{1}{2} \uX^{2}} d\uX \right)}
  {\left( \int_{\uX=\zX}^{\uX=\infty} e^{-\frac{1}{2} \uX^{2}} d\uX \right)^{2}}
  \left(
  \frac
  {\int_{\uX=\zX}^{\uX=\infty} \uX e^{-\frac{1}{2} \uX^{2}} d\uX}
  {\int_{\uX=\zX}^{\uX=\infty} e^{-\frac{1}{2} \uX^{2}} d\uX}
  -
  \frac
  {\int_{\uX=\zX+\wX}^{\uX=\infty} \uX e^{-\frac{1}{2} \uX^{2}} d\uX}
  {\int_{\uX=\zX+\wX}^{\uX=\infty} e^{-\frac{1}{2} \uX^{2}} d\uX}
  \right)
  .\\
  &\dCmt{by a change of variables (applied to each of the four integrals)}
\end{align*}
Notice in the last line that the sign of
\(  \frac{\partial}{\partial \zX} \frac{\pExpr*{\zX+\wX}}{\pExpr{\zX}}  \)
is entirely determined by the sign of the rightmost multiplicand,
\begin{gather*}
  \frac
  {\int_{\uX=\zX}^{\uX=\infty} \uX e^{-\frac{1}{2} \uX^{2}} d\uX}
  {\int_{\uX=\zX}^{\uX=\infty} e^{-\frac{1}{2} \uX^{2}} d\uX}
  -
  \frac
  {\int_{\uX=\zX+\wX}^{\uX=\infty} \uX e^{-\frac{1}{2} \uX^{2}} d\uX}
  {\int_{\uX=\zX+\wX}^{\uX=\infty} e^{-\frac{1}{2} \uX^{2}} d\uX}
,\end{gather*}
and hence, it suffices to show nonpositivity of this term, \ie that
\begin{gather}
\label{eqn:pf:corollary:main-technical:probit:2}
  \frac
  {\int_{\uX=\zX}^{\uX=\infty} \uX e^{-\frac{1}{2} \uX^{2}} d\uX}
  {\int_{\uX=\zX}^{\uX=\infty} e^{-\frac{1}{2} \uX^{2}} d\uX}
  -
  \frac
  {\int_{\uX=\zX+\wX}^{\uX=\infty} \uX e^{-\frac{1}{2} \uX^{2}} d\uX}
  {\int_{\uX=\zX+\wX}^{\uX=\infty} e^{-\frac{1}{2} \uX^{2}} d\uX}
  \leq 0
.\end{gather}
Towards verifying this last inequality, \eqref{eqn:pf:corollary:main-technical:probit:2}, let
\(  \URV \sim \N(0,1)  \)
be a standard univariate Gaussian random variable, and let \(  \VRV  \) and \(  \WRV  \) be indicator random variables given by
\(  \VRV \defeq \I( | \URV | \geq \zX )  \) and
\(  \WRV \defeq \I( | \URV | \geq \zX+\wX )  \),
respectively.
The random variable \(  | \URV |  \) is standard half-normal with density
\begin{gather}
\label{eqn:pf:corollary:main-technical:probit:5}
  \pdf{| \URV |}( \uX )
  =
  \sqrt{\frac{2}{\pi}} e^{-\frac{1}{2} \uX^{2}}
.\end{gather}
The masses of the conditioned random variables \(  \VRV=1 \Mid| | \URV |  \) and \(  \WRV=1 \Mid| | \URV |  \) are given by
\begin{gather}
\label{eqn:pf:corollary:main-technical:probit:3}
  \pdf{\VRV \Mid| | \URV |}( 1 \Mid| \uX )
  =
  \begin{cases}
  0 ,& \cIf \uX    < \zX,\\
  1 ,& \cIf \uX \geq \zX,
  \end{cases}
  \\
\label{eqn:pf:corollary:main-technical:probit:4}
  \pdf{\WRV \Mid| | \URV |}( 1 \Mid| \uX )
  =
  \begin{cases}
  0 ,& \cIf \uX    < \zX+\wX,\\
  1 ,& \cIf \uX \geq \zX+\wX.
  \end{cases}
\end{gather}
Observe:
\begin{align*}
  \pdf{\VRV}( 1 )
  &=
  \int_{\uX=0}^{\uX=\infty}
  \pdf{| \URV |}( \uX )
  \pdf{\VRV \Mid| | \URV |}( 1 \Mid| \uX )
  d\uX
  \\
  &\dCmt{by the law of total probability, definition of}
  \\
  &\dCmtx{conditional probabilities, and support of \(  \pdf{| \URV |}  \)}
  \\
  &=
  \int_{\uX=0}^{\uX=\zX}
  \pdf{| \URV |}( \uX )
  \cdot 0
  d\uX
  +
  \int_{\uX=\zX}^{\uX=\infty}
  \pdf{| \URV |}( \uX )
  \cdot 1
  d\uX
  \\
  &\dCmt{by \EQUATION \eqref{eqn:pf:corollary:main-technical:probit:3}}
  \\
  &=
  \sqrt{\frac{2}{\pi}}
  \int_{\uX=\zX}^{\uX=\infty}
  e^{-\frac{1}{2} \uX^{2}}
  d\uX
  ,\\
  &\dCmt{by \EQUATION \eqref{eqn:pf:corollary:main-technical:probit:5}}
\end{align*}
and likewise,
\begin{align*}
  \pdf{\WRV}( 1 )
  &=
  \int_{\uX=0}^{\uX=\infty}
  \pdf{| \URV |}( \uX )
  \pdf{\WRV \Mid| | \URV |}( 1 \Mid| \uX )
  d\uX
  \\
  &\dCmt{by the law of total probability, definition of}
  \\
  &\dCmtx{conditional probabilities, and support of \(  \pdf{| \URV |}  \)}
  \\
  &=
  \int_{\uX=0}^{\uX=\zX+\wX}
  \pdf{| \URV |}( \uX )
  \cdot 0
  d\uX
  +
  \int_{\uX=\zX+\wX}^{\uX=\infty}
  \pdf{| \URV |}( \uX )
  \cdot 1
  d\uX
  \\
  &\dCmt{by \EQUATION \eqref{eqn:pf:corollary:main-technical:probit:4}}
  \\
  &=
  \sqrt{\frac{2}{\pi}}
  \int_{\uX=\zX+\wX}^{\uX=\infty}
  e^{-\frac{1}{2} \uX^{2}}
  d\uX
  ,\\
  &\dCmt{by \EQUATION \eqref{eqn:pf:corollary:main-technical:probit:5}}
\end{align*}
By Bayes' theorem,
\begin{align*}
  \pdf{| \URV | \Mid| \VRV}( \uX \Mid| 1 )
  &=
  \frac
  {\pdf{| \URV |}( \uX ) \pdf{\VRV \Mid| | \URV |}( 1 \Mid| \uX )}
  {\pdf{\VRV}( 1 )}
  \\
  &=
  \begin{cases}
  0 ,& \cIf \uX < \zX ,\\
  \frac
  {\sqrt{\frac{2}{\pi}} e^{-\frac{1}{2} \uX^{2}} d\uX}
  {\sqrt{\frac{2}{\pi}} \int_{\yX=\zX}^{\yX=\infty} e^{-\frac{1}{2} \yX^{2}} d\yX}
  ,& \cIf \uX \geq \zX,
  \end{cases}
  \\
  &=
  \begin{cases}
  0 ,& \cIf \uX < \zX ,\\
  \frac
  {e^{-\frac{1}{2} \uX^{2}} d\uX}
  {\int_{\yX=\zX}^{\yX=\infty} e^{-\frac{1}{2} \yX^{2}} d\yX}
  ,& \cIf \uX \geq \zX,
  \end{cases}
\end{align*}
and
\begin{align*}
  \pdf{| \URV | \Mid| \WRV}( \uX \Mid| 1 )
  &=
  \frac
  {\pdf{| \URV |}( \uX ) \pdf{\WRV \Mid| | \URV |}( 1 \Mid| \uX )}
  {\pdf{\WRV}( 1 )}
  \\
  &=
  \begin{cases}
  0 ,& \cIf \uX < \zX+\wX ,\\
  \frac
  {\sqrt{\frac{2}{\pi}} e^{-\frac{1}{2} \uX^{2}} d\uX}
  {\sqrt{\frac{2}{\pi}} \int_{\yX=\zX+\wX}^{\yX=\infty} e^{-\frac{1}{2} \yX^{2}} d\yX}
  ,& \cIf \uX \geq \zX+\wX,
  \end{cases}
  \\
  &=
  \begin{cases}
  0 ,& \cIf \uX < \zX+\wX ,\\
  \frac
  {e^{-\frac{1}{2} \uX^{2}} d\uX}
  {\int_{\yX=\zX+\wX}^{\yX=\infty} e^{-\frac{1}{2} \yX^{2}} d\yX}
  ,& \cIf \uX \geq \zX+\wX.
  \end{cases}
\end{align*}
Therefore, in expectation,
\begin{gather*}
  \E[ | \URV | \Mid| | \URV | \geq \zX ]
  =
  \E[ | \URV | \Mid| \VRV=1 ]
  =
  \frac
  {\int_{\uX=\zX}^{\uX=\infty} \uX e^{-\frac{1}{2} \uX^{2}} d\uX}
  {\int_{\uX=\zX}^{\uX=\infty} e^{-\frac{1}{2} \uX^{2}} d\uX}
  ,\\
  \E[ | \URV | \Mid| | \URV | \geq \zX+\wX ]
  =
  \E[ | \URV | \Mid| \WRV=1 ]
  =
  \frac
  {\int_{\uX=\zX+\wX}^{\uX=\infty} \uX e^{-\frac{1}{2} \uX^{2}} d\uX}
  {\int_{\uX=\zX+\wX}^{\uX=\infty} e^{-\frac{1}{2} \uX^{2}} d\uX}
.\end{gather*}
Note that
\(  \E[ | \URV | \Mid| | \URV | \geq \zX ] \leq \E[ | \URV | \Mid| | \URV | \geq \zX+\wX ]  \)
when \(  \wX > 0  \), implying that
\begin{gather*}
  \frac
  {\int_{\uX=\zX}^{\uX=\infty} \uX e^{-\frac{1}{2} \uX^{2}} d\uX}
  {\int_{\uX=\zX}^{\uX=\infty} e^{-\frac{1}{2} \uX^{2}} d\uX}
  =
  \E[ | \URV | \Mid| | \URV | \geq \zX ]
  \leq
  \E[ | \URV | \Mid| | \URV | \geq \zX+\wX ]
  =
  \frac
  {\int_{\uX=\zX+\wX}^{\uX=\infty} \uX e^{-\frac{1}{2} \uX^{2}} d\uX}
  {\int_{\uX=\zX+\wX}^{\uX=\infty} e^{-\frac{1}{2} \uX^{2}} d\uX}
,\end{gather*}
and hence also that
\begin{gather*}
  \frac
  {\int_{\uX=\zX}^{\uX=\infty} \uX e^{-\frac{1}{2} \uX^{2}} d\uX}
  {\int_{\uX=\zX}^{\uX=\infty} e^{-\frac{1}{2} \uX^{2}} d\uX}
  -
  \frac
  {\int_{\uX=\zX+\wX}^{\uX=\infty} \uX e^{-\frac{1}{2} \uX^{2}} d\uX}
  {\int_{\uX=\zX+\wX}^{\uX=\infty} e^{-\frac{1}{2} \uX^{2}} d\uX}
  \leq 0
,\end{gather*}
as desired.
It follows from this and the earlier discussion that
\begin{align*}
  \frac{\partial}{\partial \zX}
  \frac{\pExpr*{\zX+\wX}}{\pExpr{\zX}}
  &=
  \frac{\partial}{\partial \zX}
  \frac
  {\int_{\uX=0}^{\uX=\infty} e^{-\frac{1}{2} ( \uX + \zX+\wX )^{2}} d\uX}
  {\int_{\uX=0}^{\uX=\infty} e^{-\frac{1}{2} ( \uX + \zX )^{2}} d\uX}
  \leq 0
,\end{align*}
which verifies that \CONDITION \ref{condition:assumption:p:ii} of \ASSUMPTION \ref{assumption:p} is upheld when \(  \pFn  \) is as defined for probit regression.
Having shown that both conditions of \ASSUMPTION \ref{assumption:p} are satisfied for probit regression, \STEP \ref{enum:pf:corollary:main-technical:probit:a} is completed.
%
\par 
%
%
Moving ahead with \STEP \ref{enum:pf:corollary:main-technical:probit:b}, recall that the aim here is to derive \(  \alphaX  \) and \(  \gammaX  \).
For \(  \alphaX  \), observe:
\begin{align*}
  \alphaX
  &=
  \frac{1}{\sqrt{2\pi}}
  \int_{\zX=0}^{\zX=\infty}
  e^{-\frac{1}{2} \zX^{2}}
  ( 1 - \pFn( \zX ) + \pFn( -\zX ) )
  d\zX
  \\
  &=
  \frac{2}{\sqrt{2\pi}}
  \int_{\zX=0}^{\zX=\infty}
  e^{-\frac{1}{2} \zX^{2}}
  \pFn( -\zX )
  d\zX
  \\
  &=
  \frac{2}{\sqrt{2\pi}}
  \int_{\zX=0}^{\zX=\infty}
  e^{-\frac{1}{2} \zX^{2}}
  \frac{1}{\sqrt{2\pi}}
  \int_{\uX=-\infty}^{\uX=-\betaX \zX}
  e^{-\frac{1}{2} \uX^{2}}
  d\uX
  d\zX
  \\
  &=
  \frac{2}{\sqrt{2\pi}}
  \int_{\zX=0}^{\zX=\infty}
  e^{-\frac{1}{2} \zX^{2}}
  \left(
    \frac{1}{\sqrt{2\pi}}
    \int_{\uX=-\infty}^{\uX=0}
    e^{-\frac{1}{2} \uX^{2}}
    d\uX
    -
    \frac{1}{\sqrt{2\pi}}
    \int_{\uX=-\betaX \zX}^{\uX=0}
    e^{-\frac{1}{2} \uX^{2}}
    d\uX
  \right)
  d\zX
  \\
  &=
  \frac{2}{\sqrt{2\pi}}
  \int_{\zX=0}^{\zX=\infty}
  e^{-\frac{1}{2} \zX^{2}}
  \left(
    \frac{1}{2}
    -
    \frac{1}{\sqrt{2\pi}}
    \int_{\uX=-\betaX \zX}^{\uX=0}
    e^{-\frac{1}{2} \uX^{2}}
    d\uX
  \right)
  d\zX
  \\
  &=
  \frac{1}{\sqrt{2\pi}}
  \int_{\zX=0}^{\zX=\infty}
  e^{-\frac{1}{2} \zX^{2}}
  \left(
    1
    -
    \frac{1}{\sqrt{2\pi}}
    \int_{\uX=-\betaX \zX}^{\uX=\betaX \zX}
    e^{-\frac{1}{2} \uX^{2}}
    d\uX
  \right)
  d\zX
  \\
  &=
  \frac{1}{\sqrt{2\pi}}
  \int_{\zX=0}^{\zX=\infty}
  e^{-\frac{1}{2} \zX^{2}}
  d\zX
  -
  \frac{1}{\sqrt{2\pi}}
  \int_{\zX=0}^{\zX=\infty}
  e^{-\frac{1}{2} \zX^{2}}
  \frac{1}{\sqrt{2\pi}}
  \int_{\uX=-\betaX \zX}^{\uX=\betaX \zX}
  e^{-\frac{1}{2} \uX^{2}}
  d\uX
  d\zX
  \\
  &=
  \frac{1}{2}
  -
  \frac{1}{\pi}
  \arctan( \betaX )
  \\
  &=
  \frac{1}{\pi}
  \arctan \left( \frac{1}{\betaX} \right)
,\end{align*}
where the last equality holds for \(  \betaX > 0  \).
Towards deriving a closed form expression for \(  \gammaX  \), define \(  \zetaX = 1 - \sqrt{\frac{\pi}{2}} \gammaX  \).
Then, \(  \zetaX  \) is similarly obtained as follows:
\begin{align*}
  \zetaX
  &=
  \int_{\zX=0}^{\zX=\infty}
  \zX
  e^{-\frac{1}{2} \zX^{2}}
  ( 1 - \pFn( \zX ) + \pFn( -\zX ) )
  d\zX
  \\
  &=
  2
  \int_{\zX=0}^{\zX=\infty}
  \zX
  e^{-\frac{1}{2} \zX^{2}}
  \pFn( -\zX )
  d\zX
  \\
  &=
  2
  \int_{\zX=0}^{\zX=\infty}
  \zX
  e^{-\frac{1}{2} \zX^{2}}
  \frac{1}{\sqrt{2\pi}}
  \int_{\uX=-\infty}^{\uX=-\betaX \zX}
  e^{-\frac{1}{2} \uX^{2}}
  d\uX
  d\zX
  \\
  &=
  2
  \int_{\zX=0}^{\zX=\infty}
  \zX
  e^{-\frac{1}{2} \zX^{2}}
  \left(
    \frac{1}{\sqrt{2\pi}}
    \int_{\uX=-\infty}^{\uX=0}
    e^{-\frac{1}{2} \uX^{2}}
    d\uX
    -
    \frac{1}{\sqrt{2\pi}}
    \int_{\uX=-\betaX \zX}^{\uX=0}
    e^{-\frac{1}{2} \uX^{2}}
    d\uX
  \right)
  d\zX
  \\
  &=
  2
  \int_{\zX=0}^{\zX=\infty}
  \zX
  e^{-\frac{1}{2} \zX^{2}}
  \left(
    \frac{1}{2}
    -
    \frac{1}{\sqrt{2\pi}}
    \int_{\uX=-\betaX \zX}^{\uX=0}
    e^{-\frac{1}{2} \uX^{2}}
    d\uX
  \right)
  d\zX
  \\
  &=
  \int_{\zX=0}^{\zX=\infty}
  \zX
  e^{-\frac{1}{2} \zX^{2}}
  \left(
    1
    -
    \frac{1}{\sqrt{2\pi}}
    \int_{\uX=-\betaX \zX}^{\uX=\betaX \zX}
    e^{-\frac{1}{2} \uX^{2}}
    d\uX
  \right)
  d\zX
  \\
  &=
  \int_{\zX=0}^{\zX=\infty}
  \zX
  e^{-\frac{1}{2} \zX^{2}}
  d\zX
  -
  \int_{\zX=0}^{\zX=\infty}
  \zX
  e^{-\frac{1}{2} \zX^{2}}
  \frac{1}{\sqrt{2\pi}}
  \int_{\uX=-\betaX \zX}^{\uX=\betaX \zX}
  e^{-\frac{1}{2} \uX^{2}}
  d\uX
  d\zX
  \\
  &=
  1
  -
  \sin( \arctan( \betaX ) )
  \\
  &=
  1
  -
  \frac{\betaX}{\sqrt{\betaX^{2} + 1}}
.\end{align*}
It follows that
\(
  1-\zetaX = \frac{\betaX}{\sqrt{\betaX^{2} + 1}}
,\)
and hence,
\(
  \gammaX = \frac{\sqrt{2/\pi}\betaX}{\sqrt{\betaX^{2} + 1}}
.\)
Thus,
\begin{gather*}
  \frac{1}{\GAMMAX}
  =
  \sqrt{\frac{\pi}{2} \left( 1 + \frac{1}{\betaX^{2}} \right)}
  =
  \begin{cases}
  \displaystyle
  \BigO( \frac{1}{\betaX} ) ,&\cIf \betaX \in (0,\betaXThresholdPROne),\\
  \displaystyle
  \BigO( 1 )                ,&\cIf \betaX \in [\betaXThresholdPROne, \infty) ,\\
  \end{cases}
  ,\\
  \frac{1}{\GAMMAX^{2}}
  =
  \frac{\pi}{2}
  \left( 1 + \frac{1}{\betaX^{2}} \right)
  =
  \begin{cases}
  \displaystyle
  \BigO( \frac{1}{\betaX^{2}} ) ,&\cIf \betaX \in (0,\betaXThresholdPROne) ,\\
  \displaystyle
  \BigO( 1 )                    ,&\cIf \betaX \in [\betaXThresholdPROne, \infty) ,
  \end{cases}
\end{gather*}
where
\(  \ConstbetaXThresholdPROne \defeq \ConstbetaXThresholdPROneValue  \).
Additionally, notice that
\(
  \alphaX
  =
  \frac{1}{\pi}
  \arctan \left( \frac{1}{\betaX} \right)
  \leq
  \min \left\{ \frac{1}{2}, \frac{1}{\pi \betaX} \right\}
,\)
and therefore,
\begin{align*}
  \alphaO
  \leq
  \max \left\{ \min \left\{ \frac{1}{2}, \frac{1}{\pi \betaX} \right\}, \deltaX \right\}
  =
  \begin{cases}
  \displaystyle
  \BigO( 1 )                ,&\cIf \betaX \in (0,\betaXThresholdPROne) ,\\
  \BigO( \frac{1}{\betaX} ) ,&\cIf \betaX \in [\betaXThresholdPROne, \betaXThresholdPRTwo) ,\\
  \BigO( \epsilonX )          ,&\cIf \betaX \in [\betaXThresholdPRTwo, \infty) ,
  \end{cases}
\end{align*}
where
\(  \ConstbetaXThresholdPRTwo \defeq \ConstbetaXThresholdPRTwoValue  \).
Incorporating these expressions for \(  \alphaO  \) and \(  \gammaX  \) into the sample complexity in \EQUATION \eqref{eqn:main-technical:sparse:m} of \THEOREM \ref{thm:main-technical:sparse}, and into the definitions of \(  \nuX=\nuX( \deltaX )  \) and \(  \tauX=\tauX( \deltaX )  \),
the corollary's bound holds due to \THEOREM \ref{thm:main-technical:sparse}:
\begin{gather*}
  \left\|
    \thetaStar
    -
    \frac
    {\thetaXX + \hfFn[\JCoords]( \thetaStar, \thetaXX )}
    {\| \thetaXX + \hfFn[\JCoords]( \thetaStar, \thetaXX ) \|_{2}}
  \right\|_{2}
  \leq
  \sqrt{\deltaX \| \thetaStar-\thetaXX \|_{2}}
  +
  \deltaX
,\end{gather*}
for every \(  \thetaXX \in \ParamSpace  \) and every \(  \JCoords \subseteq [\n]  \), \(  | \JCoords | \leq \k  \), uniformly with probability at least \(  1-\rho  \) as long as
\checkthis%
\begin{align*}
  \m
  &\geq
  \max \!\!\begin{array}[t]{l} \displaystyle \Biggl\{
    \frac{108\pi^{3} ( 1+\frac{1}{\betaX^{2}} )}{\deltaX}
    \log \left( \frac{24}{\rhoX} \JSPCSIZES \right)
    ,\\ \displaystyle \phantom{\Bigg\{}
    \frac{48\pi^{2} \alphaO ( 1+\frac{1}{\betaX^{2}} )}{\ConstbLD^{2} \deltaX^{2}}
    \log \left( \frac{24}{\rhoX} \JSPCSIZES \right)
    ,\\ \displaystyle \phantom{\Bigg\{}
    \frac{200 \Constb \sqrt{\frac{\pi}{2} \left( 1+\frac{1}{\betaX^{2}} \right)}}{\Constc^{2} \deltaX}
    \frac{\log \left( \frac{24}{\rhoX} \JSPCSIZES \right)}{\sqrt{\log \left( \frac{4e}{\nuX} \right)}}
    ,\\ \displaystyle \phantom{\Bigg\{}
    \frac{64}{\Constb \deltaX} \sqrt{\frac{\pi}{2}  \left(1 + \frac{1}{\betaX^{2}} \right) \log \left( \frac{4e}{\nuX} \right)} \log \left( \frac{24}{\rhoX} \right)
    ,\\ \displaystyle \phantom{\Bigg\{}
    \frac{\ConstdSD \k}{\Constb \deltaX} \sqrt{\frac{\pi}{2} \left(1 + \frac{1}{\betaX^{2}} \right) \log \left( \frac{4e}{\nuX} \right)} \log \left( \frac{1}{\nuX} \right)
  \Biggr\}  \end{array}
  \\
  &=
  \begin{cases}
  \displaystyle
  \BigO'( \frac{\k}{\betaX^{2} \epsilonX^{2}} ) ,&\cIf \betaX \in (0,\betaXThresholdPROne) ,\\
  \displaystyle
  \BigO'( \frac{\k}{\betaX \epsilonX^{2}} )     ,&\cIf \betaX \in [\betaXThresholdPROne,\betaXThresholdPRTwo) ,\\
  \displaystyle
  \BigO'( \frac{\k}{\epsilonX} )     ,&\cIf \betaX \in [\betaXThresholdPRTwo, \infty) ,
  \end{cases}
\end{align*}
where
\(
  \nuX = \nuX( \deltaX ) = \nuXEXPRPR
.\)
\end{proof}


\subsection{Proof of the Intermediate Results}
\label{outline:pf-main-technical-result|pf-intermediate}

Having completed the proofs of the main technical results, \THEOREM \ref{thm:main-technical:sparse} and \COROLLARY \ref{corollary:main-technical:logistic-regression},
the auxiliary results used therein, \LEMMAS \ref{lemma:combine}--\ref{lemma:large-dist:2}---which were introduced in \SECTION \ref{outline:pf-main-technical-result|intermediate}---are proved next.


\subsubsection{Concentration Inequalities, Expectations, and a Deterministic Bound}
\label{outline:pf-main-technical-result|pf-intermediate|concentration-ineq}

The concentration inequalities and expectations in \LEMMA \ref{lemma:concentration-ineq}, below, will be crucial to the proofs of the intermediate results, \LEMMAS \ref{lemma:combine}--\ref{lemma:large-dist:2}.
The proof this lemma is deferred to \SECTION \ref{outline:concentration-ineq}.

\begin{lemma}
\label{lemma:concentration-ineq}
Fix
\(  \sXX, \sXXX, \tX, \tXX, \deltaX \in (0,1)  \).
Fix
\(  \thetaStar \in \ParamSpace  \),
and let
\(  \JS, \JSX \subseteq 2^{[\n]}  \) and
\(  \ParamCover \subset \ParamSpace  \)
be finite sets, and define
\(  \ParamCoverX \defeq \ParamCover \setminus \Ball{\tauX}( \thetaStar )  \).
Let
\(  \kO \defeq \kOExpr  \) and
\(  \kOX \defeq \kOXExpr  \),
and let
\(  \alphaO = \alphaO( \deltaX ) = \alphaOExpr  \).
Then,
\begin{gather}
  \Pr \left( {
    \Forall{\JCoords \in \JS, \thetaX \in \ParamCoverX}{
    \left\| \frac{\hFn[\JCoords]( \thetaStar, \thetaX )}{\sqrt{2\pi}} - \E \left[ \frac{\hFn[\JCoords]( \thetaStar, \thetaX )}{\sqrt{2\pi}} \right] \right\|_{2}
    \leq
    \sqrt{\frac{( 1+\sXX )( \kO-2 ) \ADIST}{\pi \m}}
    +
    \frac{\tX \ADIST}{\pi}
    }
  } \right)
  \nonumber \\
  \geq
  1
  -
  4 | \JS | | \ParamCoverX | e^{-\frac{1}{27\pi} \m \tX^{2} \ADIST}
  -
  | \JS | | \ParamCoverX | e^{-\frac{1}{18\pi ( 1+\sXX )} \m \tX^{2} \ADIST}
  -
  | \ParamCoverX | e^{-\frac{1}{3\pi} \m \sXX^{2} \ADIST}
  \label{eqn:lemma:concentration-ineq:pr:1}
  ,\\
  \Pr \left( {
    \Forall{\JCoordsX \in \JSX}{
    \left\| \frac{\hfFn[\JCoordsX]( \thetaStar, \thetaStar )}{\sqrt{2\pi}} - \E \left[ \frac{\hfFn[\JCoordsX]( \thetaStar, \thetaStar )}{\sqrt{2\pi}} \right] \right\|_{2}
    \leq
    \sqrt{\frac{\alphaO ( 1+\sXXX )( \kOX-1 )}{\m}}
    +
    \alphaO \tXX
    }
  } \right)
  \nonumber \\
  \geq
  1
  -
  2 | \JSX | e^{-\frac{1}{12} \alphaO \m \tXX^{2}}
  -
  | \JSX | e^{-\frac{1}{8} \frac{\alphaO \m \tXX^{2}}{1+\sXXX}}
  -
  e^{-\frac{1}{3} \alphaO \m \sXXX^{2}}
  \label{eqn:lemma:concentration-ineq:pr:2}
,\end{gather}
where in expectation, for any \(  \JCoords \subseteq [\n]  \) and \(  \thetaX \in \ParamSpace  \),
\begin{gather}
  \label{eqn:lemma:concentration-ineq:ev:1}
  \E \left[ \hFn[\JCoords]( \thetaStar, \thetaX ) \right]
  =
  \thetaStar - \thetaX
  ,\\ \label{eqn:lemma:concentration-ineq:ev:2}
  \E \left[ \hfFn[\JCoords]( \thetaStar, \thetaStar ) \right]
  =
  \E \left[ \langle \hfFn[\JCoords]( \thetaStar, \thetaStar ), \thetaStar \rangle \right] \thetaStar
  =
  -\left( 1 - \sqrt{\frac{\pi}{2}} \gammaX \right) \thetaStar
  ,\\ \label{eqn:lemma:concentration-ineq:ev:4}
  \DENOM
  =
  \| \E \left[ \thetaX + \hfFn[\JCoords]( \thetaStar, \thetaX ) \right] \|_{2}
  =
  \sqrt{\frac{\pi}{2}}\gammaX
.\end{gather}
\end{lemma}

In addition to the above lemma, the following fact will facilitate the analysis in this section.

\begin{fact}
\label{fact:d_angular-d_S}
Let
\(  \Vec{\uV}, \Vec{\vV} \in \Sphere{\n}  \).
Then,
\begin{gather}
\label{eqn:fact:d_angular-d_S:1}
  \| \Vec{\uV} - \Vec{\vV} \|_{2}
  \leq
  \arccos( \langle \Vec{\uV}, \Vec{\vV} \rangle )
  \leq
  \frac{\pi}{2} \| \Vec{\uV} - \Vec{\vV} \|_{2}
.\end{gather}
\end{fact}

\begin{proof}
{\FACT \ref{fact:d_angular-d_S}}
\checkoff%
Fix \(  \Vec{\uV}, \Vec{\vV} \in \Sphere{\n}  \) arbitrarily.
To verify the first inequality in \EQUATION \eqref{eqn:fact:d_angular-d_S:1}---that
\({  \| \Vec{\uV} - \Vec{\vV} \|_{2} \leq \arccos( \langle \Vec{\uV}, \Vec{\vV} \rangle )  }\)%
---observe:
\begin{align*}
  \| \Vec{\uV} - \Vec{\vV} \|_{2}
  &=
  \sqrt{2 - 2 \langle \Vec{\uV}, \Vec{\vV} \rangle}
  \\
  &\leq
  \sqrt{2 \left( 1 - \left( 1 - \frac{\arccos^{2}( \langle \Vec{\uV}, \Vec{\vV} \rangle )}{2} \right) \right)}
  \\
  &\dCmt{by the Taylor series for the cosine function, \(  \cos(x) \geq 1 - \frac{x^{2}}{2}  \), \(  x \in \R  \)}
  \\
  &=
  \arccos( \langle \Vec{\uV}, \Vec{\vV} \rangle )
,\end{align*}
as desired.
For the second inequality in \EQUATION \eqref{eqn:fact:d_angular-d_S:1}---that
\(  \arccos( \langle \Vec{\uV}, \Vec{\vV} \rangle ) \leq \frac{\pi}{2} \| \Vec{\uV} - \Vec{\vV} \|_{2}  \)%
---note that by standard trigonometric properties,
\(  \| \Vec{\uV} - \Vec{\vV} \|_{2} = 2 \sin( \frac{\arccos( \langle \Vec{\uV}, \Vec{\vV} \rangle )}{2} )  \).
To proceed, some basic calculus is needed to examine the function \(  \frac{\sin(x)}{x}  \) on the interval \(  x \in (0,\frac{\pi}{2}]  \).
Using the quotient rule,
\begin{align*}
  \frac{d}{dx} \frac{\sin(x)}{x} = \frac{x \cos(x) - \sin(x)}{x^{2}}
,\end{align*}
where the numerator determines the sign of the above expression and has a
Taylor
series given by
\begin{align*}
  x \cos(x) - \sin(x)
  &=
  \sum_{z=0}^{\infty}
  \frac{(-1)^{z} x^{2z+1}}{(2z)!}
  -
  \sum_{z=0}^{\infty}
  \frac{(-1)^{z} x^{2z+1}}{(2z+1)!}
  =
  \sum_{z=1}^{\infty}
  \frac{(-1)^{z} x^{2z+1}}{(2z)!}
  \left( 1 - \frac{1}{2z+1} \right)
.\end{align*}
Now, it can be seen that for \(  x \in (0, \frac{\pi}{2}]  \),
\begin{align*}
  \frac{d}{dx} \frac{\sin(x)}{x}
  =
  \frac{x \cos(x) - \sin(x)}{x^{2}}
  =
  \frac{1}{x^{2}}
  \sum_{z=1}^{\infty}
  \frac{(-1)^{z} x^{2z+1}}{(2z)!}
  \left( 1 - \frac{1}{2z+1} \right)
  <
  0
,\end{align*}
which implies that \(  \frac{\sin(x)}{x}  \) decreases over the interval \(  x \in (0, \frac{\pi}{2}]  \).
Hence,
\begin{gather*}
  \inf_{x \in (0,\frac{\pi}{2}]} \frac{\sin(x)}{x}
  =
  \left. \frac{\sin(x)}{x} \right|_{x = \frac{\pi}{2}}
  =
  \frac{\sin \left( \frac{\pi}{2} \right)}{\frac{\pi}{2}}
  =
  \frac{2}{\pi}
.\end{gather*}
Then,
\begin{align*}
  \frac{\| \Vec{\uV} - \Vec{\vV} \|_{2}}{\arccos( \langle \Vec{\uV}, \Vec{\vV} \rangle )}
  =
  \frac{2 \sin \left( \frac{\arccos( \langle \Vec{\uV}, \Vec{\vV} \rangle )}{2} \right)}{\arccos( \langle \Vec{\uV}, \Vec{\vV} \rangle )}
  \geq
  \frac{2}{\pi}
,\end{align*}
implying that
\begin{gather*}
  \arccos( \langle \Vec{\uV}, \Vec{\vV} \rangle )
  \leq
  \frac{\pi}{2} \| \Vec{\uV} - \Vec{\vV} \|_{2}
,\end{gather*}
as claimed.
\end{proof}


\subsubsection{Proofs of \LEMMAS \ref{lemma:combine}--\ref{lemma:large-dist:2}}
\label{outline:pf-main-technical-result|pf-intermediate|pf}

With the above auxiliary results, \LEMMAS \ref{lemma:combine}--\ref{lemma:large-dist:2} can now be established.
We begin with the proof of \LEMMA \ref{lemma:combine}.

\begin{proof}
{\LEMMA \ref{lemma:combine}}
\mostlycheckoff%
The first step towards proving the lemma will be showing that
\begin{align*}
  \frac
  {\E [ \thetaXX+\hfFn[\JCoords]( \thetaStar, \thetaXX ) ]}
  {\| \E [ \thetaXX+\hfFn[\JCoords]( \thetaStar, \thetaXX ) ] \|_{2}}
  =
  \thetaStar
,\end{align*}
where
\(  \thetaStar, \thetaXX \in \ParamSpace  \) and \(  \JCoords \subseteq [\n]  \)
are arbitrary.
Notice that for any \(  \Vec{\uV}, \Vec{\vV}, \Vec{\wV} \in \R^{\n}  \) such that
\(  \Supp( \Vec{\wV}  ) \cup \JCoords = \Supp( \Vec{\vV} ) \cup \JCoords  \),
the following pair of equations holds:
\begin{gather}
  \label{eqn:pf:lemma:combine:1a}
  \hfFn[\JCoords]( \Vec{\uV}, \Vec{\vV} )
  =
  \hFn[\JCoords]( \Vec{\uV}, \Vec{\vV} )
  +
  \hfFn[\Supp( \Vec{\vV} ) \cup \JCoords]( \Vec{\uV}, \Vec{\uV} )
  =
  \hFn[\JCoords]( \Vec{\uV}, \Vec{\vV} )
  +
  \hfFn[\Supp( \Vec{\wV} ) \cup \JCoords]( \Vec{\uV}, \Vec{\uV} )
  ,\\ \label{eqn:pf:lemma:combine:1b}
  \hFn[\JCoords]( \Vec{\uV}, \Vec{\vV} )
  =
  \hFn[\JCoords]( \Vec{\uV}, \Vec{\wV} )
  +
  \hFn[\Supp( \Vec{\uV} ) \cup \JCoords]( \Vec{\wV}, \Vec{\vV} )
.\end{gather}
To justify these equations, observe:
\begin{align*}
  \hfFn[\JCoords]( \Vec{\uV}, \Vec{\vV} )
  &=
  \ThresholdSet{\Supp( \Vec{\uV} ) \cup \Supp( \Vec{\vV} ) \cup \JCoords} (
  \frac{\sqrt{2\pi}}{\m}
  \sep
  \CovM^{\T}
  \sep
  \frac{1}{2}
  \left( \fFn( \CovM \Vec{\uV} ) - \Sign( \CovM \Vec{\vV} ) \right)
  )
  \\
  &=
  \ThresholdSet{\Supp( \Vec{\uV} ) \cup \Supp( \Vec{\vV} ) \cup \JCoords} (
  \frac{\sqrt{2\pi}}{\m}
  \sep
  \CovM^{\T}
  \sep
  \frac{1}{2}
  \left( \Sign( \CovM \Vec{\uV} ) - \Sign( \CovM \Vec{\vV} ) \right)
  )
  \\ &\AlignSp
  +
  \ThresholdSet{\Supp( \Vec{\uV} ) \cup ( \Supp( \Vec{\vV} ) \cup \JCoords )} (
  \frac{\sqrt{2\pi}}{\m}
  \sep
  \CovM^{\T}
  \sep
  \frac{1}{2}
  \left( \fFn( \CovM \Vec{\uV} ) - \Sign( \CovM \Vec{\uV} ) \right)
  )
  \\
  &=
  \hFn[\JCoords]( \Vec{\uV}, \Vec{\vV} )
  +
  \hfFn[\Supp( \Vec{\vV} ) \cup \JCoords]( \Vec{\uV}, \Vec{\uV} )
  \\
  &=
  \hFn[\JCoords]( \Vec{\uV}, \Vec{\vV} )
  +
  \hfFn[\Supp( \Vec{\wV} ) \cup \JCoords]( \Vec{\uV}, \Vec{\uV} )
\end{align*}
and
\begin{align*}
  \hFn[\JCoords]( \Vec{\uV}, \Vec{\vV} )
  &=
  \ThresholdSet{\Supp( \Vec{\uV} ) \cup \Supp( \Vec{\vV} ) \cup \JCoords} (
  \frac{\sqrt{2\pi}}{\m}
  \sep
  \CovM^{\T}
  \sep
  \frac{1}{2}
  \left( \Sign( \CovM \Vec{\uV} ) - \Sign( \CovM \Vec{\vV} ) \right)
  )
  \\
  &=
  \ThresholdSet{\Supp( \Vec{\uV} ) \cup \Supp( \Vec{\vV} ) \cup \JCoords} (
  \frac{\sqrt{2\pi}}{\m}
  \sep
  \CovM^{\T}
  \sep
  \frac{1}{2}
  \left( \Sign( \CovM \Vec{\uV} ) - \Sign( \CovM \Vec{\wV} ) \right)
  )
  \\ &\AlignSp
  +
  \ThresholdSet{\Supp( \Vec{\uV} ) \cup \Supp( \Vec{\vV} ) \cup \JCoords} (
  \frac{\sqrt{2\pi}}{\m}
  \sep
  \CovM^{\T}
  \sep
  \frac{1}{2}
  \left( \Sign( \CovM \Vec{\wV} ) - \Sign( \CovM \Vec{\vV} ) \right)
  )
  \\
  &=
  \ThresholdSet{\Supp( \Vec{\uV} ) \cup \Supp( \Vec{\wV} ) \cup \JCoords} (
  \frac{\sqrt{2\pi}}{\m}
  \sep
  \CovM^{\T}
  \sep
  \frac{1}{2}
  \left( \Sign( \CovM \Vec{\uV} ) - \Sign( \CovM \Vec{\wV} ) \right)
  )
  \\ &\AlignSp
  +
  \ThresholdSet{\Supp( \Vec{\wV} ) \cup \Supp( \Vec{\vV} ) \cup ( \Supp( \Vec{\uV} ) \cup \JCoords )} (
  \frac{\sqrt{2\pi}}{\m}
  \sep
  \CovM^{\T}
  \sep
  \frac{1}{2}
  \left( \Sign( \CovM \Vec{\wV} ) - \Sign( \CovM \Vec{\vV} ) \right)
  )
  \\
  &=
  \hFn[\JCoords]( \Vec{\uV}, \Vec{\wV} )
  +
  \hFn[\Supp( \Vec{\uV} ) \cup \JCoords]( \Vec{\wV}, \Vec{\vV} )
.\end{align*}
Additionally, by \LEMMA \ref{lemma:concentration-ineq},
\begin{gather}
  \label{eqn:pf:lemma:combine:2}
  \E[ \hFn[\JCoords]( \Vec{\uV}, \Vec{\vV} ) ]
  =
  \Vec{\uV} - \Vec{\vV}
  ,\\ \label{eqn:pf:lemma:combine:3}
  \E[ \hfFn[\JCoords]( \Vec{\uV}, \Vec{\uV} ) ]
  =
  -\left( 1 - \sqrt{\frac{\pi}{2}} \gammaX \right) \Vec{\uV}
.\end{gather}
Thus,
\begin{align*}
  \E [ \thetaXX+\hfFn[\JCoords]( \thetaStar, \thetaXX ) ]
  &=
  \thetaXX + \E [ \hfFn[\JCoords]( \thetaStar, \thetaXX ) ]
  \\
  &=
  \thetaXX + \E[ \hFn[\JCoords]( \thetaStar, \thetaXX ) ] + \E[ \hfFn[\JCoords]( \thetaStar, \thetaStar ) ]
  \\
  &\dCmt{by \EQUATION \eqref{eqn:pf:lemma:combine:1a} and the linearity of expectation}
  \\
  &=
  \thetaXX + ( \thetaStar-\thetaXX ) - \left( 1 - \sqrt{\frac{\pi}{2}} \gammaX \right) \thetaStar
  \\
  &\dCmt{by \EQUATIONS \eqref{eqn:pf:lemma:combine:2} and \eqref{eqn:pf:lemma:combine:3}}
  \\
  &=
  \sqrt{\frac{\pi}{2}} \gammaX \thetaStar
\TagEqn{\label{eqn:pf:lemma:combine:5}}
.\end{align*}
It follows that
\begin{align}
\label{eqn:pf:lemma:combine:4}
  \frac
  {\E [ \thetaXX+\hfFn[\JCoords]( \thetaStar, \thetaXX ) ]}
  {\| \E [ \thetaXX+\hfFn[\JCoords]( \thetaStar, \thetaXX ) ] \|_{2}}
  =
  \frac{\sqrt{\hfrac{\pi}{2}} \gammaX \thetaStar}{\sqrt{\hfrac{\pi}{2}} \gammaX}
  =
  \thetaStar
,\end{align}
as claimed.
%
Having achieved the first task, the \LHS of the inequality in \EQUATION \eqref{eqn:pf:thm:main-technical:1} can now be bounded from above as follows:
\begin{align*}
  &\negphantom{\AlignSp}
  \left\| \thetaStar - \frac{\thetaXX+\hfFn[\JCoords]( \thetaStar, \thetaXX )}{\| \thetaXX+\hfFn[\JCoords]( \thetaStar, \thetaXX ) \|_{2}} \right\|_{2}
  \\
  &=
  \left\|
    \frac{\thetaXX+\hfFn[\JCoords]( \thetaStar, \thetaXX )}{\| \thetaXX+\hfFn[\JCoords]( \thetaStar, \thetaXX ) \|_{2}}
    -
    \frac{\E[ \thetaXX+\hfFn[\JCoords]( \thetaStar, \thetaXX ) ]}{\E[ \| \thetaXX+\hfFn[\JCoords]( \thetaStar, \thetaXX ) \|_{2} ]}
  \right\|_{2}
  \\
  &\dCmt{by \EQUATION \eqref{eqn:pf:lemma:combine:4}}
  \\
  &\leq
  \frac{2 \| \thetaXX+\hfFn[\JCoords]( \thetaStar, \thetaXX ) - \E[ \thetaXX+\hfFn[\JCoords]( \thetaStar, \thetaXX ) ] \|_{2}}{\| \E[ \thetaXX+\hfFn[\JCoords]( \thetaStar, \thetaXX ) ] \|_{2}}
  \\
  &\dCmt{by \FACT \ref{fact:dist-btw-normalized-vectors}}
  \\
  &=
  \frac{2 \| \hfFn[\JCoords]( \thetaStar, \thetaXX ) - \E[ \hfFn[\JCoords]( \thetaStar, \thetaXX ) ] \|_{2}}{\| \E[ \thetaXX+\hfFn[\JCoords]( \thetaStar, \thetaXX ) ] \|_{2}}
  \\
  &=
  \frac
  {2 \| \hFn[\JCoords]( \thetaStar, \thetaXX ) + \hfFn[\Supp( \thetaX ) \cup \JCoords]( \thetaStar, \thetaStar ) - \E[ \hFn[\JCoords]( \thetaStar, \thetaXX ) + \hfFn[\Supp( \thetaX ) \cup \JCoords]( \thetaStar, \thetaStar ) ] \|_{2}}
  {\| \E[ \thetaXX+\hfFn[\JCoords]( \thetaStar, \thetaXX ) ] \|_{2}}
  \\
  &\dCmt{by \EQUATION \eqref{eqn:pf:lemma:combine:1a} and the lemma's condition} 
  \\
  &=
  \frac
  {2 \| \hFn[\JCoords]( \thetaStar, \thetaX ) - \E[ \hFn[\JCoords]( \thetaStar, \thetaX ) ] \|_{2}}
  {\| \E[ \thetaXX+\hfFn[\JCoords]( \thetaStar, \thetaXX ) ] \|_{2}}
  +
  \frac
  {2\| \hFn[\Supp( \thetaStar ) \cup \JCoords]( \thetaX, \thetaXX ) - \E[ \hFn[\Supp( \thetaStar ) \cup \JCoords]( \thetaX, \thetaXX ) ] \|_{2}}
  {\| \E[ \thetaXX+\hfFn[\JCoords]( \thetaStar, \thetaXX ) ] \|_{2}}
  \\
  &\AlignSp+
  \frac
  {2 \| \hfFn[\Supp( \thetaX ) \cup \JCoords]( \thetaStar, \thetaStar ) - \E[ \hfFn[\Supp( \thetaX ) \cup \JCoords]( \thetaStar, \thetaStar ) ] \|_{2}}
  {\| \E[ \thetaXX+\hfFn[\JCoords]( \thetaStar, \thetaXX ) ] \|_{2}}
  \\
  &\dCmt{by the linearity of expectation and the triangle inequality}
  \\
  &=
  \frac
  {2 \| \hFn[\JCoords]( \thetaStar, \thetaX ) - \E[ \hFn[\JCoords]( \thetaStar, \thetaX ) ] \|_{2}}
  {\DENOM}
  +
  \frac
  {2 \| \hFn[\Supp( \thetaStar ) \cup \JCoords]( \thetaX, \thetaXX ) - \E[ \hFn[\Supp( \thetaStar ) \cup \JCoords]( \thetaX, \thetaXX ) ] \|_{2}}
  {\DENOM}
  \\
  &\AlignSp+
  \frac
  {2 \| \hfFn[\Supp( \thetaX ) \cup \JCoords]( \thetaStar, \thetaStar ) - \E[ \hfFn[\Supp( \thetaX ) \cup \JCoords]( \thetaStar, \thetaStar ) ] \|_{2}}
  {\DENOM}
  .\\
  &\dCmt{by the first equality in \EQUATION \eqref{eqn:lemma:concentration-ineq:ev:4}}
\end{align*}
This completes the proof of \LEMMA \ref{lemma:combine}.
\end{proof}

\begin{proof}
{\LEMMA \ref{lemma:large-dist:1}}
\checkoff%
Let
\(  \thetaStar \in \ParamSpace  \)
be arbitrary.
Write
\(  \ParamCoverX \defeq \ParamCover \setminus \Ball{\tauX}( \thetaStar )  \).
For the time being, fix
\(  \JCoords \in \JS  \) and \(  \thetaX \in \ParamCoverX  \)
arbitrarily---to be varied over all possible choices later---and let
\(  \JCoordsX \defeq \Supp( \thetaX ) \cup \JCoords  \).
By \EQUATION \eqref{eqn:lemma:concentration-ineq:ev:4} in \LEMMA \ref{lemma:concentration-ineq},
\begin{gather}
\label{eqn:pf:lemma:large-dist:1:2}
  \DENOM
  =
  \sqrt{\frac{\pi}{2}} \gammaX
,\end{gather}
and thus, substituting \EQUATION \eqref{eqn:pf:lemma:large-dist:1:2} into the \RHS of \EQUATION \eqref{eqn:lemma:large-dist:1:ub} in \LEMMA \ref{lemma:large-dist:1} yields
\begin{gather}
\label{eqn:pf:lemma:large-dist:1:3}
  \frac
  {2 \| \hFn[\JCoords]( \thetaStar, \thetaX ) - \E[ \hFn[\JCoords]( \thetaStar, \thetaX ) ] \|_{2}}
  {\DENOM}
  =
  \frac
  {2 \| \hFn[\JCoords]( \thetaStar, \thetaX ) - \E[ \hFn[\JCoords]( \thetaStar, \thetaX ) ] \|_{2}}
  {\sqrt{\hfrac{\pi}{2}} \gammaX}
.\end{gather}
Towards bounding the term
\({  \| \hFn[\JCoords]( \thetaStar, \thetaX ) - \E[ \hFn[\JCoords]( \thetaStar, \thetaX ) ] \|_{2}  }\)
in the numerator on the \RHS of \EQUATION \eqref{eqn:pf:lemma:large-dist:1:3},
consider \EQUATION \eqref{eqn:lemma:concentration-ineq:pr:1} in \LEMMA \ref{lemma:concentration-ineq}, where \(  \sXX, \tX \in (0,1)  \) are taken to be
\begin{gather}
\label{eqn:pf:lemma:large-dist:1:s}
  \sXX
  \defeq
  \sqrt{
    \frac{
      3\pi \log \left( \frac{3}{\rhoLDX} | \ParamCover | \right)
    }{
      \m \EDIST
    }
  }
,\end{gather}
and
\begin{align*}
  \tX
  &\defeq
  \sqrt{
    \frac{27\pi \log \left( \frac{12}{\rhoLDX} | \JS | | \ParamCover | \right)}
         {\m \EDIST}}
\TagEqn{\label{eqn:pf:lemma:large-dist:1:t}}
,\end{align*}
and where \(  \m  \) is at least
\begin{align}
\label{eqn:pf:lemma:large-dist:1:m}
  \m
  \geq
  \frac{16}{\GAMMAX^{2} \deltaX}
  \max \left\{
    27\pi \log \left( \frac{12}{\rhoLDX} | \JS | | \ParamCover | \right)
    ,
    4 ( \kO-2 )
  \right\}
.\end{align}
This bound on \(  \m  \) suffices to ensure that under the lemma's condition that
\(  \thetaX \in \ParamCoverX = \ParamCover \setminus \Ball{\tauX}( \thetaStar )  \),
the variable \(  \sXX  \) satisfies the requirement that
\(  \sXX < 1  \),
which further implies that
\(  1+\sXX < 2  \).
Hence also,
\begin{align*}
  \tX
  &=
  \sqrt{
    \frac{27\pi \log \left( \frac{12}{\rhoLDX} | \JS | | \ParamCover | \right)}
         {\m \EDIST}}
  \\
  &=
  \max \left\{
  \sqrt{
    \frac{27\pi \log \left( \frac{12}{\rhoLDX} | \JS | | \ParamCover | \right)}
         {\m \EDIST}}
  ,
  \sqrt{
    \frac{4\pi \log \left( \frac{3}{\rhoLDX} | \JS | | \ParamCover | \right)}
         {\m \EDIST}
  }
  \right\}
  \\
  &=
  \max \left\{
  \sqrt{
    \frac{27\pi \log \left( \frac{12}{\rhoLDX} | \JS | | \ParamCover | \right)}
         {\m \EDIST}}
  ,
  \sqrt{
    \frac{2\pi ( 1+\sXX ) \log \left( \frac{3}{\rhoLDX} | \JS | | \ParamCover | \right)}
         {\m \EDIST}
  }
  \right\}
\TagEqn{\label{eqn:pf:lemma:large-dist:1:t:b}}
.\end{align*}
Moreover, with these choices,
\begin{align*}
  & \negphantom{\AlignSp}
  \sqrt{\frac{\pi ( 1+\sXX )( \kO-2 ) \EDIST}{\m}}
  +
  \sqrt{\frac{\pi}{2}} \tX \EDIST
  \\
  &\leq
  \frac{1}{2} \cdot \sqrt{\frac{\pi}{8}} \gammaX \sqrt{\deltaX \EDIST}
  +
  \frac{1}{2} \cdot \sqrt{\frac{\pi}{8}} \gammaX \sqrt{\deltaX \EDIST}
  \\
  &=
  \sqrt{\frac{\pi}{8}} \gammaX \sqrt{\deltaX \EDIST}
\TagEqn{\label{eqn:pf:lemma:large-dist:1:4}}
.\end{align*}
%
\par 
%
For any random variable \(  U  \) taking values in \(  \Set{S} \subseteq \R  \), and for values \(  u \leq u' \in \Set{S}  \), the event that \(  U \leq u  \) implies \(  U \leq u'  \), and therefore,
\begin{gather}
\label{eqn:pf:lemma:large-dist:1:5}
  \Pr( U \leq u ) \leq \Pr( U \leq u' )
  ,\quad
  u \leq u'
.\end{gather}
Combining these observations with \EQUATION \eqref{eqn:lemma:concentration-ineq:pr:1} in \LEMMA \ref{lemma:concentration-ineq} yields:
\begin{align*}
  & \textstyle
  \Pr \left(
    \Forall{\JCoords \in \JS, \thetaX \in \ParamCoverX}{
    \left\| \hFn[\JCoords]( \thetaStar, \thetaX ) - \E \left[ \hFn[\JCoords]( \thetaStar, \thetaX ) \right] \right\|_{2}
    \leq
    \sqrt{\frac{\pi}{8}} \gammaX \sqrt{\deltaX \EDIST}
    }
  \right)
  \\
  & \textstyle \geq
  \Pr \left(
    \Forall{\JCoords \in \JS, \thetaX \in \ParamCoverX}{
    \left\| \hFn[\JCoords]( \thetaStar, \thetaX ) - \E \left[ \hFn[\JCoords]( \thetaStar, \thetaX ) \right] \right\|_{2}
    \leq
    \sqrt{\frac{\pi ( 1+\sXX )( \kO-2 ) \EDIST}{\m}}
    +
    \sqrt{\frac{\pi}{2}} \tX \EDIST
    }
  \right)
  \\
  &\dCmt{by \EQUATIONS \eqref{eqn:pf:lemma:large-dist:1:4} and \eqref{eqn:pf:lemma:large-dist:1:5}}
  \\
  & \textstyle \geq
  \Pr \left(
    \Forall{\JCoords \in \JS, \thetaX \in \ParamCoverX}{
    \left\| \hFn[\JCoords]( \thetaStar, \thetaX ) - \E \left[ \hFn[\JCoords]( \thetaStar, \thetaX ) \right] \right\|_{2}
    \leq
    \sqrt{\frac{2 ( 1+\sXX )( \kO-2 ) \ADIST}{\m}}
    +
    \sqrt{\frac{2}{\pi}} \tX \ADIST
    }
  \right)
  \\
  &\dCmt{\(  \ADIST \leq \frac{\pi}{2} \EDIST  \) by \FACT \ref{fact:d_angular-d_S}}
  \\
  &\geq
  1
  -
  4 | \JS | | \ParamCoverX | e^{-\frac{1}{27\pi} \m \tX^{2} \ADIST}
  -
  | \JS | | \ParamCoverX | e^{-\frac{1}{2\pi ( 1+\sXX )} \m \tX^{2} \ADIST}
  -
  | \ParamCoverX | e^{-\frac{1}{3\pi} \m \sXX^{2} \ADIST}
  \\
  &\dCmt{by \EQUATION \eqref{eqn:lemma:concentration-ineq:pr:1}}
  \\
  &\geq
  1
  -
  4 | \JS | | \ParamCoverX | e^{-\frac{1}{27\pi} \m \tX^{2} \EDIST}
  -
  | \JS | | \ParamCoverX | e^{-\frac{1}{2\pi ( 1+\sXX )} \m \tX^{2} \EDIST}
  -
  | \ParamCoverX | e^{-\frac{1}{3\pi} \m \sXX^{2} \EDIST}
  \\
  &\dCmt{\(  \ADIST \geq \EDIST  \) due to \FACT \ref{fact:d_angular-d_S}}
  \\
  &\geq
  1 - \frac{\rhoLDX}{3} - \frac{\rhoLDX}{3} - \frac{\rhoLDX}{3}
  \\
  &\dCmt{by the choice of \(  \sX, \tX  \) in \EQUATIONS \eqref{eqn:pf:lemma:large-dist:1:s} and \eqref{eqn:pf:lemma:large-dist:1:t}, respectively, and by \EQUATION \eqref{eqn:pf:lemma:large-dist:1:t:b}}
  \\
  &=
  1 - \rhoLDX
\TagEqn{\label{eqn:pf:lemma:large-dist:1:7}}
.\end{align*}
Returning to \EQUATION \eqref{eqn:pf:lemma:large-dist:1:3} and inserting \EQUATION \eqref{eqn:pf:lemma:large-dist:1:7}, it follows that if \(  \m  \) satisfies \EQUATION \eqref{eqn:pf:lemma:large-dist:1:m}, then with probability at least \(  1-\rhoLDX  \), for all \(  \JCoords \in \JS  \) and \(  \thetaX \in \ParamCoverX  \),
\begin{align*}
  \frac
  {2 \| \hFn[\JCoords]( \thetaStar, \thetaX ) - \E[ \hFn[\JCoords]( \thetaStar, \thetaX ) ] \|_{2}}
  {\DENOM}
  &=
  \frac
  {2 \| \hFn[\JCoords]( \thetaStar, \thetaX ) - \E[ \hFn[\JCoords]( \thetaStar, \thetaX ) ] \|_{2}}
  {\sqrt{\hfrac{\pi}{2}} \gammaX}
  \\
  &\leq
  \sqrt{\frac{8}{\pi}} \frac{1}{\gammaX}
  \sqrt{\frac{\pi}{8}} \gammaX
  \sqrt{\deltaX \EDIST}
  \\
  &=
  \sqrt{\deltaX \EDIST}
,\end{align*}
as desired.
\end{proof}
\begin{proof}
{\LEMMA \ref{lemma:small-dist}}
\mostlycheckoff%
Take any \(  \thetaStar \in \ParamSpace  \), and write
\(  \ParamCoverX \defeq \ParamCover \setminus \Ball{\tauX}( \thetaStar )  \) and
\(  \JSXX \defeq \{ \Supp( \thetaStar ) \cup \JCoords : \JCoords \in \JS \}  \),
where
\(  \JS \subseteq 2^{[\n]}  \)
is arbitrary.
Let
\(  \thetaX \in \ParamCoverX  \),
\(  \thetaXX \in \BallXX{2\tauX}( \thetaX )  \), and
\(  \JCoordsXX \in \JSXX  \)
be arbitrary.
Write
\(  \CovVX\VIx{\iIx} \defeq
    \ThresholdSet{\Supp( \thetaX ) \cup \JCoordsXX}( \CovV\VIx{\iIx} ) \in \R^{\n}  \),
\(  \iIx \in [\m]  \), and let
\(  \CovMX \defeq ( \CovVX\VIx{1} \,\cdots\, \CovVX\VIx{\m} )^{\T} \in \R^{\m \times \n}  \).
Using these notations, \(  \frac{1}{\sqrt{2\pi}} \hFn[\JCoordsXX]( \thetaX, \thetaXX )  \) can be expressed as follows:
\begin{align*}
  \frac{1}{\sqrt{2\pi}} \hFn[\JCoordsXX]( \thetaX, \thetaXX )
  &=
  \ThresholdSet{\Supp( \thetaX ) \cup \JCoordsXX} \left(
    \frac{1}{m}
    \sum_{\iIx=1}^{\m}
    \CovV\VIx{\iIx}
    \sep \frac{1}{2} \left( \Sign*( \langle \CovV\VIx{\iIx}, \thetaX \rangle ) - \Sign*( \langle \CovV\VIx{\iIx}, \thetaXX \rangle ) \right)
  \right)
  \\
  &=
  \frac{1}{m}
  \sum_{\iIx=1}^{\m}
  \CovVX\VIx{\iIx}
  \sep \frac{1}{2} \left( \Sign*( \langle \CovVX\VIx{\iIx}, \thetaX \rangle ) - \Sign*( \langle \CovVX\VIx{\iIx}, \thetaXX \rangle ) \right)
  \\
  &=
  \frac{1}{m}
  \sum_{\iIx=1}^{\m}
  \CovVX\VIx{\iIx}
  \sep \Sign*( \langle \CovVX\VIx{\iIx}, \thetaX \rangle )
  \sep \I( \Sign*( \langle \CovVX\VIx{\iIx}, \thetaX \rangle ) \neq \Sign*( \langle \CovVX\VIx{\iIx}, \thetaXX \rangle ) )
  \\
  &=
  \frac{1}{m}
  \CovMX^{\T}
  \Diag( \Sign*( \CovMX \thetaX ) )
  \sep \I( \Sign*( \CovMX \thetaX ) \neq \Sign*( \CovMX \thetaXX ) )
.\end{align*}
It is clear from the last line that after fixing the covariates, \(  \CovVX\VIx{\iIx}  \), \(  \iIx \in [\m]  \), the function \(  \hFn[\JCoordsXX]  \) can only take finitely many values.
Moreover, upon additionally fixing \(  \thetaX \in \ParamSpace  \), the finitely many values that can be taken by the function \(  \hFn[\JCoordsXX]( \thetaX, \cdot )  \) is determined by the number of values that can be taken by
\(  \I( \Sign*( \CovMX \thetaX ) \neq \Sign*( \CovMX \thetaXX ) ) \in \{ 0,1 \}^{\m}  \)
over all choices of \(  \thetaXX \in \BallXX{2\tauX}( \thetaX )  \).
As such, write
\(  \WS{\thetaX} \defeq \{ \I( \Sign*( \CovMX \thetaX ) \neq \Sign*( \CovMX \thetaXX ) ) : \thetaXX \in \BallXX{2\tauX}( \thetaX ) \}  \)
for \(  \thetaX \in \ParamSpace  \).
In addition, for \(  \thetaX, \thetaXX \in \ParamSpace  \), define
\(  \LRVXX{\thetaX}{\thetaXX} \defeq \| \I( \Sign( \CovM \thetaX ) \neq \Sign( \CovM \thetaXX ) ) \|_{0}  \),
and let
\(  \LRVX{\thetaX} \defeq \sup_{\thetaX' \in \BallXX{2\tauX}( \thetaX )} \LRVXX{\thetaX}{\thetaX'}  \).
While there is a na\"{i}ve upper bound of
\(  | \WS{\thetaX} | \leq 2^{\m}  \),
it turns out that a little more nuance admits a tighter bound on the cardinality of \(  \WS{\thetaX}  \) by means of the random variable \(  \LRVX{\thetaX}  \) and the following lemma.
\begin{lemma}[{\COROLLARY to \cite[{\COROLLARY 3.3}]{oymak2015near}}]
\label{lemma:pf:lemma:concentration-ineq:noiseless:small:mismatches}
Let
\(  \ConstdSD > 0  \)
be the constant defined in \DEFINITION \ref{def:univ-const}.
If
\begin{gather}
  \m
  \geq
  \frac{\ConstdSD \nO}{\nuX} \log \left( \frac{1}{\nuX} \right)
,\end{gather}
then with probability at least \(  1 - \binom{\n}{\nO} e^{-\frac{1}{64} \nuX \m}  \), the random variable \(  \LRVX{\thetaX}  \) is bounded from above by
\(  \LRVX{\thetaX} \leq \nuX \m  \)
uniformly for all
\(  \thetaX \in \SparseSphereSubspace{\nO}{\n}  \).
\end{lemma}
%
\begin{subproof}
{\LEMMA \ref{lemma:pf:lemma:concentration-ineq:noiseless:small:mismatches}}
\LEMMA \ref{lemma:pf:lemma:concentration-ineq:noiseless:small:mismatches} is a corollary to \cite[{\COROLLARY 3.3}]{oymak2015near}, which is presented below as \LEMMA \ref{lemma:oymak2015near:corollary3.3}.
%
\begin{lemma}[{(part of) \cite[{\COROLLARY 3.3}]{oymak2015near}}]
\label{lemma:oymak2015near:corollary3.3}
Let \(  \Set{U} \subseteq \R^{\n}  \).
If the set
\(  \Set{\hat{U}} \defeq \{ w \Vec{\uV} : \Vec{\uV} \in \Set{U}, w \in \R \}  \)
is a subspace with dimension
\(  \Dim \, \Set{\hat{U}} = t  \)
and
\begin{gather}
  \m
  \geq
  \frac{\ConstdSD t}{\nuX} \log \left( \frac{1}{\nuX} \right)
,\end{gather}
then
\(  \LRVXX{\Vec{\uV}}{\Vec{\vV}} \leq \nuX \m  \)
for each pair
\(  \Vec{\uV}, \Vec{\vV} \in \Set{U}  \)
such that
\(  \| \Vec{\uV} - \Vec{\vV} \|_{2} \leq \smash[b]{\frac{\nuX}{\ConstdSD \sqrt{\log \left( \frac{1}{\nuX} \right)}}}  \),
uniformly with probability at least
\(  1 - e^{-\frac{1}{64} \nuX \m}  \).
\end{lemma}
%
Resuming the verification of \LEMMA \ref{lemma:pf:lemma:concentration-ineq:noiseless:small:mismatches},
let
\(  \JSXXX \defeq \{ \JCoordsXXX \subseteq [\n] : | \JCoordsXXX | = \nO \}  \),
where
\(  | \JSXXX | = \binom{\n}{\nO}  \).
Note that
\(  \bigcup_{\JCoordsXXX \in \JSXXX} \{ \Vec{\uV} \in \Sphere{\n} : \Supp( \Vec{\uV} ) \subseteq \JCoordsXXX \} = \SparseSphereSubspace{\k}{\n}  \).
Fixing
\(  \JCoordsXXX \in \JSXXX  \)
arbitrarily, and writing
\(  \Set{U} \defeq \{ \Vec{\uV} \in \Sphere{\n} : \Supp( \Vec{\uV} ) \subseteq \JCoordsXXX \}  \) and
\(  \Set{\hat{U}} \defeq \{ w \Vec{\uV} : \Vec{\uV} \in \Set{U}, w \in \R \}  \)%
---where \(  \Set{\hat{U}}  \) has dimension \(  \Dim \, \Set{\hat{U}} = \nO  \)---consider any
\(  \thetaX \in \Set{U}  \).
Notice that because
\(  \BallXX{2\tauX}( \thetaX ) \subseteq \Set{U}  \),
it happens that
\(  \LRVX{\thetaX} = \sup_{\thetaXX \in \BallXX{2\tauX}( \thetaX )} \LRVXX{\thetaX}{\thetaXX} \leq \sup_{\thetaXX \in \Set{U}} \LRVXX{\thetaX}{\thetaXX} \leq \sup_{\thetaX', \thetaXX \in \Set{U}} \LRVXX{\thetaX'}{\thetaXX}  \).
Hence, it immediately follows from \LEMMA \ref{lemma:oymak2015near:corollary3.3} that with probability at least
\(  1 - e^{-\frac{1}{64} \nuX \m}  \),
the desired upper bound holds:
\(  \LRVX{\thetaX} \leq \sup_{\thetaX', \thetaXX \in \Set{U}} \LRVXX{\thetaX'}{\thetaXX} \leq \nuX \m  \).
By a union bound over \(  \JSXXX  \), this bound on \(  \LRVX{\thetaX}  \) holds uniformly over all
\(  \thetaX \in \SparseSphereSubspace{\nO}{\n}  \)
with probability at least
\(  1 - | \JSXXX | e^{-\frac{1}{64} \nuX \m} = 1 - \binom{\n}{\nO} e^{-\frac{1}{64} \nuX \m}  \),
as claimed.
\end{subproof}
%
%
Returning to the proof of \LEMMA \ref{lemma:small-dist}, recall that
\(  \ParamSpace = \SparseSphereSubspace{\nO}{\n}  \).
Thus, due to \LEMMA \ref{lemma:pf:lemma:concentration-ineq:noiseless:small:mismatches} and the sufficiently large choice of
\begin{gather}
\label{eqn:pf:lemma:concentration-ineq:noiseless:small:m}
  \m
  \geq
  \frac{\ConstdSD \nO}{\nuX} \log \left( \frac{1}{\nuX} \right)
\end{gather}
in \LEMMA \ref{lemma:small-dist},
with probability no less than
\(  1 - \binom{\n}{\nO} e^{-\frac{1}{64} \nuX \m}  \),
for every
\(  \thetaX \in \SparseSphereSubspace{\nO}{\n} = \ParamSpace  \) and every \(  \thetaXX \in \BallXX{2\tauX}( \thetaX )  \),
the indicator random vector
\(  \I( \Sign*( \CovMX \thetaX ) \neq \Sign*( \CovMX \thetaXX ) ) \in \{ 0,1 \}^{\m}  \)
contains at most \(  \nuX \m  \)-many nonzero entries.
Therefore, with probability at least
\(  1 - \binom{\n}{\nO} e^{-\frac{1}{64} \nuX \m}  \),
for every \(  \thetaX \in \ParamSpace  \),
\begin{gather}
\label{eqn:pf:lemma:concentration-ineq:noiseless:small:1}
  | \WS{\thetaX} |
  \leq
  \sum_{\lIx=0}^{\nuX \m}
  \binom{\m}{\lIx}
  \leq
  \left( \frac{e \m}{\nuX \m} \right)^{\nuX \m}
  =
  \left( \frac{e}{\nuX} \right)^{\nuX \m}
,\end{gather}
where the second inequality is due to a well-known bound for sums of binomial coefficients.
In light of this, for each \(  \thetaX \in \ParamSpace  \), construct the following cover, \(  \BallXCover{\thetaX} \subset \BallXX{2\tauX}( \thetaX )  \), over \(  \BallXX{2\tauX}( \thetaX )  \):
for each \(  \wS \in \WS{\thetaX}  \), insert into \(  \BallXCover{\thetaX}  \) exactly one \(  \thetaXX \in \BallXX{2\tauX}( \thetaX )  \) for which
\(  \I( \Sign*( \CovMX \thetaX ) \neq \Sign*( \CovMX \thetaXX ) ) = \wS  \).
Note that
\(  | \BallXCover{\thetaX} | = | \WS{\thetaX} |  \).
Define the random variable
\(  \MaxBallXCoverSize \defeq \max_{\thetaX' \in \ParamSpace} | \BallXCover{\thetaX'} | \equiv \max_{\thetaX' \in \ParamSpace} | \WS{\thetaX'} |  \),
and write
\(  \qX \defeq \sum_{\lIx=0}^{\nuX \m} \binom{\m}{\lIx} \leq ( \frac{e}{\nuX} )^{\nuX \m}  \).
Due to \LEMMA \ref{lemma:pf:lemma:concentration-ineq:noiseless:small:mismatches} and \EQUATION \eqref{eqn:pf:lemma:concentration-ineq:noiseless:small:1}, as well as the sufficient choice of \(  \m  \) in \EQUATION \eqref{eqn:pf:lemma:concentration-ineq:noiseless:small:m},
the random variable \(  \MaxBallXCoverSize  \) is bounded from above by
\(  \MaxBallXCoverSize \leq \qX  \)
with, once again, probability at least
\(  1 - \binom{\n}{\nO} e^{-\frac{1}{64} \nuX \m}  \).
%
\par 
%
An additional helpful technique is an orthogonal decomposition of \(  \hFn[\JCoordsXX]  \)---in this case:
\begin{align}
\label{eqn:pf:lemma:concentration-ineq:noiseless:small:6}
  \hFn[\JCoordsXX]( \thetaX, \thetaXX )
  &=
  \left\langle \hFn[\JCoordsXX]( \thetaX, \thetaXX ), \frac{\thetaX-\thetaXX}{\| \thetaX-\thetaXX \|_{2}} \right\rangle \frac{\thetaX-\thetaXX}{\| \thetaX-\thetaXX \|_{2}}
  +
  \left\langle \hFn[\JCoordsXX]( \thetaX, \thetaXX ), \frac{\thetaX+\thetaXX}{\| \thetaX+\thetaXX \|_{2}} \right\rangle \frac{\thetaX+\thetaXX}{\| \thetaX+\thetaXX \|_{2}}
  +
  \gFn[\JCoordsXX]( \thetaX, \thetaXX )
,\end{align}
where
\begin{align*}
  \gFn[\JCoordsXX]( \thetaX, \thetaXX )
  =
  \hFn[\JCoordsXX]( \thetaX, \thetaXX )
  -
  \left\langle \hFn[\JCoordsXX]( \thetaX, \thetaXX ), \frac{\thetaX-\thetaXX}{\| \thetaX-\thetaXX \|_{2}} \right\rangle \frac{\thetaX-\thetaXX}{\| \thetaX-\thetaXX \|_{2}}
  -
  \left\langle \hFn[\JCoordsXX]( \thetaX, \thetaXX ), \frac{\thetaX+\thetaXX}{\| \thetaX+\thetaXX \|_{2}} \right\rangle \frac{\thetaX+\thetaXX}{\| \thetaX+\thetaXX \|_{2}}
\end{align*}
per \EQUATION \eqref{eqn:notations:gJ:def}.
Note that similar orthogonal decompositions appear in, \eg \cite{plan2017high,friedlander2021nbiht,matsumoto2022binary,matsumoto2024robust}.
Due to \EQUATION \eqref{eqn:pf:lemma:concentration-ineq:noiseless:small:6} in combination with the linearity of expectation,
\begin{align*}
  & \negphantom{\AlignSp}
  \hFn[\JCoordsXX]( \thetaX, \thetaXX ) - \E[ \hFn[\JCoordsXX]( \thetaX, \thetaXX ) ]
  \\
  &=
  \left(
    \left\langle \hFn[\JCoordsXX]( \thetaX, \thetaXX ), \frac{\thetaX-\thetaXX}{\| \thetaX-\thetaXX \|_{2}} \right\rangle
    -
    \E \left[ \left\langle \hFn[\JCoordsXX]( \thetaX, \thetaXX ), \frac{\thetaX-\thetaXX}{\| \thetaX-\thetaXX \|_{2}} \right\rangle \right]
  \right)
  \frac{\thetaX-\thetaXX}{\| \thetaX-\thetaXX \|_{2}}
  \\
  &\AlignSp+
  \left(
    \left\langle \hFn[\JCoordsXX]( \thetaX, \thetaXX ), \frac{\thetaX+\thetaXX}{\| \thetaX+\thetaXX \|_{2}} \right\rangle
    -
    \E \left[ \left\langle \hFn[\JCoordsXX]( \thetaX, \thetaXX ), \frac{\thetaX+\thetaXX}{\| \thetaX+\thetaXX \|_{2}} \right\rangle \right]
  \right)
  \frac{\thetaX+\thetaXX}{\| \thetaX+\thetaXX \|_{2}}
  \\
  &\AlignSp+
  (
    \gFn[\JCoordsXX]( \thetaX, \thetaXX )
    -
    \E[ \gFn[\JCoordsXX]( \thetaX, \thetaXX ) ]
  )
.\end{align*}
Combining this orthogonal decomposition with the triangle inequality yields the following upper bound on the \(  \lnorm{2}  \)-distance of \(  \hFn[\JCoordsXX]( \thetaX, \thetaXX )  \) from its mean:
\begin{align*}
  \| \hFn[\JCoordsXX]( \thetaX, \thetaXX ) - \E[ \hFn[\JCoordsXX]( \thetaX, \thetaXX ) ] \|_{2}
  &\leq
  \left|
    \left\langle \hFn[\JCoordsXX]( \thetaX, \thetaXX ), \frac{\thetaX-\thetaXX}{\| \thetaX-\thetaXX \|_{2}} \right\rangle
    -
    \E \left[ \left\langle \hFn[\JCoordsXX]( \thetaX, \thetaXX ), \frac{\thetaX-\thetaXX}{\| \thetaX-\thetaXX \|_{2}} \right\rangle \right]
  \right|
  \\
  &\AlignSp+
  \left|
    \left\langle \hFn[\JCoordsXX]( \thetaX, \thetaXX ), \frac{\thetaX+\thetaXX}{\| \thetaX+\thetaXX \|_{2}} \right\rangle
    -
    \E \left[ \left\langle \hFn[\JCoordsXX]( \thetaX, \thetaXX ), \frac{\thetaX+\thetaXX}{\| \thetaX+\thetaXX \|_{2}} \right\rangle \right]
  \right|
  \\
  &\AlignSp+
  \|
    \gFn[\JCoordsXX]( \thetaX, \thetaXX )
    -
    \E[ \gFn[\JCoordsXX]( \thetaX, \thetaXX ) ]
  \|_{2}
\TagEqn{\label{eqn:pf:lemma:concentration-ineq:noiseless:small:2:1}}
.\end{align*}
%
\par 
%
Most of the remaining arguments in this proof are towards bounding the three terms on the \RHS of \EQUATION \eqref{eqn:pf:lemma:concentration-ineq:noiseless:small:2:1}.
The following lemma from \cite{matsumoto2022binary} will facilitate the bound on \EQUATION \eqref{eqn:pf:lemma:concentration-ineq:noiseless:small:2:1}.
Note that \LEMMA \ref{lemma:pf:lemma:concentration-ineq:noiseless:small:condition-L} is not an exact restatement of, but is implied by, \cite[{\LEMMA A.1}]{matsumoto2022binary} and its proof.
%
\begin{lemma}[{\dueto \cite[{\LEMMA A.1}]{matsumoto2022binary}}]
\label{lemma:pf:lemma:concentration-ineq:noiseless:small:condition-L}
Let
\(  \tX > 0  \) and \(  \uX \in (0,1)  \),
and let
\(  \thetaX, \thetaXX \in \SparseSphereSubspace{\k}{\n}  \) and \(  \JCoordsXX \in \JSXX  \).
Write
\(  \kOXX \defeq \kOXXExpr \geq \kOXXExprX  \).
Then,
\begin{gather}
  \label{eqn:lemma:pf:lemma:concentration-ineq:noiseless:small:condition-L:1}
  \Pr \left( {\textstyle
    \left| \left\langle \frac{\hFn[\JCoordsXX]( \thetaX, \thetaXX )}{\sqrt{2\pi}} , \frac{\thetaX-\thetaXX}{\| \thetaX-\thetaXX \|_{2}} \right\rangle - \E \left[ \left\langle \frac{\hFn[\JCoordsXX]( \thetaX, \thetaXX )}{\sqrt{2\pi}} , \frac{\thetaX-\thetaXX}{\| \thetaX-\thetaXX \|_{2}} \right\rangle \right] \right|
    >
    \ux \tX
  } \middle| {\textstyle
    \LRVXX{\thetaX}{\thetaXX} \leq \uX \m
  } \right)
  \leq
  2 e^{-\frac{1}{2} \uX \m \tX^{2}}
  ,\\
  \label{eqn:lemma:pf:lemma:concentration-ineq:noiseless:small:condition-L:2}
  \textstyle{ \Pr \left(
    \left| \left\langle \frac{\hFn[\JCoordsXX]( \thetaX, \thetaXX )}{\sqrt{2\pi}} , \frac{\thetaX+\thetaXX}{\| \thetaX+\thetaXX \|_{2}} \right\rangle - \E \left[ \left\langle \frac{\hFn[\JCoordsXX]( \thetaX, \thetaXX )}{\sqrt{2\pi}} , \frac{\thetaX+\thetaXX}{\| \thetaX+\thetaXX \|_{2}} \right\rangle \right] \right|
    >
    \ux \tX
  \middle|
    \LRVXX{\thetaX}{\thetaXX} \leq \uX \m
  \right) }
  \leq
  2 e^{-\frac{1}{2} \uX \m \tX^{2}}
  ,\\
  \label{eqn:lemma:pf:lemma:concentration-ineq:noiseless:small:condition-L:3}
  {\textstyle \Pr \left(
    \left\| \frac{\gFn[\JCoordsXX]( \thetaX, \thetaXX )}{\sqrt{2\pi}}  - \E \left[ \frac{\gFn[\JCoordsXX]( \thetaX, \thetaXX )}{\sqrt{2\pi}}  \right] \right\|_{2}
    >
    2 \sqrt{\frac{\kOXX \uX}{\m}}
    +
    \ux \tX
  \middle|
    \LRVXX{\thetaX}{\thetaXX} \leq \uX \m
  \right) }
  \leq
  e^{-\frac{1}{8} \uX \m \tX^{2}}
.\end{gather}
\end{lemma}
%
\newcommand{\COND}{\Forall{\thetaX' \in \ParamSpace}{\LRVX{\thetaX'} \leq \uX \m}}
\newcommand{\NEGCOND}{\ExistsST{\thetaX' \in \ParamSpace}{\LRVX{\thetaX'} > \uX \m}}
By symmetry, this implies that for any
\(  \thetaX \in \ParamSpace  \) and \(  \thetaXX \in \BallXX{2\tauX}( \thetaX )  \),
\begin{gather}
  {\textstyle
  \Pr \left(
    \left| \left\langle \frac{\hFn[\JCoordsXX]( \thetaX, \thetaXX )}{\sqrt{2\pi}}, \frac{\thetaX-\thetaXX}{\| \thetaX-\thetaXX \|_{2}} \right\rangle - \E \left[ \left\langle \frac{\hFn[\JCoordsXX]( \thetaX, \thetaXX )}{\sqrt{2\pi}}, \frac{\thetaX-\thetaXX}{\| \thetaX-\thetaXX \|_{2}} \right\rangle \right] \right|
    >
    \ux \tX
  \middle|
    \COND
  \right) }
  \leq
  2 e^{-\frac{1}{2} \uX \m \tX^{2}}
\label{eqn:lemma:pf:lemma:concentration-ineq:noiseless:small:condition-L:1:b}
  ,\\
  {\textstyle
  \Pr \left(
    \left| \left\langle \frac{\hFn[\JCoordsXX]( \thetaX, \thetaXX )}{\sqrt{2\pi}}, \frac{\thetaX+\thetaXX}{\| \thetaX+\thetaXX \|_{2}} \right\rangle - \E \left[ \left\langle \frac{\hFn[\JCoordsXX]( \thetaX, \thetaXX )}{\sqrt{2\pi}}, \frac{\thetaX+\thetaXX}{\| \thetaX+\thetaXX \|_{2}} \right\rangle \right] \right|
    >
    \ux \tX
  \middle|
    \COND
  \right) }
  \leq
  2 e^{-\frac{1}{2} \uX \m \tX^{2}}
\label{eqn:lemma:pf:lemma:concentration-ineq:noiseless:small:condition-L:2:b}
  ,\\
  {\textstyle
  \Pr \left(
    \left\| \frac{\gFn[\JCoordsXX]( \thetaX, \thetaXX )}{\sqrt{2\pi}} - \E \left[ \frac{\gFn[\JCoordsXX]( \thetaX, \thetaXX )}{\sqrt{2\pi}} \right] \right\|_{2}
    >
    2 \sqrt{\frac{\kOXX \uX}{\m}}
    +
    \ux \tX
  \middle|
    \COND
  \right) }
  \leq
  e^{-\frac{1}{8} \uX \m \tX^{2}}
\label{eqn:lemma:pf:lemma:concentration-ineq:noiseless:small:condition-L:3:b}
,\end{gather}
where notations are taken from the above lemma.
%
\par 
%
To help condense notations in the upcoming analysis, define the following indicator random variables:
\begin{gather*}
  \textstyle
  \ARV{1}{\thetaX}{\thetaXX}
  \defeq
  \I \left(
    \left| \left\langle \frac{\hFn[\JCoordsXX]( \thetaX, \thetaXX )}{\sqrt{2\pi}}, \frac{\thetaX-\thetaXX}{\| \thetaX-\thetaXX \|_{2}} \right\rangle - \E \left[ \left\langle \frac{\hFn[\JCoordsXX]( \thetaX, \thetaXX )}{\sqrt{2\pi}}, \frac{\thetaX-\thetaXX}{\| \thetaX-\thetaXX \|_{2}} \right\rangle \right] \right|
    >
    \nuX \tX
  \right)
  ,\\ \textstyle
  \ARV{2}{\thetaX}{\thetaXX}
  \defeq
  \I \left(
    \left| \left\langle \frac{\hFn[\JCoordsXX]( \thetaX, \thetaXX )}{\sqrt{2\pi}}, \frac{\thetaX+\thetaXX}{\| \thetaX+\thetaXX \|_{2}} \right\rangle - \E \left[ \left\langle \frac{\hFn[\JCoordsXX]( \thetaX, \thetaXX )}{\sqrt{2\pi}}, \frac{\thetaX+\thetaXX}{\| \thetaX+\thetaXX \|_{2}} \right\rangle \right] \right|
    >
    \nuX \tX
  \right)
  ,\\ \textstyle
  \ARV{3}{\thetaX}{\thetaXX}
  \defeq
  \I \left(
    \left\| \frac{\gFn[\JCoordsXX]( \thetaX, \thetaXX )}{\sqrt{2\pi}} - \E \left[ \frac{\gFn[\JCoordsXX]( \thetaX, \thetaXX )}{\sqrt{2\pi}} \right] \right\|_{2}
    >
    2 \sqrt{\frac{\kOXX \nuX}{\m}}
    +
    \nuX \tX
  \right)
.\end{gather*}
\renewcommand{\COND}{\Forall{\thetaX' \in \ParamSpace}{\LRVX{\thetaX'} \leq \nuX \m}}%
\renewcommand{\NEGCOND}{\ExistsST{\thetaX' \in \ParamSpace}{\LRVX{\thetaX'} > \nuX \m}}%
Now, observe:
\begin{align*}
  &
  \Pr( \ARV{1}{\thetaX}{\thetaXX}=1 \VEE \ARV{2}{\thetaX}{\thetaXX}=1 \VEE \ARV{3}{\thetaX}{\thetaXX}=1 \Mid| \COND )
  \\
  &\AlignIndent \leq
  \Pr( \ARV{1}{\thetaX}{\thetaXX}=1 \Mid| \COND )
  \\
  &\AlignIndent \AlignSp
  +
  \Pr( \ARV{2}{\thetaX}{\thetaXX}=1 \Mid| \COND )
  \\
  &\AlignIndent \AlignSp
  +
  \Pr( \ARV{3}{\thetaX}{\thetaXX}=1 \Mid| \COND )
  \\
  &\AlignIndent \dCmt{by a union bound}
  \\
  &\AlignIndent \leq
  2 e^{-\frac{1}{2} \nuX \m \tX^{2}}
  +
  2 e^{-\frac{1}{2} \nuX \m \tX^{2}}
  +
  e^{-\frac{1}{8} \nuX \m \tX^{2}}
  \\
  &\AlignIndent \dCmt{by \EQUATIONS \eqref{eqn:lemma:pf:lemma:concentration-ineq:noiseless:small:condition-L:1:b}--\eqref{eqn:lemma:pf:lemma:concentration-ineq:noiseless:small:condition-L:3:b}}
  \\
  &\AlignIndent \leq
  5 e^{-\frac{1}{8} \nuX \m \tX^{2}}
\TagEqn{\label{eqn:pf:lemma:concentration-ineq:noiseless:small:5}}
.\end{align*}
A uniform bound over all \(  \JCoordsXX \in \JSXX  \), \(  \thetaX \in \ParamCoverX  \), and \(  \thetaXX \in \BallXCover{\thetaX}  \) is then obtained as follows:
\begin{align*}
  &
  \Pr( \ExistsST{\JCoordsXX \in \JSXX, \thetaX \in \ParamCoverX, \thetaXX \in \BallXCover{\thetaX}}{\ARV{1}{\thetaX}{\thetaXX}=1 \VEE \ARV{2}{\thetaX}{\thetaXX}=1 \VEE \ARV{3}{\thetaX}{\thetaXX}=1} )
  \\
  &\AlignIndent
  =
  \Pr \!\!\!\! \begin{array}[t]{l} \displaystyle (
    \ExistsST
    {\JCoordsXX \in \JSXX, \thetaX \in \ParamCoverX, \thetaXX \in \BallXCover{\thetaX}}
    {\ARV{1}{\thetaX}{\thetaXX}=1 \VEE \ARV{2}{\thetaX}{\thetaXX}=1 \VEE \ARV{3}{\thetaX}{\thetaXX}=1}
    \\ \displaystyle \phantom{(}
    \Mid|
    \COND
  ) \end{array}
  \\
  &\AlignIndent \AlignSp
  \cdot
  \Pr(
    \COND
  )
  \\
  &\AlignIndent \AlignSp
  +
  \Pr \!\!\!\! \begin{array}[t]{l} \displaystyle (
    \ExistsST
    {\JCoordsXX \in \JSXX, \thetaX \in \ParamCoverX, \thetaXX \in \BallXCover{\thetaX}}
    {\ARV{1}{\thetaX}{\thetaXX}=1 \VEE \ARV{2}{\thetaX}{\thetaXX}=1 \VEE \ARV{3}{\thetaX}{\thetaXX}=1}
    \\ \displaystyle \phantom{(}
    \Mid|
    \NEGCOND
  ) \end{array}
  \\
  &\AlignIndent \AlignSp \phantom{+}
  \cdot
  \Pr(
    \NEGCOND
  )
  \\
  &\AlignIndent
  \dCmt{by the law of total probability and the definition of conditional probabilities}
  \\
  &\AlignIndent
  \leq
  \Pr \!\!\!\! \begin{array}[t]{l} \displaystyle (
    \ExistsST
    {\JCoordsXX \in \JSXX, \thetaX \in \ParamCoverX, \thetaXX \in \BallXCover{\thetaX}}
    {\ARV{1}{\thetaX}{\thetaXX}=1 \VEE \ARV{2}{\thetaX}{\thetaXX}=1 \VEE \ARV{3}{\thetaX}{\thetaXX}=1}
    \\ \displaystyle \phantom{(}
    \Mid|
    \COND
  ) \end{array}
  \\
  &\AlignIndent \AlignSp
  +
  \Pr(
    \NEGCOND
  )
  \\
  &\AlignIndent
  \leq
  | \JSXX | | \ParamCoverX | \qX
  \Pr (
    \ARV{1}{\thetaX}{\thetaXX}=1 \VEE \ARV{2}{\thetaX}{\thetaXX}=1 \VEE \ARV{3}{\thetaX}{\thetaXX}=1
    \Mid| \COND
  )
  \\
  &\AlignIndent \AlignSp
  +
  \Pr(
    \NEGCOND
  )
  \\
  &\AlignIndent
  \dCmt{for an arbitrary choice of \(  \JCoordsXX \in \JSXX, \thetaX \in \ParamCoverX, \thetaXX \in \BallXCover{\thetaX}  \);}
  \\
  &\AlignIndent
  \dCmt{by a union bound and an earlier discussion about the cardinality of \(  \BallXCover{\cdot}  \)}
  \\
  &\AlignIndent
  \leq
  5 | \JSXX | | \ParamCoverX | \qX e^{-\frac{1}{8} \nuX \m \tX^{2}}
  +
  \binom{\n}{\nO} e^{-\frac{1}{64} \nuX \m}
  \\
  &\AlignIndent
  \dCmt{by \EQUATION \eqref{eqn:pf:lemma:concentration-ineq:noiseless:small:5} and \LEMMA \ref{lemma:pf:lemma:concentration-ineq:noiseless:small:mismatches}}
  \\
  &\AlignIndent
  \leq
  5 | \JS | | \ParamCover | \qX e^{-\frac{1}{8} \nuX \m \tX^{2}}
  +
  \binom{\n}{\nO} e^{-\frac{1}{64} \nuX \m}
  .\\
  &\AlignIndent
  \dCmt{\(  \because | \JSXX | \leq | \JS |  \) and \(  | \ParamCoverX | \leq | \ParamCover |  \)}
\end{align*}
Note that the last line follows from recalling the definition of \(  \JSXX  \):
\(  \JSXX \defeq \{ \JCoords \cup \Supp( \thetaStar ) : \JCoords \in \JS \}  \),
which inserts at most one coordinate subset into \(  \JSXX  \) for each coordinate subset in \(  \JS  \).
%
\par 
%
Returning to \EQUATION \eqref{eqn:pf:lemma:concentration-ineq:noiseless:small:2:1} and now applying the results just derived above, the following bound holds for all \(  \JCoordsXX \in \JSXX  \), \(  \thetaX \in \ParamCoverX  \), and \(  \thetaXX \in \BallXCover{\thetaX}  \) with probability at least
\(  1 -  5 | \JS | | \ParamCover | \qX e^{-\frac{1}{8} \nuX \m \tX^{2}} - \binom{\n}{\nO} e^{-\frac{1}{64} \nuX \m}  \):
\begin{align*}
  \| \hFn[\JCoordsXX]( \thetaX, \thetaXX ) - \E[ \hFn[\JCoordsXX]( \thetaX, \thetaXX ) ] \|_{2}
  &\leq
  \left|
    \left\langle \hFn[\JCoordsXX]( \thetaX, \thetaXX ), \frac{\thetaX-\thetaXX}{\| \thetaX-\thetaXX \|_{2}} \right\rangle
    -
    \E \left[ \left\langle \hFn[\JCoordsXX]( \thetaX, \thetaXX ), \frac{\thetaX-\thetaXX}{\| \thetaX-\thetaXX \|_{2}} \right\rangle \right]
  \right|
  \\
  &\AlignSp+
  \left|
    \left\langle \hFn[\JCoordsXX]( \thetaX, \thetaXX ), \frac{\thetaX+\thetaXX}{\| \thetaX+\thetaXX \|_{2}} \right\rangle
    -
    \E \left[ \left\langle \hFn[\JCoordsXX]( \thetaX, \thetaXX ), \frac{\thetaX+\thetaXX}{\| \thetaX+\thetaXX \|_{2}} \right\rangle \right]
  \right|
  \\
  &\AlignSp+
  \|
    \gFn[\JCoordsXX]( \thetaX, \thetaXX )
    -
    \E[ \gFn[\JCoordsXX]( \thetaX, \thetaXX ) ]
  \|_{2}
  \\
  &\leq
  2\sqrt{\frac{\kOXX \nuX}{\m}} + 3 \nuX \tX
  \\
  &\leq
  5 \max \left\{
    \sqrt{\frac{\kOXX \nuX}{\m}},
    \nuX \tX
  \right\}
\TagEqn{\label{eqn:pf:lemma:concentration-ineq:noiseless:small:2}}
.\end{align*}
Set
\begin{align*}
  \tX
  &=
  \sqrt{\frac{8 \log \left( \frac{6}{\rhoSD} | \JS | | \ParamCover | \qX \right)}{\nuX \m}}
  \\
  &=
  \sqrt{
    \frac{8 \log \left( \frac{6}{\rhoSD} | \JS | | \ParamCover | \right)}{\nuX \m}
    +
    \frac{8 \log \left( \qXExpr \right)}{\nuX \m}
  }
  \\
  &\leq
  \sqrt{
    \frac{8 \log \left( \frac{6}{\rhoSD} | \JS | | \ParamCover | \right)}{\nuX \m}
    +
    8 \log \left( \frac{e}{\nuX} \right)
  }
  \\
  &\dCmt{by an earlier remark}
  \\
  &\leq
  \sqrt{\frac{8 \log \left( \frac{6}{\rhoSD} | \JS | | \ParamCover | \right)}{\nuX \m}}
  +
  \sqrt{8 \log \left( \frac{e}{\nuX} \right)}
\TagEqn{\label{eqn:pf:lemma:concentration-ineq:noiseless:small:3}}
  .\\
  &\dCmt{by the triangle inequality (in one dimension)}
\end{align*}
In accordance with \EQUATION \eqref{eqn:lemma:small-dist:m} of \LEMMA \ref{lemma:small-dist}, let
\begin{align*}
  \m
  &\geq
  \max \left\{
  \frac{200 \nuX \log \left( \frac{6}{\rhoSD} | \JS | | \ParamCover | \right)}{\left( \sqrt{\frac{\pi}{8}} \GAMMAX \ConstbSD \deltaX - \nuX \sqrt{8 \log \left( \frac{e}{\nuX} \right)} \right)^{2}}
  ,
  \frac{200 \nuX \kOXX}{\pi \GAMMAX^{2} \ConstbSD^{2} \deltaX^{2}}
  ,
  \frac{64}{\nuX} \log \left( \frac{6}{\rhoSD} \binom{\n}{\nO} \right)
  ,
  \frac{\ConstdSD \nO}{\nuX} \log \left( \frac{1}{\nuX} \right)
  \right\}
\TagEqn{\label{eqn:pf:lemma:concentration-ineq:noiseless:small:4}}
,\end{align*}
where \(  \ConstbSD, \ConstdSD > 0  \) are constants as per \DEFINITION \ref{def:univ-const} and \(  \nuX  \) is as defined in \DEFINITION \ref{def:nu-and-tau}.
Then, with probability at least
\begin{align*}
  &
  1
  -
  5 | \JS | | \ParamCover | \qX e^{-\frac{1}{8} \nuX \m \tX^{2}}
  -
  \binom{\n}{\nO} e^{-\frac{1}{64} \nuX \m}
  \\
  &\geq
  1
  -
  5 | \JS | | \ParamCover | \qX
  e^{-\frac{1}{8} \nuX \m \cdot \frac{8}{\nuX \m} \log ( \frac{6}{\rhoSD} | \JS | | \ParamCover | \qX )}
  -
  \binom{\n}{\nO}
  e^{-\frac{1}{64} \nuX \cdot \frac{64}{\nuX} \log ( \frac{6}{\rhoSD} \binom{\n}{\nO} )}
  \\
  &\dCmt{by the choices of \(  \tX, \m  \) in \EQUATIONS \eqref{eqn:pf:lemma:concentration-ineq:noiseless:small:3} and \eqref{eqn:pf:lemma:concentration-ineq:noiseless:small:4}}
  \\
  &\geq
  1 - \rhoSD
,\end{align*}
for all \(  \JCoordsXX \in \JSXX  \), \(  \thetaX \in \ParamCoverX  \), and \(  \thetaXX \in \BallXCover{\thetaX}  \),
\begin{align*}
  \| \hFn[\JCoordsXX]( \thetaX, \thetaXX ) - \E[ \hFn[\JCoordsXX]( \thetaX, \thetaXX ) ] \|_{2}
  &\leq
  5 \max \left\{
    \sqrt{\frac{\kOXX \nuX}{\m}},
    \nuX \tX
  \right\}
  \\
  &\dCmt{by \EQUATION \eqref{eqn:pf:lemma:concentration-ineq:noiseless:small:2}}
  \\
  &=
  5 \max \left\{
    \sqrt{\frac{\kOXX \nuX}{\m}},
    \sqrt{\frac{8 \nuX \log \left( \frac{6}{\rhoSD} | \JS | | \ParamCover | \right)}{\m}}
    +
    \nuX \sqrt{8 \log \left( \frac{e}{\nuX} \right)}
  \right\}
  \\
  &\dCmt{by the choice of \(  \tX  \) specified in \EQUATION \eqref{eqn:pf:lemma:concentration-ineq:noiseless:small:3}}
  \\
  &\leq
  \sqrt{\frac{\pi}{8}} \gammaX \ConstbSD \deltaX
  .\\
  &\dCmt{by the choice of \(  \m  \) specified in \EQUATION \eqref{eqn:pf:lemma:concentration-ineq:noiseless:small:4}}
  \\
  &\dCmtx{and the definition of \(  \nuX  \) in \EQUATION \eqref{eqn:def:nu-and-tau:nu}}
\end{align*}
The bound in \EQUATION \eqref{eqn:lemma:small-dist:ub} of \LEMMA \ref{lemma:small-dist} now follows: with probability at least \(  1-\rhoSD  \), uniformly for every \(  \JCoordsXX \in \JSXX  \), \(  \thetaX \in \ParamCoverX  \), and \(  \thetaXX \in \BallXCover{\thetaX}  \),
\begin{align*}
  \frac
  {2 \| \hFn[\JCoordsXX]( \thetaX, \thetaXX ) - \E[ \hFn[\JCoordsXX]( \thetaX, \thetaXX ) ] \|_{2}}
  {\DENOM}
  \leq
  \sqrt{\frac{8}{\pi}}\frac{1}{\gammaX} \sqrt{\frac{\pi}{8}} \gammaX \ConstbSD \deltaX
  =
  \ConstbSD \deltaX
.\end{align*}
\end{proof}
\begin{proof}
{\LEMMA \ref{lemma:large-dist:2}}
\checkoffbutmayberecheck%
Fix any \(  \thetaStar \in \ParamSpace  \).
Let
\(  \JCoordsX \in \JSX  \)
be arbitrary.
%
Once again, due to \EQUATION \eqref{eqn:lemma:concentration-ineq:ev:4} in \LEMMA \ref{lemma:concentration-ineq},
\begin{gather}
\label{eqn:pf:lemma:large-dist:2:2}
  \| \E[ \thetaX + \hfFn[\JCoordsX]( \thetaStar, \thetaX ) ] \|_{2}
  =
  \sqrt{\frac{\pi}{2}} \gammaX
.\end{gather}
Then, inserting \eqref{eqn:pf:lemma:large-dist:2:2} into the \RHS \EQUATION \eqref{eqn:lemma:large-dist:2:ub} in \LEMMA \ref{lemma:large-dist:2},
\begin{gather}
\label{eqn:pf:lemma:large-dist:2:3}
  \frac
  {2 \| \hfFn[\JCoordsX]( \thetaStar, \thetaStar ) - \E[ \hfFn[\JCoordsX]( \thetaStar, \thetaStar ) ] \|_{2}}
  {\DENOM}
  =
  \frac
  {2 \| \hfFn[\JCoordsX]( \thetaStar, \thetaStar ) - \E[ \hfFn[\JCoordsX]( \thetaStar, \thetaStar ) ] \|_{2}}
  {\sqrt{\hfrac{\pi}{2}} \gammaX}
.\end{gather}
The \(  \lnorm{2}  \)-distance between \(  \hfFn[\JCoordsX]( \thetaStar, \thetaStar )  \) and its mean, \ie the term
\(  \| \hfFn[\JCoordsX]( \thetaStar, \thetaStar ) - \E[ \hfFn[\JCoordsX]( \thetaStar, \thetaStar ) ] \|_{2}  \)
on the \RHS of \EQUATION \eqref{eqn:pf:lemma:large-dist:2:3},
is bound from above as follows.
Take \(  \sXXX, \tXX > 0  \) in \EQUATION \eqref{eqn:lemma:concentration-ineq:pr:2} in \LEMMA \ref{lemma:concentration-ineq} to be
\begin{gather}
\label{eqn:pf:lemma:large-dist:2:s'}
  \sXXX \defeq \sqrt{\frac{3 \log \left( \frac{3}{\rhoLDXX} \right)}{\alphaO \m}}
\end{gather}
and
\begin{align}
\label{eqn:pf:lemma:large-dist:2:t'}
  \tXX
  &\defeq
  \max \left\{
    \sqrt{\frac{3 \log \left( \frac{6}{\rhoLDXX} | \JSX | \right)}{\alphaO \m}}
    ,
    \sqrt{\frac{2 ( 1+\sXXX ) \log \left( \frac{3}{\rhoLDXX} | \JSX | \right)}{\alphaO \m}}
  \right\}
,\end{align}
and take \(  \m  \) to be bounded from below by
\begin{align}
\label{eqn:pf:lemma:large-dist:2:m:2}
  \m
  &\geq
  \max \left\{
  \frac{64 \alphaO}{\GAMMAX^{2} \ConstbLD^{2} \deltaX^{2}}
  \max \left\{
    3 \log \left( \frac{6}{\rhoLDXX} | \JS | | \ParamCover | \right)
    ,
    2 ( \kOX-1 )
  \right\}
  ,
  \frac{4}{\alphaO} \log \left( \frac{6}{\rhoLDXX} | \JS | | \ParamCover | \right)
  \right\}
%
,\end{align}
where this choice of \(  \m  \) satisfies
\begin{align*}
  \m
  &\geq
  \max \left\{
  \frac{64 \alphaO}{\GAMMAX^{2} \ConstbLD^{2} \deltaX^{2}}
  \max \left\{
    3 \log \left( \frac{6}{\rhoLDXX} | \JSX | \right)
    ,
    ( 1+\sXXX )( \kOX-1 )
  \right\}
  ,
  \frac{4}{\alphaO} \log \left( \frac{6}{\rhoLDXX} | \JSX | \right)
  \right\}
%
\end{align*}
because \(  | \JSX | \leq | \JS | | \ParamCover |  \).
This condition on \(  \m  \) also ensures that \(  \sXXX, \tXX < 1  \), as required to utilize \LEMMA \ref{lemma:concentration-ineq}.
Then, observe:
\begin{align*}
  \sqrt{\frac{2\pi \alphaO ( 1+\sXXX )( \kOX-1 )}{\m}}
  +
  \sqrt{2\pi} \alphaO \tXX
  &\leq
  \sqrt{\frac{2\pi \alphaO ( 1+\sXXX )( \kOX-1 )}{\m}}
  +
  \sqrt{\frac{6\pi \alphaO \log \left( \frac{6}{\rhoLDXX} | \JSX | \right)}{\m}}
  \\
  &\leq
  \frac{1}{2} \cdot \sqrt{\frac{\pi}{8}} \gammaX \ConstbLD \delta
  +
  \frac{1}{2} \cdot \sqrt{\frac{\pi}{8}} \gammaX \ConstbLD \delta
  \\
  &=
  \sqrt{\frac{\pi}{8}} \gammaX \ConstbLD \delta
\TagEqn{\label{eqn:pf:lemma:large-dist:2:6}}
,\end{align*}
where \(  \ConstbLD > 0  \) is a constant as per \DEFINITION \ref{def:univ-const}.
Together with \EQUATION \eqref{eqn:pf:lemma:large-dist:1:5} from the proof of \LEMMA \ref{lemma:large-dist:1}, \EQUATION \eqref{eqn:pf:lemma:large-dist:2:6} gives way to the following bound:
\begin{align*}
  &
  \Pr \left(
    \Forall{\JCoordsX \in \JSX}{
    \left\| \hFn[\JCoordsX]( \thetaStar, \thetaStar ) - \E \left[ \hFn[\JCoordsX]( \thetaStar, \thetaStar ) \right] \right\|_{2}
    \leq
    \sqrt{\frac{\pi}{8}} \gammaX \ConstbLD \deltaX
    }
  \right)
  \\
  &\geq
  \Pr \left(
    \Forall{\JCoordsX \in \JSX}{
    \left\| \hFn[\JCoordsX]( \thetaStar, \thetaStar ) - \E \left[ \hFn[\JCoordsX]( \thetaStar, \thetaStar ) \right] \right\|_{2}
    \leq
  \sqrt{\frac{2\pi \alphaO ( 1+\sXXX )( \kOX-1 )}{\m}}
  +
  \sqrt{2\pi} \alphaO \tXX
    }
  \right)
  \\
  &\dCmt{due to \EQUATIONS \eqref{eqn:pf:lemma:large-dist:1:5} and \eqref{eqn:pf:lemma:large-dist:2:6}}
  \\
  &\geq
  1
  -
  2 | \JSX | e^{-\frac{1}{3} \alphaO \m \tXX^{2}}
  -
  | \JSX | e^{-\frac{1}{2} \frac{\alphaO \m \tXX^{2}}{1+\sXXX}}
  -
  e^{-\frac{1}{3} \alphaO \m \sXXX^{2}}
  \\
  &\dCmt{by \EQUATION \eqref{eqn:lemma:concentration-ineq:pr:2} in \LEMMA \ref{lemma:concentration-ineq}}
  \\
  &\geq
  1 - \frac{\rhoLDXX}{3} - \frac{\rhoLDXX}{3} - \frac{\rhoLDXX}{3}
  \\
  &\dCmt{by the choice of \(  \sXXX, \tXX  \) in \EQUATIONS \eqref{eqn:pf:lemma:large-dist:2:s'} and \eqref{eqn:pf:lemma:large-dist:2:t'}}
  \\
  &=
  1 - \rhoLDXX
\TagEqn{\label{eqn:pf:lemma:large-dist:2:8}}
.\end{align*}
Therefore, taken together, \EQUATIONS \eqref{eqn:pf:lemma:large-dist:2:2} and \eqref{eqn:pf:lemma:large-dist:2:8} imply that if \(  \m  \) is bounded from below as in \EQUATION \eqref{eqn:pf:lemma:large-dist:2:m:2}, then with probability at least \(  1-\rhoLDXX  \), for all \(  \JCoordsX \in \JSX  \),
\begin{align*}
  \frac
  {2 \| \hfFn[\JCoordsX]( \thetaStar, \thetaStar ) - \E[ \hfFn[\JCoordsX]( \thetaStar, \thetaStar ) ] \|_{2}}
  {\DENOM}
  &=
  \frac
  {2 \| \hfFn[\JCoordsX]( \thetaStar, \thetaStar ) - \E[ \hfFn[\JCoordsX]( \thetaStar, \thetaStar ) ] \|_{2}}
  {\sqrt{\hfrac{\pi}{2}} \gammaX}
  \\
  &\leq
  \sqrt{\frac{8}{\pi}} \frac{1}{\gammaX}
  \sqrt{\frac{\pi}{8}} \gammaX
  \ConstbLD \deltaX
  \\
  &=
  \ConstbLD \deltaX
,\end{align*}
thus establishing \LEMMA \ref{lemma:large-dist:2}.
\end{proof}

\section{Proof of the Concentration Inequalities, \LEMMA \ref{lemma:concentration-ineq}}
\label{outline:concentration-ineq}


\subsection{Intermediate Results}
\label{outline:concentration-ineq|intermediate}

We return to \LEMMA \ref{lemma:concentration-ineq} to present its proof.
Towards this, the following pair of intermediate lemmas, whose proofs can be found in \SECTIONS \ref{outline:concentration-ineq|pf-noiseless} and \ref{outline:concentration-ineq|pf-noisy}, are provided below.

\begin{lemma}
\label{lemma:concentration-ineq:noiseless}
Fix
\(  \sXX, \tX, \tauX \in (0,1)  \),
and let
\(  \thetaStar \in \ParamSpace  \),
\(  \JS \subseteq 2^{[\n]}  \), and
\(  \ParamCover, \ParamCoverX \subset \ParamSpace  \),
where \(  \JS  \), \(  \ParamCover  \), and \(  \ParamCoverX  \) are finite sets, and where
\(  \ParamCoverX \defeq \ParamCover \setminus \Ball{\tauX}( \thetaStar )  \).
Let
\(  \kO \defeq \kOExpr  \).
Then,
\begin{gather}
  {
  \Pr \left(
    \ExistsST{\JCoords \in \JS, \thetaX \in \ParamCoverX}{
    \left| \left\langle \frac{\hFn[\JCoords]( \thetaStar, \thetaX )}{\sqrt{2\pi}}, \frac{\thetaStar-\thetaX}{\| \thetaStar-\thetaX \|_{2}} \right\rangle - \E \left[ \left\langle \frac{\hFn[\JCoords]( \thetaStar, \thetaX )}{\sqrt{2\pi}}, \frac{\thetaStar-\thetaX}{\| \thetaStar-\thetaX \|_{2}} \right\rangle \right] \right|
    >
    \frac{\tX \ADIST}{\pi}
    }
  \right) }
  \nonumber \\
  \leq
  2 | \JS | | \ParamCoverX | e^{-\frac{1}{3\pi} \m \tX^{2} \ADIST}
  \label{eqn:lemma:concentration-ineq:noiseless:pr:1}
  ,\\
  {
  \Pr \left(
    \ExistsST{\JCoords \in \JS, \thetaX \in \ParamCoverX}{
    \left| \left\langle \frac{\hFn[\JCoords]( \thetaStar, \thetaX )}{\sqrt{2\pi}}, \frac{\thetaStar+\thetaX}{\| \thetaStar+\thetaX \|_{2}} \right\rangle - \E \left[ \left\langle \frac{\hFn[\JCoords]( \thetaStar, \thetaX )}{\sqrt{2\pi}}, \frac{\thetaStar+\thetaX}{\| \thetaStar+\thetaX \|_{2}} \right\rangle \right] \right|
    >
    \frac{\tX \ADIST}{\pi}
    }
  \right) }
  \nonumber \\
  \leq
  2 | \JS | | \ParamCoverX | e^{-\frac{1}{3\pi} \m \tX^{2} \ADIST}
  \label{eqn:lemma:concentration-ineq:noiseless:pr:2}
  ,\\
  {
  \Pr \left(
    \ExistsST{\JCoords \in \JS, \thetaX \in \ParamCoverX}{
    \left\| \frac{\gFn[\JCoords]( \thetaStar, \thetaX )}{\sqrt{2\pi}} - \E \left[ \frac{\gFn[\JCoords]( \thetaStar, \thetaX )}{\sqrt{2\pi}} \right] \right\|_{2}
    >
    \sqrt{\frac{2 ( 1+\sXX )( \kO-2 ) \ADIST}{\m}}
    +
    \frac{\tX \ADIST}{\pi}
    }
  \right) }
  \nonumber \\
  \leq
  | \JS | | \ParamCoverX | e^{-\frac{1}{2\pi ( 1+\sXX )} \m \tX^{2} \ADIST}
  +
  | \ParamCoverX | e^{-\frac{1}{3\pi} \m \sXX^{2} \ADIST}
  \label{eqn:lemma:concentration-ineq:noiseless:pr:3}
,\end{gather}
where for any \(  \JCoords \subseteq [\n]  \) and \(  \thetaX \in \ParamSpace  \),
\begin{gather}
  \label{eqn:lemma:concentration-ineq:noiseless:ev:1}
  \E \left[ \left\langle \hFn( \thetaStar, \thetaX ), \frac{\thetaStar-\thetaX}{\| \thetaStar-\thetaX \|_{2}} \right\rangle \right]
  =
  \E \left[ \left\langle \hFn[\JCoords]( \thetaStar, \thetaX ), \frac{\thetaStar-\thetaX}{\| \thetaStar-\thetaX \|_{2}} \right\rangle \right]
  =
  \| \thetaStar-\thetaX \|_{2}
  ,\\ \label{eqn:lemma:concentration-ineq:noiseless:ev:2}
  \E \left[ \left\langle \hFn( \thetaStar, \thetaX ), \frac{\thetaStar+\thetaX}{\| \thetaStar+\thetaX \|_{2}} \right\rangle \right]
  =
  \E \left[ \left\langle \hFn[\JCoords]( \thetaStar, \thetaX ), \frac{\thetaStar+\thetaX}{\| \thetaStar+\thetaX \|_{2}} \right\rangle \right]
  = 0
  ,\\ \label{eqn:lemma:concentration-ineq:noiseless:ev:3}
  \E[ \gFn( \thetaStar, \thetaX ) ]
  =
  \E[ \gFn[\JCoords]( \thetaStar, \thetaX ) ]
  = \Vec{0}
.\end{gather}
\end{lemma}
\begin{lemma}
\label{lemma:concentration-ineq:noisy}
Fix
\(  \tX > 0  \), \(  \sXXX \in (0,1)  \), and \(  \deltaX \in (0,1)  \),
and let
\(  \JSX \subseteq 2^{[\n]}  \).
Let
\(  \kOX \defeq \kOXExpr  \),
and define
\(  \alphaO = \alphaO( \deltaX ) = \alphaOExpr  \).
Then, for \(  \thetaStar \in \ParamSpace  \),
\begin{gather}
  \label{eqn:lemma:concentration-ineq:noisy:pr:1}
  \Pr \left(
    \ExistsST{\JCoordsX \in \JSX}
    {\left| \left\langle \frac{\hfFn[\JCoordsX]( \thetaStar, \thetaStar )}{\sqrt{2\pi}}, \thetaStar \right\rangle - \E \left[ \left\langle \frac{\hfFn[\JCoordsX]( \thetaStar, \thetaStar )}{\sqrt{2\pi}}, \thetaStar \right\rangle \right] \right|
    >
    \alphaX \tX}
  \right)
  \leq
  2 | \JSX | e^{-\frac{1}{3} \alphaX \m \tX^{2}}
  ,\\ \nonumber
  \Pr \left(
    \ExistsST{\JCoordsX \in \JSX}
    {\left\| \frac{\gfFn[\JCoordsX]( \thetaStar, \thetaStar )}{\sqrt{2\pi}} - \E \left[ \frac{\gfFn[\JCoordsX]( \thetaStar, \thetaStar )}{\sqrt{2\pi}} \right]  \right\|_{2}
    >
    \sqrt{\frac{\alphaO ( 1+\sXXX )( \kOX-1 )}{\m}}
    +
    \alphaO \tX}
  \right)
  \\ \label{eqn:lemma:concentration-ineq:noisy:pr:2}
  \leq
  | \JSX | e^{-\frac{1}{2 ( 1+\sXXX )} \alphaO \m \tX^{2}}
  +
  e^{-\frac{1}{3} \alphaO \m \sXXX^{2}}
,\end{gather}
where for any \(  \JCoordsX \subseteq [\n]  \),
\begin{gather}
  \label{eqn:lemma:concentration-ineq:noisy:ev:1}
  \E \left[ \left\langle \hfFn( \thetaStar, \thetaStar ), \thetaStar \right\rangle \right]
  =
  \E \left[ \left\langle \hfFn[\JCoordsX]( \thetaStar, \thetaStar ), \thetaStar \right\rangle \right]
  = -\left( 1 - \sqrt{\frac{\pi}{2}} \gammaX \right)
  ,\\ \label{eqn:lemma:concentration-ineq:noisy:ev:2}
  \E[ \gFn( \thetaStar, \thetaStar ) ]
  =
  \E[ \gFn[\JCoordsX]( \thetaStar, \thetaStar ) ]
  = \Vec{0}
.\end{gather}
\end{lemma}


\subsection{Proof of \LEMMA \ref{lemma:concentration-ineq}}
\label{outline:concentration-ineq|pf}

We are ready to prove \LEMMA \ref{lemma:concentration-ineq} by means of the intermediate lemmas in \SECTION \ref{outline:concentration-ineq|intermediate}.

\begin{proof}
{\LEMMA \ref{lemma:concentration-ineq}}
\checkoff%
The proof of the lemma is split across \SECTIONS \ref{outline:concentration-ineq|pf|pr} and \ref{outline:concentration-ineq|pf|ev}, where the former derives \EQUATIONS \eqref{eqn:lemma:concentration-ineq:pr:1} and \eqref{eqn:lemma:concentration-ineq:pr:2} and the latter establishes \EQUATIONS \eqref{eqn:lemma:concentration-ineq:ev:1}--\eqref{eqn:lemma:concentration-ineq:ev:4}.


\subsubsection{Proof of \EQUATIONS \eqref{eqn:lemma:concentration-ineq:pr:1} and \eqref{eqn:lemma:concentration-ineq:pr:2}}
\label{outline:concentration-ineq|pf|pr}

\paragraph{Verification of \EQUATION \eqref{eqn:lemma:concentration-ineq:pr:1}} 
%
Fix
\(  \thetaStar \in \ParamSpace  \),
\(  \thetaX \in \ParamCoverX  \), and
\(  \JCoords \in \JS  \)
arbitrarily.
Towards bounding the concentration of \(  \hFn[\JCoords]( \thetaStar, \thetaX )  \) around its mean, consider the following orthogonal decomposition:
\begin{align}
\nonumber
  \hFn[\JCoords]( \thetaStar, \thetaX )
  &=
  \left\langle
    \hFn[\JCoords]( \thetaStar, \thetaX ),
    \frac
    {\thetaStar-\thetaX}
    {\| \thetaStar-\thetaX \|_{2}}
  \right\rangle
  \frac
  {\thetaStar-\thetaX}
  {\| \thetaStar-\thetaX \|_{2}}
  \\
  &\AlignSp+
  \left\langle
    \hFn[\JCoords]( \thetaStar, \thetaX ),
    \frac
    {\thetaStar+\thetaX}
    {\| \thetaStar+\thetaX \|_{2}}
  \right\rangle
  \frac
  {\thetaStar+\thetaX}
  {\| \thetaStar+\thetaX \|_{2}}
  +
  \gFn[\JCoords]( \thetaStar, \thetaX )
\label{eqn:pf:lemma:concentration-ineq:1}
,\end{align}
where, recalling \EQUATION \eqref{eqn:notations:gJ:def},
\begin{align}
\nonumber
  \gFn[\JCoords]( \thetaStar, \thetaX )
  &=
  \hFn[\JCoords]( \thetaStar, \thetaX )
  -
  \left\langle
    \hFn[\JCoords]( \thetaStar, \thetaX ),
    \frac
    {\thetaStar-\thetaX}
    {\| \thetaStar-\thetaX \|_{2}}
  \right\rangle
  \frac
  {\thetaStar-\thetaX}
  {\| \thetaStar-\thetaX \|_{2}}
  \\
  &\AlignSp-
  \left\langle
    \hFn[\JCoords]( \thetaStar, \thetaX ),
    \frac
    {\thetaStar+\thetaX}
    {\| \thetaStar+\thetaX \|_{2}}
  \right\rangle
  \frac
  {\thetaStar+\thetaX}
  {\| \thetaStar+\thetaX \|_{2}}
\label{eqn:pf:lemma:concentration-ineq:1:b}
.\end{align}
Due to \EQUATION \eqref{eqn:pf:lemma:concentration-ineq:1} and the linearity of expectation, the centered random vector
\(  \hFn[\JCoords]( \thetaStar, \thetaX ) - \E[ \hFn[\JCoords]( \thetaStar, \thetaX ) ]  \)
has the following orthogonal decomposition:
\begin{align*}
  &
  \hFn[\JCoords]( \thetaStar, \thetaX ) - \E[ \hFn[\JCoords]( \thetaStar, \thetaX ) ]
  \\
  &\AlignIndent =
  \left(
    \left\langle
      \hFn[\JCoords]( \thetaStar, \thetaX ),
      \frac
      {\thetaStar-\thetaX}
      {\| \thetaStar-\thetaX \|_{2}}
    \right\rangle
    -
    \E \left[
      \left\langle
        \hFn[\JCoords]( \thetaStar, \thetaX ),
        \frac
        {\thetaStar-\thetaX}
        {\| \thetaStar-\thetaX \|_{2}}
      \right\rangle
    \right]
  \right)
  \frac
  {\thetaStar-\thetaX}
  {\| \thetaStar-\thetaX \|_{2}}
  \\
  &\AlignIndent \AlignSp+
  \left(
    \left\langle
      \hFn[\JCoords]( \thetaStar, \thetaX ),
      \frac
      {\thetaStar+\thetaX}
      {\| \thetaStar+\thetaX \|_{2}}
    \right\rangle
    -
    \E \left[
    \left\langle
      \hFn[\JCoords]( \thetaStar, \thetaX ),
      \frac
      {\thetaStar+\thetaX}
      {\| \thetaStar+\thetaX \|_{2}}
      \right\rangle
    \right]
  \right)
  \frac
  {\thetaStar+\thetaX}
  {\| \thetaStar+\thetaX \|_{2}}
  \\
  &\AlignIndent \AlignSp+
  \left(
    \gFn[\JCoords]( \thetaStar, \thetaX )
    -
    \E[ \gFn[\JCoords]( \thetaStar, \thetaX ) ]
  \right)
\TagEqn{\label{lemma:concentration-ineq:pr:1}}
\end{align*}
due to \EQUATION \eqref{eqn:pf:lemma:concentration-ineq:1} and the linearity of expectation.
Applying the triangle inequality to the \(  \lnorm{2}  \)-norm of the orthogonal decomposition in \EQUATION \eqref{lemma:concentration-ineq:pr:1} and scaling it by a factor of \(  \frac{1}{\sqrt{2\pi}}  \) yields:
\begin{align*}
  &
  \left\| \frac{1}{\sqrt{2\pi}} \hFn[\JCoords]( \thetaStar, \thetaX ) - \E \left[ \frac{1}{\sqrt{2\pi}} \hFn[\JCoords]( \thetaStar, \thetaX ) \right] \right\|_{2}
  \\
  &\AlignIndent \leq
  \left|
    \left\langle
      \frac{1}{\sqrt{2\pi}}
      \hFn[\JCoords]( \thetaStar, \thetaX ),
      \frac
      {\thetaStar-\thetaX}
      {\| \thetaStar-\thetaX \|_{2}}
    \right\rangle
    -
    \E \left[
      \left\langle
        \frac{1}{\sqrt{2\pi}}
        \hFn[\JCoords]( \thetaStar, \thetaX ),
        \frac
        {\thetaStar-\thetaX}
        {\| \thetaStar-\thetaX \|_{2}}
      \right\rangle
    \right]
  \right|
  \\
  &\AlignIndent \AlignSp+
  \left|
    \left\langle
      \frac{1}{\sqrt{2\pi}}
      \hFn[\JCoords]( \thetaStar, \thetaX ),
      \frac
      {\thetaStar+\thetaX}
      {\| \thetaStar+\thetaX \|_{2}}
    \right\rangle
    -
    \E \left[
    \left\langle
      \frac{1}{\sqrt{2\pi}}
      \hFn[\JCoords]( \thetaStar, \thetaX ),
      \frac
      {\thetaStar+\thetaX}
      {\| \thetaStar+\thetaX \|_{2}}
      \right\rangle
    \right]
  \right|
  \\
  &\AlignIndent \AlignSp+
  \left\|
    \frac{1}{\sqrt{2\pi}}
    \gFn[\JCoords]( \thetaStar, \thetaX )
    -
    \E \left[ \frac{1}{\sqrt{2\pi}} \gFn[\JCoords]( \thetaStar, \thetaX ) \right]
  \right\|_{2}
\TagEqn{\label{lemma:concentration-ineq:pr:2}}
.\end{align*}
Due to \LEMMA \ref{lemma:concentration-ineq:noiseless}, the three terms in the last expression in \EQUATION \eqref{lemma:concentration-ineq:pr:2} are individually controlled with bounded probability as follows:
\begin{gather*}
  {
  \Pr \left(
    \ExistsST{\JCoords \in \JS, \thetaX \in \ParamCoverX}{
    \left| \left\langle \frac{\hFn[\JCoords]( \thetaStar, \thetaX )}{\sqrt{2\pi}}, \frac{\thetaStar-\thetaX}{\| \thetaStar-\thetaX \|_{2}} \right\rangle - \E \left[ \left\langle \frac{\hFn[\JCoords]( \thetaStar, \thetaX )}{\sqrt{2\pi}}, \frac{\thetaStar-\thetaX}{\| \thetaStar-\thetaX \|_{2}} \right\rangle \right] \right|
    >
    \frac{\tX \ADIST}{3\pi}
    }
  \right) }
  \\
  \leq
  2 | \JS | | \ParamCoverX | e^{-\frac{1}{27\pi} \m \tX^{2} \ADIST}
  ,\\
  {
  \Pr \left(
    \ExistsST{\JCoords \in \JS, \thetaX \in \ParamCoverX}{
    \left| \left\langle \frac{\hFn[\JCoords]( \thetaStar, \thetaX )}{\sqrt{2\pi}}, \frac{\thetaStar+\thetaX}{\| \thetaStar+\thetaX \|_{2}} \right\rangle - \E \left[ \left\langle \frac{\hFn[\JCoords]( \thetaStar, \thetaX )}{\sqrt{2\pi}}, \frac{\thetaStar+\thetaX}{\| \thetaStar+\thetaX \|_{2}} \right\rangle \right] \right|
    >
    \frac{\tX \ADIST}{3\pi}
    }
  \right) }
  \\
  \leq
  2 | \JS | | \ParamCoverX | e^{-\frac{1}{27\pi} \m \tX^{2} \ADIST}
  ,\\
  {
  \Pr \Bigl(
    \ExistsST{\JCoords \in \JS, \thetaX \in \ParamCoverX}{
    \left\| \frac{\gFn[\JCoords]( \thetaStar, \thetaX )}{\sqrt{2\pi}} - \E \left[ \frac{\gFn[\JCoords]( \thetaStar, \thetaX )}{\sqrt{2\pi}} \right]  \right\|_{2}
    >
    \sqrt{\frac{2 ( 1+\sXX )( \kO-2 ) \ADIST}{\m}}
    +
    \frac{\tX \ADIST}{3\pi}
    }
  \Bigr) }
  \\
  \leq
  | \JS | | \ParamCoverX | e^{-\frac{1}{18\pi ( 1+\sXX )} \m \tX^{2} \ADIST}
  +
  | \ParamCoverX | e^{-\frac{1}{3\pi} \m \sXX^{2} \ADIST}
.\end{gather*}
Combining the three above concentration inequalities via a union bound and complementing, it follows that with probability at least
\begin{align*}
  1
  &-
  2 | \JS | | \ParamCoverX | e^{-\frac{1}{27\pi} \m \tX^{2} \ADIST}
  -
  2 | \JS | | \ParamCoverX | e^{-\frac{1}{27\pi} \m \tX^{2} \ADIST}
  -
  | \JS | | \ParamCoverX | e^{-\frac{1}{18\pi ( 1+\sXX )} \m \tX^{2} \ADIST}
  \\
  &\AlignSp-
  | \ParamCoverX | e^{-\frac{1}{3\pi} \m \sXX^{2} \ADIST}
  \\
  &=
  1
  -
  4 | \JS | | \ParamCoverX | e^{-\frac{1}{27\pi} \m \tX^{2} \ADIST}
  -
  | \JS | | \ParamCoverX | e^{-\frac{1}{18\pi ( 1+\sXX )} \m \tX^{2} \ADIST}
  -
  | \ParamCoverX | e^{-\frac{1}{3\pi} \m \sXX^{2} \ADIST}
,\end{align*}
for all \(  \JCoords \in \JS  \) and all \(  \thetaX \in \ParamCoverX  \), the following three inequalities hold simultaneously:
\begin{gather*}
  \left|
    \left\langle
      \frac{\hFn[\JCoords]( \thetaStar, \thetaX )}{\sqrt{2\pi}}
      ,
      \frac
      {\thetaStar-\thetaX}
      {\| \thetaStar-\thetaX \|_{2}}
    \right\rangle
    -
    \E \left[
      \left\langle
        \frac{\hFn[\JCoords]( \thetaStar, \thetaX )}{\sqrt{2\pi}}
        ,
        \frac
        {\thetaStar-\thetaX}
        {\| \thetaStar-\thetaX \|_{2}}
      \right\rangle
    \right]
  \right|
  \leq
  \frac{\tX \ADIST}{3\pi}
  ,\\ 
  \left|
    \left\langle
      \frac{\hFn[\JCoords]( \thetaStar, \thetaX )}{\sqrt{2\pi}}
      ,
      \frac
      {\thetaStar+\thetaX}
      {\| \thetaStar+\thetaX \|_{2}}
    \right\rangle
    -
    \E \left[
    \left\langle
      \frac{\hFn[\JCoords]( \thetaStar, \thetaX )}{\sqrt{2\pi}}
      ,
      \frac
      {\thetaStar+\thetaX}
      {\| \thetaStar+\thetaX \|_{2}}
      \right\rangle
    \right]
  \right|
  \leq
  \frac{\tX \ADIST}{3\pi}
  ,\\ 
  \left\|
    \frac{\gFn[\JCoords]( \thetaStar, \thetaX )}{\sqrt{2\pi}}
    -
    \E \left[ \frac{\gFn[\JCoords]( \thetaStar, \thetaX )}{\sqrt{2\pi}} \right]
  \right\|_{2}
  \leq
  \sqrt{\frac{2 ( 1+\sXX )( \kO-2 ) \ADIST}{\m}}
  +
  \frac{\tX \ADIST}{3\pi}
.\end{gather*}
Then, taking this with \EQUATION \eqref{lemma:concentration-ineq:pr:2}, this implies that with probability at least
\begin{align*}
  1
  -
  4 | \JS | | \ParamCoverX | e^{-\frac{1}{27\pi} \m \tX^{2} \ADIST}
  -
  | \JS | | \ParamCoverX | e^{-\frac{1}{18\pi ( 1+\sXX )} \m \tX^{2} \ADIST}
  -
  | \ParamCoverX | e^{-\frac{1}{3\pi} \m \sXX^{2} \ADIST}
,\end{align*}
uniformly for all \(  \JCoords \in \JS  \) and all \(  \thetaX \in \ParamCoverX  \),
\begin{align*}
  \left\| \frac{\hFn[\JCoords]( \thetaStar, \thetaX )}{\sqrt{2\pi}} - \E \left[ \frac{\hFn[\JCoords]( \thetaStar, \thetaX )}{\sqrt{2\pi}} \right] \right\|_{2}
  \leq
  \sqrt{\frac{2 ( 1+\sXX )( \kO-2 ) \ADIST}{\m} }
  +
  \frac{\tX \ADIST}{\pi}
,\end{align*}
as desired.
This establishes \EQUATION \eqref{eqn:lemma:concentration-ineq:pr:1}.
%
\paragraph{Verification of \EQUATION \eqref{eqn:lemma:concentration-ineq:pr:2}} 
%
Next, \EQUATION \eqref{eqn:lemma:concentration-ineq:pr:2} is derived via an analogous technique.
Again, an orthogonal decomposition will facilitate the proof, this time into just two components:
\begin{align}
  \hfFn[\JCoords]( \thetaStar, \thetaStar )
  =
  \langle \hfFn[\JCoords]( \thetaStar, \thetaStar ), \thetaStar \rangle
  \thetaStar
  +
  \gfFn[\JCoords]( \thetaStar, \thetaStar )
\label{eqn:pf:lemma:concentration-ineq:2}
,\end{align}
where, as defined in \EQUATION \eqref{eqn:notations:gfJ:def},
\begin{align}
  \gfFn[\JCoords]( \thetaStar, \thetaStar )
  =
  \hfFn[\JCoords]( \thetaStar, \thetaStar )
  -
  \langle \hfFn[\JCoords]( \thetaStar, \thetaStar ), \thetaStar \rangle
  \thetaStar
\label{eqn:pf:lemma:concentration-ineq:2:b}
.\end{align}
By the above orthogonal decomposition in \eqref{eqn:pf:lemma:concentration-ineq:2} and the linearity of expectation,
\begin{align*}
  &
  \hfFn[\JCoords]( \thetaStar, \thetaStar ) - \E[ \hfFn[\JCoords]( \thetaStar, \thetaStar ) ]
  \\
  &\AlignIndent=
  \langle \hfFn[\JCoords]( \thetaStar, \thetaStar ), \thetaStar \rangle
  \thetaStar
  +
  \gfFn[\JCoords]( \thetaStar, \thetaStar )
  -
  \E[
    \langle \hfFn[\JCoords]( \thetaStar, \thetaStar ), \thetaStar \rangle
    \thetaStar
    +
    \gfFn[\JCoords]( \thetaStar, \thetaStar )
  ]
  \\
  &\AlignIndent=
  \left(
    \langle \hfFn[\JCoords]( \thetaStar, \thetaStar ), \thetaStar \rangle
    \thetaStar
    -
    \E[
      \langle \hfFn[\JCoords]( \thetaStar, \thetaStar ), \thetaStar \rangle
      \thetaStar
    ]
  \right)
  +
  \left(
    \gfFn[\JCoords]( \thetaStar, \thetaStar )
    -
    \E[ \gfFn[\JCoords]( \thetaStar, \thetaStar ) ]
  \right)
  \\
  &\AlignIndent=
  \left(
    \langle \hfFn[\JCoords]( \thetaStar, \thetaStar ), \thetaStar \rangle
    -
    \E[ \langle \hfFn[\JCoords]( \thetaStar, \thetaStar ), \thetaStar \rangle ]
  \right)
  \thetaStar
  +
  \left(
    \gfFn[\JCoords]( \thetaStar, \thetaStar )
    -
    \E[ \gfFn[\JCoords]( \thetaStar, \thetaStar ) ]
  \right)
.\end{align*}
Then, taking the norm of the above expressions and applying the triangle inequality to the last line,
\begin{align*}
  &
  \| \hfFn[\JCoords]( \thetaStar, \thetaStar ) - \E[ \hfFn[\JCoords]( \thetaStar, \thetaStar ) ] \|_{2}
  \\
  &\AlignIndent\leq
  \left\|
  \left(
    \langle \hfFn[\JCoords]( \thetaStar, \thetaStar ), \thetaStar \rangle
    -
    \E[
      \langle \hfFn[\JCoords]( \thetaStar, \thetaStar ), \thetaStar \rangle
    ]
  \right)
  \thetaStar
  \right\|_{2}
  +
  \left\|
    \gfFn[\JCoords]( \thetaStar, \thetaStar )
    -
    \E[ \gfFn[\JCoords]( \thetaStar, \thetaStar ) ]
  \right\|_{2}
  \\
  &\AlignIndent=
  \left|
    \langle \hfFn[\JCoords]( \thetaStar, \thetaStar ), \thetaStar \rangle
    -
    \E[
      \langle \hfFn[\JCoords]( \thetaStar, \thetaStar ), \thetaStar \rangle
    ]
  \right|
  +
  \left\|
    \gfFn[\JCoords]( \thetaStar, \thetaStar )
    -
    \E[ \gfFn[\JCoords]( \thetaStar, \thetaStar ) ]
  \right\|_{2}
.\end{align*}
Scaling this by a factor of \(  \frac{1}{\sqrt{2\pi}}  \) yields:
\begin{align*}
  &
  \left\| \frac{\hfFn[\JCoords]( \thetaStar, \thetaStar )}{\sqrt{2\pi}} - \E \left[ \frac{\hfFn[\JCoords]( \thetaStar, \thetaStar )}{\sqrt{2\pi}} \right] \right\|_{2}
  \\
  &\AlignIndent\leq
  \left|
    \left\langle \frac{\hfFn[\JCoords]( \thetaStar, \thetaStar )}{\sqrt{2\pi}} , \thetaStar \right\rangle
    -
    \E \left[
      \left\langle \frac{\hfFn[\JCoords]( \thetaStar, \thetaStar )}{\sqrt{2\pi}} , \thetaStar \right\rangle
    \right]
  \right|
  \\
  &\AlignIndent\AlignIndent+
  \left\|
    \frac{\gfFn[\JCoords]( \thetaStar, \thetaStar )}{\sqrt{2\pi}}
    -
    \E \left[ \frac{\gfFn[\JCoords]( \thetaStar, \thetaStar )}{\sqrt{2\pi}}  \right]
  \right\|_{2}
.\end{align*}
By \LEMMA \ref{lemma:concentration-ineq:noisy}, for \(  \sXXX, \tXX \in (0,1)  \),
\begin{align*}
  &\Pr \left(
    \ExistsST{\JCoordsX \in \JSX}
    {\left| \left\langle \frac{\hfFn[\JCoordsX]( \thetaStar, \thetaStar )}{\sqrt{2\pi}} , \thetaStar \right\rangle - \E \left[ \left\langle \frac{\hfFn[\JCoordsX]( \thetaStar, \thetaStar )}{\sqrt{2\pi}} , \thetaStar \right\rangle \right] \right|
    >
    \frac{\alphaO \tXX}{2}}
  \right)
  \\
  &\AlignIndent =
  \Pr \left(
    \ExistsST{\JCoordsX \in \JSX}
    {\left| \left\langle \frac{\hfFn[\JCoordsX]( \thetaStar, \thetaStar )}{\sqrt{2\pi}} , \thetaStar \right\rangle - \E \left[ \left\langle \frac{\hfFn[\JCoordsX]( \thetaStar, \thetaStar )}{\sqrt{2\pi}} , \thetaStar \right\rangle \right] \right|
    >
    \alphaX \left( \frac{\alphaO \tXX}{2 \alphaX} \right)}
  \right)
  \\
  &\AlignIndent =
  2 | \JSX | e^{-\frac{1}{12} \frac{\alphaO}{\alphaX} \alphaO \m \tXX^{2}}
  \\
  &\AlignIndent
  \dCmt{due to \LEMMA \ref{lemma:concentration-ineq:noisy}}
  \\
  &\AlignIndent \leq
  2 | \JSX | e^{-\frac{1}{12} \alphaO \m \tXX^{2}}
  ,\\
  &\AlignIndent
  \dCmt{\(  \alphaO = \alphaOExpr \geq \alphaX  \) implies \(  \tfrac{\alphaO}{\alphaX} \geq 1  \)}
\end{align*}
and
\begin{gather*}
  \Pr \left(
    \ExistsST{\JCoordsX \in \JSX}
    {\left\| \frac{\gfFn[\JCoordsX]( \thetaStar, \thetaStar )}{\sqrt{2\pi}} - \E \left[ \frac{\gfFn[\JCoordsX]( \thetaStar, \thetaStar )}{\sqrt{2\pi}} \right]  \right\|_{2}
    >
    \sqrt{\frac{\alphaO ( 1+\sXXX )( \kOX-1 )}{\m} }
    +
    \frac{\alphaO \tXX}{2} }
  \right)
  \\
  \leq
  | \JSX | e^{-\frac{1}{8 ( 1+\sXXX )} \alphaO \m \tXX^{2}}
  +
  e^{-\frac{1}{3} \alphaO \m \sXXX^{2}}
,\end{gather*}
and hence, by a union bound over the above two probabilities, with probability at least
\begin{gather*}
  1
  -
  2 | \JSX | e^{-\frac{1}{12} \alphaO \m \tXX^{2}}
  -
  | \JSX | e^{-\frac{1}{8 ( 1+\sXXX )} \alphaO \m \tXX^{2}}
  -
  e^{-\frac{1}{3} \alphaO \m \sXXX^{2}}
,\end{gather*}
the norm of the centered random vector
\(  \frac{1}{\sqrt{2\pi}} \hfFn[\JCoords]( \thetaStar, \thetaStar ) - \E[ \frac{1}{\sqrt{2\pi}} \hfFn[\JCoords]( \thetaStar, \thetaStar ) ]  \)
is bounded from above by
\begin{align*}
  \left\| \frac{\hfFn[\JCoords]( \thetaStar, \thetaStar )}{\sqrt{2\pi}} - \E \left[ \frac{\hfFn[\JCoords]( \thetaStar, \thetaStar )}{\sqrt{2\pi}} \right] \right\|_{2}
  &\leq
  \left|
    \left\langle \frac{\hfFn[\JCoords]( \thetaStar, \thetaStar )}{\sqrt{2\pi}} , \thetaStar \right\rangle
    -
    \E \left[
      \left\langle \frac{\hfFn[\JCoords]( \thetaStar, \thetaStar )}{\sqrt{2\pi}} , \thetaStar \right\rangle
    \right]
  \right|
  \\
  &\AlignIndent\AlignIndent+
  \left\|
    \frac{\gfFn[\JCoords]( \thetaStar, \thetaStar )}{\sqrt{2\pi}}
    -
    \E \left[ \frac{\gfFn[\JCoords]( \thetaStar, \thetaStar )}{\sqrt{2\pi}}  \right]
  \right\|_{2}
  \\
  &\AlignIndent\leq
  \frac{\alphaO \tXX}{2} 
  +
  \sqrt{\frac{\alphaO ( 1+\sXXX )( \kOX-1 )}{\m} }
  +
  \frac{\alphaO \tXX}{2} 
  \\
  &\AlignIndent=
  \sqrt{\frac{\alphaO ( 1+\sXXX )( \kOX-1 )}{\m} }
  +
  \alphaO \tXX
.\end{align*}
Thus, \EQUATION \eqref{eqn:lemma:concentration-ineq:pr:2} holds.


\subsubsection{Proof of \EQUATIONS \eqref{eqn:lemma:concentration-ineq:ev:1}--\eqref{eqn:lemma:concentration-ineq:ev:4}}
\label{outline:concentration-ineq|pf|ev}

Next, the four expectations, \EQUATIONS \eqref{eqn:lemma:concentration-ineq:ev:1}--\eqref{eqn:lemma:concentration-ineq:ev:4}, in \LEMMA \ref{lemma:concentration-ineq:noisy} are verified.
Let
\(  \JCoords \subseteq [\n]  \)
be an arbitrary coordinate subset.
Note that it suffices to establish the results for \(  \hFn[\JCoords]  \) as those for \(  \hFn  \) immediately follow by taking
\(  \JCoords = [\n]  \).
%
\paragraph{Verification of \EQUATION \eqref{eqn:lemma:concentration-ineq:ev:1}} 
%
Towards establishing the first expectation, \EQUATION \eqref{eqn:lemma:concentration-ineq:ev:1}, recall the orthogonal decomposition in \EQUATION \eqref{eqn:pf:lemma:concentration-ineq:2} in the proof of \EQUATION \eqref{eqn:lemma:concentration-ineq:pr:1}:
\begin{align}
  \hFn[\JCoords]( \thetaStar, \thetaX )
  &=
  \left\langle
    \hFn[\JCoords]( \thetaStar, \thetaX ),
    \frac
    {\thetaStar-\thetaX}
    {\| \thetaStar-\thetaX \|_{2}}
  \right\rangle
  \frac
  {\thetaStar-\thetaX}
  {\| \thetaStar-\thetaX \|_{2}}
  +
  \left\langle
    \hFn[\JCoords]( \thetaStar, \thetaX ),
    \frac
    {\thetaStar+\thetaX}
    {\| \thetaStar+\thetaX \|_{2}}
  \right\rangle
  \frac
  {\thetaStar+\thetaX}
  {\| \thetaStar+\thetaX \|_{2}}
  \nonumber\\
  &\AlignSp\AlignSp+
  \gFn[\JCoords]( \thetaStar, \thetaX )
\label{eqn:pf:eqn:lemma:concentration-ineq:ev:1:1}
,\end{align}
where \(  \gFn[\JCoords]( \thetaStar, \thetaX )  \) is given in \EQUATION \eqref{eqn:notations:gJ:def} or \eqref{eqn:pf:lemma:concentration-ineq:1:b}.
Hence, in expectation,
\begin{align*}
  \E[ \hFn[\JCoords]( \thetaStar, \thetaX ) ]
  &=
  \E \left[
    \left\langle
      \hFn[\JCoords]( \thetaStar, \thetaX ),
      \frac
      {\thetaStar-\thetaX}
      {\| \thetaStar-\thetaX \|_{2}}
    \right\rangle
  \right]
  \frac
  {\thetaStar-\thetaX}
  {\| \thetaStar-\thetaX \|_{2}}
  \\
  &\AlignSp\AlignSp+
  \E \left[
    \left\langle
      \hFn[\JCoords]( \thetaStar, \thetaX ),
      \frac
      {\thetaStar+\thetaX}
      {\| \thetaStar+\thetaX \|_{2}}
    \right\rangle
  \right]
  \frac
  {\thetaStar+\thetaX}
  {\| \thetaStar+\thetaX \|_{2}}
  +
  \E \left[
    \gFn[\JCoords]( \thetaStar, \thetaX )
  \right]
\TagEqn{\label{eqn:pf:eqn:lemma:concentration-ineq:ev:1:2}}
\end{align*}
by \EQUATION \eqref{eqn:pf:eqn:lemma:concentration-ineq:ev:1:1} and the linearity of expectation.
Due to \LEMMA \ref{lemma:concentration-ineq:noiseless},
\begin{gather}
  \label{eqn:pf:eqn:lemma:concentration-ineq:ev:1:3:1}
  \E \left[ \left\langle \hFn[\JCoords]( \thetaStar, \thetaX ), \frac{\thetaStar-\thetaX}{\| \thetaStar-\thetaX \|_{2}} \right\rangle \right]
  \frac{\thetaStar-\thetaX}{\| \thetaStar-\thetaX \|_{2}}
  =
  \| \thetaStar-\thetaX \|_{2}
  \frac{\thetaStar-\thetaX}{\| \thetaStar-\thetaX \|_{2}}
  =
  \thetaStar-\thetaX
  ,\\ \label{eqn:pf:eqn:lemma:concentration-ineq:ev:1:3:2}
  \E \left[ \left\langle \hFn[\JCoords]( \thetaStar, \thetaX ), \frac{\thetaStar+\thetaX}{\| \thetaStar+\thetaX \|_{2}} \right\rangle \right]
  \frac{\thetaStar+\thetaX}{\| \thetaStar+\thetaX \|_{2}}
  = 0 \cdot \frac{\thetaStar+\thetaX}{\| \thetaStar+\thetaX \|_{2}}
  =
  \Vec{0}
  ,\\ \label{eqn:pf:eqn:lemma:concentration-ineq:ev:1:3:3}
  \E[ \gFn[\JCoords]( \thetaStar, \thetaX ) ]
  = \Vec{0}
.\end{gather}
Then, by \EQUATIONS \eqref{eqn:pf:eqn:lemma:concentration-ineq:ev:1:2}--\eqref{eqn:pf:eqn:lemma:concentration-ineq:ev:1:3:3},
\begin{align*}
  &
  \E[ \hFn[\JCoords]( \thetaStar, \thetaX ) ]
  \\
  &\AlignIndent=
  \E \left[
    \left\langle
      \hFn[\JCoords]( \thetaStar, \thetaX ),
      \frac
      {\thetaStar-\thetaX}
      {\| \thetaStar-\thetaX \|_{2}}
    \right\rangle
  \right]
  \frac
  {\thetaStar-\thetaX}
  {\| \thetaStar-\thetaX \|_{2}}
  +
  \E \left[
    \left\langle
      \hFn[\JCoords]( \thetaStar, \thetaX ),
      \frac
      {\thetaStar+\thetaX}
      {\| \thetaStar+\thetaX \|_{2}}
    \right\rangle
  \right]
  \frac
  {\thetaStar+\thetaX}
  {\| \thetaStar+\thetaX \|_{2}}
  \\
  &\AlignSp\AlignSp
  +
  \E \left[
    \gFn[\JCoords]( \thetaStar, \thetaX )
  \right]
  \\
  &\AlignIndent\dCmt{by \EQUATION \eqref{eqn:pf:eqn:lemma:concentration-ineq:ev:1:2}}
  \\
  &\AlignIndent=
  \thetaStar-\thetaX
  +
  \Vec{0}
  +
  \Vec{0}
  \\
  &\AlignIndent\dCmt{by \EQUATIONS \eqref{eqn:pf:eqn:lemma:concentration-ineq:ev:1:3:1}--\eqref{eqn:pf:eqn:lemma:concentration-ineq:ev:1:3:3}}
  \\
  &\AlignIndent=
  \thetaStar-\thetaX
,\end{align*}
as claimed.
%
\paragraph{Verification of \EQUATION \eqref{eqn:lemma:concentration-ineq:ev:2}} 
%
To verify \EQUATION \eqref{eqn:lemma:concentration-ineq:ev:2}, we turn to the orthogonal decomposition in \EQUATION \eqref{eqn:pf:lemma:concentration-ineq:2} used to prove \EQUATION \eqref{eqn:lemma:concentration-ineq:pr:2}:
\begin{align*}
  \hfFn[\JCoords]( \thetaStar, \thetaStar )
  =
  \langle \hfFn[\JCoords]( \thetaStar, \thetaStar ), \thetaStar \rangle
  \thetaStar
  +
  \gfFn[\JCoords]( \thetaStar, \thetaStar )
,\end{align*}
where \(  \hfFn[\JCoords]( \thetaStar, \thetaStar )  \) is as stated in \EQUATION \eqref{eqn:notations:gfJ:def} or \eqref{eqn:pf:lemma:concentration-ineq:2:b}.
With this decomposition, due to the linearity of expectation,
\begin{align*}
  \E[ \hfFn[\JCoords]( \thetaStar, \thetaStar ) ]
\XXX{
  &=
  \E [
    \langle \hfFn[\JCoords]( \thetaStar, \thetaStar ), \thetaStar \rangle
    \thetaStar
    +
    \gfFn[\JCoords]( \thetaStar, \thetaStar )
  ]
  \\
  &=
  \E [ \langle \hfFn[\JCoords]( \thetaStar, \thetaStar ), \thetaStar \rangle \thetaStar ]
  +
  \E[ \gfFn[\JCoords]( \thetaStar, \thetaStar ) ]
  \\
}
  &=
  \E [ \langle \hfFn[\JCoords]( \thetaStar, \thetaStar ), \thetaStar \rangle ]
  \thetaStar
  +
  \E[ \gfFn[\JCoords]( \thetaStar, \thetaStar ) ]
\TagEqn{\label{eqn:pf:eqn:lemma:concentration-ineq:ev:2:1}}
,\end{align*}
where the last line uses the fact that \(  \thetaStar  \) is nonrandom.
Recall from \LEMMA \ref{lemma:concentration-ineq:noisy} that
\begin{gather*}
  \E [ \langle \hfFn[\JCoords]( \thetaStar, \thetaStar ), \thetaStar \rangle ]
  =
  - \left( 1 - \sqrt{\frac{\pi}{2}} \gammaX \right)
  ,\\
  \E[ \gfFn[\JCoords]( \thetaStar, \thetaStar ) ]
  =
  \Vec{0}
.\end{gather*}
Thus, continuing the above derivation in \eqref{eqn:pf:eqn:lemma:concentration-ineq:ev:2:1}, \EQUATION \eqref{eqn:lemma:concentration-ineq:ev:2} follows:
\begin{align*}
  \E[ \hfFn[\JCoords]( \thetaStar, \thetaStar ) ]
\XXX{
  &=
  \E [ \langle \hfFn[\JCoords]( \thetaStar, \thetaStar ), \thetaStar \rangle ]
  \thetaStar
  +
  \E[ \gfFn[\JCoords]( \thetaStar, \thetaStar ) ]
  \\
  &=
  -\left( 1 - \sqrt{\frac{\pi}{2}} \gammaX \right) \thetaStar + \Vec{0}
  \\
}
  &=
  -\left( 1 - \sqrt{\frac{\pi}{2}} \gammaX \right) \thetaStar
.\end{align*}
%
\paragraph{Verification of \EQUATION \eqref{eqn:lemma:concentration-ineq:ev:4}} 
%
Observe:
\begin{align*}
  &
  \hfFn[\JCoords]( \thetaStar, \thetaX )
  \\
  &=
  \ThresholdSet{\Supp( \thetaStar ) \cup \Supp( \thetaX ) \cup \JCoords}(
    \frac{\sqrt{2\pi}}{\m}
    \sep \CovM^{\T}
    \sep \frac{1}{2}
    ( \fFn( \CovM \thetaStar ) - \Sign( \CovM \thetaX ) )
  )
  \\
  &\dCmt{by the definition of \(  \hfFn[\JCoords]  \) in \EQUATION \eqref{eqn:notations:hfJ:def}}
  \\
  &=
  \ThresholdSet{\Supp( \thetaStar ) \cup \Supp( \thetaX ) \cup \JCoords}(
    \frac{\sqrt{2\pi}}{\m}
    \sep \CovM^{\T}
    \sep \frac{1}{2}
    ( \Sign( \CovM \thetaStar ) - \Sign( \CovM \thetaX ) )
  )
  \\
  &\AlignSp+
  \ThresholdSet{\Supp( \thetaStar ) \cup \Supp( \thetaX ) \cup \JCoords}(
    \frac{\sqrt{2\pi}}{\m}
    \sep \CovM^{\T}
    \sep \frac{1}{2}
    ( \fFn( \CovM \thetaStar ) - \Sign( \CovM \thetaStar ) )
  )
  \\
  &\dCmt{the subset thresholding operation is a linear transformation (\see \SECTIONREF \ref{outline:notations})}
  \\
  &=
  \hFn[\JCoords]( \thetaStar, \thetaX )
  +
  \hfFn[\Supp( \thetaX ) \cup \JCoords]( \thetaStar, \thetaStar )
  . \TagEqn{\label{eqn:pf:lemma:concentration-ineq:6}} \\
  &\dCmt{by the definitions of \(  \hFn[\JCoords]  \) and \(  \hfFn[\JCoords]  \) in \EQUATIONS \eqref{eqn:notations:hJ:def} and \eqref{eqn:notations:hfJ:def}, respectively}
\end{align*}
Thus,
\begin{align}
  \E[ \hfFn[\JCoords]( \thetaStar, \thetaX ) ]
  &=
  \E[
    \hFn[\JCoords]( \thetaStar, \thetaX )
    +
    \hfFn[\Supp( \thetaX ) \cup \JCoords]( \thetaStar, \thetaStar )
  ]
  =
  \E[ \hFn[\JCoords]( \thetaStar, \thetaX ) ]
  +
  \E[ \hfFn[\Supp( \thetaX ) \cup \JCoords]( \thetaStar, \thetaStar ) ]
\label{eqn:pf:lemma:concentration-ineq:3}
,\end{align}
where the first equality applies \EQUATION \eqref{eqn:pf:lemma:concentration-ineq:6} and the second equality follows from the linearity of expectation.
By \EQUATIONS \eqref{eqn:lemma:concentration-ineq:ev:1} and \eqref{eqn:lemma:concentration-ineq:ev:2}, respectively,
\begin{gather}
  \label{eqn:pf:lemma:concentration-ineq:4}
  \E[ \hFn[\JCoords]( \thetaStar, \thetaX ) ]
  =
  \thetaStar-\thetaX
  ,\\ \label{eqn:pf:lemma:concentration-ineq:5}
  \E[ \hfFn[\Supp( \thetaX ) \cup \JCoords]( \thetaStar, \thetaStar ) ]
  =
  -\left( 1 - \sqrt{\frac{\pi}{2}} \gammaX \right) \thetaStar
,\end{gather}
and therefore,
\begin{align*}
  \E[ \thetaX + \hfFn[\JCoords]( \thetaStar, \thetaX ) ]
  &=
  \thetaX
  +
  \E[ \hFn[\JCoords]( \thetaStar, \thetaX ) ]
  +
  \E[ \hfFn[\Supp( \thetaX ) \cup \JCoords]( \thetaStar, \thetaStar ) ]
  \\
  &\dCmt{by \EQUATION \eqref{eqn:pf:lemma:concentration-ineq:3}}
  \\
  &=
  \thetaX
  +
  \thetaStar - \thetaX
  -
  \left( 1 - \sqrt{\frac{\pi}{2}} \gammaX \right) \thetaStar
  \\
  &\dCmt{by \EQUATIONS \eqref{eqn:pf:lemma:concentration-ineq:4} and \eqref{eqn:pf:lemma:concentration-ineq:5}}
\XXX{
  \\
  &=
  \thetaStar - \left( 1 - \sqrt{\frac{\pi}{2}} \gammaX \right) \thetaStar
  \\
  &\dCmt{by canceling terms}
}
  \\
  &=
  \sqrt{\frac{\pi}{2}} \gammaX \thetaStar
  \TagEqn{\label{eqn:pf:lemma:concentration-ineq:7}}
.\end{align*}
Then,
\begin{gather*}
  \| \E \left[ \thetaX + \hfFn[\JCoords]( \thetaStar, \thetaX ) \right] \|_{2}
  =
  \sqrt{\frac{\pi}{2}} \gammaX \| \thetaStar \|_{2}
  =
  \sqrt{\frac{\pi}{2}} \gammaX
,\end{gather*}
where the first equality applies \EQUATION \eqref{eqn:pf:lemma:concentration-ineq:7}, the second follows from the homogeneity of the (\(  \lnorm{2}  \)-)norm, and the third equality recalls that \(  \thetaStar  \) has unit \(  \lnorm{2}  \)-norm.
This completes the proof of \LEMMA \ref{lemma:concentration-ineq}.
\end{proof}


\subsection{Proof of \LEMMA \ref{lemma:concentration-ineq:noiseless}}
\label{outline:concentration-ineq|pf-noiseless}

\begin{proof}
{\LEMMA \ref{lemma:concentration-ineq:noiseless}}
\mostlycheckoff%
The proof of the expectations, \EQUATIONS \eqref{eqn:lemma:concentration-ineq:noiseless:ev:1}--\eqref{eqn:lemma:concentration-ineq:noiseless:ev:3}, in \LEMMA \ref{lemma:concentration-ineq:noiseless} are presented in \SECTION \ref{outline:concentration-ineq|pf-noiseless|ev}, while the concentration inequalities in \EQUATIONS \eqref{eqn:lemma:concentration-ineq:noiseless:pr:1}--\eqref{eqn:lemma:concentration-ineq:noiseless:pr:3} are proved in \SECTION \ref{outline:concentration-ineq|pf-noiseless|pr}.


\subsubsection{Proof the Expectations, \EQUATIONS \eqref{eqn:lemma:concentration-ineq:noiseless:ev:1}--\eqref{eqn:lemma:concentration-ineq:noiseless:ev:3}}
\label{outline:concentration-ineq|pf-noiseless|ev}

The expectations, \EQUATIONS \eqref{eqn:lemma:concentration-ineq:noiseless:ev:1}--\eqref{eqn:lemma:concentration-ineq:noiseless:ev:3}, follow largely from work already done by \cite{matsumoto2022binary}, which is summarized below as \LEMMA \ref{lemma:pf:lemma:concentration-ineq:noiseless:ev:cond-ev}.
%
\begin{lemma}[{\dueto \cite[{\APPENDIX B}]{matsumoto2022binary}}]
\label{lemma:pf:lemma:concentration-ineq:noiseless:ev:cond-ev}
Fix
\(  \thetaStar, \thetaX \in \ParamSpace  \)
and
\(  \lX \in \ZeroTo{\m}  \).
Let
\(  \LRV \defeq \| \I( \Sign( \CovM \thetaStar ) \neq \Sign( \CovM \thetaX ) ) \|_{0}  \).
Then,
\begin{gather}
  \label{eqn:lemma:pf:lemma:concentration-ineq:noiseless:ev:cond-ev:1}
  \E_{\CovM \Mid| \LRV} \left[ \left\langle \hFn[\JCoords]( \thetaStar, \thetaX ), \frac{\thetaStar-\thetaX}{\| \thetaStar-\thetaX \|_{2}} \right\rangle \middle| \LRV=\lX \right]
  =
  \frac{\pi \lX \| \thetaStar-\thetaX \|_{2}}{\m \ADIST}
  ,\\ \label{eqn:lemma:pf:lemma:concentration-ineq:noiseless:ev:cond-ev:2}
  \E_{\CovM \Mid| \LRV} \left[ \left\langle \hFn[\JCoords]( \thetaStar, \thetaX ), \frac{\thetaStar+\thetaX}{\| \thetaStar+\thetaX \|_{2}} \right\rangle \middle| \LRV=\lX \right]
  = 0
  ,\\ \label{eqn:lemma:pf:lemma:concentration-ineq:noiseless:ev:cond-ev:3}
  \E_{\CovM \Mid| \LRV} [ \gFn[\JCoords]( \thetaStar, \thetaX ) \Mid| \LRV=\lX ]
  = \Vec{0}
.\end{gather}
\end{lemma}
%
Taking
\(  \thetaStar, \thetaX \in \ParamSpace  \)
arbitrarily, via the law of total expectation, \EQUATIONS \eqref{eqn:lemma:concentration-ineq:noiseless:ev:2} and \eqref{eqn:lemma:concentration-ineq:noiseless:ev:3} follow from \EQUATIONS \eqref{eqn:lemma:pf:lemma:concentration-ineq:noiseless:ev:cond-ev:2} and \eqref{eqn:lemma:pf:lemma:concentration-ineq:noiseless:ev:cond-ev:3}, respectively:
\begin{gather}
  \label{eqn:pf:lemma:concentration-ineq:noiseless:1:1}
  \E_{\CovM} \left[ \left\langle \hFn[\JCoords]( \thetaStar, \thetaX ), \frac{\thetaStar+\thetaX}{\| \thetaStar+\thetaX \|_{2}} \right\rangle \right]
  =
  \E_{\LRV} \left[ \E_{\CovM \Mid| \LRV} \left[ \left\langle \hFn[\JCoords]( \thetaStar, \thetaX ), \frac{\thetaStar+\thetaX}{\| \thetaStar+\thetaX \|_{2}} \right\rangle \middle| \LRV \right] \right]
  =
  \E_{\LRV} [ 0 ]
  =
  0
  ,\\ \label{eqn:pf:lemma:concentration-ineq:noiseless:1:2}
  \E_{\CovM} [ \gFn[\JCoords]( \thetaStar, \thetaX ) ]
  =
  \E_{\LRV} \left[ \E_{\CovM \Mid| \LRV} \left[ \gFn[\JCoords]( \thetaStar, \thetaX ) \middle| \LRV \right] \right]
  =
  \E_{\LRV} [ \Vec{0} ]
  =
  \Vec{0}
.\end{gather}
Note that because \(  \JCoords \subseteq [\n]  \) is arbitrary, \EQUATIONS \eqref{eqn:pf:lemma:concentration-ineq:noiseless:1:1} and \eqref{eqn:pf:lemma:concentration-ineq:noiseless:1:2} further imply that
\begin{gather*}
  \E_{\CovM} \left[ \left\langle \hFn( \thetaStar, \thetaX ), \frac{\thetaStar+\thetaX}{\| \thetaStar+\thetaX \|_{2}} \right\rangle \right]
  =
  0
  ,\\
  \E_{\CovM} [ \gFn( \thetaStar, \thetaX ) ]
  =
  \Vec{0}
\end{gather*}
by taking \(  \JCoords = [\n]  \).
%
\par 
%
Proceeding to \EQUATION \eqref{eqn:lemma:concentration-ineq:noiseless:ev:1}, define the random variable
\(  \LRV \defeq \| \I( \Sign( \CovM \thetaStar ) \neq \Sign( \CovM \thetaX ) ) \|_{0}  \)
as in \LEMMA \ref{lemma:pf:lemma:concentration-ineq:noiseless:ev:cond-ev}.
To derive
the result,
first note that \(  \LRV  \) follows a binomial distribution:
\(  \LRV \sim \Binomial( \m, \frac{1}{\pi} \ADIST )  \),
where the characterization of this random variable, \(  \LRV  \), is folklore (\seeeg \cite{charikar2002similarity}).
Then, observe:
\begin{align*}
  &
  \E \left[ \left\langle \hFn[\JCoords]( \thetaStar, \thetaX ), \frac{\thetaStar-\thetaX}{\| \thetaStar-\thetaX \|_{2}} \right\rangle \right]
  \\
  &\AlignIndent=
  \E_{\LRV} \left[ \E_{\CovM \Mid| \LRV} \left[ \left\langle \hFn[\JCoords]( \thetaStar, \thetaX ), \frac{\thetaStar-\thetaX}{\| \thetaStar-\thetaX \|_{2}} \right\rangle \middle| \LRV \right] \right]
  \\
  &\AlignIndent\dCmt{by the law of total expectation}
  \\
  &\AlignIndent=
  \sum_{\lX=0}^{\m}
  \binom{\m}{\lX}
  \left( \frac{\ADIST}{\pi} \right)^{\lX}
  \left( 1-\frac{\ADIST}{\pi} \right)^{\m-\lX}
  \E_{\CovM \Mid| \LRV} \left[ \left\langle \hFn[\JCoords]( \thetaStar, \thetaX ), \frac{\thetaStar-\thetaX}{\| \thetaStar-\thetaX \|_{2}} \right\rangle \middle| \LRV=\lX \right]
  \\
  &\AlignIndent\dCmt{by the law of the lazy statistician and the mass function of \(  {\textstyle \LRV \sim \Binomial( \m, \frac{1}{\pi} \ADIST )}    \)}
  \\
  &\AlignIndent=
  \frac{\pi \| \thetaStar-\thetaX \|_{2}}{\m \ADIST}
  \sum_{\lX=0}^{\m}
  \binom{\m}{\lX}
  \left( \frac{\ADIST}{\pi} \right)^{\lX}
  \left( 1-\frac{\ADIST}{\pi} \right)^{\m-\lX}
  \lX
  \\
  &\AlignIndent\dCmt{by \EQUATION \eqref{eqn:lemma:pf:lemma:concentration-ineq:noiseless:ev:cond-ev:1}}
  \\
  &\AlignIndent=
  \frac{\pi \| \thetaStar-\thetaX \|_{2}}{\m \ADIST}
  \E[ \LRV ]
  \\
  &\AlignIndent\dCmt{by the definition of expectation and the mass function of a binomial random variable}
  \\
  &\AlignIndent=
  \| \thetaStar-\thetaX \|_{2}
  .\\
  &\AlignIndent\dCmt{by the expectation of a binomial random variable, \(  \E[ \LRV ] = \tfrac{1}{\pi} \m \ADIST  \)}
\end{align*}
Again, because \(  \JCoords \subseteq [\n]  \) is arbitrary, it directly follows from the above derivation that
\begin{gather*}
  \E \left[ \left\langle \hFn( \thetaStar, \thetaX ), \frac{\thetaStar-\thetaX}{\| \thetaStar-\thetaX \|_{2}} \right\rangle \right]
  =
  \| \thetaStar-\thetaX \|_{2}
.\end{gather*}
This verifies \EQUATION \eqref{eqn:lemma:concentration-ineq:noiseless:ev:1}.


\subsubsection{Proof of the Concentration Inequalities, \EQUATIONS \eqref{eqn:lemma:concentration-ineq:noiseless:pr:1}--\eqref{eqn:lemma:concentration-ineq:noiseless:pr:3}}
\label{outline:concentration-ineq|pf-noiseless|pr}

Next, we turn our attention to \EQUATIONS \eqref{eqn:lemma:concentration-ineq:noiseless:pr:1}--\eqref{eqn:lemma:concentration-ineq:noiseless:pr:3}.
We begin with some preliminary analysis that will facilitate the derivations of these equations.
Initially, fix
\(  \thetaX \in \ParamCoverX  \) and \(  \JCoords \in \JS  \)
arbitrarily, where later \(  \thetaX  \) and \(  \JCoords  \) will be varied over the entire sets \(  \ParamCoverX  \) and \(  \JS  \), respectively, in union bounds.
Write
\(  \CovVX\VIx{\iIx} \defeq
    \ThresholdSet{\Supp( \thetaStar ) \cup \Supp( \thetaX ) \cup \JCoords}( \CovV\VIx{\iIx} )
    \in \R^{\n}  \),
\(  \iIx \in [\m]  \).
The definition of \(  \frac{1}{\sqrt{2\pi}} \hFn[\JCoords]( \thetaStar, \thetaX )  \) in \EQUATION \eqref{eqn:notations:hJ:def} can be rewritten as follows:
\begin{align*}
  \frac{\hFn[\JCoords]( \thetaStar, \thetaX )}{\sqrt{2\pi}}
  &=
  \ThresholdSet{\Supp( \thetaStar ) \cup \Supp( \thetaX ) \cup \JCoords} \left(
    \frac{1}{\m}
    \sum_{\iIx=1}^{\m}
    \CovV\VIx{\iIx}
   \sep
    \frac{1}{2} \left( \Sign( \langle \CovV, \thetaStar \rangle ) - \Sign( \langle \CovV, \thetaX \rangle ) \right)
  \right)
  \\
  &\dCmt{by the definition of \(  \hFn[\JCoords]  \) in \EQUATION \eqref{eqn:notations:hJ:def}}
  \\
  &=
  \frac{1}{\m}
  \sum_{\iIx=1}^{\m}
  \CovVX\VIx{\iIx}
 \sep
  \frac{1}{2} \left( \Sign( \langle \CovV\VIx{\iIx}, \thetaStar \rangle ) - \Sign( \langle \CovV\VIx{\iIx}, \thetaX \rangle ) \right)
  \\
  &\dCmt{by the linearity of the map \(  \ThresholdSet{\Supp( \thetaStar ) \cup \Supp( \thetaX ) \cup \JCoords}  \)}
  \\
  &\dCmtx{(\see \SECTIONREF \ref{outline:notations}), and by the definition of \(  \CovVX\VIx{\iIx}  \), \(  \iIx \in [\m]  \)}
  \\
  &=
  \frac{1}{\m}
  \sum_{\iIx=1}^{\m}
  \CovVX\VIx{\iIx}
 \sep
  \frac{1}{2} \left( \Sign( \langle \CovVX\VIx{\iIx}, \thetaStar \rangle ) - \Sign( \langle \CovVX\VIx{\iIx}, \thetaX \rangle ) \right)
  \\
  &\dCmt{\(  \Supp( \thetaStar ), \Supp( \thetaX ) \subseteq \Supp( \thetaStar ) \cup \Supp( \thetaX ) \cup \JCoords  \)}
  \\
  &=
  \frac{1}{\m}
  \sum_{\iIx=1}^{\m}
  \CovVX\VIx{\iIx}
 \sep
  \Sign( \langle \CovVX\VIx{\iIx}, \thetaStar \rangle )
 \sep
  \I( \Sign( \langle \CovVX\VIx{\iIx}, \thetaStar \rangle ) \neq \Sign( \langle \CovVX\VIx{\iIx}, \thetaX \rangle ) )
\TagEqn{\label{eqn:pf:eqn:lemma:concentration-ineq:noiseless:pr:5}}
  .\\
  &\dCmt{\see justification below}
\end{align*}
The last line can be verified by checking the value taken by
\(  \frac{1}{2} ( \Sign( u ) - \Sign( v ) )  \)
at \(  u, v \in \R  \) for each possible pair values of
\(  \Sign( u ), \Sign( v ) \in \{ -1,1 \}  \):
\begin{align*}
  \frac{1}{2} ( \Sign( u ) - \Sign( v ) )
  &=
  \begin{cases}
  \+ 0 ,&\cIf \Sign( u )=\Sign( v )=1,\\
  \+ 0 ,&\cIf \Sign( u )=\Sign( v )=-1 ,\\
  \+ 1 ,&\cIf \Sign( u ) = 1 \neq 1 = \Sign( v ) ,\\
  -1 ,&\cIf \Sign( u ) = -1 \neq 1 = \Sign( v ) ,\\
  \end{cases}
  \\
  &=
  \Sign( u ) \sep \I( \Sign( u ) \neq \Sign( v ) )
.\end{align*}
Therefore,
\begin{align*}
  &
  \left\langle \frac{1}{\sqrt{2\pi}} \hFn[\JCoords]( \thetaStar, \thetaX ), \frac{\thetaStar-\thetaX}{\| \thetaStar-\thetaX \|_{2}} \right\rangle
  \\
  &\AlignIndent=
  \frac{1}{\m}
  \sum_{\iIx=1}^{\m}
  \left\langle \CovVX\VIx{\iIx}, \frac{\thetaStar-\thetaX}{\| \thetaStar-\thetaX \|_{2}} \right\rangle
 \sep
  \Sign( \langle \CovVX\VIx{\iIx}, \thetaStar \rangle )
 \sep
  \I( \Sign( \langle \CovVX\VIx{\iIx}, \thetaStar \rangle ) \neq \Sign( \langle \CovVX\VIx{\iIx}, \thetaX \rangle ) )
  \\
  &\AlignIndent\dCmt{by \EQUATION \eqref{eqn:pf:eqn:lemma:concentration-ineq:noiseless:pr:5} and the linearity of inner products}
  \\ \TagEqn{\label{eqn:pf:eqn:lemma:concentration-ineq:noiseless:pr:4}}
  &\AlignIndent=
  \frac{1}{\m}
  \sum_{\iIx=1}^{\m}
  \left\langle \CovVX\VIx{\iIx}, \frac{\thetaStar-\thetaX}{\| \thetaStar-\thetaX \|_{2}} \right\rangle
 \sep
  \Sign( \left\langle \CovVX\VIx{\iIx}, \frac{\thetaStar-\thetaX}{\| \thetaStar-\thetaX \|_{2}} \right\rangle )
 \sep
  \I( \Sign( \langle \CovVX\VIx{\iIx}, \thetaStar \rangle ) \neq \Sign( \langle \CovVX\VIx{\iIx}, \thetaX \rangle ) )
  \\
  &\AlignIndent\dCmt{\see justification below}
  \\
  &\AlignIndent=
  \frac{1}{\m}
  \sum_{\iIx=1}^{\m}
  \left| \left\langle \CovVX\VIx{\iIx}, \frac{\thetaStar-\thetaX}{\| \thetaStar-\thetaX \|_{2}} \right\rangle \right|
 \sep
  \I( \Sign( \langle \CovVX\VIx{\iIx}, \thetaStar \rangle ) \neq \Sign( \langle \CovVX\VIx{\iIx}, \thetaX \rangle ) )
\TagEqn{\label{eqn:pf:eqn:lemma:concentration-ineq:noiseless:pr:6}}
  ,\\
  &\AlignIndent\dCmt{\(  u \sep \Sign( u ) = | u |  \) for any \(  u \in \R  \)}
\end{align*}
where the second to last equality, \eqref{eqn:pf:eqn:lemma:concentration-ineq:noiseless:pr:4}, follows from the observation that either the indicator term takes the value \(  0  \), or otherwise, if
\(  \Sign( \langle \CovVX\VIx{\iIx}, \thetaStar \rangle ) \neq \Sign( \langle \CovVX\VIx{\iIx}, \thetaX \rangle )  \),
then
\(  \Sign( \langle \CovVX\VIx{\iIx}, \thetaX \rangle ) = -\Sign( \langle \CovVX\VIx{\iIx}, \thetaStar \rangle )  \),
and hence,
\begin{align*}
  \Sign( \left\langle \CovVX\VIx{\iIx}, \frac{\thetaStar-\thetaX}{\| \thetaStar-\thetaX \|_{2}} \right\rangle )
  &=
  \Sign( \left\langle \CovVX\VIx{\iIx}, \thetaStar \right\rangle - \left\langle \CovVX\VIx{\iIx}, \thetaX \right\rangle )
  \\
  &=
  \Sign( \Sign( \left\langle \CovVX\VIx{\iIx}, \thetaStar \right\rangle ) \left|\left\langle \CovVX\VIx{\iIx}, \thetaStar \right\rangle\right| - \Sign( \left\langle \CovVX\VIx{\iIx}, \thetaX \right\rangle ) \left| \left\langle \CovVX\VIx{\iIx}, \thetaX \right\rangle \right| )
  \\
  &=
  \Sign( \Sign( \left\langle \CovVX\VIx{\iIx}, \thetaStar \right\rangle ) \left|\left\langle \CovVX\VIx{\iIx}, \thetaStar \right\rangle\right| + \Sign( \left\langle \CovVX\VIx{\iIx}, \thetaStar \right\rangle ) \left| \left\langle \CovVX\VIx{\iIx}, \thetaX \right\rangle \right| )
  \\
  &=
  \Sign( \Sign( \left\langle \CovVX\VIx{\iIx}, \thetaStar \right\rangle ) \left( \left|\left\langle \CovVX\VIx{\iIx}, \thetaStar \right\rangle\right| + \left| \left\langle \CovVX\VIx{\iIx}, \thetaX \right\rangle \right| \right) )
  \\
  &=
  \Sign( \left\langle \CovVX\VIx{\iIx}, \thetaStar \right\rangle )
.\end{align*}
With the above work out of the way, we are ready to derive \EQUATIONS \eqref{eqn:lemma:concentration-ineq:noiseless:pr:1}--\eqref{eqn:lemma:concentration-ineq:noiseless:pr:3}.
%
\paragraph{Verification of \EQUATION \eqref{eqn:lemma:concentration-ineq:noiseless:pr:1}} 
%
For \(  \iIx \in [\m]  \), let
\begin{gather*}
  \Ui \defeq | \langle \CovVX\VIx{\iIx}, \frac{\thetaStar-\thetaX}{\| \thetaStar-\thetaX \|_{2}} \rangle |
 \sep\I( \Sign( \langle \CovVX\VIx{\iIx}, \thetaStar \rangle ) \neq \Sign( \langle \CovVX\VIx{\iIx}, \thetaX \rangle ) )
,\end{gather*}
and let
\begin{gather*}
  \Ri \defeq \I( \Sign( \langle \CovVX\VIx{\iIx}, \thetaStar \rangle ) \neq \Sign( \langle \CovVX\VIx{\iIx}, \thetaX \rangle ) )
.\end{gather*}
Note that although this definition of \(  \Ri  \) differs slightly from a similar random variable analyzed in \cite[\APPENDIX B]{matsumoto2022binary}, nearly the same arguments as those in \cite{matsumoto2022binary} apply here, and hence we omit the
analysis here.
Due to the analysis in \cite[\APPENDIX B.1.1]{matsumoto2022binary}, the mass function of the random variable \(  \Ri  \) is given by
\begin{gather}
\label{eqn:pf:eqn:lemma:concentration-ineq:noiseless:pr:1:f_Ri}
  \pdf{\Ri}( \rX )
  =
  \begin{cases}
  1-\frac{1}{\pi} \ADIST ,& \cIf \rX=0, \\
  \frac{1}{\pi} \ADIST   ,& \cIf \rX=1,
  \end{cases}
\end{gather}
for \(  \rX \in \{ 0,1 \}  \).
For \(  \zX \in \R  \) and \(  \rX \in \{ 0,1 \}  \), the density function of the conditioned random variable \(  \Ui \Mid| \Ri  \) is given by
\begin{gather}
\label{eqn:pf:eqn:lemma:concentration-ineq:noiseless:pr:1:f_Ui|Ri}
  \pdf{\Ui \Mid| \Ri}( \zX \Mid| \rX )
  =
  \begin{cases}
  0 ,& \cIf \rX=0, \zX \neq 0, \\
  1 ,& \cIf \rX=0, \zX=0, \\
  0 ,& \cIf \rX=1, \zX<0, \\
  \frac{\pi}{\ADIST}
  \sqrt{\frac{2}{\pi}}
  e^{-\frac{1}{2} \zX^{2}}
  \frac{1}{\sqrt{2\pi}}
  \int_{\yX=-\zX \tan( \frac{1}{2} \ADIST )}^{\yX=\zX \tan( \frac{1}{2} \ADIST )}
  e^{-\frac{1}{2} \yX^{2}}
  d\yX
  ,& \cIf \rX=1, \zX \geq 0.
  \end{cases}
\end{gather}
%
\par 
%
Having specified the density function of the conditioned random variable \(  \Ui \Mid| \Ri  \), the next step is obtaining the \MGFs of the centered random variables \(  ( \Ui \Mid| \Ri ) - \E[ \Ui \Mid| \Ri ]  \) and \(  ( -\Ui \Mid| \Ri ) - \E[ -\Ui \Mid| \Ri ]  \).
To simplify notation, write \(  \muX[1] \defeq \E[ \Ui \Mid| \Ri=1 ]  \) and \(  \muX[0] \defeq \E[ \Ui \Mid| \Ri=0 ]  \), where in the latter case, 
\begin{align}
  \muX[0]
  &= \E[ \Ui \Mid| \Ri=0 ]
  = \int_{\zX=-\infty}^{\zX=\infty} \zX \pdf{\Ui \Mid| \Ri}( \zX \Mid| 0 ) d\zX
  = 0 \pdf{\Ui \Mid| \Ri}( 0 \Mid| 0 )
  = 0
\label{eqn:pf:eqn:lemma:concentration-ineq:noiseless:pr:7}
.\end{align}
Due to \cite[{\APPENDIX B.1}]{matsumoto2022binary}, the \MGF of \(  (\Ui \Mid| \Ri=1)-\E[\Ui \Mid| \Ri=1]  \), denoted by \(  \mgf{(\Ui \Mid| \Ri=1)-\E[\Ui \Mid| \Ri=1]}  \), is given and upper bounded at \(  \sX \in [0,\infty)  \) by
\begin{align*}
  &\AlignIndent
  \mgf{(\Ui \Mid| \Ri=1)-\E[\Ui \Mid| \Ri=1]}( \sX )
  \\
  &\AlignIndent=
  \E \left[ e^{\sX ( \Ui - \E[\Ui] )} \middle| \Ri=1 \right]
  \\
  &\AlignIndent=
  e^{\frac{1}{2} \sX^{2}}
  e^{-\sX \muX[1]}
  \frac{\pi}{\ADIST}
  \sqrt{\frac{2}{\pi}}
  \int_{\zX=0}^{\zX=\infty}
  e^{-\frac{1}{2} (\zX-\sX)^{2}}
  \frac{1}{\sqrt{2\pi}}
  \int_{\yX=-\zX \tan( \frac{1}{2} \ADIST )}^{\yX=\zX \tan( \frac{1}{2} \ADIST )}
  e^{-\frac{1}{2} \yX^{2}}
  d\yX\,
  d\zX
  \\
  &\AlignIndent\dCmt{by the law of the lazy statistician and \EQUATION \eqref{eqn:pf:eqn:lemma:concentration-ineq:noiseless:pr:1:f_Ui|Ri}}
  \\
  &\AlignIndent\leq
  e^{\frac{1}{2} \sX^{2}}
\TagEqn{\label{eqn:pf:eqn:lemma:concentration-ineq:noiseless:pr:2:a}}
.\end{align*}
%
Likewise, the \MGF of \(  (-\Ui \Mid| \Ri=1)-\E[-\Ui \Mid| \Ri=1]  \), denoted by \(  \mgf{(-\Ui \Mid| \Ri=1)-\E[-\Ui \Mid| \Ri=1]}  \), is \XXX{given and} upper bounded at \(  \sX \in [0,\infty)  \) by
\begin{align*}
  &
  \mgf{(-\Ui \Mid| \Ri=1)-\E[-\Ui \Mid| \Ri=1]}( \sX )
\XXX{
  \\
  &\AlignIndent=
  \E \left[ e^{\sX ( -\Ui - \E[-\Ui] )} \middle| \Ri=1 \right]
  \\
  &\AlignIndent=
  \E \left[ e^{-\sX ( \Ui - \E[\Ui] )} \middle| \Ri=1 \right]
  \\
  &\AlignIndent=
  e^{\frac{1}{2} \sX^{2}}
  e^{\sX \muX[1]}
  \frac{\pi}{\ADIST}
  \sqrt{\frac{2}{\pi}}
  \int_{\zX=0}^{\zX=\infty}
  e^{-\frac{1}{2} (\zX+\sX)^{2}}
  \frac{1}{\sqrt{2\pi}}
  \int_{\yX=-\zX \tan( \frac{1}{2} \ADIST )}^{\yX=\zX \tan( \frac{1}{2} \ADIST )}
  e^{-\frac{1}{2} \yX^{2}}
  d\yX\,
  d\zX
  \\
  &\AlignIndent\dCmt{by the law of the lazy statistician and \EQUATION \eqref{eqn:pf:eqn:lemma:concentration-ineq:noiseless:pr:1:f_Ui|Ri}}
  \\
  &\AlignIndent
}
  \leq
  e^{\frac{1}{2} \sX^{2}}
\TagEqn{\label{eqn:pf:eqn:lemma:concentration-ineq:noiseless:pr:2:b}}
.\end{align*}
On the other hand, when conditioning on \(  \Ri=0  \), the \MGF of the centered conditioned random variable \(  (\Ui \Mid| \Ri=0)-\E[\Ui \Mid| \Ri=0]  \), written \(  \mgf{(\Ui \Mid| \Ri=0)-\E[\Ui \Mid| \Ri=0]}  \), is given at \(  \sX \in [0,\infty)  \) by
\begin{align*}
  \mgf{(\Ui \Mid| \Ri=0)-\E[\Ui \Mid| \Ri=0]}( \sX )
  &=
\XXX{
  \E \left[ e^{\sX ( \Ui - \E[\Ui] )} \middle| \Ri=0 \right]
  \\
  &=
}
  \E \left[ e^{\sX \Ui} \middle| \Ri=0 \right]
  \\
  &\dCmt{using \EQUATION \eqref{eqn:pf:eqn:lemma:concentration-ineq:noiseless:pr:7}}
\XXX{
  \\
  &=
  \int_{\zX=-\infty}^{\zX=\infty} e^{\sX \zX} \pdf{\Ui \Mid| \Ri}( \zX \Mid| 0 ) d\zX
  \\
  &\dCmt{by the law of the lazy statistician and \EQUATION \eqref{eqn:pf:eqn:lemma:concentration-ineq:noiseless:pr:1:f_Ui|Ri}}
}
  \\
  &=
  e^{\sX \cdot 0} \pdf{\Ui \Mid| \Ri}( 0 \Mid| 0 )
  \\
  &\dCmt{by the law of the lazy statistician and \EQUATION \eqref{eqn:pf:eqn:lemma:concentration-ineq:noiseless:pr:1:f_Ui|Ri}, and}
  \\
  &\dCmtx{since the mass of \(  \Ui \Mid| \Ri=0  \) is entirely concentrated at \(  0   \)}
  \\
  &=
  1
\TagEqn{\label{eqn:pf:eqn:lemma:concentration-ineq:noiseless:pr:3:a}}
,\end{align*}
and the \MGF of the centered conditioned random variable \(  (-\Ui \Mid| \Ri=0)-\E[-\Ui \Mid| \Ri=0]  \), written \(  \mgf{(-\Ui \Mid| \Ri=0)-\E[-\Ui \Mid| \Ri=0]}  \), is similarly given at \(  \sX \in [0,\infty)  \) by
\begin{align*}
  \mgf{(-\Ui \Mid| \Ri=0)-\E[-\Ui \Mid| \Ri=0]}( \sX )
  =
  1
\TagEqn{\label{eqn:pf:eqn:lemma:concentration-ineq:noiseless:pr:3:b}}
.\end{align*}
%
\XXX{and the \MGF of the centered conditioned random variable \(  (-\Ui \Mid| \Ri=0)-\E[-\Ui \Mid| \Ri=0]  \), written \(  \mgf{(-\Ui \Mid| \Ri=0)-\E[-\Ui \Mid| \Ri=0]}  \), is similarly given at \(  \sX \in [0,\infty)  \) by
\begin{align*}
  \mgf{(-\Ui \Mid| \Ri=0)-\E[-\Ui \Mid| \Ri=0]}( \sX )
  &=
  \E \left[ e^{\sX ( -\Ui - \E[-\Ui] )} \middle| \Ri=0 \right]
  \\
  &=
  \E \left[ e^{-\sX \Ui} \middle| \Ri=0 \right]
  \\
  &\dCmt{by \EQUATION \eqref{eqn:pf:eqn:lemma:concentration-ineq:noiseless:pr:7}}
  \\
  &=
  \int_{\zX=-\infty}^{\zX=\infty} e^{-\sX \zX} \pdf{\Ui \Mid| \Ri}( \zX \Mid| 0 ) d\zX
  \\
  &\dCmt{by the law of the lazy statistician and \EQUATION \eqref{eqn:pf:eqn:lemma:concentration-ineq:noiseless:pr:1:f_Ui|Ri}}
  \\
  &=
  e^{-\sX \cdot 0} \pdf{\Ui \Mid| \Ri}( 0 \Mid| 0 )
  \\
  &=
  1
\TagEqn{\label{eqn:pf:eqn:lemma:concentration-ineq:noiseless:pr:3:b}}
.\end{align*}
}
Taking together the two cases for \(  \Ri=1  \) and \(  \Ri=0  \), the \MGF of \(  \Ui-\E[ \Ui ]  \), written \(  \mgf{\Ui-\E[\Ui]}  \), is given and bounded from above at \(  \sX \in [0,\infty)  \) by
\begin{align*}
  \mgf{\Ui-\E[\Ui]}( \sX )
  &=
  \E \left[ e^{\sX ( \Ui - \E[ \Ui ] )} \right]
  \\
  &=
  \pdf{\Ri}( 1 )
  \E \left[ e^{\sX ( \Ui - \E[ \Ui ] )} \middle| \Ri=1 \right]
  +
  \pdf{\Ri}( 0 )
  \E \left[ e^{\sX ( \Ui - \E[ \Ui ] )} \middle| \Ri=0 \right]
  \\
  &\dCmt{by the law of total expectation}
  \\
  &=
  \pdf{\Ri}( 1 )
  \mgf{( \Ui \Mid| \Ri=1 )-\E[ \Ui \Mid| \Ri=1 ]}( \sX )
  +
  \pdf{\Ri}( 0 )
  \mgf{( \Ui \Mid| \Ri=0 )-\E[ \Ui \Mid| \Ri=0 ]}( \sX )
  \\
  &\dCmt{by the definition of \MGFs}
  \\
  &=
  \frac{1}{\pi} \ADIST
  \mgf{( \Ui \Mid| \Ri=1 )-\E[ \Ui \Mid| \Ri=1 ]}( \sX )
  \\
  &\AlignSp+
  \left( 1-\frac{1}{\pi} \ADIST \right)
  \mgf{( \Ui \Mid| \Ri=0 )-\E[ \Ui \Mid| \Ri=0 ]}( \sX )
  \\
  &\dCmt{by \EQUATION \eqref{eqn:pf:eqn:lemma:concentration-ineq:noiseless:pr:1:f_Ri}}
  \\
  &\leq
  \frac{1}{\pi} \ADIST
  e^{\frac{1}{2} \sX^{2}}
  +
  \left( 1-\frac{1}{\pi} \ADIST \right)
  \\
  &\dCmt{by \EQUATIONS \eqref{eqn:pf:eqn:lemma:concentration-ineq:noiseless:pr:1:f_Ri}, \eqref{eqn:pf:eqn:lemma:concentration-ineq:noiseless:pr:2:a}, and \eqref{eqn:pf:eqn:lemma:concentration-ineq:noiseless:pr:3:a}}
  \\
  &=
  1 + \frac{1}{\pi} \ADIST \left( e^{\frac{1}{2} \sX^{2}} - 1 \right)
\TagEqn{\label{eqn:pf:eqn:lemma:concentration-ineq:noiseless:pr:1:1a}}
,\end{align*}
The \MGF of \(  -\Ui-\E[-\Ui]  \), denoted by \(  \mgf{-\Ui-\E[-\Ui]}  \), is similarly bounded from above by
\begin{align*}
  \mgf{-\Ui-\E[-\Ui]}( \sX )
  \leq
  1 + \frac{1}{\pi} \ADIST \left( e^{\frac{1}{2} \sX^{2}} - 1 \right)
\TagEqn{\label{eqn:pf:eqn:lemma:concentration-ineq:noiseless:pr:1:1b}}
.\end{align*}
\par 
%
Let
\(  \URV \defeq \sum_{\iIx=1}^{\m} \Ui  \).
Using the above bound on the \MGFs of \(  \Ui-\E[\Ui]  \), \(  \iIx \in [\m]  \), in \EQUATION \eqref{eqn:pf:eqn:lemma:concentration-ineq:noiseless:pr:1:1a}, the \MGF of \(  \URV-\E[\URV]  \), written \(  \mgf{\URV-\E[\URV]}  \), is given and upper bounded at \(  \sX \in [0,\infty)  \) as follows:
\begin{align*}
  \mgf{\URV-\E[\URV]}( \sX )
  &=
  \E \left[ e^{\sX ( \URV - \E[ \URV ] )} \right]
  \\
  &\dCmt{by the definition of the \MGF \(  \mgf{\URV-\E[\URV]}  \)}
  \\
  &=
  \E \left[ e^{\sX \sum_{\iIx=1}^{\m} \Ui - \E[ \Ui ]} \right]
  \\
  &\dCmt{by the definition of \(  \URV  \)}
\XXX{
  \\
  &=
  \E \left[ \prod_{\iIx=1}^{\m} e^{\sX ( \Ui - \E[ \Ui ] )} \right]
  \\
  &\dCmt{by standard facts about exponents}
}
  \\
  &=
  \prod_{\iIx=1}^{\m} \E \left[ e^{\sX ( \Ui - \E[ \Ui ] )} \right]
  \\
  &\dCmt{since \(  \Ui[1], \dots, \Ui[\m]  \) are mutually independent}
\XXX{
  \\
  &=
  \prod_{\iIx=1}^{\m} \mgf{\Ui-\E[\Ui]}( \sX )
  \\
  &\dCmt{by the definition of \(  \mgf{\Ui-\E[\Ui]}  \), \(  \iIx \in [\m]  \)}
}
  \\
  &=
  ( \mgf{\Ui-\E[\Ui]}( \sX ) )^{\m}
  \\
  &\dCmt{for any \(  \iIx \in [\m]  \);}
  \\
  &\dCmt{by the definition of \(  \mgf{\Ui-\E[\Ui]}  \), \(  \iIx \in [\m]  \), and}
  \\
  &\dCmtx{since \(  \Ui[1], \dots, \Ui[\m]  \) are identically distributed}
  \\
  &\leq
  \left( 1 + \frac{1}{\pi} \ADIST \left( e^{\frac{1}{2} \sX^{2}} - 1 \right) \right)^{\m}
  \\
  &\dCmt{by \EQUATION \eqref{eqn:pf:eqn:lemma:concentration-ineq:noiseless:pr:1:1a}}
  \\ \TagEqn{\label{eqn:pf:eqn:lemma:concentration-ineq:noiseless:pr:1:3a}}
  &\leq
  e^{\frac{1}{\pi} \m \ADIST ( e^{\frac{1}{2} \sX^{2}} - 1 )}
  ,\\
  &\dCmt{by a well-known inequality, \(  \log( 1+u ) \leq u  \) for \(  u > -1  \)}
\end{align*}
Likewise, applying \EQUATION \eqref{eqn:pf:eqn:lemma:concentration-ineq:noiseless:pr:1:1b} for the negated random variable, \(  -\URV-\E[-\URV]  \), obtains:
\begin{align*}
  \mgf{-\URV-\E[-\URV]}( \sX )
  &\leq
  e^{\frac{1}{\pi} \m \ADIST ( e^{\frac{1}{2} \sX^{2}} - 1 )}
\TagEqn{\label{eqn:pf:eqn:lemma:concentration-ineq:noiseless:pr:1:3b}}
.\end{align*}
Then, due to Bernstein (\see e.g., \cite{vershynin2018high}),
\begin{align*}
  \Pr \left( \frac{\URV}{\m} - \E \left[ \frac{\URV}{\m} \right] > \frac{1}{\pi} \tX \ADIST \right)
  &\leq
  \inf_{\sX \geq 0}
  e^{-\frac{1}{\pi} \m \sX \tX \ADIST}
  \mgf{\URV-\E[\URV]}( \sX )
  \\
  &\dCmt{due to Bernstein}
  \\
  &\leq
  \inf_{\sX \geq 0}
  e^{-\frac{1}{\pi} \m \sX \tX \ADIST}
  e^{\frac{1}{\pi} \m \ADIST ( e^{\frac{1}{2} \sX^{2}} - 1 )}
  \\
  &\dCmt{by \EQUATION \eqref{eqn:pf:eqn:lemma:concentration-ineq:noiseless:pr:1:3a}}
  \\
  &=
  \inf_{\sX \geq 0}
  e^{-\frac{1}{\pi} \m \ADIST ( \sX \tX - e^{\frac{1}{2} \sX^{2}} + 1 )}
\TagEqn{\label{eqn:pf:eqn:lemma:concentration-ineq:noiseless:pr:1:2a}}
,\end{align*}
and for the concentration on the other side, an analogous derivation gives
\begin{align*}
  \Pr \left( \frac{\URV}{\m} - \E \left[ \frac{\URV}{\m} \right] < -\frac{1}{\pi} \tX \ADIST \right)
  \leq
  \inf_{\sX \geq 0}
  e^{-\frac{1}{\pi} \m \ADIST ( \sX \tX - e^{\frac{1}{2} \sX^{2}} + 1 )}
\TagEqn{\label{eqn:pf:eqn:lemma:concentration-ineq:noiseless:pr:1:2b}}
.\end{align*}
%
\XXX{and for the concentration on the other side,
\begin{align*}
  \Pr \left( \frac{\URV}{\m} - \E \left[ \frac{\URV}{\m} \right] < -\frac{1}{\pi} \tX \ADIST \right)
  &=
  \Pr \left( -\frac{\URV}{\m} - \E \left[ -\frac{\URV}{\m} \right] > \frac{1}{\pi} \tX \ADIST \right)
  \\
  &\dCmt{by negating both sides of the inequality}
  \\
  &\leq
  \inf_{\sX \geq 0}
  e^{-\frac{1}{\pi} \m \sX \tX \ADIST}
  \mgf{-\URV-\E[-\URV]}( \sX )
  \\
  &\dCmt{due to Bernstein}
  \\
  &\leq
  \inf_{\sX \geq 0}
  e^{-\frac{1}{\pi} \m \sX \tX \ADIST}
  e^{\frac{1}{\pi} \m \ADIST ( e^{\frac{1}{2} \sX^{2}} - 1 )}
  \\
  &\dCmt{by \EQUATION \eqref{eqn:pf:eqn:lemma:concentration-ineq:noiseless:pr:1:3b}}
  \\
  &=
  \inf_{\sX \geq 0}
  e^{-\frac{1}{\pi} \m \ADIST ( \sX \tX - e^{\frac{1}{2} \sX^{2}} + 1 )}
\TagEqn{\label{eqn:pf:eqn:lemma:concentration-ineq:noiseless:pr:1:2b}}
.\end{align*}
}
To minimize the last expressions in \EQUATIONS \eqref{eqn:pf:eqn:lemma:concentration-ineq:noiseless:pr:1:2a} and \eqref{eqn:pf:eqn:lemma:concentration-ineq:noiseless:pr:1:2b}, observe that when \(  \sX, \tX > 0  \) are small,
\begin{align*}
  \left. \frac{\partial}{\partial \sX} \sX \tX - e^{\frac{1}{2} \sX^{2}} + 1 \right|_{\sX=\tX}
  =
  \left. \tX - \sX e^{\frac{1}{2} \sX^{2}} \right|_{\sX=\tX}
  \approx
  0
,\end{align*}
where
\begin{align*}
  \frac{\partial^{2}}{\partial \sX^{2}} \sX \tX - e^{\frac{1}{2} \sX^{2}} + 1
  =
  -( 1+\sX^{2} ) e^{\frac{1}{2} \sX^{2}}
  <
  0
\end{align*}
for any \(  \sX \in \R  \).
Hence, for small \(  \sX, \tX \in (0,1]  \), the expression \(  \sX \tX - e^{\frac{1}{2} \sX^{2}} + 1  \) is maximized with respect to \(  \sX  \) at approximately \(  \sX \approx \tX  \) (which minimizes \EQUATIONS \eqref{eqn:pf:eqn:lemma:concentration-ineq:noiseless:pr:1:2a} and \eqref{eqn:pf:eqn:lemma:concentration-ineq:noiseless:pr:1:2b}),
and moreover, when \(  \tX \in (0,1]  \),
\begin{align*}
  \left. \sX \tX - e^{\frac{1}{2} \sX^{2}} + 1 \right|_{\sX=\tX}
  =
  \tX^{2} - e^{\frac{1}{2} \tX^{2}} + 1
  \geq
  \frac{\tX^{2}}{3}
.\end{align*}
Therefore, taking \(  \sX=\tX  \) in \EQUATIONS \eqref{eqn:pf:eqn:lemma:concentration-ineq:noiseless:pr:1:2a} and \eqref{eqn:pf:eqn:lemma:concentration-ineq:noiseless:pr:1:2b} yields the following concentration inequalities for \(  \tX \in (0,1]  \):
\begin{gather*}
  \Pr \left( \frac{\URV}{\m} - \E \left[ \frac{\URV}{\m} \right] > \frac{1}{\pi} \tX \ADIST \right)
  \leq
  e^{-\frac{1}{3\pi} \m \tX^{2} \ADIST}
  ,\\
  \Pr \left( \frac{\URV}{\m} - \E \left[ \frac{\URV}{\m} \right] < -\frac{1}{\pi} \tX \ADIST \right)
  \leq
  e^{-\frac{1}{3\pi} \m \tX^{2} \ADIST}
,\end{gather*}
and combining these via a union bound subsequently obtains:
\begin{gather*}
  \Pr \left( \left| \frac{\URV}{\m} - \E \left[ \frac{\URV}{\m} \right] \right| > \frac{1}{\pi} \tX \ADIST \right)
  \leq
  2 e^{-\frac{1}{3\pi} \m \tX^{2} \ADIST}
.\end{gather*}
This, along with \EQUATION \eqref{eqn:pf:eqn:lemma:concentration-ineq:noiseless:pr:6} and the definition of the random variable \(  \URV  \), immediately implies that
\begin{gather*}
  \Pr \left(
    \left| \left\langle \frac{\hFn[\JCoords]( \thetaStar, \thetaX )}{\sqrt{2\pi}}, \frac{\thetaStar-\thetaX}{\| \thetaStar-\thetaX \|_{2}} \right\rangle - \E \left[ \left\langle \frac{\hFn[\JCoords]( \thetaStar, \thetaX )}{\sqrt{2\pi}}, \frac{\thetaStar-\thetaX}{\| \thetaStar-\thetaX \|_{2}} \right\rangle \right] \right|
    >
    \frac{\tX \ADIST}{\pi}
  \right)
  \nonumber \\
  \leq
  2 e^{-\frac{1}{3\pi} \m \tX^{2} \ADIST}
.\end{gather*}
Then, \EQUATION \eqref{eqn:lemma:concentration-ineq:noiseless:pr:1} follows from union bounds over all \(  \JCoords \in \JS  \) and \(  \thetaX \in \ParamCoverX  \):
\begin{gather*}
  \Pr \left(
    \ExistsST{\JCoords \in \JS, \thetaX \in \ParamCoverX}{
    \left| \left\langle \frac{\hFn[\JCoords]( \thetaStar, \thetaX )}{\sqrt{2\pi}}, \frac{\thetaStar-\thetaX}{\| \thetaStar-\thetaX \|_{2}} \right\rangle - \E \left[ \left\langle \frac{\hFn[\JCoords]( \thetaStar, \thetaX )}{\sqrt{2\pi}}, \frac{\thetaStar-\thetaX}{\| \thetaStar-\thetaX \|_{2}} \right\rangle \right] \right|
    >
    \frac{\tX \ADIST}{\pi}
    }
  \right)
  \nonumber \\
  \leq
  2 | \JS | | \ParamCoverX | e^{-\frac{1}{3\pi} \m \tX^{2} \ADIST}
,\end{gather*}
as desired.
%
\paragraph{Verification of \EQUATION \eqref{eqn:lemma:concentration-ineq:noiseless:pr:2}} 
%
\let\oldvX\vX%
\let\vX\zX%
Next, the concentration inequality in \EQUATION \eqref{eqn:lemma:concentration-ineq:noiseless:pr:2} will be verified.
Again, take any \(  \JCoords \in \JS  \) and \(  \thetaX \in \ParamCoverX  \).
Write the random variables
\(  \Vi \defeq \langle \CovVX\VIx{\iIx}, \frac{\thetaStar+\thetaX}{\| \thetaStar+\thetaX \|_{2}} \rangle  \), \(  \iIx \in [\m]  \),
and carry over the notation of the random variables
\(  \Ri \defeq \I( \Sign( \langle \CovVX\VIx{\iIx}, \thetaStar \rangle ) \neq \Sign( \langle \CovVX\VIx{\iIx}, \thetaX \rangle ) )  \), \(  \iIx \in [\m]  \),
from the verification of \EQUATION \eqref{eqn:lemma:concentration-ineq:noiseless:pr:1}, whose mass functions are given in \EQUATION \eqref{eqn:pf:eqn:lemma:concentration-ineq:noiseless:pr:1:f_Ri}.
Due to \cite[\LEMMA B.6]{matsumoto2022binary}, the conditioned random variable \(  \Vi \Mid| \Ri=1  \) is standard Gaussian, \ie
\(  ( \Vi \Mid| \Ri=1 ) \sim \N(0,1)  \),
and thus, the density function of \(  \Vi \Mid| \Ri  \) is given for \(  \zX \in \R  \) and \(  \rX \in \{ 0,1 \}  \) by:
\begin{gather*}
  \pdf{\Vi \Mid| \Ri}( \zX \Mid| \rX )
  =
  \begin{cases}
  0                                             ,& \cIf \rX=0, \zX \neq 0 ,\\
  1                                             ,& \cIf \rX=0, \zX = 0    ,\\
  \frac{1}{\sqrt{2\pi}} e^{-\frac{1}{2} \zX^{2}} ,& \cIf \rX=1.
  \end{cases}
\end{gather*}
Since \(  \Vi \Mid| \Ri=1  \) is standard Gaussian, its expectation is
\(  \E[ \Vi \Mid| \Ri=1 ] = 0  \),
while also,
\begin{gather*}
  \E[ \Vi \Mid| \Ri=0 ]
  = \int_{\zX=-\infty}^{\zX=\infty} \zX \pdf{\Vi \Mid| \Ri}( \zX \Mid| 0 ) d\zX
  = 0 \pdf{\Vi \Mid| \Ri}( 0 \Mid| 0 )
  = 0
.\end{gather*}
With this, the \MGFs of the (centered) conditioned random variables \(  \Vi \Mid| \Ri  \) and \(  -\Vi \Mid| \Ri  \) are obtained for \(  \sX \in [0,\infty)  \) as follows:
\begin{align*}
  \mgf{\Vi \Mid| \Ri=1}( \sX )
  &=
  \E \left[ e^{\sX \Vi} \middle| \Ri=1 \right]
  \\
  &=
  \frac{1}{\sqrt{2\pi}}
  \int_{\zX=-\infty}^{\zX=\infty}
  e^{\sX \zX}
  e^{-\frac{1}{2} \zX^{2}}
  d\zX
  \\
  &=
  e^{\frac{1}{2} \sX^{2}}
  \frac{1}{\sqrt{2\pi}}
  \int_{\zX=-\infty}^{\zX=\infty}
  e^{-\frac{1}{2} (\zX-\sX)^{2}}
  d\zX
  \\ \TagEqn{\label{eqn:pf:eqn:lemma:concentration-ineq:noiseless:pr:2:1a}}
  &=
  e^{\frac{1}{2} \sX^{2}}
  ,\\
  &\dCmt{by evaluating the density of a mean-\(  \sX  \), variance-\(  1  \)}
  \\
  &\dCmtIndent \text{Gaussian random variable over its entire support}
\end{align*}
and by an analogous derivation,
\begin{align*}
  \mgf{-\Vi \Mid| \Ri=1}( \sX )
  &=
  e^{\frac{1}{2} \sX^{2}}
\TagEqn{\label{eqn:pf:eqn:lemma:concentration-ineq:noiseless:pr:2:1b}}
.\end{align*}
%
\XXX{and likewise,
\begin{align*}
  \mgf{-\Vi \Mid| \Ri=1}( \sX )
  &=
  \E \left[ e^{-\sX \Vi} \middle| \Ri=1 \right]
  \\
  &=
  \frac{1}{\sqrt{2\pi}}
  \int_{\zX=-\infty}^{\zX=\infty}
  e^{-\sX \zX}
  e^{-\frac{1}{2} \zX^{2}}
  d\zX
  \\
  &=
  e^{\frac{1}{2} \sX^{2}}
  \frac{1}{\sqrt{2\pi}}
  \int_{\zX=-\infty}^{\zX=\infty}
  e^{-\frac{1}{2} (\zX+\sX)^{2}}
  d\zX
  \\ \TagEqn{\label{eqn:pf:eqn:lemma:concentration-ineq:noiseless:pr:2:1b}}
  &=
  e^{\frac{1}{2} \sX^{2}}
  .\\
  &\dCmt{by evaluating the density of a mean-\(  ( -\sX )  \), variance-\(  1  \)}
  \\
  &\dCmtIndent \text{Gaussian random variable over its entire support}
\end{align*}
}
Additionally,
\begin{gather}
  \label{eqn:pf:eqn:lemma:concentration-ineq:noiseless:pr:2:2a}
  \mgf{\Vi \Mid| \Ri=0}( \sX )
  =
  \E \left[ e^{\sX \Vi} \middle| \Ri=0 \right]
  =
  \int_{\vX=-\infty}^{\vX=\infty} e^{\sX \vX} \pdf{\Vi \Mid| \Ri}( \vX \Mid| 0 ) d\vX
  =
  e^{\sX \cdot 0}
  =
  1
  ,\\ \label{eqn:pf:eqn:lemma:concentration-ineq:noiseless:pr:2:2b}
  \mgf{-\Vi \Mid| \Ri=0}( \sX )
  =
  \E \left[ e^{-\sX \Vi} \middle| \Ri=0 \right]
  =
  \int_{\vX=-\infty}^{\vX=\infty} e^{-\sX \vX} \pdf{\Vi \Mid| \Ri}( \vX \Mid| 0 ) d\vX
  =
  e^{-\sX \cdot 0}
  =
  1
.\end{gather}
It follows that the \MGF of \(  \Vi  \) is bounded from above by
\begin{align*}
  \mgf{\Vi}( \sX )
  &=
  \E \left[ e^{\sX \Vi} \right]
  \\
  &=
  \pdf{\Ri}(1)
  \E \left[ e^{\sX \Vi} \middle| \Ri=1 \right]
  +
  \pdf{\Ri}(0)
  \E \left[ e^{\sX \Vi} \middle| \Ri=0 \right]
  \\
  &\dCmt{by the law of total expectation}
  \\
  &=
  \pdf{\Ri}(1)
  \mgf{\Vi \Mid| \Ri=1}( \sX )
  +
  \pdf{\Ri}(0)
  \mgf{\Vi \Mid| \Ri=0}( \sX )
  \\
  &\dCmt{by the definitions of \(  \mgf{\Vi \Mid| \Ri=1}, \mgf{\Vi \Mid| \Ri=0}  \)}
\XXX{
  \\
  &=
  \frac{1}{\pi} \ADIST
  \mgf{\Vi \Mid| \Ri=1}( \sX )
  +
  \left( 1 - \frac{1}{\pi} \ADIST \right)
  \mgf{\Vi \Mid| \Ri=0}( \sX )
  \\
  &\dCmt{by \EQUATION \eqref{eqn:pf:eqn:lemma:concentration-ineq:noiseless:pr:1:f_Ri}}
}
  \\
  &\leq
  \frac{1}{\pi} \ADIST
  e^{\frac{1}{2} \sX^{2}}
  +
  \left( 1 - \frac{1}{\pi} \ADIST \right)
  \\
  &\dCmt{by \EQUATIONS \eqref{eqn:pf:eqn:lemma:concentration-ineq:noiseless:pr:1:f_Ri}, \eqref{eqn:pf:eqn:lemma:concentration-ineq:noiseless:pr:2:1a}, and \eqref{eqn:pf:eqn:lemma:concentration-ineq:noiseless:pr:2:2a}}
  \\
  &=
  1 + \frac{1}{\pi} \ADIST \left( e^{\frac{1}{2} \sX^{2}} - 1 \right)
\TagEqn{\label{eqn:pf:eqn:lemma:concentration-ineq:noiseless:pr:2:3a}}
,\end{align*}
and likewise, the \MGF of \(  -\Vi  \) is bounded by 
\begin{align*}
  \mgf{-\Vi}( \sX )
  &\leq
  1 + \frac{1}{\pi} \ADIST \left( e^{\frac{1}{2} \sX^{2}} - 1 \right)
\TagEqn{\label{eqn:pf:eqn:lemma:concentration-ineq:noiseless:pr:2:3b}}
.\end{align*}
%
\XXX{and
\begin{align*}
  \mgf{-\Vi}( \sX )
  &=
  \E \left[ e^{-\sX \Vi} \right]
  \\
  &=
  \pdf{\Ri}(1)
  \E \left[ e^{-\sX \Vi} \middle| \Ri=1 \right]
  +
  \pdf{\Ri}(0)
  \E \left[ e^{-\sX \Vi} \middle| \Ri=0 \right]
  \\
  &\dCmt{by the law of total expectation}
  \\
  &=
  \pdf{\Ri}(1)
  \mgf{-\Vi \Mid| \Ri=1}( \sX )
  +
  \pdf{\Ri}(0)
  \mgf{-\Vi \Mid| \Ri=0}( \sX )
  \\
  &\dCmt{by the definitions of \(  \mgf{-\Vi \Mid| \Ri=1}, \mgf{-\Vi \Mid| \Ri=0}  \)}
  \\
  &=
  \frac{1}{\pi} \ADIST
  \mgf{-\Vi \Mid| \Ri=1}( \sX )
  +
  \left( 1 - \frac{1}{\pi} \ADIST \right)
  \mgf{-\Vi \Mid| \Ri=0}( \sX )
  \\
  &\dCmt{by \EQUATION \eqref{eqn:pf:eqn:lemma:concentration-ineq:noiseless:pr:1:f_Ri}}
  \\
  &\leq
  \frac{1}{\pi} \ADIST
  e^{\frac{1}{2} \sX^{2}}
  +
  \left( 1 - \frac{1}{\pi} \ADIST \right)
  \\
  &\dCmt{by \EQUATIONS \eqref{eqn:pf:eqn:lemma:concentration-ineq:noiseless:pr:2:1b} and \eqref{eqn:pf:eqn:lemma:concentration-ineq:noiseless:pr:2:2b}}
  \\
  &=
  1 + \frac{1}{\pi} \ADIST \left( e^{\frac{1}{2} \sX^{2}} - 1 \right)
\TagEqn{\label{eqn:pf:eqn:lemma:concentration-ineq:noiseless:pr:2:3b}}
.\end{align*}
}
Let \(  \VRV \defeq \sum_{\iIx=1}^{\m} \Vi  \), where in expectation,
\begin{align*}
  \E[\VRV]
\XXX{
  &=
  \E \left[ \sum_{\iIx=1}^{\m} \Vi \right]
  \\
  &\dCmt{by the definition of \(  \VRV  \)}
  \\
}
  &=
  \sum_{\iIx=1}^{\m} \E[\Vi]
  \\
  &\dCmt{by by the definition of \(  \VRV  \) and the linearity of expectation}
  \\
  &=
  \sum_{\iIx=1}^{\m} \pdf{\Ri}(0) \E[ \Vi \Mid| \Ri=0 ] + \pdf{\Ri}(1) \E[ \Vi \Mid| \Ri=1 ]
  \\
  &\dCmt{by the law of total expectation}
  \\
  &=
  \sum_{\iIx=1}^{\m} \pdf{\Ri}(0) \cdot 0 + \pdf{\Ri}(1) \cdot 0
  \\
  &\dCmt{as argued earlier}
  \\
  &=
  0
.\end{align*}
Then, the \MGFs of the (centered) random variables \(  \VRV  \) and \(  -\VRV  \) are given and upper bounded at \(  \sX \in [0,\infty)  \) as follows:
\begin{align*}
  \mgf{\VRV}( \sX )
  &=
  \left( \mgf{\Vi}( \sX ) \right)^{\m}
  \\
  &\dCmt{for any \(  \iIx \in [\m]  \);}
  \\
  &\dCmt{since \(  \Vi[1], \dots, \Vi[\m]  \) are identically distributed}
  \\
  &\leq
  \left( 1 + \frac{1}{\pi} \ADIST \left( e^{\frac{1}{2} \sX^{2}} - 1 \right) \right)^{\m}
  \\
  &\dCmt{by \EQUATION \eqref{eqn:pf:eqn:lemma:concentration-ineq:noiseless:pr:2:3a}}
  \\ \TagEqn{\label{eqn:pf:eqn:lemma:concentration-ineq:noiseless:pr:2:4a}}
  &\leq
  e^{\frac{1}{\pi} \m \ADIST ( e^{\frac{1}{2} \sX^{2}} - 1 )}
  ,\\
  &\dCmt{by a well-known inequality, \(  \log( 1+u ) \leq u  \) for \(  u > -1  \)}
\end{align*}
and likewise,
\begin{align*}
  \mgf{-\VRV}( \sX )
\XXX{
  &=
  \left( \mgf{-\Vi}( \sX ) \right)^{\m}
  \\
  &\dCmt{for any \(  \iIx \in [\m]  \);}
  \\
  &\dCmt{since \(  -\Vi[1], \dots, -\Vi[\m]  \) are identically distributed}
  \\
  &\leq
  \left( 1 + \frac{1}{\pi} \ADIST \left( e^{\frac{1}{2} \sX^{2}} - 1 \right) \right)^{\m}
  \\
  &\dCmt{by \EQUATION \eqref{eqn:pf:eqn:lemma:concentration-ineq:noiseless:pr:2:3b}}
  \\
}
  &\leq
  e^{\frac{1}{\pi} \m \ADIST ( e^{\frac{1}{2} \sX^{2}} - 1 )}
  \TagEqn{\label{eqn:pf:eqn:lemma:concentration-ineq:noiseless:pr:2:4b}}
\XXX{  .\\
  &\dCmt{\(  {\textstyle \log ( 1 + \frac{1}{\pi} \ADIST ( e^{\frac{1}{2} \sX^{2}} - 1 ) ) \leq \frac{1}{\pi} \ADIST ( e^{\frac{1}{2} \sX^{2}} - 1 )}  \)}
  \\
  &\dCmtIndent\text{(by a well-known inequality)}
}
.\end{align*}
Now, observe:
\begin{align*}
  &
  \Pr \left(
    \frac{\VRV}{\m} - \E \left[ \frac{\VRV}{\m} \right] > \frac{1}{\pi} \tX \ADIST
  \middle|
    \Ri=1
  \right)
  \\
  &\AlignSp\leq
  \inf_{\sX \geq 0}
  e^{-\frac{1}{\pi} \m \sX \tX \ADIST}
  \mgf{\VRV}( \sX )
  \\
  &\AlignSp\dCmt{due to Bernstein (\see e.g., \cite{vershynin2018high})}
  \\
  &\AlignSp\leq
  \inf_{\sX \geq 0}
  e^{-\frac{1}{\pi} \m \sX \tX \ADIST}
  e^{\frac{1}{\pi} \m \ADIST ( e^{\frac{1}{2} \sX^{2}} - 1 )}
  \\
  &\AlignSp\dCmt{by \EQUATION \eqref{eqn:pf:eqn:lemma:concentration-ineq:noiseless:pr:2:4a}}
  \\ \TagEqn{\label{eqn:pf:eqn:lemma:concentration-ineq:noiseless:pr:2:5a}}
  &\AlignSp\leq
  e^{\frac{1}{3\pi} \m \tX^{2} \ADIST}
  ,\\
  &\AlignSp\dCmt{as argued earlier in the proof of \EQUATION \eqref{eqn:lemma:concentration-ineq:noiseless:pr:1}}
\end{align*}
and on the other side:
\begin{align*}
  &
  \Pr \left(
    \frac{\VRV}{\m} - \E \left[ \frac{\VRV}{\m} \right] < -\frac{1}{\pi} \tX \ADIST
  \middle|
    \Ri=1
  \right)
\XXX{
  \\
  &\AlignSp=
  \Pr \left(
    -\frac{\VRV}{\m} - \E \left[ -\frac{\VRV}{\m} \right] > \frac{1}{\pi} \tX \ADIST
  \middle|
    \Ri=1
  \right)
  \\
  &\AlignSp\leq
  \inf_{\sX \geq 0}
  e^{-\frac{1}{\pi} \m \sX \tX \ADIST}
  \mgf{-\VRV}( \sX )
  \\
  &\AlignSp\dCmt{due to Bernstein (\see e.g., \cite{vershynin2018high})}
  \\
  &\AlignSp\leq
  \inf_{\sX \geq 0}
  e^{-\frac{1}{\pi} \m \sX \tX \ADIST}
  e^{\frac{1}{\pi} \m \ADIST ( e^{\frac{1}{2} \sX^{2}} - 1 )}
  \\
  &\AlignSp\dCmt{by \EQUATION \eqref{eqn:pf:eqn:lemma:concentration-ineq:noiseless:pr:2:4b}}
  \\
  &\AlignSp\leq
}
  \leq
  e^{\frac{1}{3\pi} \m \tX^{2} \ADIST}
  \TagEqn{\label{eqn:pf:eqn:lemma:concentration-ineq:noiseless:pr:2:5b}}
\XXX{
  .\\
  &\AlignSp\dCmt{as argued earlier in the proof of \EQUATION \eqref{eqn:lemma:concentration-ineq:noiseless:pr:1}}
}
.\end{align*}
By a union bound over the above pair of inequalities in \EQUATIONS \eqref{eqn:pf:eqn:lemma:concentration-ineq:noiseless:pr:2:5a} and \eqref{eqn:pf:eqn:lemma:concentration-ineq:noiseless:pr:2:5b},
\begin{gather*}
  \Pr \left( \left| \frac{\VRV}{\m} - \E \left[ \frac{\VRV}{\m} \right] \right| > \frac{1}{\pi} \tX \ADIST \right)
  \leq
  2 e^{-\frac{1}{3\pi} \m \tX^{2} \ADIST}
,\end{gather*}
and therefore, recalling the definitions of the random variables \(  \Vi  \), \(  \iIx \in [\m]  \), and their relationship to \(  \hFn[\JCoords]( \thetaStar, \thetaX )  \), the above concentration inequality further implies that
\begin{gather*}
  \textstyle
  \Pr \left(
    \left| \left\langle \frac{\hFn[\JCoords]( \thetaStar, \thetaX )}{\sqrt{2\pi}}, \frac{\thetaStar+\thetaX}{\| \thetaStar+\thetaX \|_{2}} \right\rangle - \E \left[ \left\langle \frac{\hFn[\JCoords]( \thetaStar, \thetaX )}{\sqrt{2\pi}}, \frac{\thetaStar+\thetaX}{\| \thetaStar+\thetaX \|_{2}} \right\rangle \right] \right|
    >
    \frac{1}{\pi} \tX \ADIST
  \right)
  \nonumber \\
  \leq
  2 e^{-\frac{1}{3\pi} \m \tX^{2} \ADIST}
.\end{gather*}
Lastly, by union bounding over all \(  \JCoords \in \JS  \) and \(  \thetaX \in \ParamCoverX  \), \EQUATION \eqref{eqn:lemma:concentration-ineq:noiseless:pr:2} follows:
\begin{gather*}
  \textstyle
  \Pr \left(
    \ExistsST{\JCoords \in \JS, \thetaX \in \ParamCoverX}{
    \left| \left\langle \frac{\hFn[\JCoords]( \thetaStar, \thetaX )}{\sqrt{2\pi}}, \frac{\thetaStar+\thetaX}{\| \thetaStar+\thetaX \|_{2}} \right\rangle - \E \left[ \left\langle \frac{\hFn[\JCoords]( \thetaStar, \thetaX )}{\sqrt{2\pi}}, \frac{\thetaStar+\thetaX}{\| \thetaStar+\thetaX \|_{2}} \right\rangle \right] \right|
    >
    \frac{1}{\pi} \tX \ADIST
    }
  \right)
  \nonumber \\
  \leq
  2 | \JS | | \ParamCoverX | e^{-\frac{1}{3\pi} \m \tX^{2} \ADIST}
.\end{gather*}
%
\let\vX\oldvX%
%
\paragraph{Verification of \EQUATION \eqref{eqn:lemma:concentration-ineq:noiseless:pr:3}} 
%
Towards deriving the third concentration inequality, \EQUATION \eqref{eqn:lemma:concentration-ineq:noiseless:pr:3}, consider an orthonormal basis,
\(  \{ \Vec{\vV}\VIx{1}, \dots, \Vec{\vV}\VIx{\kX} \} \subset \R^{\n}  \),
for the subspace
\(  \Set{V} \defeq \{ \Vec{v} \in \R^{\n} : \Supp( \Vec{v} ) \subseteq \Supp( \thetaStar ) \cup \Supp( \thetaX ) \cup \JCoords \}  \),
where
\(  \kX \defeq | \Supp( \thetaStar ) \cup \Supp( \thetaX ) \cup \JCoords |  \),
and where
\(  \Vec{\vV}\VIx{\kX-1} \defeq \frac{\thetaStar-\thetaX}{\| \thetaStar-\thetaX \|_{2}}  \) and
\(  \Vec{\vV}\VIx{\kX}   \defeq \frac{\thetaStar+\thetaX}{\| \thetaStar+\thetaX \|_{2}}  \).
Then, the orthogonal decomposition of \(  \frac{1}{\sqrt{2\pi}} \gFn[\JCoords]( \thetaStar, \thetaX )  \) using this basis is given and subsequently rewritten as follows:
\begin{align*}
  \gFn[\JCoords]( \thetaStar, \thetaX )
  &=
  \sum_{\jIx=1}^{\kX}
  \left\langle \frac{\gFn[\JCoords]( \thetaStar, \thetaX )}{\sqrt{2\pi}} , \Vec{\vV}\VIx{\jIx} \right\rangle \Vec{\vV}\VIx{\jIx}
  \\
  &=
  \sum_{\jIx=1}^{\kX}
  \left\langle
    \frac{\hFn[\JCoords]( \thetaStar, \thetaX )}{\sqrt{2\pi}} 
    -
    \left\langle \frac{\hFn[\JCoords]( \thetaStar, \thetaX )}{\sqrt{2\pi}} , \Vec{\vV}\VIx{\kX-1} \right\rangle
    \Vec{\vV}\VIx{\kX-1}
    -
    \left\langle \frac{\hFn[\JCoords]( \thetaStar, \thetaX )}{\sqrt{2\pi}} , \Vec{\vV}\VIx{\kX} \right\rangle
    \Vec{\vV}\VIx{\kX}
    ,
    \Vec{\vV}\VIx{\jIx}
  \right\rangle
  \Vec{\vV}\VIx{\jIx}
  \\
  &\dCmt{by the choice of \(  \Vec{\vV}\VIx{\kX-1} = \tfrac{\thetaStar-\thetaX}{\| \thetaStar-\thetaX \|_{2}}, \Vec{\vV}\VIx{\kX}   = \tfrac{\thetaStar+\thetaX}{\| \thetaStar+\thetaX \|_{2}}  \)}
  \\
  &=
  \sum_{\jIx=1}^{\kX}
  \left\langle
    \frac{\hFn[\JCoords]( \thetaStar, \thetaX )}{\sqrt{2\pi}} ,
    \Vec{\vV}\VIx{\jIx}
  \right\rangle
  \Vec{\vV}\VIx{\jIx}
  -
  \left\langle \frac{\hFn[\JCoords]( \thetaStar, \thetaX )}{\sqrt{2\pi}} , \Vec{\vV}\VIx{\kX-1} \right\rangle
  \Vec{\vV}\VIx{\kX-1}
  -
  \left\langle \frac{\hFn[\JCoords]( \thetaStar, \thetaX )}{\sqrt{2\pi}} , \Vec{\vV}\VIx{\kX} \right\rangle
  \Vec{\vV}\VIx{\kX}
  \\
  &\dCmt{due to the orthogonality of \(  \Vec{\vV}\VIx{1}, \dots, \Vec{\vV}\VIx{\kX}  \)}
  \\
  &=
  \sum_{\jIx=1}^{\kX-2}
  \left\langle
    \frac{\hFn[\JCoords]( \thetaStar, \thetaX )}{\sqrt{2\pi}} ,
    \Vec{\vV}\VIx{\jIx}
  \right\rangle
  \Vec{\vV}\VIx{\jIx}
  \\
  &=
  \frac{1}{\m}
  \sum_{\iIx=1}^{\m}
  \sum_{\jIx=1}^{\kX-2}
  \langle \CovVX\VIx{\iIx}, \Vec{\vV}\VIx{\jIx} \rangle
  \Vec{\vV}\VIx{\jIx}
 \sep
  \Sign( \langle \CovVX\VIx{\iIx}, \thetaStar \rangle )
 \sep
  \I( \Sign( \langle \CovVX\VIx{\iIx}, \thetaStar \rangle ) \neq \Sign( \langle \CovVX\VIx{\iIx}, \thetaX \rangle ) )
  .\\
  &\dCmt{by \EQUATION \eqref{eqn:pf:eqn:lemma:concentration-ineq:noiseless:pr:5}}
\end{align*}
%
\par 
%
Note that
\(  \thetaStar, \thetaX \in \Span( \{ \Vec{\vV}\VIx{\kX-1}, \Vec{\vV}\VIx{\kX} \} )  \),
which implies by the orthogonality of the set
\(  \{ \Vec{\vV}\VIx{1}, \dots, \Vec{\vV}\VIx{\kX} \}  \)
that
\(   \thetaStar, \thetaX \perp \Vec{\vV}\VIx{\jIx}  \)
for every \(  \jIx \in [\kX-2]  \).
Thus, applying standard facts about Gaussians, for each \(  \iIx \in [\m]  \), \(  \jIx \in [\kX-2]  \), there is an equivalence in distribution:
\begin{gather*}
  \langle \CovVX\VIx{\iIx}, \Vec{\vV}\VIx{\jIx} \rangle
  \sep
  \Sign( \langle \CovVX\VIx{\iIx}, \thetaStar \rangle )
  \sep
  \I( \Sign( \langle \CovVX\VIx{\iIx}, \thetaStar \rangle ) \neq \Sign( \langle \CovVX\VIx{\iIx}, \thetaX \rangle ) )
  \sim
  \Wij
  \defeq
  \Zij \Yi
,\end{gather*}
where for each \(  \iIx \in [\m]  \), \(  \jIx \in [\kX]  \), the random variable
\(  \Zij \sim \N(0,1)  \)
is standard Gaussian and
\begin{gather*}
  \Yi \defeq \Sign( \Zij[\iIx][\kX-1] ) \sep \I( \Sign( \Zij[\iIx][\kX-1] ) \neq \Sign( \Zij[\iIx][\kX] ) )
.\end{gather*}
Notice that the random variables
\(  \{ \Zij \}_{\iIx \in [\m], \jIx \in [\kX-2]}  \) 
are \iid and also independent of
\(  \Zij[\iIx][\kX-1]  \), \(  \Zij[\iIx][\kX]  \), and \(  \Yi  \), \(  \iIx \in [\m]  \).
Moreover,
\begin{gather*}
  \frac{1}{\m}
  \sum_{\iIx=1}^{\m}
  \sum_{\jIx=1}^{\kX-2}
  \langle \CovVX\VIx{\iIx}, \Vec{\vV}\VIx{\jIx} \rangle
  \Vec{\vV}\VIx{\jIx}
  \sep
  \Sign( \langle \CovVX\VIx{\iIx}, \thetaStar \rangle )
  \sep
  \I( \Sign( \langle \CovVX\VIx{\iIx}, \thetaStar \rangle ) \neq \Sign( \langle \CovVX\VIx{\iIx}, \thetaX \rangle ) )
  \sim
  \frac{1}{\m}
  \sum_{\iIx=1}^{\m}
  \sum_{\jIx=1}^{\kX-2}
  \Wij
  \Vec{\vV}\VIx{\jIx}
.\end{gather*}
Due to the rotational invariance of Gaussians,
\begin{align*}
  &
  \langle \CovVX\VIx{\iIx}, \Vec{\vV}\VIx{\jIx} \rangle
  \sep
  \Sign( \langle \CovVX\VIx{\iIx}, \thetaStar \rangle )
  \sep
  \I( \Sign( \langle \CovVX\VIx{\iIx}, \thetaStar \rangle ) \neq \Sign( \langle \CovVX\VIx{\iIx}, \thetaX \rangle ) )
  \\
  &\AlignIndent
  \sim
  \langle \CovVX\VIx{\iIx}, \ej \rangle
  \sep
  \Sign( \langle \CovVX\VIx{\iIx}, \thetaStar \rangle )
  \sep
  \I( \Sign( \langle \CovVX\VIx{\iIx}, \thetaStar \rangle ) \neq \Sign( \langle \CovVX\VIx{\iIx}, \thetaX \rangle ) )
,\end{align*}
where
\(  \ej \defeq \BVec{\{ \jIx \}} \in \R^{\n}  \)
is the \(  \jIx\Th  \) standard basis vector for \(  \R^{\n}  \) in which the \(  \jIx\Th  \) entry is set to \(  1  \) and all other entries are set to \(  0  \).
Hence, \WLOG, the analysis will proceed under the assumption that the first \(  ( \kX-2 )  \)-many \(  \jIx\Th  \) basis vectors are \(  \Vec{\vV}\VIx{\jIx} = \ej  \), \(  \jIx \in [\kX-2]  \).
Under this assumption, the random vector, \(  \Vec{\URV}  \), which is given by
\begin{gather*}
  \Vec{\URV}
  \defeq
  \frac{1}{\m}
  \sum_{\iIx=1}^{\m}
  \sum_{\jIx=1}^{\kX-2}
  \Wij \Vec{\vV}\VIx{\jIx}
  =
  \frac{1}{\m}
  \sum_{\iIx=1}^{\m}
  \sum_{\jIx=1}^{\kX-2}
  \Wij \ej
,\end{gather*}
has \(  \jIx\Th  \) entries, \(  \jIx \in [\n]  \),
\begin{gather*}
  \Vec*{\URV}\VIx{\jIx}
  =
  \begin{cases}
  0                                    ,& \cIf \jIx \in [\n] \setminus [\kX-2] ,\\
  \frac{1}{\m} \sum_{\iIx=1}^{\m} \Wij ,& \cIf \jIx \in [\kX-2]                .
  \end{cases}
\end{gather*}
%
\par 
%
For \(  \iIx \in [\m]  \), define the random variable
\(  \Ri \defeq \I( \Sign( \Zij[\iIx][\kX-1] ) \neq \Sign( \Zij[\iIx][\kX] ) )  \),
whose a mass function given at \(  \rX \in \{ 0,1 \}  \) by
\begin{gather*}
  \pdf{\Ri}( \rX )
  =
  \begin{cases}
  1-\frac{1}{\pi} \ADIST ,& \cIf \rX=0, \\
  \frac{1}{\pi} \ADIST   ,& \cIf \rX=1.
  \end{cases}
\end{gather*}
As in the verification of \EQUATION \eqref{eqn:lemma:concentration-ineq:noiseless:pr:1}, this mass function can be derived by way of an approach similar to that which appears in \cite[{\APPENDIX B.1.1}]{matsumoto2022binary}.
Additionally, write the random vector
\(  \Vec{\RRV} \defeq ( \Ri[1], \dots, \Ri[\m] )  \),
whose entries are \iid[,] and let
\(  \LRV \defeq \| \Vec{\RRV} \|_{0}  \).
Because each random variable \(  \Zij  \), \(  \jIx \in [\kX-2]  \), is independent of \(  \Zij[\iIx][\kX-1]  \) and \(  \Zij[\iIx][\kX]  \), it is also independent of \(  \Sign( \Zij[\iIx][\kX-1] )  \) and \(  \Ri  \),
where \(  \Sign( \Zij[\iIx][\kX-1] )  \) follows a Rademacher distribution.
Since mean-\(  0  \) Gaussian random variables have the same distribution as their negations, there are the following equivalences in distribution: \(  -\Zij \sim \Zij \sim \N(0,1)  \) and \(  \Zij \sep \Sign( \Zij[\iIx][\kX-1] ) \sim \Zij \sim \N(0,1)  \) (\see e.g., \cite[{\APPENDIX B}]{matsumoto2022binary} for a formal argument).
Hence,
\(  ( \Wij \Mid| \Ri=1 ) \sim \N(0,1)  \).
Since the random variables \(  \Wij[1], \dots, \Wij[\m]  \) are \iid[,] it follows that
\(  ( \Vec*{\URV}\VIx{\jIx} \Mid| \LRV=\lX ) \sim ( \Vec*{\URV}\VIx{\jIx} \Mid| \Vec{\RRV}=\Vec{\rX} ) \sim \N( 0, \frac{\lX}{\m^{2}} )  \)
for each \(  \jIx \in [\kX-2]  \) and an arbitrary choice of \(  \Vec{\rX} \in \{ 0,1 \}^{\m}  \), and where \(  \lX \defeq \| \Vec{\rX} \|_{0}  \).
(A more rigorous analysis can employ the law of total probability.)
Therefore, \(  \Vec{\URV}  \) is a \(  \frac{\sqrt{\lX}}{\m}  \)-\subgaussian random vector with support of cardinality \(  \| \Vec{\URV} \|_{0} = \kX-2  \).
%
\par 
%
Before proceeding, two results are introduced to facilitate the proof.
%
\begin{lemma}[{\LEMMA \cite[\LEMMA A.2]{matsumoto2022binary}}]
\label{lemma:count-mismatch}
Fix
\(  \sXX \in (0,1)  \).
Let
\(  \Vec{\ZRVX}\VIx{1}, \dots, \Vec{\ZRVX}\VIx{\m} \sim \N( \Vec{0}, \Id{\n} )  \),
and let
\(  \Vec{\uV}, \Vec{\vV} \in \Sphere{\n}  \).
Define the random variable
\(  \LRV \defeq | \{ \iIx \in [\m] : \Sign( \langle \Vec{\ZRVX}\VIx{\iIx}, \Vec{\uV} \rangle ) \neq \Sign( \langle \Vec{\ZRVX}\VIx{\iIx}, \Vec{\vV} \rangle ) \} |  \).
Then,
\begin{gather}
\label{eqn:lemma:count-mismatch:ev}
  \mu_{\LRV} \defeq \E[ \LRV ] = \frac{\m \arccos( \langle \Vec{\uV}, \Vec{\vV} \rangle )}{\pi}
\end{gather}
and
\begin{gather}
\label{eqn:lemma:count-mismatch:pr}
  \Pr \left( \LRV > ( 1+\sXX ) \mu_{\LRV} \right) \leq e^{-\frac{1}{3\pi} \m \sXX^{2} \arccos( \langle \Vec{\uV}, \Vec{\vV} \rangle )}
.\end{gather}
\end{lemma}

\begin{lemma}
\label{lemma:norm-subgaussian:1}
Fix \(  \tXXX, \sigma  > 0  \) and \(  0 < \dX \leq \n  \).
Let
\(  \ICoords \subseteq [\n]  \), \(  | \ICoords | = \dX  \),
and
\(  \Vec{\XRV} \sim \N( \Vec{0}, \sigma^{2} \sum_{\jIx \in \ICoords} \ej \ej^{\T} )  \).
Then,
\begin{gather}
\label{eqn:lemma:norm-subgaussian:1:1}
  \Pr \left(
    \| \Vec{\XRV} - \E[ \Vec{\XRV} ] \|_{2}
    >
    \sqrt{\dX} \sigma
    +
    \tXXX
  \right)
  \leq
  \Pr \left(
    \| \Vec{\XRV} - \E[ \Vec{\XRV} ] \|_{2}
    >
    \E[ \| \Vec{\XRV} \|_{2} ]
    +
    \tXXX
  \right)
  \leq
  e^{-\frac{1}{2 \sigma^{2}} \tXXX^{2}}
.\end{gather}
\end{lemma}
%
\begin{subproof}
{\LEMMA \ref{lemma:norm-subgaussian:1}}
Note that
\(  \| \Vec{\XRV} - \E[ \Vec{\XRV} ] \|_{2} = \| \Vec{\XRV} \|_{2}  \)
due to the lemma's condition that \(  \Vec{\XRV}  \) is zero-mean.
By standard properties of Gaussians,
the expected \(  \lnorm{2}  \)-norm of \(  \Vec{\XRV}  \) is bound from above by
\(  \E[ \| \Vec{\XRV} \|_{2} ] \leq \sqrt{\dX} \sigma  \).
Due to a well-known concentration inequality for Lipschitz functions of \subgaussian random vectors (\seeeg \cite{wainwright2019high}), and noting that the \(  \lnorm{2}  \)-norm is \(  1  \)-Lipschitz, the claimed inequality holds:
\begin{align*}
  \Pr \left( \| \Vec{\XRV} - \E[ \Vec{\XRV} ] \|_{2} > \sqrt{\dX} \sigma + \tXXX \right)
  &=
  \Pr( \| \Vec{\XRV} \|_{2} > \sqrt{\dX} \sigma + \tXXX )
  \\
  &\leq
  \Pr( \| \Vec{\XRV} \|_{2} > \E[ \| \Vec{\XRV} \|_{2} ] + \tXXX )
  \\
  &\leq
  e^{-\frac{1}{2 \sigma^{2}} \tXXX^{2}}
,\end{align*}
as desired.
\end{subproof}
%
Fixing \(  \sXX \in (0,1)  \), the random variable \(  \LRV  \) exceeds
\(  \LRV > ( 1+\sXX ) \frac{1}{\pi} \m \ADIST  \)
with probability at most
\(  e^{-\frac{1}{3\pi} \m \sXX^{2} \ADIST}  \)
by \LEMMA \ref{lemma:count-mismatch}.
Additionally, by an earlier observation,
\(  \E[ \Vec{\URV} \Mid| \LRV=\lX ] = \Vec{0}  \),
and thus, due to \LEMMA \ref{lemma:norm-subgaussian:1}, for \(  \tXXX > 0  \),
\begin{align}
  \nonumber
  \Pr \left( \| \Vec{\URV} - \E[ \Vec{\URV} ] \|_{2} > \frac{\sqrt{( \kX-2 ) \lX}}{\m} + \tXXX \middle| \LRV \leq \lX \right)
  &\leq
  \Pr \left( \| \Vec{\URV} - \E[ \Vec{\URV} ] \|_{2} > \frac{\sqrt{( \kX-2 ) \lX}}{\m} + \tXXX \middle| \LRV = \lX \right)
  \\ \label{eqn:pf:eqn:lemma:concentration-ineq:noiseless:pr:3:cond-on-L}
  &\leq
  e^{-\frac{\m^{2} \tXXX^{2}}{2 \lX}}
,\end{align}
where in particular, taking
\(  \lX = ( 1+\sXX ) \frac{1}{\pi} \m \ADIST  \)
and
\(  \tXXX = \frac{1}{\pi} \tX \ADIST  \)
in \EQUATION \eqref{eqn:pf:eqn:lemma:concentration-ineq:noiseless:pr:3:cond-on-L} and noting that
\begin{align}
\label{eqn:pf:eqn:lemma:concentration-ineq:noiseless:pr:3:1}
  \kX
  &= | \Supp( \thetaStar ) \cup \Supp( \thetaX ) \cup \JCoords | \leq \min \{ | \Supp( \thetaStar ) | + | \Supp( \thetaX ) | + | \JCoords |, \n \}
  \nonumber\\
  &\leq \min \{ 2\k + \max_{\JCoords' \in \JS} | \JCoords' |, \n \} = \kO
\end{align}
for any \(  \JCoords \in \JS  \),
the following holds:
\begin{align*}
  &
  \textstyle
  \Pr \Bigl(
    \| \Vec{\URV} - \E[ \Vec{\URV} ] \|_{2}
    >
    \sqrt{\frac{1}{\pi \m} ( 1+\sXX )( \kO-2 ) \ADIST}
    +
    \frac{1}{\pi} \tX \ADIST
    \\
    &\AlignIndent\AlignIndent \textstyle
    \Bigl| \LRV \leq ( 1+\sXX ) \frac{1}{\pi} \m \ADIST
  \Bigr)
  \\
  & \textstyle \AlignIndent \leq
  \Pr \Bigl(
    \| \Vec{\URV} - \E[ \Vec{\URV} ] \|_{2}
    >
    \sqrt{\frac{1}{\pi \m} ( 1+\sXX )( \kX-2 ) \ADIST}
    +
    \frac{1}{\pi} \tX \ADIST
    \\ & \textstyle \AlignIndent\AlignIndent\AlignIndent\AlignIndent
    \Bigl|
    \LRV \leq ( 1+\sXX ) \frac{1}{\pi} \m \ADIST
  \Bigr)
  \\
  & \AlignIndent \dCmt{by \EQUATION \eqref{eqn:pf:eqn:lemma:concentration-ineq:noiseless:pr:3:1}, \(  \kX \leq \kO  \)}
  \\
  & \AlignIndent \leq
  e^{-\frac{1}{2\pi ( 1+\sXX )} \m \tX^{2} \ADIST}
  .\\
  &\AlignIndent\dCmt{by \EQUATION \eqref{eqn:pf:eqn:lemma:concentration-ineq:noiseless:pr:3:cond-on-L}}
  \\ \TagEqn{\label{eqn:pf:eqn:lemma:concentration-ineq:noiseless:pr:3:2}}
\end{align*}
Combining the above arguments obtains:
\begin{align*}
  &
  \textstyle
  \Pr \Bigl(
    \| \Vec{\URV} - \E[ \Vec{\URV} ] \|_{2}
    >
    \sqrt{\frac{1}{\pi \m} ( 1+\sXX )( \kO-2 ) \ADIST}
    +
    \frac{1}{\pi} \tX \ADIST
  \Bigr)
  \\
  & \textstyle
  \AlignIndent \leq
  \Pr \Bigl( \| \Vec{\URV} - \E[ \Vec{\URV} ] \|_{2} > \sqrt{\frac{1}{\pi \m} ( 1+\sXX )( \kO-2 ) \ADIST} + \frac{1}{\pi} \tX \ADIST
  \\
  & \AlignIndent\AlignIndent\AlignIndent\AlignIndent \textstyle
  \Bigl| \LRV \leq ( 1+\sXX ) \frac{1}{\pi} \m \ADIST \Bigr)
  \\
  &\AlignIndent\AlignIndent+
  \Pr \left( \LRV > ( 1+\sXX ) \frac{1}{\pi} \m \ADIST \right)
  \\
  & \AlignIndent \leq
  e^{-\frac{1}{2\pi ( 1+\sXX )} \m \tX^{2} \ADIST}
  +
  e^{-\frac{1}{3\pi} \m \sXX^{2} \ADIST}
  .\\
  & \AlignIndent \dCmt{by \EQUATION \eqref{eqn:pf:eqn:lemma:concentration-ineq:noiseless:pr:3:2} and \LEMMA
  \ref{lemma:count-mismatch}}
\end{align*}
Recalling the equivalences in distribution described earlier in the proofs, it directly follows that
\begin{gather*}
  \textstyle
  \Pr \left( \left\| \frac{\gFn[\JCoords]( \thetaStar, \thetaX )}{\sqrt{2\pi}} - \E \left[ \frac{\gFn[\JCoords]( \thetaStar, \thetaX )}{\sqrt{2\pi}} \right] \right\|_{2} > \sqrt{\frac{1}{\pi \m} ( 1+\sXX )( \kO-2 ) \ADIST} + \frac{1}{\pi} \tX \ADIST \right)
  \\
  \leq
  e^{-\frac{1}{2\pi ( 1+\sXX )} \m \tX^{2} \ADIST}
  +
  e^{-\frac{1}{3\pi} \m \sXX^{2} \ADIST}
\end{gather*}
for any single choice of
\(  \JCoords \in \JS  \) and \(  \thetaX \in \ParamCoverX  \).
Then, union bounds over \(  \JS  \) and \(  \ParamCoverX  \) yields the concentration inequality in \EQUATION \eqref{eqn:lemma:concentration-ineq:noiseless:pr:3}:
\begin{gather*}
  \textstyle
  \Pr \left(
    \ExistsST{\JCoords \in \JS, \thetaX \in \ParamCoverX}{
    \left\| \frac{\gFn[\JCoords]( \thetaStar, \thetaX )}{\sqrt{2\pi}} - \E \left[ \frac{\gFn[\JCoords]( \thetaStar, \thetaX )}{\sqrt{2\pi}} \right] \right\|_{2} > \sqrt{\frac{( 1+\sXX )( \kO-2 ) \ADIST}{\pi \m} } + \frac{\tX \ADIST}{\pi}
    }
  \right)
  \\
  \leq
  | \JS | | \ParamCoverX | e^{-\frac{1}{2\pi ( 1+\sXX )} \m \tX^{2} \ADIST}
  +
  | \ParamCoverX | e^{-\frac{1}{3\pi} \m \sXX^{2} \ADIST}
,\end{gather*}
as claimed.
\end{proof}



\subsection{Proof of \LEMMA \ref{lemma:concentration-ineq:noisy}}
\label{outline:concentration-ineq|pf-noisy}

\begin{proof}
{\LEMMA \ref{lemma:concentration-ineq:noisy}}
\checkoff%
The proof of the lemma is split across a few subsections within this section, \SECTION \ref{outline:concentration-ineq|pf-noisy}:
\SECTION \ref{outline:concentration-ineq|pf-noisy|1} is devoted to \EQUATIONS \eqref{eqn:lemma:concentration-ineq:noisy:ev:1} and \eqref{eqn:lemma:concentration-ineq:noisy:pr:1}, while \SECTION \ref{outline:concentration-ineq|pf-noisy|2} derives \EQUATIONS \eqref{eqn:lemma:concentration-ineq:noisy:ev:2} and \eqref{eqn:lemma:concentration-ineq:noisy:pr:2}.
Lastly, \SECTION \ref{outline:concentration-ineq|pf-noisy|pf-f1,f2} proves an intermediate result.


\subsubsection{Proof of \EQUATIONS \eqref{eqn:lemma:concentration-ineq:noisy:pr:1} and \eqref{eqn:lemma:concentration-ineq:noisy:ev:1}}
\label{outline:concentration-ineq|pf-noisy|1}

%
Fix
\(  \thetaStar \in \ParamSpace  \) and \(  \JCoordsX \in \JSX  \)
arbitrarily.
Write
\(  \CovVX\VIx{\iIx} \defeq \ThresholdSet{\Supp( \thetaStar ) \cup \JCoordsX}( \CovV\VIx{\iIx} )  \),
\(  \iIx \in [\m]  \).
As similarly seen earlier, \(  \frac{1}{\sqrt{2\pi}} \hfFn[\JCoordsX]( \thetaStar, \thetaStar )  \) can be written as follows:
\begin{align}
  \frac{1}{\sqrt{2\pi}} \hfFn[\JCoordsX]( \thetaStar, \thetaStar )
  &=
  \ThresholdSet{\Supp( \thetaStar ) \cup \JCoordsX} \left(
    \frac{1}{\m}
    \sum_{\iIx=1}^{\m}
    \CovV\VIx{\iIx}
    \sep \frac{1}{2} \left( \fFn( \langle \CovV, \thetaStar \rangle ) - \Sign( \langle \CovV, \thetaStar \rangle ) \right)
  \right)
  \\
  &\dCmt{by the definition of \(  \hfFn[\JCoords]  \) in \EQUATION \eqref{eqn:notations:hfJ:def}}
  \\
  &=
  \frac{1}{\m}
  \sum_{\iIx=1}^{\m}
  \ThresholdSet{\Supp( \thetaStar ) \cup \JCoordsX}( \CovV\VIx{\iIx} )
  \sep \frac{1}{2} \left( \fFn( \langle \CovV, \thetaStar \rangle ) - \Sign( \langle \CovV, \thetaStar \rangle ) \right)
  \\
  &\dCmt{by the linearity of the subset thresholding operation (\see \SECTIONREF \ref{outline:notations})}
  \\
  &=
  \frac{1}{\m}
  \sum_{\iIx=1}^{\m}
  \CovVX\VIx{\iIx}
  \sep \frac{1}{2} \left( \fFn( \langle \CovVX\VIx{\iIx}, \thetaStar \rangle ) - \Sign( \langle \CovVX\VIx{\iIx}, \thetaStar \rangle ) \right)
  \\
  &\dCmt{by the definition of \(  \CovVX\VIx{\iIx}  \), \(  \iIx \in [\m]  \)}
  \\
  &=
  -\frac{1}{\m}
  \sum_{\iIx=1}^{\m}
  \CovVX\VIx{\iIx}
  \sep \Sign( \langle \CovVX\VIx{\iIx}, \thetaStar \rangle )
  \sep \I( \fFn( \langle \CovVX\VIx{\iIx}, \thetaStar \rangle ) \neq \Sign( \langle \CovVX\VIx{\iIx}, \thetaStar \rangle ) )
\TagEqn{\label{eqn:pf:lemma:concentration-ineq:noisy:16}}
,\end{align}
and thus,
\begin{align*}
  \left\langle \frac{1}{\sqrt{2\pi}} \hfFn[\JCoordsX]( \thetaStar, \thetaStar ), \thetaStar \right\rangle
  &=
  -\frac{1}{\m}
  \sum_{\iIx=1}^{\m}
  \langle \CovVX\VIx{\iIx}, \thetaStar \rangle
  \sep
  \Sign( \langle \CovVX\VIx{\iIx}, \thetaStar \rangle )
  \sep
  \I( \fFn( \langle \CovVX\VIx{\iIx}, \thetaStar \rangle ) \neq \Sign( \langle \CovVX\VIx{\iIx}, \thetaStar \rangle ) )
  \\
  &=
  -\frac{1}{\m}
  \sum_{\iIx=1}^{\m}
  | \langle \CovVX\VIx{\iIx}, \thetaStar \rangle |
  \sep
  \I( \fFn( \langle \CovVX\VIx{\iIx}, \thetaStar \rangle ) \neq \Sign( \langle \CovVX\VIx{\iIx}, \thetaStar \rangle ) )
\TagEqn{\label{eqn:pf:lemma:concentration-ineq:noisy:1}}
.\end{align*}
Note that justifications for some of the steps taken above can be obtained by extending those appearing in the proof of \LEMMA \ref{lemma:concentration-ineq:noiseless}.
%
\par 
%
The first step towards deriving \EQUATIONS \eqref{eqn:lemma:concentration-ineq:noisy:pr:1} and \eqref{eqn:lemma:concentration-ineq:noisy:ev:1} is characterizing the distribution of each \(  \iIx\Th  \) summand, \(  \iIx \in [\m]  \), in \EQUATION \eqref{eqn:pf:lemma:concentration-ineq:noisy:1}.
Let
\(  \Zi \sim \N(0,1)  \) and \(  \Ri \defeq \I( \fFn( \Zi ) \neq \Sign( \Zi ) )  \),
\(  \iIx \in [\m]  \).
Then, each \(  \iIx\Th  \) summand, \(  \iIx \in [\m]  \), follows the same distribution as
\begin{align*}
  | \langle \CovVX\VIx{\iIx}, \thetaStar \rangle |
  \sep \I( \fFn( \langle \CovVX\VIx{\iIx}, \thetaStar \rangle ) \neq \Sign( \langle \CovVX\VIx{\iIx}, \thetaStar \rangle ) )
  \sim
  \Wi
  \defeq
  | \Zi | \sep \I( \fFn( \Zi ) \neq \Sign( \Zi ) )
  =
  | \Zi | \Ri
.\end{align*}
The density and mass functions of \(  | \Zi |  \) and \(  \Ri  \), respectively, are given by
\begin{gather}
\label{eqn:pf:lemma:concentration-ineq:noisy:4}
  \pdf{| \Zi |}( \zX )
  =
  \begin{cases}
  \sqrt{\frac{2}{\pi}} e^{-\frac{1}{2} \zX^{2}} ,& \cIf \zX \geq 0, \\
  0                                             ,& \cIf \zX = 0,
  \end{cases}
  \\
\label{eqn:pf:lemma:concentration-ineq:noisy:5}
  \pdf{\Ri}( \rX )
  =
  \begin{cases}
  1-\alphaX ,& \cIf \rX = 0, \\
  \alphaX   ,& \cIf \rX = 1.
  \end{cases}
\end{gather}
Additionally, the mass function of the conditioned random variable \(  \Ri=1 \Mid| \Zi  \) is given by
\begin{gather}
\label{eqn:pf:lemma:concentration-ineq:noisy:2}
  \pdf{\Ri \Mid| \Zi}( 1 \Mid| \zX )
  =
  \begin{cases}
  \pFn( \zX )   ,& \cIf \zX < 0,   \\
  1-\pFn( \zX ) ,& \cIf \zX \geq 0,
  \end{cases}
\end{gather}
and the mass function of the conditioned random variable \(  \Ri=1 \Mid| |\Zi|   \) is given by
\begin{gather}
\label{eqn:pf:lemma:concentration-ineq:noisy:8}
  \pdf{\Ri \Mid| |\Zi|}( 1 \Mid| \zX )
  =
  \begin{cases}
  0                         ,& \cIf \zX < 0,   \\
  \frac{1}{2} (\pExpr{\zX}) ,& \cIf \zX \geq 0,
  \end{cases}
\end{gather}
where the latter case---when \(  \zX \geq 0  \)---is obtained as follows: 
\begin{align*}
  \pdf{\Ri \Mid| |\Zi|}( 1 \Mid| \zX )
  &=
  \pdf{\Ri \Mid| |\Zi|, \Zi}( 1 \Mid| \zX, \zX )
  \pdf{\Zi \Mid| |\Zi|}( \zX | \zX )
  +
  \pdf{\Ri \Mid| |\Zi|, \Zi}( 1 \Mid| \zX, -\zX )
  \pdf{\Zi \Mid| |\Zi|}( \zX | -\zX )
  \\
  &\dCmt{by the law of total probability and the definition of conditional probabilities}
  \\
  &\dCmtx{(and the observation that \(  \pdf{\Zi \Mid| |\Zi|}( \zX' | \zX ) = 0  \) whenever \(  | \zX' | \neq \zX  \), \(  \zX' \in \R, \zX \geq 0  \))}
  \\
  &=
  \frac{1}{2}
  \pdf{\Ri \Mid| |\Zi|, \Zi}( 1 \Mid| \zX, \zX )
  +
  \frac{1}{2}
  \pdf{\Ri \Mid| |\Zi|, \Zi}( 1 \Mid| \zX, -\zX )
  \\
  &\dCmt{by symmetry}
  \\
  &=
  \frac{1}{2}
  \pdf{\Ri \Mid| \Zi}( 1 \Mid| \zX )
  +
  \frac{1}{2}
  \pdf{\Ri \Mid| \Zi}( 1 \Mid| -\zX )
  \\
  &\dCmt{because \(  \Zi  \) completely determines \(  |\Zi|  \), which implies \(  ( \Ri \Mid| |\Zi|, \Zi ) \sim ( \Ri \Mid| \Zi )  \)}
  \\
  &=
  \frac{1}{2}
  ( 1-\pFn( \zX ) )
  +
  \frac{1}{2}
  \pFn( -\zX )
  \\
  &\dCmt{by \EQUATION \eqref{eqn:pf:lemma:concentration-ineq:noisy:2}}
  \\
  &=
  \frac{1}{2} (\pExpr{\zX})
.\end{align*}
Note that
\(  ( \Wi \Mid| \Ri=1 ) \sim ( |\Zi|\Ri \Mid| \Ri=1 ) \sim ( |\Zi| \Mid| \Ri=1 )  \).
Thus, via Bayes' theorem, for \(  \zX \in \R  \),
\begin{align*}
  \pdf{\Wi \Mid| \Ri}( \zX \Mid| 1 )
  &=
  \pdf{|\Zi| \Mid| \Ri}( \zX \Mid| 1 )
  \\
  &\dCmt{by the above remark}
  \\
  &=
  \frac{\pdf{|\Zi|}( \zX ) \pdf{\Ri \Mid| |\Zi|}( 1 \Mid| \zX )}{\pdf{\Ri}(1)}
  \\
  &\dCmt{by Bayes' theorem}
  \\
  &=
  \frac{\sqrt{\frac{2}{\pi}} e^{-\frac{1}{2} \zX^{2}} \frac{1}{2} (\pExpr{\zX})}{\alphaX}
  \\
  &\dCmt{by \EQUATIONS \eqref{eqn:pf:lemma:concentration-ineq:noisy:4}, \eqref{eqn:pf:lemma:concentration-ineq:noisy:5}, and \eqref{eqn:pf:lemma:concentration-ineq:noisy:8}}
  \\
  &=
  \frac{1}{\sqrt{2\pi} \alphaX} e^{-\frac{1}{2} \zX^{2}} (\pExpr{\zX})
,\end{align*}
and therefore, taking together the above work, the density of the conditioned random variable \(  \Wi \Mid| \Ri  \) is given for \(  \rX \in \{ 0,1 \}  \) and \(  \zX \in \R  \) by
\begin{gather}
\label{eqn:pf:lemma:concentration-ineq:noisy:3}
  \pdf{\Wi \Mid| \Ri}( \zX \Mid| \rX )
  =
  \begin{cases}
  0                                                                    ,& \cIf \rX=0, \zX \neq 0, \\
  1                                                                    ,& \cIf \rX=0, \zX = 0, \\
  0                                                                    ,& \cIf \rX=1, \zX < 0, \\
  \frac{1}{\sqrt{2\pi} \alphaX} e^{-\frac{1}{2} \zX^{2}} (\pExpr{\zX}) ,& \cIf \rX=1, \zX \geq 0.
  \end{cases}
\end{gather}
In expectation, when conditioning on \(  \Ri=0  \),
\begin{align*}
  \E[ \Wi \Mid| \Ri=0 ]
  &=
  \int_{\zX=-\infty}^{\zX=\infty} \zX \pdf{\Wi \Mid| \Ri}( \zX \Mid| 0 ) d\zX
  \\
  &=
  0 \pdf{\Wi \Mid| \Ri}( 0 \Mid| 0 )
  \\
  &\dCmt{due to \EQUATION \eqref{eqn:pf:lemma:concentration-ineq:noisy:3}}
  \\
  &=
  0
\TagEqn{\label{eqn:pf:lemma:concentration-ineq:noisy:6}}
,\end{align*}
and when conditioning on \(  \Ri=1  \),
\begin{align*}
  \E[ \Wi \Mid| \Ri=1 ]
  &=
  \int_{\zX=-\infty}^{\zX=\infty} \zX \pdf{\Wi \Mid| \Ri}( \zX \Mid| 1 ) d\zX
  \\
  &=
  \int_{\zX=0}^{\zX=\infty}
  \frac{1}{\sqrt{2\pi} \alphaX} \zX e^{-\frac{1}{2} \zX^{2}} (\pExpr{\zX})
  d\zX
  \\
  &\dCmt{by \EQUATION \eqref{eqn:pf:lemma:concentration-ineq:noisy:3}}
  \\
  &=
  \int_{\zX=0}^{\zX=\infty}
  \frac{1}{\sqrt{2\pi} \alphaX} \zX e^{-\frac{1}{2} \zX^{2}}
  d\zX
  -
  \int_{\zX=0}^{\zX=\infty}
  \frac{1}{\sqrt{2\pi} \alphaX} \zX e^{-\frac{1}{2} \zX^{2}} (\pFn( \zX ) - \pFn( -\zX ))
  d\zX
  \\
  &=
  \frac{1}{\sqrt{2\pi} \alphaX}
  -
  \frac{\gammaX}{2\alphaX}
  \\
  &\dCmt{by the definition of \(  \gammaX  \)}
  \\ \TagEqn{\label{eqn:pf:lemma:concentration-ineq:noisy:7}}
  &=
  \EWRValue
.\end{align*}

With this preliminary work completed, we now proceed to the derivations of \EQUATIONS \eqref{eqn:lemma:concentration-ineq:noisy:pr:1} and \eqref{eqn:lemma:concentration-ineq:noisy:ev:1}, starting with the former.
%
\paragraph{Verification of \EQUATION \eqref{eqn:lemma:concentration-ineq:noisy:pr:1}} 
%
Having obtained the density function and expectations for the conditioned random variable \(  \Wi \Mid| \Ri  \), the expectation of the random variable \(  | \langle \CovVX\VIx{\iIx}, \thetaStar \rangle | \sep \I( \fFn( \langle \CovVX\VIx{\iIx}, \thetaStar \rangle ) \neq \Sign( \langle \CovVX\VIx{\iIx}, \thetaStar \rangle ) )  \) is now calculated as follows: 
\begin{align*}
  \E[ | \langle \CovVX\VIx{\iIx}, \thetaStar \rangle | \sep \I( \fFn( \langle \CovVX\VIx{\iIx}, \thetaStar \rangle ) \neq \Sign( \langle \CovVX\VIx{\iIx}, \thetaStar \rangle ) ) ]
  &=
  \E[ \Wi ]
  \\
  &=
  \pdf{\Ri}( 0 ) \E[ \Wi \Mid| \Ri=0 ]
  +
  \pdf{\Ri}( 1 ) \E[ \Wi \Mid| \Ri=1 ]
  \\
  &\dCmt{by the law of total expectation}
  \\
  &=
  ( 1-\alphaX )
  0
  +
  \alphaX
  \EWRValue
  \\
  &\dCmt{by \EQUATIONS \eqref{eqn:pf:lemma:concentration-ineq:noisy:5}, \eqref{eqn:pf:lemma:concentration-ineq:noisy:6}, and \eqref{eqn:pf:lemma:concentration-ineq:noisy:7}}
  \\
  &=
  \frac{\sqrt{\hfrac{2}{\pi}}-\gammaX}{2}
.\end{align*}
By the linearity of expectation, it follows that
\begin{align*}
  \E \left[ \left\langle \hfFn[\JCoordsX]( \thetaStar, \thetaStar ), \thetaStar \right\rangle \right]
  &=
  \E \left[
    -\frac{\sqrt{2\pi}}{\m}
    \sum_{\iIx=1}^{\m}
    | \langle \CovVX\VIx{\iIx}, \thetaStar \rangle |
    \sep \I( \fFn( \langle \CovVX\VIx{\iIx}, \thetaStar \rangle ) \neq \Sign( \langle \CovVX\VIx{\iIx}, \thetaStar \rangle ) )
  \right]
  \\
  &=
  -\frac{\sqrt{2\pi}}{\m}
  \sum_{\iIx=1}^{\m}
  \E \left[
    | \langle \CovVX\VIx{\iIx}, \thetaStar \rangle |
    \sep \I( \fFn( \langle \CovVX\VIx{\iIx}, \thetaStar \rangle ) \neq \Sign( \langle \CovVX\VIx{\iIx}, \thetaStar \rangle ) )
  \right]
  \\
  &=
  -\frac{\sqrt{2\pi}}{\m}
  \sum_{\iIx=1}^{\m}
  \frac{\sqrt{\hfrac{2}{\pi}}-\gammaX}{2}
  \\
  &=
  -\left( 1 - \sqrt{\frac{\pi}{2}} \gammaX \right)
,\end{align*}
as claimed.
This completes the derivation of \EQUATION \eqref{eqn:lemma:concentration-ineq:noisy:pr:1}.
%
\paragraph{Verification of \EQUATION \eqref{eqn:lemma:concentration-ineq:noisy:ev:1}} 
%
Next, \EQUATION \eqref{eqn:lemma:concentration-ineq:noisy:ev:1} is derived.
This derivation is based on the \MGFs of the centered conditioned random variables \(  ( \Wi \Mid| \Ri ) - \E[ \Wi \Mid| \Ri ]  \) and \(  ( -\Wi \Mid| \Ri ) - \E[ -\Wi \Mid| \Ri ]  \), which are denoted by \(  \mgf{( \Wi \Mid| \Ri ) - \E[ \Wi \Mid| \Ri ]}  \) and \(  \mgf{( -\Wi \Mid| \Ri ) - \E[ -\Wi \Mid| \Ri ]}  \), respectively.
Write
\(  \muX[0] \defeq \E[ \Wi \Mid| \Ri=0 ] = 0  \) and
\(  \muX[1] \defeq \E[ \Wi \Mid| \Ri=1 ] = \frac{\sqrt{\hfrac{2}{\pi}}-\gammaX}{2\alphaX} \),
where these expectations were calculated previously in \EQUATIONS \eqref{eqn:pf:lemma:concentration-ineq:noisy:6} and \eqref{eqn:pf:lemma:concentration-ineq:noisy:7}.
Conditioned on \(  \Ri=1  \), the \MGFs are given at \(  \sX \in [0,\infty)  \) by
\begin{align*}
  \mgf{(\Wi \Mid| \Ri=1)-\E[\Wi \Mid| \Ri=1]}(\sX)
  &=
  \E \left[
    e^{\sX ( \Wi-\E[\Wi] )}
  \middle|
    \Ri=1
  \right]
  \\
  &=
  \int_{\zX=0}^{\zX=\infty}
  \frac{1}{\sqrt{2\pi} \alphaX} e^{\sX( \zX-\muX[1] )} e^{-\frac{1}{2} \zX^{2}} (\pExpr{\zX})
  d\zX
  \\
  &=
  \frac{1}{\alphaX}
  e^{\frac{1}{2} \sX^{2}}
  \frac{1}{\sqrt{2\pi}}
  e^{-\sX \muX[1]}
  \int_{\zX=0}^{\zX=\infty}
  e^{-\frac{1}{2} (\zX-\sX)^{2}} (\pExpr{\zX})
  d\zX
  \\
  &=
  \frac{1}{\alphaX}
  e^{\frac{1}{2} \sX^{2}}
  \fFnX( \sX )
,\end{align*}
and
\begin{align*}
  \mgf{(-\Wi \Mid| \Ri=1)-\E[-\Wi \Mid| \Ri=1]}(\sX)
  &=
  \E \left[
    e^{\sX ( -\Wi-\E[-\Wi] )}
  \middle|
    \Ri=1
  \right]
  \\
  &=
  \int_{\zX=0}^{\zX=\infty}
  \frac{1}{\sqrt{2\pi} \alphaX} e^{\sX( \zX+\muX[1] )} e^{-\frac{1}{2} \zX^{2}} (\pExpr{\zX})
  d\zX
  \\
  &=
  \frac{1}{\alphaX}
  e^{\frac{1}{2} \sX^{2}}
  \frac{1}{\sqrt{2\pi}}
  e^{\sX \muX[1]}
  \int_{\zX=0}^{\zX=\infty}
  e^{-\frac{1}{2} (\zX+\sX)^{2}} (\pExpr{\zX})
  d\zX
  \\
  &=
  \frac{1}{\alphaX}
  e^{\frac{1}{2} \sX^{2}}
  \fFnXX( \sX )
,\end{align*}
where
\begin{gather}
\label{eqn:pf:lemma:concentration-ineq:noisy:f1}
  \fFnX( \sX )
  \defeq
  \frac{1}{\sqrt{2\pi}}
  e^{-\sX \muX[1]}
  \int_{\zX=0}^{\zX=\infty}
  e^{-\frac{1}{2} (\zX-\sX)^{2}} (\pExpr{\zX})
  d\zX
  ,\\
\label{eqn:pf:lemma:concentration-ineq:noisy:f2}
  \fFnXX( \sX )
  \defeq
  \frac{1}{\sqrt{2\pi}}
  e^{\sX \muX[1]}
  \int_{\zX=0}^{\zX=\infty}
  e^{-\frac{1}{2} (\zX+\sX)^{2}} (\pExpr{\zX})
  d\zX
.\end{gather}
%
Before proceeding, the following lemma is introduced to facilitate upper bounds on the \MGFs, \(  \mgf{(\Wi \Mid| \Ri=1)-\E[\Wi \Mid| \Ri=1]}  \) and \(  \mgf{(-\Wi \Mid| \Ri=1)-\E[-\Wi \Mid| \Ri=1]}  \).
Its proof is deferred to \SECTION \ref{outline:concentration-ineq|pf-noisy|pf-f1,f2}.
%
\begin{lemma}
\label{lemma:pf:lemma:concentration-ineq:noisy:f1,f2}
Let \(  \fFnX, \fFnXX : \R \to \R  \) be the functions defined in \EQUATIONS \eqref{eqn:pf:lemma:concentration-ineq:noisy:f1} and \eqref{eqn:pf:lemma:concentration-ineq:noisy:f2}.
Then,
\begin{gather}
\label{eqn:pf:lemma:concentration-ineq:noisy:f1:ub}
  \sup_{\sX \geq 0} \fFnX( \sX ) = \fFnX( 0 )
  ,\\
\label{eqn:pf:lemma:concentration-ineq:noisy:f2:ub}
  \sup_{\sX \geq 0} \fFnXX( \sX ) = \fFnXX( 0 )
,\end{gather}
where
\begin{gather}
\label{eqn:pf:lemma:concentration-ineq:noisy:f1(0)}
  \fFnX( 0 ) = \alphaX
  ,\\
\label{eqn:pf:lemma:concentration-ineq:noisy:f2(0)}
  \fFnXX( 0 ) = \alphaX
.\end{gather}
\end{lemma}
%
Due to \EQUATIONS \eqref{eqn:pf:lemma:concentration-ineq:noisy:f1:ub}--\eqref{eqn:pf:lemma:concentration-ineq:noisy:f2(0)} in \LEMMA \ref{lemma:pf:lemma:concentration-ineq:noisy:f1,f2}, the \MGFs of \(  {( \Wi \Mid| \Ri=1 )} - {\E[ \Wi \Mid| \Ri=1 ]}  \) and \(  {( -\Wi \Mid| \Ri=1 )} - {\E[ -\Wi \Mid| \Ri=1 ]}  \) are now upper bounded by
\begin{gather}
\label{eqn:pf:lemma:concentration-ineq:noisy:11}
  \mgf{(\Wi \Mid| \Ri=1)-\E[\Wi \Mid| \Ri=1]}(\sX)
  =
  \frac{1}{\alphaX}
  e^{\frac{1}{2} \sX^{2}}
  \fFnX( \sX )
  \leq
  e^{\frac{1}{2} \sX^{2}}
  ,\\
\label{eqn:pf:lemma:concentration-ineq:noisy:12}
  \mgf{(-\Wi \Mid| \Ri=1)-\E[-\Wi \Mid| \Ri=1]}(\sX)
  =
  \frac{1}{\alphaX}
  e^{\frac{1}{2} \sX^{2}}
  \fFnXX( \sX )
  \leq
  e^{\frac{1}{2} \sX^{2}}
.\end{gather}
%
\par 
%
On the other hand, when conditioned on \(  \Ri=0  \), the \MGFs of these random variables \(  {( \Wi \Mid| \Ri=0 )} - {\E[ \Wi \Mid| \Ri=0 ]}  \) and \(  ( -\Wi \Mid| \Ri=0 ) - \E[ -\Wi \Mid| \Ri=0 ]  \), written as \(  \mgf{(\Wi \Mid| \Ri=0)-\E[\Wi \Mid| \Ri=0]}  \) and \(  \mgf{(-\Wi \Mid| \Ri=0)-\E[-\Wi \Mid| \Ri=0]}  \), are obtained as follows for \(  \sX \in [0,\infty)  \):
\begin{align*}
  \mgf{(\Wi \Mid| \Ri=0)-\E[\Wi \Mid| \Ri=0]}(\sX)
  &=
  \E \left[
    e^{\sX( \Wi-\muX[0] )}
  \middle|
    \Ri=0
  \right]
  \\
  &=
  \E \left[
    e^{\sX\Wi}
  \middle|
    \Ri=0
  \right]
  \\
  &=
  e^{\sX \cdot 0} \pdf{\Wi \Mid| \Ri}( 0 \Mid| 0 )
  \\
  &=
  1
\TagEqn{\label{eqn:pf:lemma:concentration-ineq:noisy:13}}
,\end{align*}
and likewise,
\begin{align*}
  \mgf{(-\Wi \Mid| \Ri=0)-\E[-\Wi \Mid| \Ri=0]}(\sX)
\XXX{
  &=
  \E \left[
    e^{\sX( -\Wi+\muX[0] )}
  \middle|
    \Ri=0
  \right]
  \\
  &=
  \E \left[
    e^{-\sX\Wi}
  \middle|
    \Ri=0
  \right]
  \\
  &=
  e^{-\sX \cdot 0} \pdf{\Wi \Mid| \Ri}( 0 \Mid| 0 )
  \\
}
  &=
  1
\TagEqn{\label{eqn:pf:lemma:concentration-ineq:noisy:14}}
.\end{align*}
%
\par 
%
Now, consider the sum of the random variables \(  \Wi  \), \(  \iIx \in [\m]  \).
Write
\(  \WRV \defeq \sum_{\iIx=1}^{\m} \Wi  \),
and let
\(  \Vec{\RRV} \defeq ( \Ri[1], \dots, \Ri[\m] )  \).
Fixing
\(  \Vec{\rX} \in \{ 0,1 \}^{\m}  \),
the \MGF of \(  ( \WRV \Mid| \Vec{\RRV}=\Vec{\rX} ) - \E[ \WRV \Mid| \Vec{\RRV}=\Vec{\rX} ]  \), written \(  \mgf{( \WRV \Mid| \Vec{\RRV}=\Vec{\rX} ) - \E[ \WRV \Mid| \Vec{\RRV}=\Vec{\rX} ]}( \sX )  \), is given and upper bounded at \(  \sX \in [0,\infty)  \) by
\begin{align*}
  \mgf{( \WRV \Mid| \Vec{\RRV}=\Vec{\rX} ) - \E[ \WRV \Mid| \Vec{\RRV}=\Vec{\rX} ]}( \sX )
\XXX{
  &=
  \E \left[
    e^{\sX( \WRV - \E[ \WRV ] )}
  \middle|
    \Vec{\RRV}=\Vec{\rX}
  \right]
  \\
  &\dCmt{by the definition of \MGFs}
  \\
}
  &=
  \E \left[
    e^{\sX \sum_{\iIx=1}^{\m} \Wi - \E[ \Wi ]}
  \middle|
    \Vec{\RRV}=\Vec{\rX}
  \right]
  \\
  &\dCmt{by the definition of \MGFs and by the definition of \(  \WRV  \)} 
  \\
\XXX{  &=
  \E \left[
    \prod_{\iIx=1}^{\m}
    \left( e^{\sX( \Wi - \E[ \Wi ] )} \middle| \Ri=\rX[\iIx] \right)
  \right]
  \\
  &\dCmt{each random variable \(  \Wi  \), \(  \iIx \in [\m]  \), is independent of}
  \\
  &\dCmtIndent\text{the random variables \(  \{ \Ri[\iIx'] \}_{\iIx' \neq \iIx}  \)}
  \\
}
  &=
  \prod_{\iIx=1}^{\m}
  \E \left[
    e^{\sX( \Wi - \E[ \Wi ] )}
  \middle|
    \Ri=\rX[\iIx]
  \right]
  \\
  &\dCmt{each \(  \Wi  \), \(  \iIx \in [\m]  \), is independent of \(  \{ \Ri[\iIx'] \}_{\iIx' \neq \iIx}  \); and}
  \\
  &\dCmt{\(  \{ ( \Wi[1] \Mid| \Ri[1] ), \dots, ( \Wi[\m] \Mid| \Ri[\m] ) \}  \) are mutually independent}
\XXX{
  \\
  &=
  \prod_{\substack{\iIx=1 :\\ \rX[\iIx]=1}}^{\m}
  \E \left[
    e^{\sX( \Wi - \E[ \Wi ] )}
  \middle|
    \Ri=1
  \right]
  \prod_{\substack{\iIx=1 :\\ \rX[\iIx]=0}}^{\m}
  \E \left[
    e^{\sX( \Wi - \E[ \Wi ] )}
  \middle|
    \Ri=0
  \right]
  \\
  &\dCmt{by partitioning the values of the index of multiplication}
}
  \\
  &=
  \prod_{\substack{\iIx=1 :\\ \rX[\iIx]=1}}^{\m}
  \mgf{( \Wi \Mid| \Ri=1 ) - \E[ \Wi \Mid| \Ri=1 ]}( \sX )
  \prod_{\substack{\iIx=1 :\\ \rX[\iIx]=0}}^{\m}
  \mgf{( \Wi \Mid| \Ri=0 ) - \E[ \Wi \Mid| \Ri=0 ]}( \sX )
  \\
  &\dCmt{by partitioning the values of the index of multiplication}
  \\
  &\dCmtx{and by the definition of \(  \mgf{( \Wi \Mid| \Ri ) - \E[ \Wi \Mid| \Ri ]}  \), \(  \iIx \in [\m]  \)}
  \\
  &\leq
  \prod_{\substack{\iIx=1 :\\ \rX[\iIx]=1}}^{\m}
  e^{\frac{1}{2} \sX^{2}}
  \prod_{\substack{\iIx=1 :\\ \rX[\iIx]=0}}^{\m}
  1
  \\
  &\dCmt{by \EQUATIONS \eqref{eqn:pf:lemma:concentration-ineq:noisy:11} and \eqref{eqn:pf:lemma:concentration-ineq:noisy:13}}
  \\
  &=
  e^{\frac{1}{2} \| \Vec{\rX} \|_{0} \sX^{2}}
\TagEqn{\label{eqn:pf:lemma:concentration-ineq:noisy:15}}
.\end{align*}
It follows that
\begin{align*}
  \Pr \left(
    \frac{\WRV}{\m} - \E \left[ \frac{\WRV}{\m} \right] > \alphaX \tX
  \middle|
    \Vec{\RRV}=\Vec{\rX}
  \right)
  &\leq
  \inf_{\sX \geq 0}
  e^{-\alphaX \m \sX \tX}
  \mgf{( \WRV \Mid| \Vec{\RRV}=\Vec{\rX} ) - \E[ \WRV \Mid| \Vec{\RRV}=\Vec{\rX} ]}( \sX )
  \\
  &\dCmt{due to Bernstein (\seeeg \cite{vershynin2018high})}
  \\
  \TagEqn{\label{eqn:pf:lemma:concentration-ineq:noisy:9}}
  &\leq
  \inf_{\sX \geq 0}
  e^{-\alphaX \m \sX \tX}
  e^{\frac{1}{2} \| \Vec{\rX} \|_{0} \sX^{2}}
  .\\
  &\dCmt{by \EQUATION \eqref{eqn:pf:lemma:concentration-ineq:noisy:15}}
\end{align*}
Additionally, note that
\begin{align}
\label{eqn:pf:lemma:concentration-ineq:noisy:10}
  \pdf{\Vec{\RRV}}( \Vec{\rX} ) = \alphaX^{\| \Vec{\rX} \|_{0}} (1-\alphaX)^{\m-\| \Vec{\rX} \|_{0}}
.\end{align}
Then,
\begin{align*}
  \Pr \left(
    \frac{\WRV}{\m} - \E \left[ \frac{\WRV}{\m} \right] > \alphaX \tX
  \right)
  &=
  \sum_{\Vec{\rX} \in \{ 0,1 \}^{\m}}
  \pdf{\Vec{\RRV}}( \Vec{\rX} )
  \Pr \left(
    \frac{\WRV}{\m} - \E \left[ \frac{\WRV}{\m} \right] > \alphaX \tX
  \middle|
    \Vec{\RRV}=\Vec{\rX}
  \right)
  \\
  &\dCmt{by the law of total expectation}
  \\
  &\leq
  \sum_{\Vec{\rX} \in \{ 0,1 \}^{\m}}
  \alphaX^{\| \Vec{\rX} \|_{0}} ( 1-\alphaX )^{\m-\| \Vec{\rX} \|_{0}}
  \inf_{\sX \geq 0}
  e^{-\alphaX \m \sX \tX}
  e^{\frac{1}{2} \| \Vec{\rX} \|_{0} \sX^{2}}
  \\
  &\dCmt{by \EQUATIONS \eqref{eqn:pf:lemma:concentration-ineq:noisy:9} and \eqref{eqn:pf:lemma:concentration-ineq:noisy:10}}
  \\
  &\leq
  \inf_{\sX \geq 0}
  e^{-\alphaX \m \sX \tX}
  \sum_{\Vec{\rX} \in \{ 0,1 \}^{\m}}
  ( \alphaX e^{\frac{1}{2} \sX^{2}} )^{\| \Vec{\rX} \|_{0}}
  ( 1-\alphaX )^{\m-\| \Vec{\rX} \|_{0}}
\XXX{  \\
  &=
  \inf_{\sX \geq 0}
  e^{-\alphaX \m \sX \tX}
  \sum_{\Ell=0}^{\m}
  \sum_{\substack{\Vec{\rX} \in \{ 0,1 \}^{\m} :\\ \| \Vec{\rX} \|_{0} = \Ell}}
  ( \alphaX e^{\frac{1}{2} \sX^{2}} )^{\Ell}
  ( 1-\alphaX )^{\m-\Ell}
  \\
  &\dCmt{by partitioning the values of the index of summation}
}
  \\
  &=
  \inf_{\sX \geq 0}
  e^{-\alphaX \m \sX \tX}
  \sum_{\Ell=0}^{\m}
  \binom{\m}{\Ell}
  ( \alphaX e^{\frac{1}{2} \sX^{2}} )^{\Ell}
  ( 1-\alphaX )^{\m-\Ell}
  \\
  &\dCmt{by partitioning the values of the index of}
  \\
  &\dCmtx{summation according to \(  \Ell = \| \Vec{\rX} \|_{0}  \)}
\XXX{
  \\
  &=
  \inf_{\sX \geq 0}
  e^{-\alphaX \m \sX \tX}
  \left( \alphaX e^{\frac{1}{2} \sX^{2}} + 1 - \alphaX \right)^{\m}
  \\
  &\dCmt{by the binomial theorem}
}
  \\
  &=
  \inf_{\sX \geq 0}
  e^{-\alphaX \m \sX \tX}
  \left( 1 + \alphaX ( e^{\frac{1}{2} \sX^{2}} - 1 ) \right)^{\m}
  \\
  &\dCmt{by the binomial theorem}
\XXX{
  \\
  &\leq
  \inf_{\sX \geq 0}
  e^{-\alphaX \m \sX \tX}
  e^{\alphaX \m ( e^{\frac{1}{2} \sX^{2}} - 1 )}
  \\
  &\dCmt{\(  {\textstyle 1 + \alphaX ( e^{\frac{1}{2} \sX^{2}} - 1 ) = e^{\log( 1 + \alphaX ( e^{\frac{1}{2} \sX^{2}} - 1 ) )} \leq e^{\alphaX ( e^{\frac{1}{2} \sX^{2}} - 1 )}}  \),}
  \\
  &\dCmtIndent\text{where the inequality is due to a well-known result}
}
  \\
  &\leq
  \inf_{\sX \geq 0}
  e^{-\alphaX \m ( \sX \tX - e^{\frac{1}{2} \sX^{2}} + 1 )}
  .\\
  &\dCmt{by a well-known inequality, \(  \log( 1+u ) \leq u  \) for \(  u > -1  \)} 
\end{align*}
As discussed earlier in the proof of \LEMMA \ref{lemma:concentration-ineq:noiseless}, \(  \sX \tX - e^{\frac{1}{2} \sX^{2}} + 1  \) is maximized with respect to \(  \sX  \) for \(  \sX, \tX \in (0,1)  \) when roughly \(  \sX \approx \tX  \) because
\begin{gather*}
  \left. \frac{\partial}{\partial \sX} \sX \tX - e^{\frac{1}{2} \sX^{2}} + 1 \right|_{\sX=\tX}
  =
  \left. \tX - \sX e^{\frac{1}{2} \sX^{2}} \right|_{\sX=\tX}
  \approx
  0
\end{gather*}
for small \(  \sX, \tX > 0  \), and because for all \(  \sX \in \R  \),
\begin{align*}
  \frac{\partial^{2}}{\partial \sX^{2}} \sX \tX - e^{\frac{1}{2} \sX^{2}} + 1
  =
  -( 1+\sX^{2} ) e^{\frac{1}{2} \sX^{2}}
  <
  0
.\end{align*}
Hence, we will take \(  \sX=\tX  \).
In addition, recall that
\begin{gather*}
  \left. \sX \tX - e^{\frac{1}{2} \sX^{2}} + 1 \right|_{\sX=\tX}
  \geq
  \frac{\tX^{2}}{3}
.\end{gather*}
It follows that
\begin{gather*}
  \Pr \left(
    \WRV - \E[ \WRV ] > \alphaX \tX
  \right)
  \leq
  e^{-\frac{1}{3} \alphaX \m \tX^{2}}
,\end{gather*}
and therefore, due to the design of \(  \WRV  \),
\begin{gather*}
  \Pr \left(
    \left\langle \frac{\hfFn[\JCoordsX]( \thetaStar, \thetaStar )}{\sqrt{2\pi}} , \thetaStar \right\rangle
    -
    \E \left[ \left\langle \frac{\hfFn[\JCoordsX]( \thetaStar, \thetaStar )}{\sqrt{2\pi}} , \thetaStar \right\rangle \right]
    >
    \alphaX \tX
  \right)
  \leq
  e^{-\frac{1}{3} \alphaX \m \tX^{2}}
.\end{gather*}
Moreover, by a nearly identical argument (omitted here), the other side of the bound is obtained:
\begin{gather*}
  \Pr \left(
    \WRV - \E[ \WRV ] < -\alphaX \tX
  \right)
  =
  \Pr \left(
    -\WRV - \E[ -\WRV ] > \alphaX \tX
  \right)
  \leq
  e^{-\frac{1}{3} \alphaX \m \tX^{2}}
,\end{gather*}
and thus,
\begin{gather*}
  \Pr \left(
    \left\langle \frac{\hfFn[\JCoordsX]( \thetaStar, \thetaStar )}{\sqrt{2\pi}} , \thetaStar \right\rangle
    -
    \E \left[ \left\langle \frac{\hfFn[\JCoordsX]( \thetaStar, \thetaStar )}{\sqrt{2\pi}} , \thetaStar \right\rangle \right]
    <
    -\alphaX \tX
  \right)
  \leq
  e^{-\frac{1}{3} \alphaX \m \tX^{2}}
.\end{gather*}
Combining the above inequalities into a two-sided bound via a union bound yields:
\begin{gather*}
  \Pr \left(
    \left|
    \left\langle \frac{\hfFn[\JCoordsX]( \thetaStar, \thetaStar )}{\sqrt{2\pi}} , \thetaStar \right\rangle
    -
    \E \left[ \left\langle \frac{\hfFn[\JCoordsX]( \thetaStar, \thetaStar )}{\sqrt{2\pi}} , \thetaStar \right\rangle \right]
    \right|
    >
    \alphaX \tX
  \right)
  \leq
  2 e^{-\frac{1}{3} \alphaX \m \tX^{2}}
.\end{gather*}
Lastly, union bounding over all \(  \JCoordsX \in \JSX  \), the desired uniform concentration inequality follows:
\begin{gather*}
  \Pr \left(
    \ExistsST{\JCoordsX \in \JSX}
    {\left|
    \left\langle \frac{\hfFn[\JCoordsX]( \thetaStar, \thetaStar )}{\sqrt{2\pi}} , \thetaStar \right\rangle
    -
    \E \left[ \left\langle \frac{\hfFn[\JCoordsX]( \thetaStar, \thetaStar )}{\sqrt{2\pi}} , \thetaStar \right\rangle \right]
    \right|
    >
    \alphaX \tX}
  \right)
  \leq
  2 | \JSX | e^{-\frac{1}{3} \alphaX \m \tX^{2}}
.\end{gather*}


\subsubsection{Proof of the \EQUATIONS \eqref{eqn:lemma:concentration-ineq:noisy:pr:2} and \eqref{eqn:lemma:concentration-ineq:noisy:ev:2}}
\label{outline:concentration-ineq|pf-noisy|2}

We begin with some preliminary analysis to characterize a few random variables of interest.
As in the derivations of \EQUATIONS \eqref{eqn:lemma:concentration-ineq:noisy:pr:1} and \eqref{eqn:lemma:concentration-ineq:noisy:ev:1} in \SECTION \ref{outline:concentration-ineq|pf-noisy|1}, consider an arbitrary coordinate subset
\(  \JCoordsX \in \JSX  \),
recall \EQUATION \eqref{eqn:pf:lemma:concentration-ineq:noisy:16} and \eqref{eqn:pf:lemma:concentration-ineq:noisy:1}:
\begin{gather*}
  \frac{1}{\sqrt{2\pi}} \hfFn[\JCoordsX]( \thetaStar, \thetaStar )
  =
  -\frac{1}{\m}
  \sum_{\iIx=1}^{\m}
  \CovVX\VIx{\iIx}
  \sep \Sign( \langle \CovVX\VIx{\iIx}, \thetaStar \rangle )
  \sep \I( \fFn( \langle \CovVX\VIx{\iIx}, \thetaStar \rangle ) \neq \Sign( \langle \CovVX\VIx{\iIx}, \thetaStar \rangle ) )
  ,\\
  \left\langle \frac{1}{\sqrt{2\pi}} \hfFn[\JCoordsX]( \thetaStar, \thetaStar ), \thetaStar \right\rangle
  =
  -\frac{1}{\m}
  \sum_{\iIx=1}^{\m}
  \langle \CovVX\VIx{\iIx}, \thetaStar \rangle
  \sep \Sign( \langle \CovVX\VIx{\iIx}, \thetaStar \rangle )
  \sep \I( \fFn( \langle \CovVX\VIx{\iIx}, \thetaStar \rangle ) \neq \Sign( \langle \CovVX\VIx{\iIx}, \thetaStar \rangle ) )
,\end{gather*}
where
\(  \CovVX\VIx{\iIx} \defeq \ThresholdSet{\Supp( \thetaStar ) \cup \JCoordsX}( \CovV\VIx{\iIx} )  \).
Thus,
\begin{align*}
  &
  \frac{1}{\sqrt{2\pi}} \hfFn[\JCoordsX]( \thetaStar, \thetaStar )
  -
  \left\langle \frac{1}{\sqrt{2\pi}} \hfFn[\JCoordsX]( \thetaStar, \thetaStar ), \thetaStar \right\rangle
  \thetaStar
  \\
  &\AlignIndent=
  -\frac{1}{\m}
  \sum_{\iIx=1}^{\m}
  \left( \CovVX\VIx{\iIx} - \langle \CovVX\VIx{\iIx}, \thetaStar \rangle \thetaStar \right)
  \sep \Sign( \langle \CovVX\VIx{\iIx}, \thetaStar \rangle )
  \sep \I( \fFn( \langle \CovVX\VIx{\iIx}, \thetaStar \rangle ) \neq \Sign( \langle \CovVX\VIx{\iIx}, \thetaStar \rangle ) )
.\end{align*}
Let
\(  \kX \defeq | \Supp( \thetaStar ) \cup \JCoordsX |  \),
and denote the \(  \kX  \)-dimensional subspace of vectors whose support is a (possibly improper) subset of \(  \Supp( \thetaStar ) \cup \JCoordsX  \) by
\(  \Set{V} \defeq \{ \Vec{\vV} \in \R^{\n} : \Supp( \Vec{\vV} ) \subseteq \Supp( \thetaStar ) \cup \JCoordsX \}  \).
Let
\(  \{ \Vec{\vV}\VIx{1}, \dots, \Vec{\vV}\VIx{\kX} \} \subset \Set{V}  \)
be an orthonormal basis of \(  \Set{V}  \), where
\(  \Vec{\vV}\VIx{\kX} = \thetaStar  \).
Then, for each \(  \iIx \in [\m]  \), since the vector \(  \CovVX\VIx{\iIx}  \) is contained in \(  \Set{V}  \), it is orthogonally decomposed with this bases as:
\begin{gather*}
  \CovVX\VIx{\iIx}
  =
  \sum_{\jIx=1}^{\kX}
  \langle \CovVX\VIx{\iIx}, \Vec{\vV}\VIx{\jIx} \rangle \Vec{\vV}\VIx{\jIx}
,\end{gather*}
while \(  \CovVX\VIx{\iIx} - \langle \CovVX\VIx{\iIx}, \thetaStar \rangle \thetaStar  \)---which is likewise an element in the vector subspace \(  \Set{V}  \)---is orthogonally decomposed as:
\begin{align*}
  \CovVX\VIx{\iIx} - \langle \CovVX\VIx{\iIx}, \thetaStar \rangle \thetaStar
  &=
  \sum_{\jIx=1}^{\kX}
  \bigl(
    \langle \CovVX\VIx{\iIx}, \Vec{\vV}\VIx{\jIx} \rangle
    -
    \langle \CovVX\VIx{\iIx}, \thetaStar \rangle
    \langle \thetaStar, \Vec{\vV}\VIx{\jIx} \rangle
  \bigr)
  \Vec{\vV}\VIx{\jIx}
\XXX{
  \\
  &=
  \langle \CovVX\VIx{\iIx} - \langle \CovVX\VIx{\iIx}, \thetaStar \rangle \thetaStar, \Vec{\vV}\VIx{\kX} \rangle \Vec{\vV}\VIx{\kX}
  +
  \sum_{\jIx=1}^{\kX-1}
  \langle \CovVX\VIx{\iIx} - \langle \CovVX\VIx{\iIx}, \thetaStar \rangle \thetaStar, \Vec{\vV}\VIx{\jIx} \rangle \Vec{\vV}\VIx{\jIx}
  \\
  &\dCmt{by separating out the last term from the summation}
}
  \\
  &=
  \bigl(
    \langle \CovVX\VIx{\iIx}, \thetaStar \rangle
    -
    \langle \CovVX\VIx{\iIx}, \thetaStar \rangle
    \langle \thetaStar, \thetaStar \rangle
  \bigr)
  \thetaStar
  +
  \sum_{\jIx=1}^{\kX-1}
  \bigl(
    \langle \CovVX\VIx{\iIx}, \Vec{\vV}\VIx{\jIx} \rangle
    -
    \langle \CovVX\VIx{\iIx}, \thetaStar \rangle
    \langle \thetaStar, \Vec{\vV}\VIx{\jIx} \rangle
  \bigr)
  \Vec{\vV}\VIx{\jIx}
  \\
  &\dCmt{by separating out the \(  \kX\Th  \) term from the summation,}
  \\
  &\dCmtx{and since \(  \Vec{\vV}\VIx{\kX} = \thetaStar  \) by design}
  \\
\XXX{
  &=
  \bigl(
    \langle \CovVX\VIx{\iIx}, \thetaStar \rangle
    -
    \langle \CovVX\VIx{\iIx}, \thetaStar \rangle
  \bigr)
  \thetaStar
  +
  \sum_{\jIx=1}^{\kX-1}
  \langle \CovVX\VIx{\iIx} - \langle \CovVX\VIx{\iIx}, \thetaStar \rangle \thetaStar, \Vec{\vV}\VIx{\jIx} \rangle \Vec{\vV}\VIx{\jIx}
  \\
  &\dCmt{by distributivity and recalling that \(  \langle \thetaStar, \thetaStar \rangle = \| \thetaStar \|_{2}^{2} = 1  \)}
  \\
}
  &=
  \sum_{\jIx=1}^{\kX-1}
  \bigl(
    \langle \CovVX\VIx{\iIx}, \Vec{\vV}\VIx{\jIx} \rangle
    -
    \langle \CovVX\VIx{\iIx}, \thetaStar \rangle
    \langle \thetaStar, \Vec{\vV}\VIx{\jIx} \rangle
  \bigr)
  \Vec{\vV}\VIx{\jIx}
  \\
  &\dCmt{recalling that \(  \langle \thetaStar, \thetaStar \rangle = \| \thetaStar \|_{2}^{2} = 1  \)}
   \\
  &=
  \sum_{\jIx=1}^{\kX-1}
  \langle \CovVX\VIx{\iIx}, \Vec{\vV}\VIx{\jIx} \rangle \Vec{\vV}\VIx{\jIx}
  .\\
  &\dCmt{for \(  \jIx \neq \kX  \), \(  \langle \thetaStar, \Vec{\vV}\VIx{\jIx} \rangle = \langle \Vec{\vV}\VIx{\kX}, \Vec{\vV}\VIx{\jIx} \rangle = 0  \) since \(  \Vec{\vV}\VIx{\jIx} \perp \Vec{\vV}\VIx{\kX}  \) by design}
\end{align*}
Using this orthogonal decomposition,
\begin{align*}
  &
  \frac{1}{\sqrt{2\pi}} \hfFn[\JCoordsX]( \thetaStar, \thetaStar )
  -
  \left\langle \frac{1}{\sqrt{2\pi}} \hfFn[\JCoordsX]( \thetaStar, \thetaStar ), \thetaStar \right\rangle
  \thetaStar
  \\
  &\AlignIndent=
  -\frac{1}{\m}
  \sum_{\iIx=1}^{\m}
  \left( \CovVX\VIx{\iIx} - \langle \CovVX\VIx{\iIx}, \thetaStar \rangle \thetaStar \right)
  \sep
  \Sign( \langle \CovVX\VIx{\iIx}, \thetaStar \rangle )
  \sep
  \I( \fFn( \langle \CovVX\VIx{\iIx}, \thetaStar \rangle ) \neq \Sign( \langle \CovVX\VIx{\iIx}, \thetaStar \rangle ) )
  \\
  &\AlignIndent=
  -\frac{1}{\m}
  \sum_{\iIx=1}^{\m}
  \sum_{\jIx=1}^{\kX-1}
  \langle \CovVX\VIx{\iIx}, \Vec{\vV}\VIx{\jIx} \rangle \Vec{\vV}\VIx{\jIx}
  \sep
  \Sign( \langle \CovVX\VIx{\iIx}, \thetaStar \rangle )
  \sep
  \I( \fFn( \langle \CovVX\VIx{\iIx}, \thetaStar \rangle ) \neq \Sign( \langle \CovVX\VIx{\iIx}, \thetaStar \rangle ) )
\TagEqn{\label{eqn:pf:lemma:concentration-ineq:noisy:17}}
.\end{align*}
Note that by a well-known property of Gaussian vectors, \(  \langle \CovVX\VIx{\iIx}, \Vec{\vV}\VIx{\jIx} \rangle \sim \N(0,1)  \) for each \(  \iIx \in [\m]  \), \(  \jIx \in [\kX]  \), and moreover, due to the orthogonality of \(  \Vec{\vV}\VIx{1}, \dots, \Vec{\vV}\VIx{\kX}  \), the random variables \(  \langle \CovVX\VIx{\iIx}, \Vec{\vV}\VIx{1} \rangle, \dots, \langle \CovVX\VIx{\iIx}, \Vec{\vV}\VIx{\kX} \rangle  \) are mutually independent.
In particular, the random variable \(  \langle \CovVX\VIx{\iIx}, \Vec{\vV}\VIx{\kX} \rangle = \langle \CovVX\VIx{\iIx}, \thetaStar \rangle  \) is independent of every \(  \langle \CovVX\VIx{\iIx}, \Vec{\vV}\VIx{\jIx} \rangle  \), \(  \jIx \in [\kX-1]  \).
Therefore, for every \(  \iIx \in [\m]  \), each \(  \jIx\Th  \) summand, \(  \jIx \in [\kX-1]  \), in \eqref{eqn:pf:lemma:concentration-ineq:noisy:17} follows the same distribution as
\begin{gather}
\label{eqn:pf:lemma:concentration-ineq:noisy:19}
  \langle \CovVX\VIx{\iIx}, \Vec{\vV}\VIx{\jIx} \rangle \Vec{\vV}\VIx{\jIx}
  \sep
  \Sign( \langle \CovVX\VIx{\iIx}, \thetaStar \rangle )
  \sep
  \I( \fFn( \langle \CovVX\VIx{\iIx}, \thetaStar \rangle ) \neq \Sign( \langle \CovVX\VIx{\iIx}, \thetaStar \rangle ) )
  \sim
  \Yi \Zij \Vec{\vV}\VIx{\jIx} = \Uij \Vec{\vV}\VIx{\jIx}
,\end{gather}
where
\(  \Zij[\iIx][1], \dots, \Zij[\iIx][\kX] \sim \N(0,1)  \)
are \iid Gaussian random variables, and where
\begin{gather*}
  \Ri \defeq \I( \fFn( \Zij[\iIx][\kX] ) \neq \Sign( \Zij[\iIx][\kX] ) )
  ,\\
  \Yi \defeq \Sign( \Zij[\iIx][\kX] ) \Ri
  ,\\
  \Uij \defeq \Yi \Zij
.\end{gather*}
Conditioned on \(  \Ri=1  \), the random variable is distributed as \(  ( \Yi \Mid| \Ri \neq 0 ) \sim \{ -1,1 \}  \), uniformly, and is independently of all \(  \Zij  \), \(  \jIx \in [\kX-1]  \).
Hence, conditioned on \(  \Ri  \), the random variable \(  \Uij \Mid| \Ri  \) has a density function given for \(  \zX \in \R  \) and \(  \rX \in \{ 0,1 \}  \) by
\begin{align*}
  \pdf{\Uij \Mid| \Ri}( \zX \Mid| \rX )
  &=
  \begin{cases}
  0, & \cIf \zX \neq 0, \rX=0, \\
  1, & \cIf \zX = 0,    \rX=0, \\
  \pdf{\Zij}( \zX ) \pdf{\Yi \Mid| \Ri}( 1 \Mid| 1 ) + \pdf{-\Zij}( \zX ) \pdf{\Yi \Mid| \Ri}( -1 \Mid| 1 ), & \cIf \rX=1,
  \end{cases}
  \\
  &=
  \begin{cases}
  0, & \cIf \zX \neq 0, \rX=0, \\
  1, & \cIf \zX = 0,    \rX=0, \\
  \frac{1}{2} \pdf{\Zij}( \zX ) + \frac{1}{2} \pdf{-\Zij}( \zX ), & \cIf \rX=1,
  \end{cases}
  \\
  &=
  \begin{cases}
  0, & \cIf \zX \neq 0, \rX=0, \\
  1, & \cIf \zX = 0,    \rX=0, \\
  \frac{1}{\sqrt{2\pi}} e^{-\frac{1}{2} \zX^{2}}, & \cIf \rX=1,
  \end{cases}
\TagEqn{\label{eqn:pf:lemma:concentration-ineq:noisy:18}}
\end{align*}
where the third case on the \RHS of the first equality is due to the law of total probability, the definition of conditional probabilities, the independence of \(  ( \Yi \Mid| \Ri = 1 )  \) and \(  \Zij  \), \(  \jIx \in [\kX-1]  \), and remarks made in the proof of \LEMMA \ref{lemma:concentration-ineq:noiseless}.
\EQUATION \eqref{eqn:pf:lemma:concentration-ineq:noisy:18} implies
\(  ( \Uij \Mid| \Ri=1 ) \sim \Zij  \).
Additionally, due to an earlier discussion in \SECTION \ref{outline:concentration-ineq|pf-noisy|1}, the mass function of the random variable \(  \Ri  \) is given for \(  \rX \in \{ 0,1 \}  \) by
\begin{align*}
  \pdf{\Ri}( \rX )
  =
  \begin{cases}
  1-\alphaX ,& \cIf \rX=0, \\
  \alphaX   ,& \cIf \rX=1.
  \end{cases}
\end{align*}
%
\par 
%
This concludes the preliminary work.
We now proceed to the derivations of \EQUATIONS \eqref{eqn:lemma:concentration-ineq:noisy:pr:2} and \eqref{eqn:lemma:concentration-ineq:noisy:ev:2}, beginning with the latter.
%
\paragraph{Verification of \EQUATION \eqref{eqn:lemma:concentration-ineq:noisy:ev:2}} 
%
It is now possible to verify \EQUATION \eqref{eqn:lemma:concentration-ineq:noisy:ev:2}.
Combining the above arguments with the law of total expectation and the definition of conditional expectations, the expectation of \(  \Uij  \) is calculated as follows:
\begin{align*}
  \E[ \Uij ]
  &=
  \pdf{\Ri}(0) \E[ \Uij \Mid| \Ri=0 ] + \pdf{\Ri}(1) \E[ \Uij \Mid| \Ri=1 ]
  \\
  &=
  ( 1-\alphaX ) 0 + \alphaX \int_{\zX=-\infty}^{\zX=\infty} \frac{1}{\sqrt{2\pi}} \zX e^{-\frac{1}{2} \zX^{2}} d\zX
  \\
  &=
  0
\TagEqn{\label{eqn:pf:lemma:concentration-ineq:noisy:22}}
.\end{align*}
\EQUATION \eqref{eqn:lemma:concentration-ineq:noisy:ev:2} now follows:
\begin{align*}
  \E[ \gfFn[\JCoordsX]( \thetaStar, \thetaStar ) ]
\XXX{
  \\
  &=
  \E \left[
    \frac{1}{\sqrt{2\pi}} \hfFn[\JCoordsX]( \thetaStar, \thetaStar )
    -
    \left\langle \frac{1}{\sqrt{2\pi}} \hfFn[\JCoordsX]( \thetaStar, \thetaStar ), \thetaStar \right\rangle
    \thetaStar
  \right]
  \\
  &\dCmt{by the definition of \(  \gfFn[\JCoordsX]  \) in \EQUATION \eqref{eqn:notations:gfJ:def}}
}
  &=
  -\frac{1}{\m}
  \sum_{\iIx=1}^{\m}
  \sum_{\jIx=1}^{\kX-1}
  \E \left[
    \langle \CovVX\VIx{\iIx}, \Vec{\vV}\VIx{\jIx} \rangle
    \sep
    \Sign( \langle \CovVX\VIx{\iIx}, \thetaStar \rangle )
    \sep
    \I( \fFn( \langle \CovVX\VIx{\iIx}, \thetaStar \rangle ) \neq \Sign( \langle \CovVX\VIx{\iIx}, \thetaStar \rangle ) )
  \right]
  \Vec{\vV}\VIx{\jIx}
  \\
  &\dCmt{by the definition of \(  \gfFn[\JCoordsX]  \) in \EQUATION \eqref{eqn:notations:gfJ:def} and by \EQUATION \eqref{eqn:pf:lemma:concentration-ineq:noisy:17}}
  \\
  &\dCmtx{and the linearity of expectation}
  \\
  &=
  -\frac{1}{\m}
  \sum_{\iIx=1}^{\m}
  \sum_{\jIx=1}^{\kX-1}
  \E[ \Uij ] \Vec{\vV}\VIx{\jIx}
  \\
  &\dCmt{by \EQUATION \eqref{eqn:pf:lemma:concentration-ineq:noisy:19}}
  \\
  &=
  \Vec{0}
  ,\\
  &\dCmt{by \EQUATION \eqref{eqn:pf:lemma:concentration-ineq:noisy:22}}
\end{align*}
as desired.
%
\paragraph{Verification of \EQUATION \eqref{eqn:lemma:concentration-ineq:noisy:pr:2}} 
%
The verification of \EQUATION \eqref{eqn:lemma:concentration-ineq:noisy:pr:2} will largely rely on the work already accomplished above, as well as \LEMMA \ref{lemma:norm-subgaussian:1}, which is restated below for convenience as \LEMMA \ref{lemma:norm-subgaussian:2}.
%
\begin{lemma}
\label{lemma:norm-subgaussian:2}
%
Fix \(  \tXXX, \sigma  > 0  \) and \(  0 < \dX \leq \n  \).
Let
\(  \ICoords \subseteq [\n]  \), \(  | \ICoords | = \dX  \).
Let
\(  \Vec{\XRV} \sim \N( \Vec{0}, \sigma^{2} \sum_{\jIx \in \ICoords} \ej \ej^{\T} )  \).
Then,
\begin{gather}
\label{eqn:lemma:norm-subgaussian:2:1}
  \Pr \left(
    \| \Vec{\XRV} - \E[ \Vec{\XRV} ] \|_{2}
    >
    \sqrt{\dX} \sigma
    +
    \tXXX
  \right)
  \leq
  \Pr \left(
    \| \Vec{\XRV} - \E[ \Vec{\XRV} ] \|_{2}
    >
    \E[ \| \Vec{\XRV} \|_{2} ]
    +
    \tXXX
  \right)
  \leq
  e^{-\frac{1}{2 \sigma^{2}} \tXXX^{2}}
.\end{gather}
\end{lemma}
%
Recall that for each \(  \jIx \in [\kX]  \),
\begin{gather*}
  \langle \CovVX\VIx{\iIx}, \Vec{\vV}\VIx{\jIx} \rangle \Vec{\vV}\VIx{\jIx}
  \sep
  \Sign( \langle \CovVX\VIx{\iIx}, \thetaStar \rangle )
  \sep
  \I( \fFn( \langle \CovVX\VIx{\iIx}, \thetaStar \rangle ) \neq \Sign( \langle \CovVX\VIx{\iIx}, \thetaStar \rangle ) )
  \sim
  \Yi \Zij \Vec{\vV}\VIx{\jIx} = \Uij \Vec{\vV}\VIx{\jIx}
\end{gather*}
per \EQUATION \eqref{eqn:pf:lemma:concentration-ineq:noisy:19}.
Due to the rotational invariance of Gaussians, we will assume going forward, without loss of generality, that the basis vectors, \(  \Vec{\vV}\VIx{1}, \dots, \Vec{\vV}\VIx{\kX}  \), are simply the first \(  \kX  \) standard basis vectors of \(  \R^{\n}  \), i.e., \(  \Vec{\vV}\VIx{\jIx} = \ej  \), \(  \jIx \in [\kX]  \), where \(  \ej \in \R^{\n}  \) is the vector in which  the \(  \jIx\Th  \) entry is \(  1  \) and all other entries are \(  0  \).
Note that under this assumption, \(  \thetaStar = \Vec{\vV}\VIx{\kX} = \ej[\kX]  \), but this, again, does not lose any generality.
For \(  \jIx \in [\kX-1]  \), let
\begin{gather*}
  \Uj \defeq \frac{1}{\m} \sum_{\iIx=1}^{\m} \Uij
,\end{gather*}
and let
\begin{gather*}
  \Vec{\URV} \defeq \sum_{\jIx=1}^{\kX-1} \Uj \ej
.\end{gather*}
Writing the random vector
\(  \Vec{\RRV} \defeq ( \Ri[1], \dots, \Ri[\m] )  \),
and fixing
\(  \Vec{\rX} \in \{ 0,1 \}^{\m}  \),
notice that
\begin{align*}
  ( \Uj \Mid| \Vec{\RRV}=\Vec{\rX} )
  &=
  \frac{1}{\m} \sum_{\iIx \in \Supp( \Vec{\rX} )} ( \Uij \Mid| \Vec{\RRV}=\Vec{\rX} )
\XXX{
  \\
  &=
  \frac{1}{\m} \sum_{\iIx \in \Supp( \Vec{\rX} )} ( \Uij \Mid| \Ri=\rX[\iIx] )
  \\
  &\dCmt{each \(  \Uij  \), \(  \iIx \in [\m]  \), is independent of \(  \{ \Ri[\iIx'] \}_{\iIx' \neq \iIx}  \)}
}
  \\
  &=
  \frac{1}{\m} \sum_{\iIx \in \Supp( \Vec{\rX} )} ( \Uij \Mid| \Ri=1 )
  \\
  &\dCmt{each \(  \Uij  \), \(  \iIx \in [\m]  \), is independent of \(  \{ \Ri[\iIx'] \}_{\iIx' \neq \iIx}  \),}
  \\
  &\dCmtx{and since for each \(  \iIx \in \Supp( \Vec{\rX} )  \), \(  \rX[\iIx]=1  \)}
  \\
  &\sim
  \frac{1}{\m} \sum_{\iIx \in \Supp( \Vec{\rX} )} \Zij
  ,\\
  &\dCmt{by the density of \(  \Uij \Mid| \Ri=1  \) in \EQUATION \eqref{eqn:pf:lemma:concentration-ineq:noisy:18}}
\end{align*}
and therefore,
\(  ( \Uj \Mid| \Vec{\RRV}=\Vec{\rX} ) \sim \N( 0, \frac{\| \Vec{\rX} \|_{0}}{\m^{2}} )  \).
Moreover, letting
\(  \LRV \defeq \| \Vec{\RRV} \|_{0}  \) and \(  \EllX \in \ZeroTo{\m}  \),
it also happens that
\(  ( \Uj \Mid| \LRV = \EllX ) \sim \N( 0, \frac{\EllX}{\m^{2}} )  \)
since the random variables, \(  \Zij  \), \(  \iIx \in [\m]  \), are \iid
(This can be more formally argued using the density function of the conditioned random variable \(  \Uj \Mid| \Vec{\RRV}  \), combined with the law of probability.)
Writing the \(  ( \kX-1 )  \)-sparse random vector
\begin{align*}
  \Vec{\URV} \defeq \sum_{\jIx=1}^{\kX-1} \Uj \ej
,\end{align*}
it follows that the conditioned random vector \(  \Vec{\URV} \Mid| \LRV = \EllX  \) follows a zero-mean Gaussian distribution such that
\begin{align*}
  ( \Vec{\URV} \Mid| \LRV = \EllX )
  =
  \left( \sum_{\jIx=1}^{\kX-1} \Uj \ej \middle| \LRV = \EllX \right)
  \sim
  \N \left( \Vec{0}, \frac{\EllX}{\m^{2}} \sum_{\jIx=1}^{\kX-1} \ej \ej^{\T} \right)
,\end{align*}
and hence, \(  \Vec{\URV} \Mid| \LRV \leq \EllX  \) is at most \(  \frac{\sqrt{\EllX}}{\m}  \)-\subgaussian with mean
\(  \E[ \Vec{\URV} \Mid| \LRV \leq \EllX ] = \Vec{0}  \)
and support of cardinality
\(  \| \Vec{\URV} \|_{0} = \kX-1  \).
Therefore, by \LEMMA \ref{lemma:norm-subgaussian:2} and standard properties of Gaussians,
\begin{align}
\label{eqn:pf:lemma:concentration-ineq:noisy:20}
  &
  \Pr \left(
    \| \Vec{\URV} - \E[ \Vec{\URV} ] \|_{2}
    >
    \frac{\sqrt{( \kX-1 ) \EllX}}{\m}
    +
    \alphaO \tX
  \middle|
    \LRV \leq \EllX
  \right)
  \\
  &\AlignIndent\leq
  \Pr \left(
    \| \Vec{\URV} - \E[ \Vec{\URV} ] \|_{2}
    >
    \frac{\sqrt{( \kX-1 ) \EllX}}{\m}
    +
    \alphaO \tX
  \middle|
    \LRV = \EllX
  \right)
  \\
  &\AlignIndent\leq
  e^{-\frac{\m^{2} \alphaO^{2} \tX^{2}}{2 \EllX}}
.\end{align}
Additionally, it is folklore that
\(  \LRV = \sum_{\iIx=1}^{\m} \Ri \sim \Binomial( \m, \alphaX )  \)
with
\(  \E[ \LRV ] = \E[ \sum_{\iIx=1}^{\m} \Ri ] = \alphaX \m  \)
(\seeeg \cite{charikar2002similarity}).
Thus, letting
\(  \LRVO \sim \Binomial( \m, \alphaO )  \)
be a binomial random variable with mean
\(  \E[ \LRVO ] = \alphaO \m  \),
by recalling that
\(  \alphaO \defeq \alphaOExpr \geq \alphaX  \),
and by a standard concentration inequality for binomial random variables, for \(  \sXXX \in (0,1)  \),
\begin{gather}
\label{eqn:pf:lemma:concentration-ineq:noisy:21}
  \Pr( \LRV > ( 1+\sXXX ) \alphaO \m )
  \leq
  \Pr( \LRVO > ( 1+\sXXX ) \alphaO \m )
  \leq
  e^{-\frac{1}{3} \alphaO \m \sXXX^{2}}
.\end{gather}
Applying the law of total probability gives way to
\begin{align*}
  &\negphantom{\AlignSp}
  \Pr \left(
    \| \Vec{\URV} - \E[ \Vec{\URV} ] \|_{2}
    >
    \sqrt{\frac{\alphaO ( 1+\sXXX ) ( \kX-1 )}{\m} }
    +
    \alphaO \tX
  \right)
  \\
  &=
  \Pr \left(
    \| \Vec{\URV} - \E[ \Vec{\URV} ] \|_{2} > \sqrt{\frac{\alphaO ( 1+\sXXX ) ( \kX-1 ) }{\m} } + \alphaO \tX
  \middle|
    \LRV \leq ( 1+\sXXX ) \alphaO \m
  \right)
  \Pr( \LRV \leq ( 1+\sXXX ) \alphaO \m )
  \\
  &\AlignSp+
  \Pr \left(
    \| \Vec{\URV} - \E[ \Vec{\URV} ] \|_{2} > \sqrt{\frac{\alphaO ( 1+\sXXX ) ( \kX-1 )}{\m} } + \alphaO \tX
  \middle|
    \LRV > ( 1+\sXXX ) \alphaO \m
  \right)
  \Pr( \LRV > ( 1+\sXXX ) \alphaO \m )
  \\
  &\dCmt{by the law of total probability and the definition of conditional probabilities}
  \\
  &\leq
  \Pr \left(
    \| \Vec{\URV} - \E[ \Vec{\URV} ] \|_{2} > \sqrt{\frac{\alphaO ( 1+\sXXX ) ( \kX-1 )}{\m} } + \alphaO \tX
  \middle|
    \LRV \leq ( 1+\sXXX ) \alphaO \m
  \right)
  +
  \Pr( \LRV > ( 1+\sXXX ) \alphaO \m )
  \\
  &\leq
  e^{-\frac{1}{2 ( 1+\sXXX )} \alphaO \m \tX^{2}}
  +
  e^{-\frac{1}{3} \alphaO \m \sXXX^{2}}
  \TagEqn{\label{eqn:pf:lemma:concentration-ineq:noisy:2:1}}
  .\\
  &\dCmt{by \EQUATIONS \eqref{eqn:pf:lemma:concentration-ineq:noisy:20} and \eqref{eqn:pf:lemma:concentration-ineq:noisy:21} and because \(  {\textstyle \frac{\alphaO^{2} \m^{2} \tX^{2}}{2 \alphaO \m ( 1+\sXXX )} = \frac{\alphaO \m \tX^{2}}{2 ( 1+\sXXX )}}  \)}
\end{align*}
For an arbitrary fixing of \(  \JCoordsX \in \JSX  \), the variables \(  \kX  \) and \(  \kOX  \) satisfy
\begin{gather*}
  \kX = | \Supp( \thetaStar ) \cup \JCoordsX | \leq \min \{ | \Supp( \thetaStar ) | + | \JCoordsX |, \n \} \leq \min \{ \k + \max_{\JCoordsX' \in \JSX} | \JCoordsX' |, \n \} = \kOX
.\end{gather*}
It follows from this and the above bound in \EQUATION \eqref{eqn:pf:lemma:concentration-ineq:noisy:2:1} that
\begin{align*}
  &\negphantom{\AlignIndent}
  \Pr \left(
    \left\| \frac{\gfFn[\JCoordsX]( \thetaStar, \thetaStar )}{\sqrt{2\pi}}  - \E \left[ \frac{\gfFn[\JCoordsX]( \thetaStar, \thetaStar )}{\sqrt{2\pi}}  \right]  \right\|_{2}
    >
    \sqrt{\frac{\alphaO ( 1+\sXXX ) ( \kOX-1 )}{\m} }
    +
    \alphaO \tX
  \right)
  \\
  &\leq
  \Pr \left(
    \left\| \frac{\gfFn[\JCoordsX]( \thetaStar, \thetaStar )}{\sqrt{2\pi}}  - \E \left[ \frac{\gfFn[\JCoordsX]( \thetaStar, \thetaStar )}{\sqrt{2\pi}}  \right]  \right\|_{2}
    >
    \sqrt{\frac{\alphaO ( 1+\sXXX ) ( \kX-1 )}{\m} }
    +
    \alphaO \tX
  \right)
  \\
  &\dCmt{due to the above remark that \(  \kX \leq \kOX  \)}
  \\
  &\leq
  e^{-\frac{1}{2 ( 1+\sXXX )} \alphaO \m \tX^{2}}
  +
  e^{-\frac{1}{3} \alphaO \m \sXXX^{2}}
  \TagEqn{\label{eqn:pf:lemma:concentration-ineq:noisy:2:2}}
  .\\
  &\dCmt{by \EQUATION \eqref{eqn:pf:lemma:concentration-ineq:noisy:2:1} and the definition of \(  \Vec{\URV}  \)}
\end{align*}
To obtain a uniform result over all \(  \JCoordsX \in \JSX  \), a union bound over \(  \JSX  \) can be applied to the probability corresponding to the first term on the \RHS of the above inequality, \eqref{eqn:pf:lemma:concentration-ineq:noisy:2:2}, yielding \EQUATION \eqref{eqn:lemma:concentration-ineq:noisy:pr:2}:
\begin{gather*}
  \Pr \left(
    \ExistsST{\JCoordsX \in \JSX}
    {\left\| \frac{\gfFn[\JCoordsX]( \thetaStar, \thetaStar )}{\sqrt{2\pi}}  - \E \left[ \frac{\gfFn[\JCoordsX]( \thetaStar, \thetaStar )}{\sqrt{2\pi}}  \right]  \right\|_{2}
    >
    \sqrt{\frac{\alphaO ( 1+\sXXX ) ( \kOX-1 )}{\m} }
    +
    \alphaO \tX}
  \right)
  \\
  \leq
  | \JSX | e^{-\frac{1}{2 ( 1+\sXXX )} \alphaO \m \tX^{2}}
  +
  e^{-\frac{1}{3} \alphaO \m \sXXX^{2}}
,\end{gather*}
as desired.
\end{proof}


\subsubsection{Proof of \LEMMA \ref{lemma:pf:lemma:concentration-ineq:noisy:f1,f2}}
\label{outline:concentration-ineq|pf-noisy|pf-f1,f2}

This section establishes the auxiliary result, \LEMMA \ref{lemma:pf:lemma:concentration-ineq:noisy:f1,f2}, stated and used in the proof of \LEMMA \ref{lemma:concentration-ineq:noisy}.

\begin{proof}
{\LEMMA \ref{lemma:pf:lemma:concentration-ineq:noisy:f1,f2}}
\checkoff%
Recall the definitions of the functions \(  \fFnX, \fFnXX : \R \to \R  \) from \EQUATIONS \eqref{eqn:pf:lemma:concentration-ineq:noisy:f1} and \eqref{eqn:pf:lemma:concentration-ineq:noisy:f2}, respectively:
\begin{gather*}
  \fFnX( \sX )
  \defeq
  \frac{1}{\sqrt{2\pi}}
  e^{-\sX \muX[1]}
  \int_{\zX=0}^{\zX=\infty}
  e^{-\frac{1}{2} (\zX-\sX)^{2}} (\pExpr{\zX})
  d\zX
  =
  \frac{1}{\sqrt{2\pi}}
  e^{-\sX \muX[1]}
  \int_{\zX=0}^{\zX=\infty}
  e^{-\frac{1}{2} (\zX-\sX)^{2}} \pExprFn( \zX )
  d\zX
  ,\\
  \fFnXX( \sX )
  \defeq
  \frac{1}{\sqrt{2\pi}}
  e^{\sX \muX[1]}
  \int_{\zX=0}^{\zX=\infty}
  e^{-\frac{1}{2} (\zX+\sX)^{2}} (\pExpr{\zX})
  d\zX
  =
  \frac{1}{\sqrt{2\pi}}
  e^{\sX \muX[1]}
  \int_{\zX=0}^{\zX=\infty}
  e^{-\frac{1}{2} (\zX+\sX)^{2}} \pExprFn( \zX )
  d\zX
,\end{gather*}
where
\(  \muX[1] = \frac{\sqrt{\hfrac{2}{\pi}}-\gammaX}{2\alphaX}  \)
and
\(  \pExprFn( \zX ) = \pExpr{\zX}  \), \(  \zX \in \R  \).
Due to \CONDITION \ref{condition:assumption:p:i} of \ASSUMPTION \ref{assumption:p}, the function \(  \pFn  \) is nondecreasing over the real line, which implies that \(  \pExprFn  \) is nonincreasing.
Additionally, by \CONDITION \ref{condition:assumption:p:ii} of \ASSUMPTION \ref{assumption:p}, \(  \pExprFn  \) satisfies
\begin{gather*}
  \frac{\pExprFn( \zX+\wX )}{\pExprFn( \zX )}
  \geq
  \frac{\pExprFn( \zXX+\wX )}{\pExprFn( \zXX )}
\end{gather*}
for
\(  \zX \leq \zXX \in [0,\infty)  \) and \(  \wX > 0  \).
%
\par 
%
Next, we will establish the lemma's result for \(  \fFnX( 0 )  \) and \(  \fFnXX( 0 )  \).
%
\paragraph{Verification of \EQUATIONS \eqref{eqn:pf:lemma:concentration-ineq:noisy:f1(0)} and \eqref{eqn:pf:lemma:concentration-ineq:noisy:f2(0)}}
%
\EQUATIONS \eqref{eqn:pf:lemma:concentration-ineq:noisy:f1(0)} and \eqref{eqn:pf:lemma:concentration-ineq:noisy:f2(0)} in the lemma are simple to verify:
\begin{align*}
  \fFnX( 0 )
  &=
  \frac{1}{\sqrt{2\pi}}
  e^{-0 \muX[1]}
  \int_{\zX=0}^{\zX=\infty}
  e^{-\frac{1}{2} (\zX-0)^{2}} (\pExpr{\zX})
  d\zX
  \\
  &=
  \frac{1}{\sqrt{2\pi}}
  \int_{\zX=0}^{\zX=\infty}
  e^{-\frac{1}{2} \zX^{2}} (\pExpr{\zX})
  d\zX
  \\
  &=
  \alphaX
,\end{align*}
and likewise,
\begin{align*}
  \fFnXX( 0 )
  &=
  \frac{1}{\sqrt{2\pi}}
  e^{0 \muX[1]}
  \int_{\zX=0}^{\zX=\infty}
  e^{-\frac{1}{2} (\zX+0)^{2}} (\pExpr{\zX})
  d\zX
  \\
  &=
  \frac{1}{\sqrt{2\pi}}
  \int_{\zX=0}^{\zX=\infty}
  e^{-\frac{1}{2} \zX^{2}} (\pExpr{\zX})
  d\zX
  \\
  &=
  \alphaX
,\end{align*}
where the last equality in each derivation follows directly from the definition of \(  \alphaX  \) in \EQUATION \eqref{eqn:notations:alpha:def}.
%
\paragraph{Verification of \EQUATION \eqref{eqn:pf:lemma:concentration-ineq:noisy:f1:ub}} 
%
Moving on to the upper bound on \(  \fFnX  \) in \EQUATION \eqref{eqn:pf:lemma:concentration-ineq:noisy:f1:ub}, it suffices to show that
\(  \frac{d}{d\sX} \fFnX( \sX ) \leq 0  \)
for all \(  \sX \geq 0  \) since this implies, by basic calculus, that
\(  \fFnX( \sX ) \leq \fFnX( 0 )  \)
over the interval \(  \sX \in [0,\infty)  \).
Observe:
\begin{gather*}
  \frac{d}{d\sX} \fFnX( \sX )
  =
  e^{-\sX \muX[1]}
  \frac{1}{\sqrt{2\pi}}
  \int_{\zX=0}^{\zX=\infty}
  ( \zX-\sX-\muX[1] )
  e^{-\frac{1}{2} ( \zX-\sX )^{2}}
  \pExprFn( \zX )
  d\zX
.\end{gather*}
When \(  \sX=0  \), the desired inequality,
\(  \frac{d}{d\sX} \fFnX( \sX ) \leq 0  \),
is true:
\begin{align*}
  \frac{d}{d\sX} \fFnX( 0 )
  &=
  \frac{1}{\sqrt{2\pi}}
  \int_{\zX=0}^{\zX=\infty}
  ( \zX-\muX[1] )
  e^{-\frac{1}{2} \zX^{2}}
  \pExprFn( \zX )
  d\zX
  \\
  &=
  \frac{1}{\sqrt{2\pi}}
  \int_{\zX=0}^{\zX=\infty}
  \zX
  e^{-\frac{1}{2} \zX^{2}}
  \pExprFn( \zX )
  d\zX
  -
  \muX[1]
  \frac{1}{\sqrt{2\pi}}
  \int_{\zX=0}^{\zX=\infty}
  e^{-\frac{1}{2} \zX^{2}}
  \pExprFn( \zX )
  d\zX
  \\
  &=
  \frac{\sqrt{\hfrac{2}{\pi}}-\gammaX}{2}
  -
  \frac{\sqrt{\hfrac{2}{\pi}}-\gammaX}{2\alphaX}
  \alphaX
  \\
  &\dCmt{by the definitions of \(  \alphaX, \gammaX  \) in \EQUATIONS \eqref{eqn:notations:alpha:def} and \eqref{eqn:notations:gamma:def}, respectively,}
  \\
  &\dCmtIndent\text{and an earlier remark in \EQUATION \eqref{eqn:pf:lemma:concentration-ineq:noisy:7} that \(  {\textstyle \muX[1] = \tfrac{\sqrt{\hfrac{2}{\pi}}-\gammaX}{2\alphaX}}  \)}
  \\
  &=
  0
.\end{align*}
%
\par 
%
On the other hand, the case when \(  \sX > 0  \) will require more work.
\renewcommand{\wX}{\sX}
Notice that
\begin{gather}
\label{eqn:pf:lemma:pf:lemma:concentration-ineq:noisy:f1,f2:1}
  \frac{d}{d\sX} \fFnX( \wX )
  =
  e^{-\sX \muX[1]}
  \frac{1}{\sqrt{2\pi}}
  \int_{\zX=0}^{\zX=\infty}
  ( \zX-\wX-\muX[1] )
  e^{-\frac{1}{2} ( \zX-\wX )^{2}}
  \pExprFn( \zX )
  d\zX
  < 0
\end{gather}
if and only if
\begin{gather}
\label{eqn:pf:lemma:pf:lemma:concentration-ineq:noisy:f1,f2:2}
  \int_{\zX=0}^{\zX=\infty}
  ( \zX-\wX-\muX[1] )
  e^{-\frac{1}{2} ( \zX-\wX )^{2}}
  \pExprFn( \zX )
  d\zX
  < 0
,\end{gather}
and similarly,
\begin{gather}
\label{eqn:pf:lemma:pf:lemma:concentration-ineq:noisy:f1,f2:3}
  \frac{d}{d\sX} \fFnX( 0 )
  =
  \frac{1}{\sqrt{2\pi}}
  \int_{\zX=0}^{\zX=\infty}
  ( \zX-\muX[1] )
  e^{-\frac{1}{2} \zX^{2}}
  \pExprFn( \zX )
  d\zX
  = 0
\end{gather}
if and only if
\begin{gather}
\label{eqn:pf:lemma:pf:lemma:concentration-ineq:noisy:f1,f2:4}
  \int_{\zX=0}^{\zX=\infty}
  ( \zX-\muX[1] )
  e^{-\frac{1}{2} \zX^{2}}
  \pExprFn( \zX )
  d\zX
  = 0
.\end{gather}
We already have that
\(  \frac{d}{d\sX} \fFnX( 0 ) = 0  \),
which implies by the above observation that \EQUATION \eqref{eqn:pf:lemma:pf:lemma:concentration-ineq:noisy:f1,f2:4} also holds.
%
\par 
%
The next argument focuses in on the former biconditional statement---in particular, the establishment of \EQUATION \eqref{eqn:pf:lemma:pf:lemma:concentration-ineq:noisy:f1,f2:2}.
To derive \EQUATION \eqref{eqn:pf:lemma:pf:lemma:concentration-ineq:noisy:f1,f2:2}, the interval of integration on its \LHS is partitioned into three intervals:
\begin{align*}
  &\negphantom{\AlignSp}
  \int_{\zX=0}^{\zX=\infty}
  ( \zX-\wX-\muX[1] )
  e^{-\frac{1}{2} ( \zX-\wX )^{2}}
  \pExprFn( \zX )
  d\zX
  \\
  &=
  \int_{\zX=0}^{\zX=\wX}
  ( \zX-\wX-\muX[1] )
  e^{-\frac{1}{2} ( \zX-\wX )^{2}}
  \pExprFn( \zX )
  d\zX
  +
  \int_{\zX=\wX}^{\zX=\wX+\muX[1]}
  ( \zX-\wX-\muX[1] )
  e^{-\frac{1}{2} ( \zX-\wX )^{2}}
  \pExprFn( \zX )
  d\zX
  \\
  &\AlignSp+
  \int_{\zX=\wX+\muX[1] }^{\zX=\infty}
  ( \zX-\wX-\muX[1] )
  e^{-\frac{1}{2} ( \zX-\wX )^{2}}
  \pExprFn( \zX )
  d\zX
  \\
  &=
  \int_{\zX=-\wX}^{\zX=0}
  ( \zX-\muX[1] )
  e^{-\frac{1}{2} \zX^{2}}
  \pExprFn( \zX+\wX )
  d\zX
  +
  \int_{\zX=0}^{\zX=\muX[1]}
  ( \zX-\muX[1] )
  e^{-\frac{1}{2} \zX^{2}}
  \pExprFn( \zX+\wX )
  d\zX
  \\
  &\AlignSp+
  \int_{\zX=\muX[1] }^{\zX=\infty}
  ( \zX-\muX[1] )
  e^{-\frac{1}{2} \zX^{2}}
  \pExprFn( \zX+\wX )
  d\zX
\TagEqn{\label{eqn:pf:lemma:pf:lemma:concentration-ineq:noisy:f1,f2:5}}
,\end{align*}
where the second equality applies a change of variables.
Clearly, the first of the three integrals in the last expression in \eqref{eqn:pf:lemma:pf:lemma:concentration-ineq:noisy:f1,f2:5} is negative when \(  \sX > 0  \):
\begin{gather}
\label{eqn:pf:lemma:pf:lemma:concentration-ineq:noisy:f1,f2:9}
  \int_{\zX=-\wX}^{\zX=0}
  ( \zX-\muX[1] )
  e^{-\frac{1}{2} \zX^{2}}
  \pExprFn( \zX+\wX )
  d\zX
  < 0
,\end{gather}
and thus, if the second and third integrals in the last expression in \eqref{eqn:pf:lemma:pf:lemma:concentration-ineq:noisy:f1,f2:5} sum to a nonpositive value, then \EQUATION \eqref{eqn:pf:lemma:pf:lemma:concentration-ineq:noisy:f1,f2:2}---and hence also \EQUATION \eqref{eqn:pf:lemma:pf:lemma:concentration-ineq:noisy:f1,f2:1}---will hold.
We will now show that this nonpositivity
indeed occurs.
Note the following property of an expression related to the integrand:
\begin{gather*}
  ( \zX-\muX[1] )
  e^{-\frac{1}{2} \zX^{2}}
  \pExprFn( \zX )
  \leq 0
  ,\quad
  \zX \in [0,\muX[1]]
  ,\\
  ( \zX-\muX[1] )
  e^{-\frac{1}{2} \zX^{2}}
  \pExprFn( \zX )
  \geq 0
  ,\quad
  \zX \in [\muX[1],\infty)
,\end{gather*}
which implies that
\begin{gather}
\label{eqn:pf:lemma:pf:lemma:concentration-ineq:noisy:f1,f2:10}
  - ( \zX-\muX[1] )
  e^{-\frac{1}{2} \zX^{2}}
  \pExprFn( \zX )
  =
  | ( \zX-\muX[1] )
  e^{-\frac{1}{2} \zX^{2}}
  \pExprFn( \zX ) |
  ,\quad
  \zX \in [0,\muX[1]]
  ,\\
\label{eqn:pf:lemma:pf:lemma:concentration-ineq:noisy:f1,f2:11}
  ( \zX-\muX[1] )
  e^{-\frac{1}{2} \zX^{2}}
  \pExprFn( \zX )
  =
  | ( \zX-\muX[1] )
  e^{-\frac{1}{2} \zX^{2}}
  \pExprFn( \zX ) |
  ,\quad
  \zX \in [\muX[1],\infty)
.\end{gather}
Then, for the second integral in \eqref{eqn:pf:lemma:pf:lemma:concentration-ineq:noisy:f1,f2:5}, observe:
\begin{align*}
  \int_{\zX=0}^{\zX=\muX[1]}
  ( \zX-\muX[1] )
  e^{-\frac{1}{2} \zX^{2}}
  \pExprFn( \zX+\wX )
  d\zX
  &=
  \int_{\zX=0}^{\zX=\muX[1]}
  | ( \zX-\muX[1] )
  e^{-\frac{1}{2} \zX^{2}}
  \pExprFn( \zX ) |
  \left( -\frac{\pExprFn( \zX+\wX )}{\pExprFn( \zX )} \right)
  d\zX
  \\
  &\dCmt{by \EQUATION \eqref{eqn:pf:lemma:pf:lemma:concentration-ineq:noisy:f1,f2:10}}
  \\
  &\leq
  \int_{\zX=0}^{\zX=\muX[1]}
  | ( \zX-\muX[1] ) )
  e^{-\frac{1}{2} \zX^{2}}
  \pExprFn( \zX ) |
  \left( -\frac{\pExprFn( \muX[1]+\wX )}{\pExprFn( \muX[1] )} \right)
  d\zX
  \\
  &\dCmt{by \CONDITION \ref{condition:assumption:p:ii} of \ASSUMPTION \ref{assumption:p},}
  \\
  &\dCmtIndent \text{and because \(  \zX \leq \muX[1]  \) for all \(  \zX \in [0, \muX[1]]  \)}
  \\
  &=
  \frac{\pExprFn( \muX[1]+\wX )}{\pExprFn( \muX[1] )}
  \int_{\zX=0}^{\zX=\muX[1]}
  ( \zX-\muX[1] )
  e^{-\frac{1}{2} \zX^{2}}
  \pExprFn( \zX )
  d\zX
\TagEqn{\label{eqn:pf:lemma:pf:lemma:concentration-ineq:noisy:f1,f2:6}}
  .\\
  &\dCmt{by \EQUATION \eqref{eqn:pf:lemma:pf:lemma:concentration-ineq:noisy:f1,f2:10}}
\end{align*}
Similarly, for the third integral from \eqref{eqn:pf:lemma:pf:lemma:concentration-ineq:noisy:f1,f2:5}, observe:
\begin{align*}
  \int_{\zX=\muX[1] }^{\zX=\infty}
  ( \zX-\muX[1] )
  e^{-\frac{1}{2} \zX^{2}}
  \pExprFn( \zX+\wX )
  d\zX
  &=
  \int_{\zX=\muX[1] }^{\zX=\infty}
  | ( \zX-\muX[1] )
  e^{-\frac{1}{2} \zX^{2}}
  \pExprFn( \zX ) |
  \frac{\pExprFn( \zX+\wX )}{\pExprFn( \zX )}
  d\zX
  \\
  &\dCmt{by \EQUATION \eqref{eqn:pf:lemma:pf:lemma:concentration-ineq:noisy:f1,f2:11}}
  \\
  &\leq
  \int_{\zX=\muX[1] }^{\zX=\infty}
  | ( \zX-\muX[1] )
  e^{-\frac{1}{2} \zX^{2}}
  \pExprFn( \zX ) |
  \frac{\pExprFn( \muX[1]+\wX )}{\pExprFn( \muX[1] )}
  d\zX
  \\
  &\dCmt{by \CONDITION \ref{condition:assumption:p:ii} of \ASSUMPTION \ref{assumption:p},}
  \\
  &\dCmtIndent \text{and because \(  \zX \geq \muX[1]  \) for all \(  \zX \in [\muX[1], \infty)  \)}
  \\
  &=
  \frac{\pExprFn( \muX[1]+\wX )}{\pExprFn( \muX[1] )}
  \int_{\zX=\muX[1] }^{\zX=\infty}
  ( \zX-\muX[1] )
  e^{-\frac{1}{2} \zX^{2}}
  \pExprFn( \zX )
  d\zX
\TagEqn{\label{eqn:pf:lemma:pf:lemma:concentration-ineq:noisy:f1,f2:7}}
  .\\
  &\dCmt{by \EQUATION \eqref{eqn:pf:lemma:pf:lemma:concentration-ineq:noisy:f1,f2:11}}
\end{align*}
Then, the sum of the two integrals is bounded from above as follows:
\begin{align*}
  &
  \int_{\zX=0}^{\zX=\muX[1]}
  ( \zX-\muX[1] )
  e^{-\frac{1}{2} \zX^{2}}
  \pExprFn( \zX+\wX )
  d\zX
  +
  \int_{\zX=\muX[1] }^{\zX=\infty}
  ( \zX-\muX[1] )
  e^{-\frac{1}{2} \zX^{2}}
  \pExprFn( \zX+\wX )
  d\zX
  \\
  &\AlignIndent\leq
  \frac{\pExprFn( \muX[1]+\wX )}{\pExprFn( \muX[1] )}
  \int_{\zX=0}^{\zX=\muX[1]}
  ( \zX-\muX[1] )
  e^{-\frac{1}{2} \zX^{2}}
  \pExprFn( \zX )
  d\zX
  +
  \frac{\pExprFn( \muX[1]+\wX )}{\pExprFn( \muX[1] )}
  \int_{\zX=\muX[1] }^{\zX=\infty}
  ( \zX-\muX[1] )
  e^{-\frac{1}{2} \zX^{2}}
  \pExprFn( \zX )
  d\zX
  \\
  &\AlignIndent\dCmt{by \EQUATIONS \eqref{eqn:pf:lemma:pf:lemma:concentration-ineq:noisy:f1,f2:6} and \eqref{eqn:pf:lemma:pf:lemma:concentration-ineq:noisy:f1,f2:7}}
  \\
  &\AlignIndent=
  \frac{\pExprFn( \muX[1]+\wX )}{\pExprFn( \muX[1] )}
  \left(
    \int_{\zX=0}^{\zX=\infty}
    \zX
    e^{-\frac{1}{2} \zX^{2}}
    \pExprFn( \zX )
    d\zX
    -
    \muX[1]
    \int_{\zX=0}^{\zX=\infty}
    e^{-\frac{1}{2} \zX^{2}}
    \pExprFn( \zX )
    d\zX
  \right)
  \\
  &\AlignIndent=
  \frac{\pExprFn( \muX[1]+\wX )}{\pExprFn( \muX[1] )}
  \left(
    \left( 1 - \sqrt{\frac{\pi}{2}} \gammaX \right)
    -
    \frac{\sqrt{\hfrac{2}{\pi}} - \gammaX}{2\alphaX}
    \sqrt{2\pi} \alphaX
  \right)
  \\
  &\AlignIndent\dCmt{by the definitions of \(  \alphaX, \gammaX  \) in \EQUATIONS \eqref{eqn:notations:alpha:def} and \eqref{eqn:notations:gamma:def}, respectively,}
  \\
  &\AlignIndent\dCmtIndent\text{and because \(  {\textstyle \muX[1] = \frac{\sqrt{\hfrac{2}{\pi}} - \gammaX}{2\alphaX}}  \) as in \EQUATION \eqref{eqn:pf:lemma:concentration-ineq:noisy:7}}
  \\
  &\AlignIndent=
  0
\TagEqn{\label{eqn:pf:lemma:pf:lemma:concentration-ineq:noisy:f1,f2:8}}
.\end{align*}
Substituting \EQUATIONS \eqref{eqn:pf:lemma:pf:lemma:concentration-ineq:noisy:f1,f2:9} and \eqref{eqn:pf:lemma:pf:lemma:concentration-ineq:noisy:f1,f2:8} into \EQUATION \eqref{eqn:pf:lemma:pf:lemma:concentration-ineq:noisy:f1,f2:5}, it follows that for \(  \sX > 0  \),
\begin{align*}
  &\negphantom{\AlignSp}
  \int_{\zX=0}^{\zX=\infty}
  ( \zX-\wX-\muX[1] )
  e^{-\frac{1}{2} ( \zX-\wX )^{2}}
  \pExprFn( \zX )
  d\zX
  \\
  &=
  \int_{\zX=-\wX}^{\zX=0}
  ( \zX-\muX[1] )
  e^{-\frac{1}{2} \zX^{2}}
  \pExprFn( \zX+\wX )
  d\zX
  +
  \int_{\zX=0}^{\zX=\muX[1]}
  ( \zX-\muX[1] )
  e^{-\frac{1}{2} \zX^{2}}
  \pExprFn( \zX+\wX )
  d\zX
  \\
  &\AlignSp+
  \int_{\zX=\muX[1] }^{\zX=\infty}
  ( \zX-\muX[1] )
  e^{-\frac{1}{2} \zX^{2}}
  \pExprFn( \zX+\wX )
  d\zX
  \\
  &\dCmt{by \EQUATION \eqref{eqn:pf:lemma:pf:lemma:concentration-ineq:noisy:f1,f2:5}}
  \\
  &<
  \int_{\zX=0}^{\zX=\muX[1]}
  ( \zX-\muX[1] )
  e^{-\frac{1}{2} \zX^{2}}
  \pExprFn( \zX+\wX )
  d\zX
  +
  \int_{\zX=\muX[1] }^{\zX=\infty}
  ( \zX-\muX[1] )
  e^{-\frac{1}{2} \zX^{2}}
  \pExprFn( \zX+\wX )
  \\
  &\dCmt{by \EQUATION \eqref{eqn:pf:lemma:pf:lemma:concentration-ineq:noisy:f1,f2:9}}
  \\
  &\leq
  0
  .\\
  &\dCmt{by \EQUATION \eqref{eqn:pf:lemma:pf:lemma:concentration-ineq:noisy:f1,f2:8}}
\end{align*}
In short, the above work has established that
\begin{align*}
  \int_{\zX=0}^{\zX=\infty}
  ( \zX-\wX-\muX[1] )
  e^{-\frac{1}{2} ( \zX-\wX )^{2}}
  \pExprFn( \zX )
  d\zX
  < 0
\end{align*}
when \(  \sX > 0  \), and that
\(  \frac{d}{d\sX} \fFnX( 0 ) = 0  \).
Therefore, by \EQUATIONS \eqref{eqn:pf:lemma:pf:lemma:concentration-ineq:noisy:f1,f2:1} and \eqref{eqn:pf:lemma:pf:lemma:concentration-ineq:noisy:f1,f2:2}, as well as the earlier discussion, it happens that
\(  \frac{d}{d\sX} \fFnX( \sX ) \leq 0  \)
for all
\(  \sX \geq 0  \).
By basic calculus, this implies that
\begin{gather*}
  \sup_{\sX \geq 0} \fFnX( \sX ) = \fFnX( 0 )
,\end{gather*}
verifying \EQUATION \eqref{eqn:pf:lemma:concentration-ineq:noisy:f1:ub}.
%
\paragraph{Verification of \EQUATION \eqref{eqn:pf:lemma:concentration-ineq:noisy:f2:ub}} 
%
\EQUATION \eqref{eqn:pf:lemma:concentration-ineq:noisy:f2:ub} can be derived through an analogous approach.
As such, most of the analysis to upper bound \(  \fFnXX  \) falls onto showing that
\(  \frac{d}{d\sX} \fFnXX( \sX ) \leq 0  \)
for all \(  \sX \geq 0  \), from which it will directly follow that
\(  \fFnXX( \sX ) \leq \fFnXX( 0 )  \)
for \(  \sX \geq 0  \).
The derivative of \(  \fFnXX  \) with respect to \(  \sX  \) is given by
\begin{gather*}
  \frac{d}{d\sX} \fFnXX( \sX )
  =
  e^{\sX \muX[1]}
  \frac{1}{\sqrt{2\pi}}
  \int_{\zX=0}^{\zX=\infty}
  ( \muX[1]-\zX-\sX )
  e^{-\frac{1}{2} ( \zX+\sX )^{2}}
  \pExprFn( \zX )
  d\zX
.\end{gather*}
At \(  \sX=0  \), this evaluates to
\begin{align*}
  \frac{d}{d\sX} \fFnXX( 0 )
  &=
  \frac{1}{\sqrt{2\pi}}
  \int_{\zX=0}^{\zX=\infty}
  ( \muX[1]-\zX )
  e^{-\frac{1}{2} \zX^{2}}
  \pExprFn( \zX )
  d\zX
  \\
  &=
  \muX[1]
  \frac{1}{\sqrt{2\pi}}
  \int_{\zX=0}^{\zX=\infty}
  e^{-\frac{1}{2} \zX^{2}}
  \pExprFn( \zX )
  d\zX
  -
  \frac{1}{\sqrt{2\pi}}
  \int_{\zX=0}^{\zX=\infty}
  \zX
  e^{-\frac{1}{2} \zX^{2}}
  \pExprFn( \zX )
  d\zX
  \\
  &=
  \frac{\sqrt{\hfrac{2}{\pi}} - \gammaX}{2\alphaX} \alphaX
  -
  \frac{\sqrt{\hfrac{2}{\pi}} - \gammaX}{2}
  \\
  &\dCmt{by the definitions of \(  \alphaX, \gammaX  \) in \EQUATIONS \eqref{eqn:notations:alpha:def} and \eqref{eqn:notations:gamma:def}, respectively,}
  \\
  &\dCmtIndent\text{and an earlier remark in \EQUATION \eqref{eqn:pf:lemma:concentration-ineq:noisy:7} that \(  {\textstyle \muX[1] = \frac{\sqrt{\hfrac{2}{\pi}} - \gammaX}{2\alphaX}}  \)}
  \\
  &=
  0
\TagEqn{\label{eqn:pf:lemma:pf:lemma:concentration-ineq:noisy:f1,f2:f2:7}}
,\end{align*}
which verifies the desired nonpositivity of \(  \frac{d}{d\sX} \fFnXX  \) in the case when \(  \sX=0  \), and which further implies that
\begin{gather}
\label{eqn:pf:lemma:pf:lemma:concentration-ineq:noisy:f1,f2:f2:1}
  \int_{\zX=0}^{\zX=\infty}
  ( \muX[1]-\zX )
  e^{-\frac{1}{2} \zX^{2}}
  \pExprFn( \zX )
  d\zX
  = 0
.\end{gather}
%
\par 
%
On the other hand, towards the case whee \(  \sX > 0  \),
note the following biconditional statement for \(  \wX > 0  \):
\begin{gather}
\label{eqn:pf:lemma:pf:lemma:concentration-ineq:noisy:f1,f2:f2:2}
  \frac{d}{d\sX} \fFnXX( \wX )
  =
  e^{\wX \muX[1]}
  \frac{1}{\sqrt{2\pi}}
  \int_{\zX=0}^{\zX=\infty}
  ( \muX[1]-\zX-\wX )
  e^{-\frac{1}{2} ( \zX+\wX )^{2}}
  \pExprFn( \zX )
  d\zX
  < 0
\end{gather}
if and only if
\begin{gather}
\label{eqn:pf:lemma:pf:lemma:concentration-ineq:noisy:f1,f2:f2:3}
  \int_{\zX=0}^{\zX=\infty}
  ( \muX[1]-\zX-\wX )
  e^{-\frac{1}{2} ( \zX+\wX )^{2}}
  \pExprFn( \zX )
  d\zX
  < 0
.\end{gather}
The next step is establishing the inequality in \eqref{eqn:pf:lemma:pf:lemma:concentration-ineq:noisy:f1,f2:f2:3} for \(  \sX > 0  \).
Throughout the upcoming analysis, take \(  \sX > 0  \) arbitrarily.
The interval of integration appearing on the \LHS of \EQUATION \eqref{eqn:pf:lemma:pf:lemma:concentration-ineq:noisy:f1,f2:f2:3} can be partitioned according to where the integrand takes positive verses nonpositive values:
\(  \zX \in [0, \muX[1]-\wX)  \) and \(  \zX \in [\muX[1]-\wX, \infty)  \),
respectively.
Hence, the integral in \eqref{eqn:pf:lemma:pf:lemma:concentration-ineq:noisy:f1,f2:f2:3} can be rewritten as:
\begin{align*}
  &\negphantom{\AlignSp}
  \int_{\zX=0}^{\zX=\infty}
  ( \muX[1]-\zX-\wX )
  e^{-\frac{1}{2} ( \zX+\wX )^{2}}
  \pExprFn( \zX )
  d\zX
  \\
  &=
  \int_{\zX=0}^{\zX=\muX[1]-\wX}
  ( \muX[1]-\zX-\wX )
  e^{-\frac{1}{2} ( \zX+\wX )^{2}}
  \pExprFn( \zX )
  d\zX
  +
  \int_{\zX=\muX[1]-\wX}^{\zX=\infty}
  ( \muX[1]-\zX-\wX )
  e^{-\frac{1}{2} ( \zX+\wX )^{2}}
  \pExprFn( \zX )
  d\zX
\TagEqn{\label{eqn:pf:lemma:pf:lemma:concentration-ineq:noisy:f1,f2:f2:4}}
.\end{align*}
The first of the two terms on the \RHS of \EQUATION \eqref{eqn:pf:lemma:pf:lemma:concentration-ineq:noisy:f1,f2:f2:4} is bounded from above by
\begin{align*}
  &
  \int_{\zX=0}^{\zX=\muX[1]-\wX}
  ( \muX[1]-\zX-\wX )
  e^{-\frac{1}{2} ( \zX+\wX )^{2}}
  \pExprFn( \zX )
  d\zX
  \\
  &\AlignSp<
  \int_{\zX=0}^{\zX=\muX[1]-\wX}
  ( \muX[1]-\zX )
  e^{-\frac{1}{2} ( \zX+\wX )^{2}}
  \pExprFn( \zX )
  d\zX
  \\
  &\AlignSp<
  \int_{\zX=0}^{\zX=\muX[1]}
  ( \muX[1]-\zX )
  e^{-\frac{1}{2} ( \zX+\wX )^{2}}
  \pExprFn( \zX )
  d\zX
  \\
  &\AlignSp\dCmt{the integrand is nonnegative on the interval \(  \zX \in [0,\muX[1]]  \)}
  \\
  &\AlignSp<
  \int_{\zX=0}^{\zX=\muX[1]}
  ( \muX[1]-\zX )
  e^{-\frac{1}{2} \zX^{2}}
  \pExprFn( \zX )
  d\zX
\TagEqn{\label{eqn:pf:lemma:pf:lemma:concentration-ineq:noisy:f1,f2:f2:5}}
,\end{align*}
while the second term on the \RHS of \eqref{eqn:pf:lemma:pf:lemma:concentration-ineq:noisy:f1,f2:f2:4} is upper bounded by
\begin{align*}
  &\negphantom{\AlignSp}
  \int_{\zX=\muX[1]-\wX}^{\zX=\infty}
  ( \muX[1]-\zX-\wX )
  e^{-\frac{1}{2} ( \zX+\wX )^{2}}
  \pExprFn( \zX )
  d\zX
  \\
  &=
  \int_{\zX=\muX[1]}^{\zX=\infty}
  ( \muX[1]-\zX )
  e^{-\frac{1}{2} \zX^{2}}
  \pExprFn( \zX-\wX )
  d\zX
  \\
  &=
  \int_{\zX=\muX[1]}^{\zX=\infty}
 |  ( \muX[1]-\zX )
  e^{-\frac{1}{2} \zX^{2}} |
  ( -\pExprFn( \zX-\wX ) )
  d\zX
  \\
  &\dCmt{since \(  {\textstyle -( \muX[1]-\zX ) e^{-\frac{1}{2} \zX^{2}} = | ( \muX[1]-\zX ) e^{-\frac{1}{2} \zX^{2}} |}  \) for \(  \zX \geq \muX[1]  \)}
  \\
  &\leq
  \int_{\zX=\muX[1]}^{\zX=\infty}
  | ( \muX[1]-\zX ) )
  e^{-\frac{1}{2} \zX^{2}} |
  ( -\pExprFn( \zX ) )
  d\zX
  \\
  &\dCmt{since \(  \pExprFn  \) is nonincreasing}
  \\
  &=
  \int_{\zX=\muX[1]}^{\zX=\infty}
  ( \muX[1]-\zX )
  e^{-\frac{1}{2} \zX^{2}}
  \pExprFn( \zX )
  d\zX
\TagEqn{\label{eqn:pf:lemma:pf:lemma:concentration-ineq:noisy:f1,f2:f2:6}}
.\end{align*}
Combining \EQUATIONS \eqref{eqn:pf:lemma:pf:lemma:concentration-ineq:noisy:f1,f2:f2:5} and \eqref{eqn:pf:lemma:pf:lemma:concentration-ineq:noisy:f1,f2:f2:6} into \EQUATION \eqref{eqn:pf:lemma:pf:lemma:concentration-ineq:noisy:f1,f2:f2:4} yields:
\begin{align*}
  &\negphantom{\AlignSp}
  \int_{\zX=0}^{\zX=\infty}
  ( \muX[1]-\zX-\wX )
  e^{-\frac{1}{2} ( \zX+\wX )^{2}}
  \pExprFn( \zX )
  d\zX
  \\
  &<
  \int_{\zX=0}^{\zX=\muX[1]}
  ( \muX[1]-\zX )
  e^{-\frac{1}{2} \zX^{2}}
  \pExprFn( \zX )
  d\zX
  +
  \int_{\zX=\muX[1]}^{\zX=\infty}
  ( \muX[1]-\zX )
  e^{-\frac{1}{2} \zX^{2}}
  \pExprFn( \zX )
  d\zX
  \\
  &\dCmt{by \EQUATIONS \eqref{eqn:pf:lemma:pf:lemma:concentration-ineq:noisy:f1,f2:f2:5} and \eqref{eqn:pf:lemma:pf:lemma:concentration-ineq:noisy:f1,f2:f2:6}}
  \\
  &=
  \int_{\zX=0}^{\zX=\infty}
  ( \muX[1]-\zX )
  e^{-\frac{1}{2} \zX^{2}}
  \pExprFn( \zX )
  d\zX
  \\
  &=
  0
  .\\
  &\dCmt{by \EQUATION \eqref{eqn:pf:lemma:pf:lemma:concentration-ineq:noisy:f1,f2:f2:1}}
\end{align*}
Thus, for every \(  \sX > 0  \), \EQUATION \eqref{eqn:pf:lemma:pf:lemma:concentration-ineq:noisy:f1,f2:f2:3} holds, implying that \EQUATION \eqref{eqn:pf:lemma:pf:lemma:concentration-ineq:noisy:f1,f2:f2:2} is also true due to the biconditional relationship between this pair of equations (which was stated earlier in the proof)---that is, it indeed happens that
\(  \frac{d}{d\sX} \fFnXX( \sX ) < 0  \)
for \(  \sX > 0  \).
Moreover, the derivation in \eqref{eqn:pf:lemma:pf:lemma:concentration-ineq:noisy:f1,f2:f2:7} showed that
\(  \frac{d}{d\sX} \fFnXX( 0 ) = 0  \).
It follows that
\(  \frac{d}{d\sX} \fFnXX( \sX ) \leq 0  \)
for all \(  \sX \geq 0  \), and therefore, due to standard facts about derivatives, \EQUATION \eqref{eqn:pf:lemma:concentration-ineq:noisy:f2:ub} holds:
\begin{gather*}
  \sup_{\sX \geq 0} \fFnXX( \sX ) = \fFnXX( 0 )
,\end{gather*}
concluding the proof of \LEMMA \ref{lemma:pf:lemma:concentration-ineq:noisy:f1,f2}.
\renewcommand{\wX}{w} 
\end{proof}



\end{document}